\documentclass{amsart}
\usepackage{amssymb}
\usepackage{amsfonts}

\newtheorem{theorem}{Theorem}
\theoremstyle{plain}

\newtheorem{axiom}{Axiom}

\newtheorem{corollary}{Corollary}

\newtheorem{definition}{Definition}

\newtheorem{lemma}{Lemma}
\newtheorem{notation}{Notation}

\newtheorem{proposition}{Proposition}
\newtheorem{remark}{Remark}

\numberwithin{equation}{section}
\input{tcilatex}

\begin{document}
\title[Flag Hardy spaces on $\mathbb{H}^{n}$]{Flag Hardy spaces and
Marcinkiewicz multipliers on the Heisenberg group: an expanded version}
\author{Yongsheng Han}
\address{Department of Mathematics\\
Auburn University\\
Auburn, AL 36849, U.S.A.}
\email{hanyong\TEXTsymbol{\backslash}@mail.auburn.edu}
\author{Guozhen Lu}
\address{Department of Mathematics\\
Wayne State University\\
Detroit, MI 48202, U.S.A.}
\thanks{Research supported in part by U.S. NSF Grant DMS0901761.}
\email{gzlu\TEXTsymbol{\backslash}@math.wayne.edu}
\author{Eric Sawyer}
\address{Department of Mathematics and Statistics\\
McMaster University\\
Hamilton, Ontario, Canada}
\thanks{Research supported in part by a grant from NSERC}
\email{sawyer\TEXTsymbol{\backslash}@mcmaster.ca}
\date{August 2, 2012}
\maketitle

\begin{abstract}
Marcinkiewicz multipliers are $L^{p}$ bounded for $1<p<\infty $ on the
Heisenberg group $\mathbb{H}^{n}\simeq \mathbb{C}^{n}\times \mathbb{R}$ (D.
Muller, F. Ricci and E. M. Stein \cite{MRS}, \cite{MRS2}). This is
surprising in that this class of multipliers is invariant under a two
parameter group of dilations on $\mathbb{C}^{n}\times \mathbb{R}$, while
there is \emph{no} two parameter group of \emph{automorphic} dilations on $%
\mathbb{H}^{n}$. This lack of automorphic dilations underlies the inability
of classical one or two parameter Hardy space theory to handle Marcinkiewicz
multipliers on $\mathbb{H}^{n}$ when $0<p\leq 1$.

We address this deficiency by developing a theory of \emph{flag} Hardy
spaces $H_{flag}^{p}$ on the Heisenberg group, $0<p\leq 1$, that is in a
sense `intermediate' between the classical Hardy spaces $H^{p}$ and the
product Hardy spaces $H_{product}^{p}$ on $\mathbb{C}^{n}\times \mathbb{R}$
(A. Chang and R. Fefferman (\cite{CF1}, \cite{CF2}, \cite{F1}, \cite{F2}, 
\cite{F3}). We show that flag singular integral operators, which include the
aforementioned Marcinkiewicz multipliers, are bounded on $H_{flag}^{p}$, as
well as from $H_{flag}^{p}$ to $L^{p}$, for $0<p\leq 1$. We also
characterize the dual spaces of $H_{flag}^{1}$ and $H_{flag}^{p}$, and
establish a Calder\'{o}n-Zygmund decomposition that yields standard
interpolation theorems for the flag Hardy spaces $H_{flag}^{p}$. In
particular, this recovers the $L^{p}$ results in \cite{MRS} (but not the
sharp results in \cite{MRS2}) by interpolating between those for $%
H_{flag}^{p}$ and $L^{2}$.
\end{abstract}

\tableofcontents

\section{Introduction}

Classical Calder\'{o}n-Zygmund theory centers around singular integrals
associated with the Hardy-Littlewood maximal operator $M$ that commutes with
the usual one-parameter family of dilations on $\mathbb{R}^{n}$, $\delta
\cdot x=(\delta x_{1},...,\delta x_{n})$ for $\delta >0$. On the other hand, 
\emph{product} Calder\'{o}n-Zygmund theory centers around singular integrals
associated with the \emph{strong} maximal function $M_{S}$ that commutes
with the multi-parameter dilations on $\mathbb{R}^{n}$, $\delta \cdot
x=(\delta _{1}x_{1},...,\delta _{n}x_{n})$ for $\delta =(\delta
_{1},...,\delta _{n})\in \mathbb{R}_{+}^{n}$. All of these dilations are
group automorphisms on $\mathbb{R}^{n}$. The strong maximal function (\cite%
{JMZ}) is given by 
\begin{equation}
M_{S}(f)(x)=\sup\limits_{x\in R}{\frac{{1}}{{|R|}}}\int\limits_{R}|f(y)|dy,
\label{def strong max}
\end{equation}%
where the supremum is taken over the family of all rectangles $R$ with sides
parallel to the axes.

For Calder\'{o}n-Zygmund theory in the product setting, one considers
operators of the form $Tf=K\ast f,$ where $K$ is homogeneous, that is, ${%
\delta _{1}}...{\delta _{n}}K(\delta \cdot x)=K(x),$ or, more generally, $%
K(x)$ satisfies the certain differential inequalities and cancellation
conditions such that the kernels ${\delta _{1}}...{\delta _{n}}K(\delta
\cdot x)$ also satisfy the same bounds. Such operators have been studied for
example in Gundy-Stein (\cite{GS}), R. Fefferman and Stein \cite{FS}, R.
Fefferman (\cite{F1}), Chang and R. Fefferman (\cite{CF1}, \cite{CF2}, \cite%
{CF3}), Journ\'{e} (\cite{J1}, \cite{J2}), Pipher \cite{P}, and others. More
recently, Hankel theorems and commutators have been treated in the product
setting by Ferguson and Lacey \cite{FeLa} and Lacey and Terwilliger \cite%
{LaTe}.

On the other hand, multi-parameter analysis has only recently been developed
for $L^{p}$ theory with $1<p<\infty $ when the underlying multi-parameter
structure is not explicit, but \emph{implicit}, as in the flag
multi-parameter structure studied in \cite{NRS} and its counterpart on the
Heisenberg group $\mathbb{H}^{n}$ studied in \cite{MRS} and \cite{MRS2}. In
these latter two papers the authors obtained the surprising result that
certain Marcinkiewicz multipliers, invariant under a two-parameter group of
dilations on $\mathbb{C}^{n}\times \mathbb{R}$, are bounded on $L^{p}\left( 
\mathbb{H}^{n}\right) $, \emph{despite} the absence of a two-parameter
automorphic group of dilations on $\mathbb{H}^{n}$. This striking result
exploited an implicit product, or \emph{semiproduct,} structure underlying
the group multiplication in $\mathbb{H}^{n}\simeq \mathbb{C}^{n}\times 
\mathbb{R}$. In contrast to this, it is not hard to see that the class of
flag singular integrals considered there is \emph{not} in general bounded on
the standard one-parameter Hardy space $H^{1}\left( \mathbb{H}^{n}\right) $
(see e.g. Theorem \ref{endpoint fails} in Section 11 below). The lesson
learned here is that Hardy space theories for $0<p\leq 1$ must be tailored
to the invariance properties of the class of singular integral operators
under consideration.

The goal of this paper is to develop for the Heisenberg group a theory of 
\emph{flag} Hardy spaces $H_{flag}^{p}$, with $0<p\leq 1$, that serve to
provide a suitable endpoint theory for Marcinkiewicz multipliers on the
Heisenberg group. This answers in part a question posed to one of us by E.M.
Stein in 1998. We remark that the first two authors have treated a \emph{%
Euclidean} flag structure in \textit{\cite{HL1},} a precursor of this paper
that is not intended for publication.

The flag theory for the Heisenberg group is most conveniently explained when 
$p=1$ in the more general context of spaces $\left( X,\rho ,d\mu \right) $
of homogeneous type \cite{CW}, which already include Euclidean spaces $%
\mathbb{R}^{N}$ and stratified graded nilpotent Lie groups such as the
Heisenberg groups $\mathbb{H}^{n}=\mathbb{C}^{n}\times \mathbb{R}$. We may
assume here that $\rho $ and $d\mu $ are connected by the equivalence%
\begin{equation}
\mu \left( B_{\rho }\left( x,r\right) \right) \approx r\text{ where }B_{\rho
}\left( x,r\right) =\left\{ y\in X:\rho \left( x,y\right) <r\right\} .
\label{special rho}
\end{equation}%
In particular, the usual structure on Euclidean space $\mathbb{R}^{n}$ is
given by $\rho \left( x,y\right) =\left\vert x-y\right\vert ^{n}$ and $d\mu
\left( x\right) =dx$.

Recall that one of several equivalent definitions of the Hardy space $%
H^{1}\left( X\right) $ is given as the set of $f\in \left( C^{\eta }\left(
X\right) \right) ^{\ast }$ with%
\begin{equation*}
\left\Vert f\right\Vert _{H^{1}\left( X\right) }\equiv \left\Vert g\left(
f\right) \right\Vert _{L^{1}\left( d\mu \right) }<\infty ,
\end{equation*}%
where the Littlewood-Paley $g$-function $g\left( f\right) $ is given by%
\begin{equation*}
g\left( f\right) =\left\{ \sum\limits_{j=-\infty }^{\infty }\left\vert
E_{j}f\right\vert ^{2}\right\} ^{{\frac{{1}}{{2}}}},
\end{equation*}%
and where $\left\{ E_{j}\right\} _{j=-\infty }^{\infty }$ is an appropriate
Littlewood-Paley decomposition of the identity on $L^{2}\left( d\mu \right) $%
.

The \emph{product} Hardy space $H_{product}^{1}\left( X\times X^{\prime
}\right) $ corresponding to a product of homogeneous spaces $\left( X,\rho
,d\mu \right) $ and $\left( X^{\prime },\rho ^{\prime },d\mu ^{\prime
}\right) $ is given as the set of $f\in \left( C^{\eta }\left( X\times
X^{\prime }\right) \right) ^{\ast }$ with%
\begin{equation*}
\left\Vert f\right\Vert _{H_{product}^{1}\left( X\times X^{\prime }\right)
}\equiv \left\Vert g_{product}\left( f\right) \right\Vert _{L^{1}\left( d\mu
\times d\mu ^{\prime }\right) }<\infty ,
\end{equation*}%
where the \emph{product} Littlewood-Paley $g$-function $g_{product}\left(
f\right) $ is given by%
\begin{equation*}
g_{product}\left( f\right) =\left\{ \sum\limits_{j,j^{\prime }=-\infty
}^{\infty }\left\vert D_{j}D_{j^{\prime }}^{\prime }f\right\vert
^{2}\right\} ^{{\frac{{1}}{{2}}}},
\end{equation*}%
and where $\left\{ D_{j}\right\} _{j=-\infty }^{\infty }$ and $\left\{
D_{j^{\prime }}^{\prime }\right\} _{j^{\prime }=-\infty }^{\infty }$ are
Littlewood-Paley decompositions of the identities on $L^{2}\left( d\mu
\right) $\ and $L^{2}\left( d\mu ^{\prime }\right) $ respectively (and act
separately on the respective distinct variables). Note that if $j=j^{\prime
} $ then $D_{j}D_{j^{\prime }}^{\prime }=D_{j}D_{j}^{\prime }$ satisfies
estimates similar to those for $E_{j}$ in the standard \emph{one-parameter}
Hardy space $H^{1}\left( X\times X^{\prime }\right) $. Thus we see that%
\begin{eqnarray*}
g_{product}\left( f\right) &=&\left\{ \sum\limits_{j,j^{\prime }=-\infty
}^{\infty }\left\vert D_{j}D_{j^{\prime }}^{\prime }f\right\vert
^{2}\right\} ^{{\frac{{1}}{{2}}}}\geq \left\{ \sum\limits_{j}^{\infty
}\left\vert D_{j}D_{j}^{\prime }f\right\vert ^{2}\right\} ^{{\frac{{1}}{{2}}}%
} \\
&\approx &\left\{ \sum\limits_{j}^{\infty }\left\vert E_{j}f\right\vert
^{2}\right\} ^{{\frac{{1}}{{2}}}}=g\left( f\right) ,
\end{eqnarray*}%
and so we have the inclusion%
\begin{equation*}
H_{product}^{1}\left( X\times X^{\prime }\right) \subset H^{1}\left( X\times
X^{\prime }\right) .
\end{equation*}

Now we specialize the space of homogeneous type $X$ to be the Heisenberg
group $\mathbb{H}^{n}=\mathbb{C}^{n}\times \mathbb{R}$. The \emph{flag}
structure on the Heisenberg group $\mathbb{H}^{n}$ arises in an intermediate
manner, namely as a homogeneous space structure derived from the Heisenberg
multiplication law that is adapted to the product of the homogeneous spaces $%
\mathbb{C}^{m}$ and $\mathbb{R}$. The appropriate definition of the flag
Hardy space $H_{flag}^{1}\left( \mathbb{H}^{n}\right) $ is already suggested
in \cite{MRS2}, where a Littlewood-Paley $g$-function $g_{flag}$ is
introduced that is adapted to the flag structure on the Heisenberg group $%
\mathbb{H}^{n}=\mathbb{C}^{n}\times \mathbb{R}$:%
\begin{equation*}
g_{flag}\left( f\right) =\left\{ \sum\limits_{j,k=-\infty }^{\infty
}\left\vert E_{k}D_{j}f\right\vert ^{2}\right\} ^{{\frac{{1}}{{2}}}},
\end{equation*}%
where $\left\{ D_{j}\right\} _{j=-\infty }^{\infty }$ is the standard
Littlewood-Paley decomposition of the identity on $L^{2}\left( \mathbb{H}%
^{n}\right) $, and $\left\{ E_{k}\right\} _{k=-\infty }^{\infty }$ is the
standard Littlewood-Paley decomposition of the identity on $L^{2}\left( 
\mathbb{R}\right) $. One can then define $H_{flag}^{1}(\mathbb{H}^{n})$ to
consist of appropriate `distributions' $f$ on $\mathbb{H}^{n}$ with%
\begin{equation*}
\left\Vert f\right\Vert _{H_{flag}^{1}\left( \mathbb{H}^{n}\right) }\equiv
\left\Vert g_{flag}\left( f\right) \right\Vert _{L^{1}\left( \mathbb{H}%
^{n}\right) }<\infty .
\end{equation*}%
Now for $k\leq 2j$, it turns out that $E_{k}D_{j}$ is essentially the \emph{%
one-parameter} Littlewood-Paley function $D_{j}$; while for $k>2j$ it turns
out that $E_{k}D_{j}$ is essentially the \emph{product} Littlewood-Paley
function $E_{k}F_{j}$ where $\left\{ F_{j}\right\} _{j=-\infty }^{\infty }$
is the standard Littlewood-Paley decomposition of the identity on $%
L^{2}\left( \mathbb{C}^{n}\right) $. Thus we see that $g_{flag}\left(
f\right) $ is a \emph{semiproduct} Littlewood-Paley function satisfying%
\begin{eqnarray*}
g_{product}\left( f\right) &\lessapprox &g_{flag}\left( f\right) \lessapprox
g\left( f\right) , \\
H_{product}^{1}\left( X\times X^{\prime }\right) &\subset
&H_{flag}^{1}\left( X\times X^{\prime }\right) \subset H^{1}\left( X\times
X^{\prime }\right) .
\end{eqnarray*}%
We describe this structure as `semiproduct' since only \emph{vertical}
Heisenberg rectangles (which are essentially unions of contiguous Heisenberg
balls of fixed radius stacked one on top of the other) arise as essentially
the supports of the components $E_{k}D_{j}$ when $k>2j$. When $k\leq 2j$,
the support of $E_{k}D_{j}$ is essentially a Heisenberg cube. Thus no \emph{%
horizontal} rectangles arise and the structure is `semiproduct'.

Of course, we must also address the nature of the `distributions' referred
to above, and for this we will use a lifting technique introduced in \cite%
{MRS} to define \emph{projected} flag molecular spaces $\mathcal{M}%
_{flag}\left( \mathbb{H}^{n}\right) $, and then the aforementioned
distributions will be elements of the dual space $\mathcal{M}_{flag}\left( 
\mathcal{H}^{n}\right) ^{\prime }$. We also show that these distributions
are `essentially' the same as those obtained from the dual of a more
familiar \emph{moment} flag molecular space $\mathcal{M}_{F}\left( \mathcal{H%
}^{n}\right) $. Finally, we mention that a theory of flag Hardy spaces can
also be developed with the techniques used here, but without recourse to any
notion of `distributions', by simply defining $H_{abstract}^{p}\left( 
\mathbb{H}^{n}\right) $ to be the abstract completion of the metric space 
\begin{equation*}
X^{p}\left( \mathbb{H}^{n}\right) \equiv \left\{ f\in L^{2}\left( \mathbb{H}%
^{n}\right) :g_{flag}\left( f\right) \in L^{p}\left( \mathbb{H}^{n}\right)
\right\}
\end{equation*}%
with metric%
\begin{equation*}
d\left( f_{1},f_{2}\right) \equiv \left\Vert g_{flag}\left(
f_{1}-f_{2}\right) \right\Vert _{L^{p}\left( \mathbb{H}^{n}\right) }^{p},\ \
\ \ \ f_{j}\in X^{p}\left( \mathbb{H}^{n}\right) .
\end{equation*}%
We show that the abstract space $H_{abstract}^{p}\left( \mathbb{H}%
^{n}\right) $, whose elements are realized only as equivalence classes of
Cauchy sequences, is in fact isomorphic to the space $H_{flag}^{p}\left( 
\mathbb{H}^{n}\right) $, whose elements have the advantage of being realized
as a subspace of distributions, namely those $f$ in $\mathcal{M}%
_{flag}\left( \mathbb{H}^{n}\right) ^{\prime }$ whose flag Littlewood-Paley
function $g_{flag}\left( f\right) $ belongs to $L^{p}\left( \mathbb{H}%
^{n}\right) $.

In Part 1 of the paper we define flag Hardy spaces and state our results. In
Part 2 we give the proofs, and in Part 3 we begin a modest extension of the
theory to spaces of homogeneous type. In particular, we construct there a
dyadic grid adapted to the flag structure, which is essential for our
treatment of the Heisenberg group.

\part{Flag Hardy spaces: definitions and results}

Our point of departure is to develop a \emph{wavelet} Calder\'{o}n
reproducing formula associated with the given two-parameter structure as in 
\cite{MRS2}, and then to prove a Plancherel-P\^{o}lya type inequality in
this setting. This will provide the flexibility needed to define flag Hardy
spaces and prove boundedness of flag singular integrals, duality and
interpolation theorems for these spaces. To explain the novelty in this
approach more carefully, we point out the following three types of
reproducing formula derived from the original idea of Calder\'{o}n:%
\begin{eqnarray*}
f\left( x\right) &=&\int_{0}^{\infty }\psi _{t}\ast \psi _{t}\ast f\left(
x\right) \frac{dt}{t}, \\
f\left( x\right) &=&\sum_{j\in \mathbb{Z}}\widetilde{D_{j}}D_{j}f\left(
x\right) , \\
f(x) &=&\sum\limits_{j}\sum\limits_{I}\left\{ \left\vert I\right\vert \left(
\psi _{j}\ast f\right) \left( x_{I}\right) \right\} {\widetilde{\psi }}%
_{j}(x,x_{I}).
\end{eqnarray*}

We refer to the formula in the first line as a\emph{\ continuous} Calder\'{o}%
n reproducing formula, its advantage being the use of \emph{compactly}
supported components $\psi _{t}$ that are \emph{repeated}. We refer to the
second formula as a \emph{discrete} Calder\'{o}n reproducing formula, in
which $D_{j}$ is generally a compactly supported $\emph{nonconvolution}$
operator in a space of homogeneous type, and $\widetilde{D_{j}}$ is no
longer compactly supported but satisfies \emph{molecular} estimates. In
certain cases, such as in Euclidean space, it is possible to use the Fourier
transform to obtain a discrete decomposition with repeated convolution
operators $D_{j}=\psi _{j}$.

Finally, we refer to the third formula as a \emph{wavelet} Calder\'{o}n
reproducing formula, which can also be developed in a space of homogeneous
type. For example, such formulas were first developed by Frazier and Jawerth
in certain situations in \cite{FJ}. The advantage of the third formula is
that it expresses $f$ as a sum of molecular, or wavelet-like, functions ${%
\widetilde{\psi }}_{j}\left( x,x_{I}\right) $ with coefficients $\left\vert
I\right\vert \left( \psi _{j}\ast f\right) \left( x_{I}\right) $ that are
obtained by evaluating $\psi _{j}\ast f$ at \emph{any} convenient point in
the set $I$ from a dyadic decomposition at scale $2^{j}$ of the space. As a
consequence, we can replace the coefficient $\left\vert I\right\vert \left(
\psi _{j}\ast f\right) \left( x_{I}\right) $ with either the supremum or
infimum of such choices and retain appropriate estimates (see Theorem \ref%
{P-P} below). We note in passing that the collection of functions $\left\{ {%
\widetilde{\psi }}_{j}\left( x,x_{I}\right) \right\} _{j,I}$ form a \emph{%
Riesz} basis for $L^{2}$. In certain cases when such functions form an \emph{%
orthogonal} basis, the decomposition is referred to as a \emph{wavelet}
decomposition, and it is from this that we borrow our terminology.

This `wavelet' scheme is particularly useful in dealing with the Hardy
spaces $H^{p}$ for $0<p\leq 1$, and using this we will show that flag
singular integral operators are bounded on $H_{flag}^{p}$ for all $0<p\leq 1$%
, and furthermore show that these operators are bounded from $H_{flag}^{p}$
to $L^{p}$ for all $0<p\leq 1$. These ideas can also be applied in the pure
product setting to provide a different approach to proving $H_{product}^{p}$
to $L^{p}$ boundedness than that used by R. Fefferman, albeit requiring more
smoothness, and thus bypass both the action of singular integral operators
on rectangle atoms, and the use of Journ\'{e}'s covering lemma.

\bigskip

We now recall the example of implicit multiparameter structure that provides
the main motivation for this paper. In \cite{MRS}, D. Muller, F. Ricci and
E. M. Stein uncovered a new class of flag singular integrals on
Heisenberg(-type) groups that arose in the investigation of Marcinkiewicz
multipliers. To be more precise, let $m(\mathcal{L},iT)$ be the
Marcinkiewicz multiplier operator, where $\mathcal{L}$ is the sub-Laplacian, 
$T$ is the central element of the Lie algebra on the Heisenberg group $%
\mathbb{H}^{n}=\mathbb{C}^{n}\times \mathbb{R},$ and $m$ satisfies
Marcinkiewicz conditions as in \cite{MRS}. It was proved in \cite{MRS} that
the kernel of $m(\mathcal{L},iT)$ satisfies standard one-parameter Calder%
\'{o}n-Zygmund type estimates associated with automorphic dilations in the
region where $\left\vert t\right\vert <\left\vert z\right\vert ^{2},$ and
two-parameter Calder\'{o}n-Zygmund type estimates in the region where $%
\left\vert t\right\vert \geq \left\vert z\right\vert ^{2}$.

The proof of $L^{p}$ boundedness of $m(\mathcal{L},iT)$ given in \cite{MRS}
requires lifting the operator to a larger group, $\mathbb{H}^{n}\times 
\mathbb{R}$. This lifts $K$, the kernel of $m(\mathcal{L},iT)$ on $\mathbb{H}%
^{n},$ to a \emph{product} kernel $\widetilde{K}$ on $\mathbb{H}^{n}\times 
\mathbb{R}.$ The lifted kernel $\widetilde{K}$ is constructed so that it
projects to $K$ by 
\begin{equation*}
K(z,t)=\int_{-\infty }^{\infty }{\widetilde{K}}(z,t-u,u)du
\end{equation*}%
taken in the sense of distributions. The operator $\widetilde{{T}}$
corresponding to the product kernel $\widetilde{{K}}$ can be dealt with in
terms of tensor products of operators, and one can obtain their $L^{p}$
boundedness from the known pure product theory. Finally, the $L^{p}$
boundedness of the operator with kernel $K$ follows from the transference
method of Coifman and Weiss (\cite{CW}), using the projection $\pi :\mathbb{H%
}^{n}\times \mathbb{R}\rightarrow \mathbb{H}^{n}$ by $\pi ((z,t),u)=(z,t+u)$%
. One of our main results, Corollary \pageref{Marcin} below, is an extension
of the boundedness of $m(\mathcal{L},iT)$ to flag Hardy spaces $H_{flag}^{p}$
for all $0<p\leq 1$, and follows from the boundedness of flag singular
integrals on $H_{flag}^{p}$.

In a subsequent paper \cite{MRS2}, D. Muller, F. Ricci and E.M. Stein
obtained the same boundedness results but with optimal regularity on the
multipliers. This required working directly on the group without lifting to
a product, and led to the introduction of a \emph{continuous} flag
Littlewood-Paley $g$-function and a corresponding \emph{continuous} Calder%
\'{o}n reproducing formula. We remark that one of the main features of our
extension of these results to $H^{p}$ for $0<p\leq 1$ is the construction of
a \emph{wavelet} Calder\'{o}n reproducing formula.

We note that the regularity satisfied by flag singular kernels is better
than that of the product singular kernels. More precisely, the singularity
of the standard pure product kernel on $\mathbb{C}^{n}\times \mathbb{R}$ is
contained in the union $\{(z,0)\}\cup \{(0,u)\}$ of two subspaces, while the
singularity of $K(z,u)$, the flag singular kernel on $\mathbb{H}^{n}=\mathbb{%
C}^{n}\times \mathbb{R}$ defined by (\ref{defflag}) below, is contained in a
single subspace $\{(0,y)\}$, but is more singular on yet a smaller subspace $%
\{(0,0)\}$, a situation described neatly in terms of the flag (or
filtration) of subspaces, $\left\{ \left( 0,0\right) \right\} \subsetneqq
\left\{ \left( 0,u\right) \right\} \subsetneqq \mathbb{H}^{n}$. Some natural
questions that arise now are these.

\bigskip

\textbf{Question 1:} What is the correct definition of a flag Hardy space $%
H_{flag}^{p}$ on the Heisenberg group for $0<p\leq 1$ in order that both (%
\textbf{1}) Maricinkiewicz multipliers are bounded, and (\textbf{2}) the $%
L^{p}$ results are recovered by interpolation?

\textbf{Question 2:} What is the correct definition of spaces $BMO_{flag}$
of bounded mean oscillation, and more generally, what is the duality theory
for $H_{flag}^{p}$?

\textbf{Question 3:} What is the relationship between classical Hardy spaces 
$H^{p}$ and the flag Hardy spaces $H_{flag}^{p}$?

\bigskip

We address these questions as follows. As in the $L^{p}$ theory for $p>1$
considered in \cite{MRS}, one is naturally tempted to establish Hardy space
theory under the implicit two-parameter structure associated with flag
singular integrals by invoking the method of lifting to the pure product
setting together with the transference method in [CW]. However, this direct
lifting method is not readily adaptable to the case of $p<1$ because the
transference method is not known to be valid . A different approach
centering on the use of a continuous flag Littlewood-Paley $g$-function was
carried out in \cite{MRS2}. This suggests that the flag Hardy space $%
H_{flag}^{p}$ for $0<p\leq 1$ should be defined in terms of this or a
similar $g$-function. Crucial for this is the use of a space of test
functions arising from the lifting technique in \cite{MRS}, and a \emph{%
wavelet} Calder\'{o}n reproducing formula adapted to these test functions.
Here is the order in which we implement these ideas.

\bigskip

(1) We first use the $L^{p}$ theory of Littlewood-Paley square functions $%
g_{flag}$ as in \cite{MRS2} to develop a Plancherel-Polya type inequality.

(2) We next define the flag Hardy spaces $H_{flag}^{p}$ using the flag $g$%
-function $g_{flag}$ together with a space of test functions that is
motivated by the lifting technique in \cite{MRS}. We then develop the theory
of Hardy spaces $H_{flag}^{p}$ associated to the two-parameter flag
structures and the boundedness of flag singular integrals on these spaces.
We also establish the boundedness of flag singular integrals from $%
H_{flag}^{p}$ to $L^{p}$.

(3) We then turn to duality theory for the flag Hardy space $H_{flag}^{p}$
and introduce the dual space $CMO_{_{flag}}^{p}$. In particular we establish
the duality between $H_{_{flag}}^{1}$ and the space $BMO_{_{flag}}$. We then
establish the boundedness of flag singular integrals on $BMO_{_{flag}}$. It
is worthwhile to point out that in the classical one-parameter or pure
product case, $BMO$ is related to the concept of Carleson measure. The space 
$CMO_{_{flag}}^{p}$ for all $0<p\leq 1,$ as the dual space of $H_{flag}^{p}$
introduced in this paper, is then defined by a generalized Carleson measure
condition.

(4) We then establish a Calder\'{o}n-Zygmund decomposition lemma for any $%
H_{flag}^{p}$ function ($0<p<\infty $) in terms of functions in $%
H_{_{flag}}^{p_{1}}$ and $H_{_{flag}}^{p_{2}}$ with $0<p_{1}<p<p_{2}<\infty $%
. This gives rise to an interpolation theorem between $H_{_{flag}}^{p_{1}}$
and $H_{_{flag}}^{p_{2}}$ for any $0<p_{2}<p_{1}<\infty $ ($%
H_{_{flag}}^{p}=L^{p}$ for $1<p<\infty $).

(5) Finally, we construct an example of a Marcinkiewicz multiplier that
fails to be bounded on the classical one-parameter Hardy space $H^{1}\left( 
\mathbb{H}^{n}\right) $.

\bigskip

We devote the remainder of Part 1 to a detailed description of our approach
and results. For the most part we deal exclusively with the Heisenberg group 
$\mathbb{H}^{n}=\mathbb{C}^{n}\times \mathbb{R}$, but in constructing a
dyadic grid on $\mathbb{H}^{n}$, it is convenient to also consider the more
general case of special products of spaces of homogeneous type $\left(
X,\rho ,\mu \right) $. Proofs will be given in Parts 2 and 3 of the paper.

\section{The square function on the Heisenberg group}

We begin with an \emph{implicit} two-parameter continuous variant of the
Littlewood-Paley square function that is introduced in \cite{MRS2}. For this
we need the standard Calder\'{o}n reproducing formula on the Heisenberg
group. Note that spectral theory was used in place of the Calder\'{o}n
reproducing formula in \cite{MRS2}.

\begin{theorem}
\label{GM}(Corollary 1 of \cite{GM}) There is $\psi \in C^{\infty }\left( 
\mathbb{H}^{n}\right) $ satisfying 
\begin{eqnarray*}
\text{\textbf{either} }\psi &\in &\mathcal{S}\left( \mathbb{H}^{n}\right) 
\text{ and all moments of }\psi \text{ vanish,} \\
\text{\textbf{or} }\psi &\in &C_{c}^{\infty }\left( \mathbb{H}^{n}\right) 
\text{ and all arbitrarily large moments of }\psi \text{ vanish,}
\end{eqnarray*}%
such that the following Calder\'{o}n reproducing formula holds:%
\begin{equation*}
f=\int_{0}^{\infty }\psi _{s}^{\vee }\ast \psi _{s}\ast f\frac{ds}{s},\ \ \
\ \ f\in L^{2}\left( \mathbb{H}^{n}\right) ,
\end{equation*}%
where $\ast $ is Heisenberg convolution, $\psi ^{\vee }\left( \zeta \right) =%
\overline{\psi \left( \zeta ^{-1}\right) }$ and $\psi _{s}\left( z,t\right)
=s^{-2n-2}\psi \left( \frac{z}{s},\frac{u}{s^{2}}\right) $ for $s>0$.
\end{theorem}

\begin{remark}
We will usually assume that $\psi $ above has compact support. However, it
will sometimes be convenient for us if the component functions $\psi _{s}$
have \emph{infinitely} many vanishing moments. In particular we can then use
the \emph{same} component functions to define the flag Hardy spaces for 
\emph{all} $0<p<\infty $ (the smaller $p$ is the more vanishing moments are
required to obtain necessary decay of singular integrals). Thus we will
sometimes sacrifice the property of having compactly supported component
functions.
\end{remark}

We now wish to extend this formula to encompass the flag structure on the
Heisenberg group $\mathbb{H}^{n}$.

\subsection{The component functions\label{component}}

Following \cite{MRS2}, we construct a Littlewood-Paley \emph{component}
function $\psi $ defined on $\mathbb{H}^{n}\simeq \mathbb{C}^{n}\times 
\mathbb{R}$, given by the partial convolution $\ast _{2}$ in the second
variable only: 
\begin{equation*}
\psi (z,u)=\psi ^{(1)}\ast _{2}\psi ^{(2)}(z,u)=\int_{\mathbb{R}}\psi
^{(1)}(z,u-v)\psi ^{(2)}(v)dv,\ \ \ \ \ \left( z,u\right) \in \mathbb{C}%
^{n}\times \mathbb{R},
\end{equation*}%
where $\psi ^{(1)}\in \mathcal{S}(\mathbb{H}^{n}\mathbb{)}$ is as in Theorem %
\ref{GM}, and $\psi ^{\left( 2\right) }\in \mathcal{S}\left( \mathbb{R}%
\right) $ satisfies 
\begin{equation*}
\int_{0}^{\infty }|\widehat{{\psi }^{(2)}}(t\eta )|^{2}\frac{dt}{t}=1
\end{equation*}%
for all $\eta \in \mathbb{R}\backslash \{0\}$, along with the moment
conditions 
\begin{eqnarray*}
\int\limits_{\mathbb{H}^{n}}z^{\alpha }u^{\beta }\psi ^{(1)}(z,u)dzdu &=&0,\
\ \ \ \ \left\vert \alpha \right\vert +2\beta \leq M, \\
\int\limits_{\mathbb{R}}v^{\gamma }\psi ^{(2)}(v)dv &=&0,\ \ \ \ \ \gamma
\geq 0.
\end{eqnarray*}%
Here the positive integer $M$ may be taken arbitrarily large when the
support of $\psi ^{(1)}$ is compact, and may be infinite otherwise.

Thus we have%
\begin{equation}
f(z,u)=\int_{0}^{\infty }\int_{0}^{\infty }\check{\psi}_{s,t}\ast \psi
_{s,t}\ast f(z,u)\frac{ds}{s}\frac{dt}{t},\ \ \ \ \ f\in L^{2}\left( \mathbb{%
H}^{n}\right) ,  \label{CRF st}
\end{equation}%
where the functions $\psi _{s,t}$ are given by%
\begin{equation}
\psi _{s,t}\left( z,u\right) =\psi _{s}^{(1)}\ast _{2}\psi _{t}^{(2)}\left(
z,u\right) ,  \label{defpsits}
\end{equation}%
with%
\begin{equation*}
\psi _{s}^{(1)}(z,u)=s^{-2n-2}\psi ^{(1)}(\frac{z}{s},\frac{u}{s^{2}})\,\,%
\text{and}\,\,\psi _{t}^{(2)}(v)=t^{-1}\psi ^{(2)}(\frac{v}{t}),
\end{equation*}%
and where the integrals in (\ref{CRF st}) converge in $L^{2}\left( \mathbb{H}%
^{n}\right) $. Indeed, 
\begin{eqnarray*}
\check{\psi}_{s,t}\ast _{\mathbb{H}^{n}}\psi _{s,t}\ast _{\mathbb{H}%
^{n}}f\left( z,u\right) &=&\left( \check{\psi}_{s}^{(1)}\ast _{2}\psi
_{t}^{(2)}\right) \ast _{\mathbb{H}^{n}}\left( \check{\psi}_{s}^{(1)}\ast
_{2}\psi _{t}^{(2)}\right) \ast _{\mathbb{H}^{n}}f\left( z,u\right) \\
&=&\left( \check{\psi}_{s}^{(1)}\ast _{\mathbb{H}^{n}}\check{\psi}%
_{s}^{(1)}\right) \ast _{\mathbb{H}^{n}}\left( \psi _{t}^{(2)}\ast _{\mathbb{%
R}}\psi _{t}^{(2)}\right) \ast _{2}f\left( z,u\right)
\end{eqnarray*}%
yields (\ref{CRF st}) upon invoking the standard Calder\'{o}n reproducing
formula on $\mathbb{R}$ and then Theorem \ref{GM} on $\mathbb{H}^{n}$:%
\begin{eqnarray*}
&&\int_{0}^{\infty }\int_{0}^{\infty }\check{\psi}_{s,t}\ast _{\mathbb{H}%
^{n}}\psi _{s,t}\ast _{\mathbb{H}^{n}}f(z,u)\frac{ds}{s}\frac{dt}{t} \\
&=&\int_{0}^{\infty }\check{\psi}_{s}^{(1)}\ast _{\mathbb{H}^{n}}\check{\psi}%
_{s}^{(1)}\ast _{\mathbb{H}^{n}}\left\{ \int_{0}^{\infty }\psi
_{t}^{(2)}\ast _{\mathbb{R}}\psi _{t}^{(2)}\ast _{2}f\left( z,u\right) \frac{%
dt}{t}\right\} \frac{ds}{s} \\
&=&\int_{0}^{\infty }\check{\psi}_{s}^{(1)}\ast _{\mathbb{H}^{n}}\check{\psi}%
_{s}^{(1)}\ast _{\mathbb{H}^{n}}f\left( z,u\right) \frac{ds}{s}=f\left(
z,u\right) .
\end{eqnarray*}

For $f\in L^{p}$, $1<p<\infty $, the \emph{continuous} Littlewood-Paley
square function $g_{flag}(f)$ of $f$ is defined by 
\begin{equation*}
g_{flag}(f)(z,u)=\left\{ \int_{0}^{\infty }\int_{0}^{\infty }|\psi
_{s,t}\ast f(z,u)|^{2}\frac{ds}{s}\frac{dt}{t}\right\} ^{{\frac{{1}}{{2}}}},
\end{equation*}%
Note that we have the \emph{flag} moment conditions, so-called because they
include only half of the product moment conditions associated with the
product $\mathbb{C}^{n}\times \mathbb{R}$:%
\begin{equation}
\int\limits_{\mathbb{R}}u^{\alpha }\psi (z,u)du=0,\ \ \ \ \ \text{for all }%
\alpha \in \mathbb{Z}_{+}\text{ and }z\in \mathbb{C}^{n}.  \label{momcond}
\end{equation}%
Indeed, with the change of variable $u^{\prime }=u-v$ and the binomial
theorem 
\begin{equation*}
\left( u^{\prime }+v\right) ^{\beta }=\sum_{\beta =\gamma +\delta }c_{\gamma
,\delta }\left( u^{\prime }\right) ^{\gamma }v^{\delta },
\end{equation*}%
we have%
\begin{eqnarray*}
\int\limits_{\mathbb{R}}u^{\alpha }\psi (z,u)du &=&\int_{\mathbb{R}%
}u^{\alpha }\left\{ \int_{\mathbb{R}}\psi ^{(2)}(u-v)\psi
^{(1)}(z,v)dv\right\} du \\
&=&\int_{\mathbb{R}}\left\{ \int_{\mathbb{R}}\left( u^{\prime }+v\right)
^{\alpha }\psi ^{(2)}(u^{\prime })du\right\} \psi ^{(1)}(z,v)dv \\
&=&\sum_{\alpha =\gamma +\delta }c_{\gamma ,\delta }\int_{\mathbb{R}}\left\{
\int_{\mathbb{R}}\left( u^{\prime }\right) ^{\gamma }\psi ^{(2)}(u^{\prime
})du^{\prime }\right\} v^{\delta }\psi ^{(1)}(z,v)dv \\
&=&\sum_{\alpha =\gamma +\delta }c_{\gamma ,\delta }\int_{\mathbb{R}}\left\{
0\right\} v^{\delta }\psi ^{(1)}(z,v)dv=0,
\end{eqnarray*}%
for all $\alpha \in \mathbb{Z}_{+}$ and each $z\in \mathbb{C}^{n}$. Note
that as a consequence the \emph{full} moments $\int\limits_{\mathbb{H}%
^{n}}z^{\alpha }u^{\beta }\psi (z,u)du$ all vanish, but that in general the
partial moments $\int\limits_{\mathbb{C}^{n}}z^{\alpha }\psi (z,u)dz$ do 
\emph{not} vanish.

\begin{remark}
As observed in \cite{NRSW}, there is a weak cancellation substitute for this
failure to vanish, namely an \emph{estimate} for $\int\limits_{\mathbb{C}%
^{n}}z^{\alpha }\psi (z,u)dz$ that is derived from the vanishing moments of $%
\psi ^{(1)}(z,v)$ and the smoothness of $\psi ^{(2)}(u)$ via the identity%
\begin{eqnarray*}
\int\limits_{\mathbb{C}^{n}}z^{\alpha }\psi (z,u)dz &=&\int\limits_{\mathbb{C%
}^{n}}\int\limits_{\mathbb{R}}z^{\alpha }\psi ^{(1)}(z,v)\psi ^{(2)}(u-v)dzdv
\\
&=&\int\limits_{\mathbb{C}^{n}}\int\limits_{\mathbb{R}}z^{\alpha }\psi
^{(1)}(z,v)\left[ \psi ^{(2)}(u-v)-\psi ^{(2)}(u)\right] dzdv.
\end{eqnarray*}%
We will not pursue this further here, as it seems to have no role in our
development.
\end{remark}

We will also consider the associated \emph{sequence} of component functions $%
\left\{ \psi _{j,k}\right\} _{j,k\in \mathbb{Z}}$ where the functions $\psi
_{j,k}$ are given by 
\begin{equation}
\psi _{j,k}(z,u)=\psi _{j}^{(1)}\ast _{2}\psi _{k}^{(2)}(z,u),
\label{defpsijk}
\end{equation}%
with%
\begin{equation*}
\psi _{j}^{(1)}(z,u)=2^{\alpha j(2n+2)}\psi ^{(1)}(2^{\alpha j}z,2^{2\alpha
j}u)\,\,\text{and}\,\,\psi _{k}^{(2)}(v)=2^{2\alpha k}\psi ^{(2)}(2^{2\alpha
k}v),
\end{equation*}%
and $\psi ^{(1)}$ and $\psi ^{(2)}$ as above. Here $\alpha $ is a small
positive constant that will be fixed in Theorem \ref{discretethm} below,
where we establish a \emph{wavelet} Calder\'{o}n reproducing formula using
this sequence of component functions for small $\alpha $. We then have a
corresponding \emph{discrete} (convolution) Littlewood-Paley square function 
$g_{flag}(f)$ defined by 
\begin{equation*}
g_{flag}(f)(z,u)=\left\{ \sum\limits_{j}\sum\limits_{k}|\psi _{j,k}\ast
f(z,u)|^{2}\right\} ^{{\frac{{1}}{{2}}}}.
\end{equation*}%
This should be compared with the analogous square function in \cite{MRS2}.

\begin{remark}
The terminology "implicit two-parameter structure" is inspired by the fact
that the functions $\psi _{s,t}(z,u)$ and $\psi _{j,k}(z,u)$ are not dilated
directly from $\psi (z,u)$, but rather from a lifting of $\psi $ to a
product function. It is the subtle convolution $\ast _{2}$ that facilitates
a passage from one-parameter "cubes" to two-parameter "rectangles" as
dictated by the geometry of the kernels considered.
\end{remark}

\subsection{Square function inequalities}

Altogether we have from above that%
\begin{equation}
f(z,u)=\int_{0}^{\infty }\int_{0}^{\infty }\psi _{s,t}\ast \psi _{s,t}\ast
f(z,u)\frac{ds}{s}\frac{dt}{t},\ \ \ \ \ f\in L^{2}\left( \mathbb{H}%
^{n}\right) .  \label{CRF}
\end{equation}%
Note that if one considers the integral on the right hand side as defining
an operator acting on $f$, then by the construction of\ the function $\psi $%
, it is a flag singular integral operator. Using iteration and the
vector-valued Littlewood-Paley estimate, together with the Calder\'{o}n
reproducing formula on $L^{2}$, allows us to obtain $L^{p}$ estimates for $%
g_{flag}$, $1<p<\infty $, in Theorem \ref{p>1} below. This should be
compared to the variant in Proposition 4.1 of \cite{MRS2} for $g$-functions
constructed from spectral theory for $\mathcal{L}$ and $T$.

\begin{theorem}
\label{p>1}Let $1<p<\infty $. There exist constants $C_{1}$ and $C_{2}$
depending on $n$ and $p$ such that 
\begin{equation*}
C_{1}\Vert f\Vert _{p}\leq \Vert g_{flag}(f)\Vert _{p}\leq C_{2}\Vert f\Vert
_{p},\ \ \ \ \ f\in L^{p}\left( \mathbb{H}^{n}\right) .
\end{equation*}
\end{theorem}

In order to state our results for flag singular integrals on $\mathbb{H}^{n}$%
, we need to recall some definitions given in \cite{NRS}. Following \cite%
{NRS}, we begin with the definition of a class of distributions on Euclidean
space $\mathbb{R}^{N}$. A $k-normalized$ bump function on a space $\mathbb{R}%
^{N}$ is a $C^{k}-$function supported on the unit ball with $C^{k}$ norm
bounded by $1$. As pointed out in \cite{NRS}, the definitions given below
are independent of the choices of $k\geq 1$, and thus we will simply refer
to a "normalized bump function" without specifying the index $k$.

We define a flag convolution kernel on the Heisenberg group in analogy with
Definition 2.1.1 in \cite{NRS}.

\begin{definition}
\label{defflag}A flag convolution kernel on $\mathbb{H}^{n}=\mathbb{C}%
^{n}\times \mathbb{R}$ is a distribution $K$ on $\mathbb{R}^{2n+1}$ which
coincides with a $C^{\infty }$ function away from the coordinate subspace $%
\{(0,u)\}\subset \mathbb{H}^{n}$, where $0\in \mathbb{C}^{n}$ and $u\in 
\mathbb{R}$, and satisfies

\begin{enumerate}
\item (Differential Inequalities) For any multi-indices $\alpha =(\alpha
_{1},\cdots ,\alpha _{n})$, $\beta =(\beta _{1},\cdots ,\beta _{m})$ 
\begin{equation*}
|\partial _{z}^{\alpha }\partial _{u}^{\beta }K(z,u)|\leq C_{\alpha ,\beta
}\left\vert z\right\vert ^{-2n-|\alpha |}\cdot \left( \left\vert
z\right\vert ^{2}+\left\vert u\right\vert \right) ^{-1-|\beta |}
\end{equation*}%
for all $(z,u)\in \mathbb{H}^{n}$ with $z\neq 0$.

\item (Cancellation Condition) 
\begin{equation*}
|\int_{\mathbb{R}}\partial _{z}^{\alpha }K(z,u)\phi _{1}(\delta u)du|\leq
C_{\alpha }|z|^{-2n-|\alpha |}
\end{equation*}%
for every multi-index $\alpha $ and every normalized bump function $\phi
_{1} $ on $\mathbb{R}$ and every $\delta >0$; 
\begin{equation*}
|\int_{\mathbb{C}^{n}}\partial _{u}^{\beta }K(z,u)\phi _{2}(\delta z)dz|\leq
C_{\gamma }|u|^{-1-|\beta |}
\end{equation*}%
for every multi-index $\beta $ and every normalized bump function $\phi _{2}$
on $\mathbb{C}^{n}$ and every $\delta >0$; and 
\begin{equation*}
|\int_{\mathbb{H}^{n}}K(z,u)\phi _{3}(\delta _{1}z,\delta _{2}u)dzdu|\leq C
\end{equation*}%
for every normalized bump function $\phi _{3}$ on $\mathbb{H}^{n}$ and every 
$\delta _{1}>0$ and $\delta _{2}>0$.
\end{enumerate}
\end{definition}

As in \cite{MRS}, we may always assume that a flag kernel $K(z,u)$ is
integrable on $\mathbb{H}^{n}$ by using a smooth truncation argument.

\bigskip

Informally, we can now define the flag Hardy space $H_{flag}^{p}\left( 
\mathbb{H}^{n}\right) $ on the Heisenberg group for $0<p\leq 1$ by%
\begin{equation*}
H_{flag}^{p}\left( \mathbb{H}^{n}\right) =\left\{ f\text{ a \emph{%
distribution} on }\mathbb{H}^{n}:g_{flag}(f)\in L^{p}(\mathbb{H}%
^{n})\right\} ,
\end{equation*}%
and for $f\in H_{flag}^{p}\left( \mathbb{H}^{n}\right) $ define 
\begin{equation*}
\Vert f\Vert _{H_{flag}^{p}}=\Vert g_{flag}(f)\Vert _{p}.
\end{equation*}

Of course we need to give a precise definition of \emph{distribution} in
this context. A natural question also arises as to whether or not the
resulting definition is independent of the choice of component functions $%
\psi _{j,k}$ in the definition of the square function $g_{flag}$. Moreover,
to study $H_{flag}^{p}$-boundedness of flag singular integrals and establish
a duality theory for $H_{flag}^{p}$, this definition is difficult to use
when $0<p\leq 1$. Instead, we need to approximately discretize the
quasi-norm of $H_{flag}^{p}$. In order to obtain this discrete $H_{flag}^{p}$
quasi-norm we will prove certain Plancherel-P\^{o}lya-type inequalities, and
the main tool used in proving such inequalities is the wavelet Calder\'{o}n
reproducing formula that we define below. To be more specific, we will prove
that the formula (\ref{CRF}) converges in certain spaces of test functions $%
\mathcal{M}_{flag}^{M}(\mathbb{H}^{n})$ adapted to the flag structure, and
thus also in the dual spaces $\mathcal{M}_{flag}^{M}(\mathbb{H}^{n})^{\prime
}$ (see Theorem \ref{discretethm} below). Furthermore, using an
approximation procedure and an almost orthogonality argument, we prove in
Theorem \ref{discretethm} below a \emph{wavelet} Calder\'{o}n reproducing
formula which expresses $f$ as a Fourier-like \emph{series} of molecules or
`wavelets' $\left( z,u\right) \rightarrow {\widetilde{\psi }}%
_{j,k}(z,u,z_{I},u_{J})$ with coefficients $\psi _{j,k}\ast f(z_{I},u_{J})$.

We will describe this formula explicitly in Section 3 below, and in order to
do so, we will use the \emph{flag} dyadic decomposition%
\begin{equation*}
\mathbb{H}^{n}=\overset{\cdot }{\cup }_{\left( \alpha ,\tau \right) \in
K_{j}}\mathcal{S}_{j,\alpha ,\tau }
\end{equation*}%
of the Heisenberg group given in Theorem \ref{Heisenberg grid} below (this
is a `hands on' variant of the tiling construction in Stricharz \cite{Str}),
as well as the notion of Heisenberg rectangles%
\begin{equation*}
\mathcal{R}_{\mathcal{S}_{j,\alpha ,\tau }}^{\mathcal{S}_{k,\beta ,\upsilon
}}\left( ver\right) \text{ and }\mathcal{R}_{\mathcal{S}_{j,\alpha ,\tau }}^{%
\mathcal{S}_{k,\beta ,\upsilon }}\left( hor\right)
\end{equation*}%
given in Definition \ref{Hrect} in Section \ref{s Heisenberg grid} below
when $j\leq k$ and $\mathcal{S}_{j,\alpha ,\tau }$ and $\mathcal{S}_{k,\beta
,\upsilon }$ are dyadic cubes in $\mathbb{H}^{n}$ with $\mathcal{S}%
_{j,\alpha ,\tau }\subset \mathcal{S}_{k,\beta ,\upsilon }$. Recall that%
\begin{equation*}
\left\{ I\right\} _{I\text{ dyadic}}=\left\{ I_{\alpha }^{j}\right\} _{j\in 
\mathbb{Z}\text{ and }\alpha \in 2^{j}\mathbb{Z}^{2n}}
\end{equation*}%
is the usual dyadic grid in $\mathbb{C}^{n}$ and that%
\begin{equation*}
\left\{ J\right\} _{J\text{ dyadic}}=\left\{ J_{\tau }^{k}\right\} _{k\in 
\mathbb{Z}\text{ and }\tau \in 2^{k}\mathbb{Z}}
\end{equation*}%
is the usual dyadic grid in $\mathbb{R}$. The projection of the dyadic cube $%
\mathcal{S}_{j,\alpha ,\tau }$ onto $\mathbb{C}^{n}$ is the dyadic cube $%
I_{\alpha }^{j}$, and $\mathcal{R}_{\mathcal{S}_{j,\alpha ,\tau }}^{\mathcal{%
S}_{k,\beta ,\upsilon }}\left( ver\right) $ (respectively $\mathcal{R}_{%
\mathcal{S}_{j,\alpha ,\tau }}^{\mathcal{S}_{k,\beta ,\upsilon }}\left(
hor\right) $) plays the role of the dyadic rectangle $I_{\alpha }^{j}\times
J_{\upsilon }^{2k}$ (repsectively $I_{\beta }^{k}\times J_{\tau }^{2j}$). In
the Heisenberg group, these rectangles necessarily `rotate' with the group
structure.

\begin{notation}
\label{notaterect}It will be convenient to use the suggestive, if somewhat
imprecise, notation%
\begin{equation*}
R=\mathcal{R}=I\times J=I_{\alpha }^{j}\times J_{\upsilon }^{2k}
\end{equation*}%
for the dyadic rectangle $\mathcal{R}_{\mathcal{S}_{j,\alpha ,\tau }}^{%
\mathcal{S}_{k,\beta ,\upsilon }}\left( ver\right) $, etc. It should be
emphasized that $R=\mathcal{R}=I\times J$ is \emph{not} a product set, but
rather a dyadic Heisenberg rectangle $\mathcal{R}_{\mathcal{S}_{j,\alpha
,\tau }}^{\mathcal{S}_{k,\beta ,\upsilon }}\left( ver\right) $ that serves
as a Heisenberg substitute for the actual product set $I_{\alpha }^{j}$
times $J_{\upsilon }^{2k}$. Thus we will say that the dyadic rectangle $R=%
\mathcal{R}=I\times J$ has \emph{side lengths} $\ell \left( I\right) =2^{j}$
and $\ell \left( J\right) =2^{2k}$. For $j\leq k$, the collection of all
dyadic Heisenberg rectangles $R=\mathcal{R}=I\times J$ with side lengths $%
2^{j}$ and $2^{2k}$ will be denoted by%
\begin{equation*}
\mathcal{R}\left( 2^{j}\times 2^{2k}\right) \equiv \left\{ R=\mathcal{R}%
=I\times J=I_{\alpha }^{j}\times J_{\upsilon }^{2k}=\mathcal{R}_{\mathcal{S}%
_{j,\alpha ,\tau }}^{\mathcal{S}_{k,\beta ,\upsilon }}\left( ver\right) :%
\mathcal{S}_{j,\alpha ,\tau }\subset \mathcal{S}_{k,\beta ,\upsilon
}\right\} .
\end{equation*}
\end{notation}

\begin{description}
\item[Caution] For $k\leq j$, the support of the component function $\psi
_{j,k}$ defined in (\ref{defpsijk}) is essentially a vertical Heisenberg
rectangle $I\times J$ having side lengths $\ell \left( I\right) =2^{-j}$ and 
$\ell \left( J\right) =2^{-2k}$. Note the passage from $j,k$ to $-j,-k$.
\end{description}

\subsection{Standard test functions}

We now discuss the issues arising in giving a precise definition of the flag
Hardy space $H_{flag}^{p}\left( \mathbb{H}^{n}\right) $ as elements in the
dual of familiar test spaces. We begin by introducing the test spaces $%
\mathcal{M}_{flag}^{M}(\mathbb{H}^{n})$ associated with the flag structure
on $\mathbb{H}^{n}$ that are obtained by \emph{projecting} the corresponding
product test spaces $\mathcal{M}_{product}^{M}\left( \mathbb{H}^{n}\times 
\mathbb{R}\right) $ onto $\mathbb{H}^{n}$. Our definitions here will
encompass the entire range $0<p\leq 1$. For this we use the projection of
functions $F$ defined on $\mathbb{H}^{n}\times \mathbb{R}$ \ to functions $%
f=\pi F$ defined on $\mathbb{H}^{n}$\ as introduced in \cite{MRS}:%
\begin{equation}
f\left( z,u\right) =\left( \pi F\right) \left( z,u\right) \equiv \int_{%
\mathbb{R}}F\left( \left( z,u-v\right) ,v\right) dv.  \label{projection}
\end{equation}%
We will also use the notation $\pi F=F_{\flat }$ as in \cite{MRS}. Recall
that $2n+1$ is the Euclidean dimension of the Heisenberg group $\mathbb{H}%
^{n}=\mathbb{C}^{n}\times \mathbb{R}$ and that $Q=2n+2$ is the homogeneous
dimension of $\mathbb{H}^{n}$.

Note that in this notation, the component function $\psi \left( z,u\right) $
in Subsection 2.1 above is given by $\pi \Psi \left( z,u\right) $ where%
\begin{equation}
\Psi \left( z,u,v\right) \equiv \psi ^{\left( 1\right) }\left( z,u\right)
\psi ^{\left( 2\right) }\left( v\right) .  \label{Psi}
\end{equation}

We now define an appropriate \emph{product} molecular space $\mathcal{M}%
_{product}^{M_{1},M_{2},M}$ on $\mathbb{H}^{n}\times \mathbb{R}$ with three
parameters $M_{1},M_{2},M$.

\begin{remark}
Note that in the definition below, we require \emph{equally} many moments
and derivatives in each of the $u$ and $v$ variables, and exactly \emph{twice%
} as many moments and derivatives in the $z$ variable. The integer $M$
controls the decay of the function, the integer $M_{1}$ controls the total
number of moments, and the integer $M_{2}$ controls the total weighted
number of derivatives permitted.
\end{remark}

\begin{definition}
\label{product molecular}Let $M,M_{1},M_{2}\in \mathbb{N}$ be positive
integers and let $0<\delta \leq 1$. The \emph{product} molecular space $%
\mathcal{M}_{product}^{M+\delta ,M_{1},M_{2}}\left( \mathbb{H}^{n}\times 
\mathbb{R}\right) $ consists of all functions $F\left( \left( z,u\right)
,v\right) $ on $\mathbb{H}^{n}\times \mathbb{R}$ satisfying the product
moment conditions%
\begin{eqnarray}
\int_{\mathbb{H}^{n}}z^{\alpha }u^{\beta }F\left( \left( z,u\right)
,v\right) dzdu &=&0\text{ and }\int_{\mathbb{R}}v^{\gamma }F\left( \left(
z,u\right) ,v\right) dv=0  \label{Mo1} \\
\text{for all }\left\vert \alpha \right\vert +2\beta &\leq &M_{1}\text{ and }%
2\gamma \leq M_{1},  \notag
\end{eqnarray}%
and such that there is a nonnegative constant $A$ satisfying the following
four differential inequalities:%
\begin{eqnarray}
&&\left\vert \partial _{z}^{\alpha }\partial _{u}^{\beta }\partial
_{v}^{\gamma }F\left( \left( z,u\right) ,v\right) \right\vert \leq A\frac{1}{%
\left( 1+\left\vert z\right\vert ^{2}+\left\vert u\right\vert \right) ^{%
\frac{Q+M+\left\vert \alpha \right\vert +2\beta +\delta }{2}}}\frac{1}{%
\left( 1+\left\vert v\right\vert \right) ^{1+M+\gamma +\delta }},
\label{DI1} \\
&&\ \ \ \ \ \ \ \ \ \ \text{for all }\left\vert \alpha \right\vert +2\beta
\leq M_{2}\text{ and }2\gamma \leq M_{2},  \notag
\end{eqnarray}%
\begin{eqnarray}
&&\left\vert \partial _{z}^{\alpha }\partial _{u}^{\beta }\partial
_{v}^{\gamma }F\left( \left( z,u\right) ,v\right) -\partial _{z}^{\alpha
}\partial _{u}^{\beta }\partial _{v}^{\gamma }F\left( \left( z^{\prime
},u^{\prime }\right) ,v\right) \right\vert  \label{DI2} \\
&&\ \ \ \ \ \ \ \ \ \ \ \ \ \ \ \leq A\frac{\left\vert \left( z,u\right)
\circ \left( z^{\prime },u^{\prime }\right) ^{-1}\right\vert ^{\delta }}{%
\left( 1+\left\vert z\right\vert ^{2}+\left\vert u\right\vert \right) ^{%
\frac{Q+M+M_{2}+2\delta }{2}}}\frac{1}{\left( 1+\left\vert v\right\vert
\right) ^{1+M+\gamma +\delta }}  \notag \\
&&\text{for all }\left\vert \alpha \right\vert +2\beta =M_{2}\text{ and }%
\left\vert \left( z,u\right) \circ \left( z^{\prime },u^{\prime }\right)
^{-1}\right\vert \leq \frac{1}{2}\left( 1+\left\vert z\right\vert
^{2}+\left\vert u\right\vert \right) ^{\frac{1}{2}}.  \notag
\end{eqnarray}%
\begin{eqnarray}
&&\left\vert \partial _{z}^{\alpha }\partial _{u}^{\beta }\partial
_{v}^{\gamma }F\left( \left( z,u\right) ,v\right) -\partial _{z}^{\alpha
}\partial _{u}^{\beta }\partial _{v}^{\gamma }F\left( \left( z,u\right)
,v^{\prime }\right) \right\vert  \label{DI3} \\
&&\ \ \ \ \ \ \ \ \ \ \leq A\frac{1}{\left( 1+\left\vert z\right\vert
^{2}+\left\vert u\right\vert \right) ^{\frac{Q+M+\left\vert \alpha
\right\vert +2\beta +\delta }{2}}}\frac{\left\vert v-v^{\prime }\right\vert
^{\delta }}{\left( 1+\left\vert v\right\vert \right) ^{1+M+\frac{M_{2}}{2}%
+2\delta }},  \notag \\
&&\ \ \ \ \ \ \ \ \ \ \text{for all }\left\vert \alpha \right\vert +2\beta
\leq M_{2}\text{ and }2\gamma =M_{2}  \notag \\
&&\ \ \ \ \ \ \ \ \ \ \ \ \ \ \ \text{and }\left\vert v-v^{\prime
}\right\vert \leq \frac{1}{2}\left( 1+\left\vert v\right\vert \right) , 
\notag
\end{eqnarray}%
\begin{eqnarray}
&&\left\vert \left[ \partial _{z}^{\alpha }\partial _{u}^{\beta }\partial
_{v}^{\gamma }F\left( \left( z,u\right) ,v\right) -\partial _{z}^{\alpha
}\partial _{u}^{\beta }\partial _{v}^{\gamma }F\left( \left( z^{\prime
},u^{\prime }\right) ,v\right) \right] \right.  \label{DI4} \\
&&\ \ \ \ \ \left. -\left[ \partial _{z}^{\alpha }\partial _{u}^{\beta
}\partial _{v}^{\gamma }F\left( \left( z,u\right) ,v^{\prime }\right)
-\partial _{z}^{\alpha }\partial _{u}^{\beta }\partial _{v}^{\gamma }F\left(
\left( z^{\prime },u^{\prime }\right) ,v^{\prime }\right) \right] \right\vert
\notag \\
&&\ \ \ \ \ \ \ \ \ \ \leq A\frac{\left\vert \left( z,u\right) \circ \left(
z^{\prime },u^{\prime }\right) ^{-1}\right\vert ^{\delta }}{\left(
1+\left\vert z\right\vert ^{2}+\left\vert u\right\vert \right) ^{\frac{%
Q+M+M_{2}+2\delta }{2}}}\frac{\left\vert v-v^{\prime }\right\vert ^{\delta }%
}{\left( 1+\left\vert v\right\vert \right) ^{1+M+\frac{M_{2}}{2}+2\delta }} 
\notag \\
&&\ \ \ \ \ \ \ \ \ \ \text{for all }\left\vert \alpha \right\vert +2\beta
=M_{2}\text{ and }2\gamma =M_{2}  \notag \\
&&\ \ \ \ \ \ \ \ \ \ \ \ \ \ \ \text{and }\left\vert \left( z,u\right)
\circ \left( z^{\prime },u^{\prime }\right) ^{-1}\right\vert \leq \frac{1}{2}%
\left( 1+\left\vert z\right\vert ^{2}+\left\vert u\right\vert \right) ^{%
\frac{1}{2}}  \notag \\
&&\ \ \ \ \ \ \ \ \ \ \ \ \ \ \ \text{and }\left\vert v-v^{\prime
}\right\vert \leq \frac{1}{2}\left( 1+\left\vert v\right\vert \right) , 
\notag
\end{eqnarray}
\end{definition}

The space $\mathcal{M}_{product}^{M+\delta ,M_{1},M_{2}}\left( \mathbb{H}%
^{n}\times \mathbb{R}\right) $ becomes a Banach space under the norm defined
by the least nonnegative number $A$ for which the above four inequalities
hold.

\bigskip

Now we define the\ \emph{flag} molecular space $\mathcal{M}_{flag}^{M+\delta
,M_{1},M_{2}}\left( \mathbb{H}^{n}\right) $ as the projection of $\mathcal{M}%
_{product}^{M+\delta ,M_{1},M_{2}}\left( \mathbb{H}^{n}\times \mathbb{R}%
\right) $ under the map $\pi $ given in (\ref{projection}). We refer to $%
\mathcal{M}_{flag}^{M+\delta ,M_{1},M_{2}}\left( \mathbb{H}^{n}\right) $ as
a \emph{projected} flag molecular space when contrasting it with the more
familiar \emph{moment} flag molecular space defined below.

\begin{definition}
\label{flag molecular}Let $M,M_{1},M_{2}\in \mathbb{N}$ be positive integers
and $0<\delta \leq 1$. The \emph{flag} molecular space $\mathcal{M}%
_{flag}^{M+\delta ,M_{1},M_{2}}\left( \mathbb{H}^{n}\right) $ consists of
all functions $f$ on $\mathbb{H}^{n}$ such that there is $F\in \mathcal{M}%
_{product}^{M+\delta ,M_{1},M_{2}}\left( \mathbb{H}^{n}\times \mathbb{R}%
\right) $ with $f=\pi F=F_{\flat }$. Define a norm on $\mathcal{M}%
_{flag}^{M+\delta ,M_{1},M_{2}}\left( \mathbb{H}^{n}\right) $ by%
\begin{equation*}
\left\Vert f\right\Vert _{\mathcal{M}_{flag}^{M+\delta ,M_{1},M_{2}}\left( 
\mathbb{H}^{n}\right) }\equiv \inf_{F:f=\pi F}\left\Vert F\right\Vert _{%
\mathcal{M}_{product}^{M+\delta ,M_{1},M_{2}}\left( \mathbb{H}^{n}\times 
\mathbb{R}\right) }.
\end{equation*}
\end{definition}

Thus the norm on $\mathcal{M}_{flag}^{M+\delta ,M_{1},M_{2}}\left( \mathbb{H}%
^{n}\right) $ is the quotient norm%
\begin{equation*}
\left\Vert f\right\Vert _{\mathcal{M}_{flag}^{M+\delta ,M_{1},M_{2}}\left( 
\mathbb{H}^{n}\right) }=\mathcal{M}_{product}^{M+\delta ,M_{1},M_{2}}\left( 
\mathbb{H}^{n}\times \mathbb{R}\right) /\pi ^{-1}\left( \left\{ 0\right\}
\right) ,
\end{equation*}%
and $\mathcal{M}_{flag}^{M+\delta ,M_{1},M_{2}}\left( \mathbb{H}^{n}\right) $
is a Banach space.

We record here an intertwining formula for $\pi $ and a convolution operator 
$T$ on $\mathbb{H}^{n}$. Let 
\begin{equation*}
Tf\left( z,u\right) =K\ast _{\mathbb{H}^{n}}f\left( z,u\right) =\int_{%
\mathbb{H}^{n}}K\left( \left( z,u\right) \circ \left( z^{\prime },u^{\prime
}\right) ^{-1}\right) f\left( z^{\prime },u^{\prime }\right) dz^{\prime
}du^{\prime }.
\end{equation*}%
Extend $T$ to an operator $\widetilde{T}=T\otimes \delta _{0}$ on the group $%
\mathbb{H}^{n}\times \mathbb{R}$ by acting $T$ in the $\mathbb{H}^{n}$
factor only:%
\begin{equation*}
\widetilde{T}F\left( \left( z,u\right) ,v\right) =\int_{\mathbb{H}%
^{n}}K\left( \left( z,u\right) \circ \left( z^{\prime },u^{\prime }\right)
^{-1}\right) F\left( z^{\prime },u^{\prime },v\right) dz^{\prime }du^{\prime
}.
\end{equation*}

\begin{lemma}
\label{intertwine}Let $T$ be a convolution operator on $\mathbb{H}^{n}$ and
let $\widetilde{T}=T\otimes \delta _{0}$ be its extension to $\mathbb{H}%
^{n}\times \mathbb{R}$ defined above. Then%
\begin{equation*}
T\left( \pi F\right) \left( z,u\right) =\pi \left( \widetilde{T}F\right)
\left( z,u\right) .
\end{equation*}
\end{lemma}

\begin{proof}
Formally we have%
\begin{eqnarray*}
T\left( \pi F\right) \left( z,u\right) &=&\int_{\mathbb{H}^{n}}K\left(
\left( z,u\right) \circ \left( z^{\prime },u^{\prime }\right) ^{-1}\right)
\left( \pi F\right) \left( z^{\prime },u^{\prime }\right) dz^{\prime
}du^{\prime } \\
&=&\int_{\mathbb{H}^{n}}K\left( \left( z,u\right) \circ \left( z^{\prime
},u^{\prime }\right) ^{-1}\right) \left\{ \int_{\mathbb{R}}F\left( z^{\prime
},u^{\prime }-v,v\right) dv\right\} dz^{\prime }du^{\prime } \\
&=&\int_{\mathbb{H}^{n}}\int_{\mathbb{R}}K\left( z-z^{\prime },u-u^{\prime
}+2\func{Im}\overline{z^{\prime }}z\right) F\left( z^{\prime },u^{\prime
}-v,v\right) dvdz^{\prime }du^{\prime }.
\end{eqnarray*}%
Now make the change of variable $w^{\prime }=u^{\prime }-v$ to get%
\begin{eqnarray*}
T\left( \pi F\right) \left( z,u\right) &=&\int_{\mathbb{H}^{n}}\int_{\mathbb{%
R}}K\left( z-z^{\prime },u-w^{\prime }-v+2\func{Im}\overline{z^{\prime }}%
z\right) F\left( z^{\prime },w^{\prime },v\right) dvdz^{\prime }dw^{\prime }
\\
&=&\int_{\mathbb{R}}\left\{ \int_{\mathbb{H}^{n}}K\left( \left( z,u-v\right)
\circ \left( z^{\prime },w^{\prime }\right) ^{-1}\right) F\left( z^{\prime
},w^{\prime },v\right) dz^{\prime }dw^{\prime }\right\} dv \\
&=&\int_{\mathbb{R}}\left\{ \widetilde{T}F\left( z,u-v,v\right) \right\}
dv=\pi \left( \widetilde{T}F\right) \left( z,u\right) .
\end{eqnarray*}
\end{proof}

Later in the paper we will fix $M_{1}=M_{2}=M$ and denote $\mathcal{M}%
_{flag}^{M+\delta ,M_{1},M_{2}}\left( \mathbb{H}^{n}\right) $ simply by $%
\mathcal{M}_{flag}^{M+\delta }\left( \mathbb{H}^{n}\right) $, but for now we
will allow $M_{1}$ and $M_{2}$ to remain independent of $M$ in order to
further analyze the space $\mathcal{M}_{flag}^{M+\delta ,M_{1},M_{2}}\left( 
\mathbb{H}^{n}\right) $.

\subsubsection{An analysis of the projected flag molecular space}

Lemma \ref{containments} below shows that functions $f\left( z,u\right) $ in
the projected flag molecular space $\mathcal{M}_{flag}^{M+\delta
,M_{1},M_{2}}\left( \mathbb{H}^{n}\right) $ have moments in the $u$ variable
alone, as well as \emph{more} moments in the $\left( z,u\right) $ variable
than we might expect. We refer loosely to this situation as having \emph{%
half-product} moments. There is a more familiar space of test functions $%
\mathsf{M}_{F}^{M+\delta ,M_{1},M_{2}}\left( \mathbb{H}^{n}\right) $ defined
below with half-product moments, that \emph{avoids} the operation of
projection, and that is closely related to the projected test space $%
\mathcal{M}_{flag}^{M+\delta ,M_{1},M_{2}}\left( \mathbb{H}^{n}\right) $. We
refer to $\mathsf{M}_{F}^{M+\delta ,M_{1},M_{2}}\left( \mathbb{H}^{n}\right) 
$ as the \emph{moment} flag molecular space. While we do not know if the
spaces $\mathcal{M}_{flag}^{M+\delta ,M_{1},M_{2}}\left( \mathbb{H}%
^{n}\right) $ and $\mathsf{M}_{F}^{M+\delta ,M_{1},M_{2}}\left( \mathbb{H}%
^{n}\right) $ coincide, the embeddings in Lemma \ref{containments} below
show they are essentially the same.

\begin{definition}
Let $M,M_{1},M_{2}\in \mathbb{N}$ be positive integers and $0<\delta \leq 1$%
. Define the \emph{moment\ }flag molecular space $\mathsf{M}_{F}^{M+\delta
,M_{1},M_{2}}\left( \mathbb{H}^{n}\right) $ to consist of all functions $f$
on $\mathbb{H}^{n}$ satisfying the moment conditions%
\begin{eqnarray*}
\int_{\mathbb{H}^{n}}z^{\alpha }u^{\beta }f\left( z,u\right) dzdu &=&0\text{
for all }\left\vert \alpha \right\vert \leq M_{1}\text{, }\left\vert \alpha
\right\vert +2\beta \leq 2M_{1}+2, \\
\int_{\mathbb{R}}u^{\gamma }f\left( z,u\right) du &=&0\text{ for all }\gamma
\leq M_{1},
\end{eqnarray*}%
and such that there is a nonnegative constant $A$ satisfying the following
two differential inequalities:%
\begin{equation*}
\left\vert \partial _{z}^{\alpha }\partial _{u}^{\beta }f\left( z,u\right)
\right\vert \leq A\frac{1}{\left( 1+\left\vert z\right\vert ^{2}+\left\vert
u\right\vert \right) ^{\frac{Q+M+\left\vert \alpha \right\vert +2\beta }{2}}}%
\text{ for all }\left\vert \alpha \right\vert +2\beta \leq M_{2},
\end{equation*}%
\begin{eqnarray*}
&&\left\vert \partial _{z}^{\alpha }\partial _{u}^{\beta }f\left( z,u\right)
-\partial _{z}^{\alpha }\partial _{u}^{\beta }f\left( z^{\prime },u^{\prime
}\right) \right\vert \leq A\frac{\left\vert \left( z,u\right) \circ \left(
z^{\prime },u^{\prime }\right) ^{-1}\right\vert ^{\delta }}{\left(
1+\left\vert z\right\vert ^{2}+\left\vert u\right\vert \right) ^{\frac{%
Q+M+\delta +M_{2}}{2}}} \\
&&\text{for all }\left\vert \alpha \right\vert +2\beta =M_{2}\text{ and }%
\left\vert \left( z,u\right) \circ \left( z^{\prime },u^{\prime }\right)
^{-1}\right\vert \leq \frac{1}{2}\left( 1+\left\vert z\right\vert
^{2}+\left\vert u\right\vert \right) ^{\frac{1}{2}}.
\end{eqnarray*}
\end{definition}

Note that the moment conditions in the definition of $\mathsf{M}%
_{F}^{M+\delta ,M_{1},M_{2}}\left( \mathbb{H}^{n}\right) $ permit larger $%
\beta $'$s$ depending on $\left\vert \alpha \right\vert $ than in the
definition of $\mathcal{M}_{flag}^{M+\delta ,M_{1},M_{2}}\left( \mathbb{H}%
^{n}\right) $. The space $\mathsf{M}_{F}^{M+\delta ,M_{1},M_{2}}\left( 
\mathbb{H}^{n}\right) $ becomes a Banach space under the norm defined by the
least nonnegative number $A$ for which the above two inequalities hold.

\begin{lemma}
\label{containments}The spaces $\mathcal{M}_{flag}^{M+\delta
,M_{1},M_{2}}\left( \mathbb{H}^{n}\right) $ and $\mathsf{M}_{F}^{M+\delta
,M_{1},M_{2}}\left( \mathbb{H}^{n}\right) $ satisfy the following
containments%
\begin{equation*}
\mathsf{M}_{F}^{3M+\delta +M_{2},M_{1},2M_{2}+4}\left( \mathbb{H}^{n}\right)
\subset \mathcal{M}_{flag}^{M+\delta ,M_{1},M_{2}}\left( \mathbb{H}%
^{n}\right) \subset \mathsf{M}_{F}^{M+\delta ,M_{1},M_{2}}\left( \mathbb{H}%
^{n}\right) ,
\end{equation*}%
which are continuous: 
\begin{equation*}
\left\Vert f\right\Vert _{\mathsf{M}_{F}^{M+\delta ,M_{1},M_{2}}\left( 
\mathbb{H}^{n}\right) }\lesssim \left\Vert f\right\Vert _{\mathcal{M}%
_{flag}^{M+\delta ,M_{1},M_{2}}\left( \mathbb{H}^{n}\right) }\lesssim
\left\Vert f\right\Vert _{\mathsf{M}_{F}^{3M+\delta
+M_{2},M_{1},2M_{2}+4}\left( \mathbb{H}^{n}\right) }.
\end{equation*}
\end{lemma}

\begin{remark}
The importance of the `projected' flag molecular space $\mathcal{M}%
_{flag}^{M+\delta ,M_{1},M_{2}}\left( \mathbb{H}^{n}\right) $ lies in the
existence of a \emph{wavelet} Calder\'{o}n reproducing formula for this
space of test functions - see Theorem \ref{discretethm} below. We do not
know if such a reproducing formula holds for the `moment' flag space $%
\mathsf{M}_{F}^{M+\delta ,M_{1},M_{2}}\left( \mathbb{H}^{n}\right) $, but
the embeddings in Lemma \ref{containments} will identify the distributions
in the dual space $\mathcal{M}_{flag}^{M+\delta ,M_{1},M_{2}}\left( \mathbb{H%
}^{n}\right) ^{\prime }$ as being `roughly' those in a dual space $\mathsf{M}%
_{F}^{M^{\prime }+\delta ,M_{1}^{\prime },M_{2}^{\prime }}\left( \mathbb{H}%
^{n}\right) ^{\prime }$.
\end{remark}

\begin{remark}
The integer $M_{1}$ that controls the number of moments in $\mathsf{M}%
_{F}^{M+\delta ,M_{1},M_{2}}\left( \mathbb{H}^{n}\right) $ remains the same
in both the smaller space $\mathsf{M}_{F}^{3M+\delta
+M_{2},M_{1},2M_{2}+4}\left( \mathbb{H}^{n}\right) $ and the larger space $%
\mathsf{M}_{F}^{M+\delta ,M_{1},M_{2}}\left( \mathbb{H}^{n}\right) $.
However, we lose both derivatives and decay in passing from the smaller to
the larger space.
\end{remark}

While we cannot say that $H_{flag}^{p}\left( \mathbb{H}^{n}\right) $ is a 
\emph{subspace} of the more familiar one-parameter Hardy space $H^{p}\left( 
\mathbb{H}^{n}\right) $, we can show that the quotient space 
\begin{equation*}
Q_{flag}^{p}\left( \mathbb{H}^{n}\right) \equiv H_{flag}^{p}\left( \mathbb{H}%
^{n}\right) \ /\ \mathsf{M}_{F}^{M^{\prime }+\delta ,M_{1}^{\prime
},M_{2}^{\prime }}\left( \mathbb{H}^{n}\right) ^{\bot }
\end{equation*}%
of $H_{flag}^{p}\left( \mathbb{H}^{n}\right) $ can be identified with a
closed subspace of the corresponding quotient space%
\begin{equation*}
Q^{p}\left( \mathbb{H}^{n}\right) \equiv H^{p}\left( \mathbb{H}^{n}\right) \
/\ \mathsf{M}_{F}^{M^{\prime }+\delta ,M_{1}^{\prime },M_{2}^{\prime
}}\left( \mathbb{H}^{n}\right) ^{\bot }
\end{equation*}%
of $H^{p}\left( \mathbb{H}^{n}\right) $, thus giving a sense in which the
distributions we use to define $H_{flag}^{p}\left( \mathbb{H}^{n}\right) $
are `roughly' the same as those used to define $H^{p}\left( \mathbb{H}%
^{n}\right) $. See Section 10 below for details.

\section{The wavelet Calder\'{o}n reproducing formula}

We can now state our \emph{wavelet} Calder\'{o}n reproducing formula for the
flag structure in terms of the projected product test spaces 
\begin{equation*}
\mathcal{M}_{flag}^{M+\delta }\left( \mathbb{H}^{n}\right) \equiv \mathcal{M}%
_{flag}^{M+\delta ,M.M}\left( \mathbb{H}^{n}\right) ,
\end{equation*}%
defined by projecting the product test spaces%
\begin{equation*}
\mathcal{M}_{product}^{M+\delta }\left( \mathbb{H}^{n}\times \mathbb{R}%
\right) \equiv \mathcal{M}_{product}^{M+\delta ,M,M}\left( \mathbb{H}%
^{n}\times \mathbb{R}\right) .
\end{equation*}%
We remind the reader that Euclidean versions of such reproducing formulas
were obtained by Frazier and Jawerth \cite{FJ} using the Fourier transform
together with the very special property that $\mathbb{R}^{n}$ is \emph{tiled}
by the compact abelian torus $\mathbb{T}^{n}$ and its discrete dual group,
the lattice $\mathbb{Z}^{n}$.

It is convenient to introduce some new notation for the dyadic rectangles
defined in Notation \ref{notaterect}. Given $0<\alpha <1$ and a positive
integer $N$, we write%
\begin{eqnarray*}
\mathsf{R}\left( j,k\right) &\equiv &\mathcal{R}\left( 2^{-\alpha \left(
j+N\right) }\times 2^{-2\alpha \left( k+N\right) }\right) , \\
\mathsf{Q}\left( j\right) &\equiv &\mathcal{R}\left( 2^{-\alpha \left(
j+N\right) }\times 2^{-2\alpha \left( j+N\right) }\right) .
\end{eqnarray*}%
Now for $\mathcal{Q}\in \mathsf{Q}\left( j\right) $ let $\left( z_{\mathcal{Q%
}},u_{\mathcal{Q}}\right) $ be any \emph{fixed} point in the cube $\mathcal{Q%
}$; and for $\mathcal{R}\in \mathsf{R}\left( j,k\right) $ with $k<j$ let $%
\left( z_{\mathcal{R}},u_{\mathcal{R}}\right) $ be any \emph{fixed} point in
the rectangle $\mathcal{R}$. Let us write the collection of \emph{all}
dyadic cubes as 
\begin{equation*}
\mathsf{Q}\equiv \bigcup_{j\in \mathbb{Z}}\mathsf{Q}\left( j\right) ,
\end{equation*}%
and the collection of \emph{all strictly vertical} dyadic rectangles as%
\begin{equation*}
\mathsf{R}_{vert}\equiv \bigcup_{j>k}\mathsf{R}\left( j,k\right) .
\end{equation*}%
We now set%
\begin{eqnarray*}
\psi _{\mathcal{Q}}^{\prime } &=&\psi _{j}^{\left( 1\right) }\text{ if }%
\mathcal{Q}\in \mathsf{Q}\left( j\right) , \\
\psi _{\mathcal{R}}^{\prime } &=&\psi _{j,k}=\psi _{j}^{\left( 1\right)
}\ast _{2}\psi _{k}^{\left( 2\right) }\text{ if }\mathcal{R}\in \mathsf{R}%
\left( j,k\right) ,
\end{eqnarray*}%
where $\psi _{j,k}$ are as in (\ref{defpsijk}). Given an appropriate
distribution $f$ on $\mathbb{H}^{n}$, we define its \emph{wavelet
coefficients} $f_{\mathcal{Q}}$ and $f_{\mathcal{R}}$ by%
\begin{eqnarray*}
f_{\mathcal{Q}} &=&\psi _{\mathcal{Q}}^{\prime }\ast f\left( z_{\mathcal{Q}%
},u_{\mathcal{Q}}\right) \text{ if }\mathcal{Q}\in \mathsf{Q}, \\
f_{\mathcal{R}} &=&\psi _{\mathcal{R}}^{\prime }\ast f\left( z_{\mathcal{R}%
},u_{\mathcal{R}}\right) \text{ if }\mathcal{R}\in \mathsf{R}_{vert},\text{
i.e. when }j>k.
\end{eqnarray*}%
Here is the \emph{wavelet} Calder\'{o}n reproducing formula.

\begin{theorem}
\label{discretethm}Suppose notation is as above. Then there are associated
functions $\widetilde{\psi }_{\mathcal{Q}},\widetilde{\psi }_{\mathcal{R}%
}\in \mathcal{M}_{flag}^{M+\delta }\left( \mathbb{H}^{n}\right) $ for $%
\mathcal{Q}\in \mathsf{Q}$ and $\mathcal{R}\in \mathsf{R}_{vert}$ satisfying 
\begin{eqnarray*}
\left\Vert \widetilde{\psi }_{\mathcal{Q}}\right\Vert _{\mathcal{M}%
_{flag}^{M+\delta }\left( \mathbb{H}^{n}\right) } &\lesssim &\left\Vert \psi
_{\mathcal{Q}}^{\prime }\right\Vert _{\mathcal{M}_{flag}^{M+\delta }\left( 
\mathbb{H}^{n}\right) },\ \ \ \ \ \mathcal{Q}\in \mathsf{Q}, \\
\left\Vert \widetilde{\psi }_{\mathcal{R}}\right\Vert _{\mathcal{M}%
_{flag}^{M+\delta }\left( \mathbb{H}^{n}\right) } &\lesssim &\left\Vert \psi
_{\mathcal{R}}^{\prime }\right\Vert _{\mathcal{M}_{flag}^{M+\delta }\left( 
\mathbb{H}^{n}\right) },\ \ \ \ \ \mathcal{R}\in \mathsf{R}_{vert},
\end{eqnarray*}%
and 
\begin{equation}
f\left( z,u\right) =\sum_{\mathcal{Q}\in \mathsf{Q}}f_{\mathcal{Q}}\ 
\widetilde{\psi }_{\mathcal{Q}}\left( z,u\right) +\sum_{\mathcal{R}\in 
\mathsf{R}_{vert}}f_{\mathcal{R}}\ \widetilde{\psi }_{\mathcal{R}}\left(
z,u\right) ,\ \ \ \ \ \left( z,u\right) \in \mathbb{H}^{n},
\label{defdiscreteCRF}
\end{equation}%
where the series in (\ref{defdiscreteCRF}) converges in three spaces:

\begin{enumerate}
\item in $L^{p}\left( \mathbb{H}^{n}\right) $ for $1<p<\infty $,

\item in the Banach space $\mathcal{M}_{flag}^{M^{\prime }+\delta }(\mathbb{H%
}^{n})$ for $M^{\prime }$ large enough,

\item and in the corresponding dual space $\mathcal{M}_{flag}^{M^{\prime
}+\delta }(\mathbb{H}^{n})^{\prime }$ for $M^{\prime }$ large enough.
\end{enumerate}
\end{theorem}

\begin{remark}
Note that only \emph{half} of the collection of dyadic rectangles, namely
the vertical ones $\mathsf{R}_{vert}$, are used in the wavelet Calder\'{o}n
reproducing formula. This is a reflection of the implicit product structure
inherent in the Heisenberg group $\mathbb{H}^{n}$.
\end{remark}

\subsection{Plancherel-P\^{o}lya inequalities and flag Hardy spaces}

The wavelet Calder\'{o}n reproducing formula (\ref{defdiscreteCRF}) yields
the following Plancherel-P\^{o}lya type inequalities (cf \cite{PlPo}, \cite%
{PlPo2}). We use the notation $A\approx B$ to indicate that two quantities $%
A $ and $B$ are comparable.

\begin{theorem}
\label{P-P}Suppose $\psi ^{(1)},\phi ^{(1)}\in \mathcal{S}(\mathbb{C}^{n})$
and $\psi ^{(2)},\phi ^{(2)}\in \mathcal{S}\mathbb{(R)}$ and let 
\begin{equation*}
\psi (z,u)=\int\limits_{\mathbb{R}}\psi ^{(1)}(z,u-v)\psi ^{(2)}(v)dv,
\end{equation*}%
\begin{equation*}
\phi (z,u)=\int\limits_{\mathbb{R}}\phi ^{(1)}(z,u-v)\psi ^{(2)}(v)dv,
\end{equation*}%
be two component functions that each satisfy the conditions in Subsubsection %
\ref{component}. Then with $\mathsf{Q}$ , $\mathsf{R}_{vert}$, $\psi _{%
\mathcal{Q}}^{\prime }$ ands $\psi _{\mathcal{R}}^{\prime }$ as above, and
similarly for $\phi $, and for $f\in \mathcal{M}_{flag}^{M+\delta }\left( 
\mathbb{H}^{n}\right) ^{\prime }$, $0<p<\infty $, and $M$ chosen large
enough depending on $n$ and $p$, 
\begin{equation*}
\left\Vert \left\{ \sum_{\mathcal{Q}\in \mathsf{Q}}\sup_{\left( z^{\prime
},u^{\prime }\right) \in \mathcal{Q}}\left\vert \psi _{\mathcal{Q}}^{\prime
}\ast f\left( z^{\prime },u^{\prime }\right) \right\vert ^{2}\chi _{\mathcal{%
Q}}\left( z,u\right) +\sum_{\mathcal{R}\in \mathsf{R}_{vert}}\sup_{\left(
z^{\prime },u^{\prime }\right) \in \mathcal{R}}\left\vert \psi _{\mathcal{R}%
}^{\prime }\ast f\left( z^{\prime },u^{\prime }\right) \right\vert ^{2}\chi
_{\mathcal{R}}\left( z,u\right) \right\} ^{\frac{1}{2}}\right\Vert
_{L^{p}\left( \mathbb{H}^{n}\right) }
\end{equation*}%
\begin{equation*}
\approx \left\Vert \left\{ \sum_{\mathcal{Q}\in \mathsf{Q}}\inf_{\left(
z^{\prime },u^{\prime }\right) \in \mathcal{Q}}\left\vert \phi _{\mathcal{Q}%
}^{\prime }\ast f\left( z^{\prime },u^{\prime }\right) \right\vert ^{2}\chi
_{\mathcal{Q}}\left( z,u\right) +\sum_{\mathcal{R}\in \mathsf{R}%
_{vert}}\inf_{\left( z^{\prime },u^{\prime }\right) \in \mathcal{R}%
}\left\vert \phi _{\mathcal{R}}^{\prime }\ast f\left( z^{\prime },u^{\prime
}\right) \right\vert ^{2}\chi _{\mathcal{R}}\left( z,u\right) \right\} ^{%
\frac{1}{2}}\right\Vert _{L^{p}\left( \mathbb{H}^{n}\right) }.
\end{equation*}
\end{theorem}

The Plancherel-P\^{o}lya type inequalities in Theorem \ref{P-P} will prove
useful in establishing properties of the \emph{wavelet} Littlewood-Paley $g$%
-function 
\begin{equation*}
g_{flag}(f)(z,u)=\left\{ \sum_{\mathcal{Q}\in \mathsf{Q}}\left\vert \psi _{%
\mathcal{Q}}^{\prime }\ast f\left( z_{\mathcal{Q}},u_{\mathcal{Q}}\right)
\right\vert ^{2}\chi _{\mathcal{Q}}\left( z,u\right) +\sum_{\mathcal{R}\in 
\mathsf{R}_{vert}}\left\vert \psi _{\mathcal{R}}^{\prime }\ast f\left( z_{%
\mathcal{R}},u_{\mathcal{R}}\right) \right\vert ^{2}\chi _{\mathcal{R}%
}\left( z,u\right) \right\} ^{{\frac{{1}}{{2}}}},
\end{equation*}%
where we are using the notation of Theorem \ref{discretethm} and Theorem \ref%
{P-P}.

We can now give a precise definition of the \emph{flag} Hardy spaces.

\begin{definition}
\label{defflagHardy}Let $0<p<\infty $. Then for $M$ sufficiently large
depending on $n$ and $p$ we define the \emph{flag} Hardy space $%
H_{flag}^{p}\left( \mathbb{H}^{n}\right) $ on the Heisenberg group by%
\begin{equation*}
H_{flag}^{p}\left( \mathbb{H}^{n}\right) =\left\{ f\in \mathcal{M}%
_{flag}^{M+\delta }\left( \mathbb{H}^{n}\right) ^{\prime }:g_{flag}(f)\in
L^{p}\left( \mathbb{H}^{n}\right) \right\} ,
\end{equation*}%
and for $f\in H_{flag}^{p}\left( \mathbb{H}^{n}\right) $ we set 
\begin{equation}
\Vert f\Vert _{H_{flag}^{p}}=\Vert g_{flag}(f)\Vert _{p}.  \label{H^p norm}
\end{equation}
\end{definition}

\begin{remark}
We can take $M$ in Definition \ref{defflagHardy} to satisfy 
\begin{equation*}
M\geq M_{n,p}\equiv \left( 2n+2\right) \left[ \frac{2}{p}-1\right] +1.
\end{equation*}%
We have not computed the optimal value of $M_{n,p}$.
\end{remark}

It is easy to see using Theorem \ref{P-P} that the Hardy space $H_{flag}^{p}$
in Definition \ref{defflagHardy} is well defined and that the $H_{flag}^{p}$
norm of $f$ is equivalent to the $L^{p}$ norm of $g_{flag}\left( f\right) $.
By use of the Plancherel-P\^{o}lya-type inequalities, we will prove the
boundedness of flag singular integrals on $H_{flag}^{p}$ below.

\subsection{Boundedness of singular integrals and Marcinkiewicz multipliers}

Our main theorem is the $H_{flag}^{p}\rightarrow H_{flag}^{p}$ boundedness
of flag singular integrals.

\begin{theorem}
\label{flagintbounded}Suppose that $T$ is a flag singular integral with the
kernel $K(z,u)$ as in Definition \ref{defflag}. Then $T$ is bounded on $%
H_{flag}^{p}$ for $0<p\leq 1$. Namely, for all $0<p\leq 1$ there exists a
constant $C_{p,n}$ such that 
\begin{equation*}
\Vert Tf\Vert _{H_{flag}^{p}}\leq C_{p,n}\Vert f\Vert _{H_{flag}^{p}}.
\end{equation*}
\end{theorem}

To obtain the $H_{flag}^{p}\rightarrow L^{p}$ boundedness of flag singular
integrals, we prove the following general result:

\begin{theorem}
\label{flagintHtoL}Let $0<p\leq 1$. If $T$ is a linear operator which is
bounded simultaneously on $L^{2}(\mathbb{R}^{2n+1})$ and $H_{flag}^{p}(%
\mathbb{H}^{n}),$ then $T$ can be extended to a bounded operator from $%
H_{flag}^{p}(\mathbb{H}^{n})$ to $L^{p}(\mathbb{R}^{2n+1})$.
\end{theorem}

\begin{remark}
From the proof given in the next part of the paper, we see that this result
holds in a larger setting, which includes the classical one-parameter and
product Hardy spaces and the Hardy spaces on spaces of homogeneous type.
Thus this provides an alternative approach, albeit requiring more
smoothness, to that using R. Fefferman's rectangle atom criterion on
boundedness \cite{F1}, andJourn\'{e}'s geometric lemma (see \cite{J1}, \cite%
{J2} and \cite{P}).
\end{remark}

In particular, for flag singular integrals we can deduce the following:

\begin{corollary}
\label{Hp to Lp}Let $T$ be a flag singular integral as in Theorem 1.4. Then $%
T$ is bounded from $H_{flag}^{p}(\mathbb{H}^{n})$ to $L^{p}(\mathbb{R}%
^{2n+1})$ for $0<p\leq 1$.
\end{corollary}

\begin{remark}
The conclusions of both Theorem \ref{flagintbounded} and Corollary \ref{Hp
to Lp}, persist if we\ only require the moment and smoothness conditions on
the flag kernel in Definition \ref{defflag} to hold for $\left\vert \alpha
\right\vert ,\beta \leq N_{n,p}$ where $N_{n,p}<\infty $ is taken
sufficiently large.
\end{remark}

As a consequence, we can extend the Marcinkiewicz multiplier theorem in \cite%
{MRS} (see Lemma 2.1 there) to flag Hardy spaces for $0<p\leq 1$. To
describe this extension, recall the standard sub-Laplacian $\mathcal{L}$ on
the Heisenberg group 
\begin{equation*}
\mathbb{H}^{n}=\mathbb{C}^{n}\times \mathbb{R}=\left\{ \left( z,t\right)
:z=\left( z_{j}\right) _{j=1}^{n},\ z_{j}=x_{j}+iy_{j}\in \mathbb{C},\ t\in 
\mathbb{R}\right\} ,
\end{equation*}%
defined by%
\begin{equation*}
\mathcal{L}\equiv -\sum_{j=1}^{n}\left( X_{j}^{2}+Y_{j}^{2}\right) ,\ \ \ \
\ X_{j}=\frac{\partial }{\partial x_{j}}+2y_{j}\frac{\partial }{\partial t}%
,\ Y_{j}=\frac{\partial }{\partial y_{j}}-2x_{j}\frac{\partial }{\partial t}.
\end{equation*}%
The operators $\mathcal{L}$ and $T=\frac{\partial }{\partial t}$ commute,
and so do their spectral measures $dE_{1}\left( \xi \right) $ and $%
dE_{2}\left( \eta \right) $. Given a bounded function $m\left( \xi ,\eta
\right) $ on $\mathbb{R}_{+}\times \mathbb{R}$, define the multiplier
operator $m\left( \mathcal{L},iT\right) $ on $L^{2}\left( \mathbb{H}%
^{n}\right) $ by%
\begin{equation*}
m\left( \mathcal{L},iT\right) =\int \int_{\mathbb{R}_{+}\times \mathbb{R}%
}m\left( \xi ,\eta \right) dE_{1}\left( \xi \right) dE_{2}\left( \eta
\right) .
\end{equation*}%
Then $m\left( \mathcal{L},iT\right) $ is automatically bounded on $%
L^{2}\left( \mathbb{H}^{n}\right) $, and if we impose Marcinkiewicz
conditions on the multiplier, we obtain boundedness on flag Hardy spaces,
this despite the fact that $m$ is invariant under a two-parameter family of
dilations $\delta _{\left( s,t\right) }$ which are group automorphisms only
when $s=t^{2}$.

\begin{corollary}
\label{Marcin}Let $0<p\leq 1$, and suppose that $m\left( \xi ,\eta \right) $
is a bounded function defined on $\mathbb{R}_{+}\times \mathbb{R}$
satisfying the Marcinkiewicz conditions%
\begin{equation*}
\left\vert \left( \xi \partial _{\xi }\right) ^{\alpha }\left( \eta \partial
_{\eta }\right) ^{\beta }m\left( \xi ,\eta \right) \right\vert \leq
C_{\alpha ,\beta },
\end{equation*}%
for all $\left\vert \alpha \right\vert ,\beta \leq N_{n,p}$, where $%
N_{n,p}<\infty $ is taken sufficiently large. Then $m\left( \mathcal{L}%
,iT\right) $ is a bounded operator on $H_{flag}^{p}\left( \mathbb{H}%
^{n}\right) $ for $0<p\leq 1$.
\end{corollary}

The Corollary follows from the results above together with Theorem 3.1 in 
\cite{MRS}, which shows that the kernel $K\left( z,u\right) $ of a
Marcinkiewicz multiplier $m\left( \mathcal{L},iT\right) $ satisfies the
conditions defining a flag convolution kernel in Definition \ref{defflag}.

\subsection{Carleson measures and duality}

To study the dual space of $H_{flag}^{p},$ we introduce the Carleson measure
space $CMO_{flag}^{p}$, so-called because of the example due to L. Carleson
in \cite{Car}.

\begin{notation}
It will often be convenient from now on to bundle the set $\mathsf{Q}$ of
all dyadic cubes and the set $\mathsf{R}_{vert}$ of all vertical dyadic
rectangles into a single set 
\begin{equation*}
\mathsf{R}_{+}=\mathsf{Q}\cup \mathsf{R}_{vert}
\end{equation*}%
consisting of all dyadic cubes and all vertical dyadic rectangles. We also
write 
\begin{equation*}
\psi _{\mathcal{R}}=\left\{ 
\begin{array}{ccc}
\psi _{\mathcal{Q}}^{\prime } & \text{ if } & \mathcal{R}=\mathcal{Q}\in 
\mathsf{Q} \\ 
\psi _{\mathcal{R}}^{\prime } & \text{ if } & \mathcal{R}\in \mathsf{R}%
_{vert}%
\end{array}%
\right. .
\end{equation*}
\end{notation}

\begin{definition}
Let $\psi _{j,k}$ be as in (\ref{defpsijk}) with notation as above. We say
that $f\in CMO_{flag}^{p}$ if $f\in \mathcal{M}_{flag}^{M+\delta }\left( 
\mathbb{H}^{n}\right) ^{\prime }$ and the norm $\Vert f\Vert
_{CMO_{flag}^{p}}$ is finite where%
\begin{equation*}
\Vert f\Vert _{CMO_{flag}^{p}}\equiv \sup_{\Omega }\left\{ \frac{1}{%
\left\vert \Omega \right\vert ^{\frac{2}{p}-1}}\sum_{\mathcal{R}\in \mathsf{R%
}_{+}}\int_{\Omega }\sum_{\mathcal{R}\subset \Omega }\left\vert \psi _{%
\mathcal{R}}\ast f\left( z,u\right) \right\vert ^{2}\chi _{\mathcal{R}%
}\left( z,u\right) dzdu\right\} ^{\frac{1}{2}}
\end{equation*}%
for all open sets $\Omega $ in $\mathbb{H}^{n}$ with finite measure.
\end{definition}

Note that the Carleson measure condition is used with the implicit
two-parameter structure in $CMO_{flag}^{p}$. When $p=1$, we denote the space 
$CMO_{flag}^{1}$ as usual by $BMO_{flag}$. To see that the space $%
CMO_{flag}^{p}$ is well defined, one needs to show that the definition of $%
CMO_{flag}^{p}$ is independent of the choice of the component functions $%
\psi _{j,k}.$ This can be proved just as for the Hardy space $H_{flag}^{p}$,
using the following Plancherel-P\^{o}lya-type inequality.

\begin{theorem}
\label{P-PCar}Suppose $\psi ,\phi $ satisfy the conditions as in Theorem \ref%
{P-P}. Then for $f\in \mathcal{M}_{flag}^{M+\delta }\left( \mathbb{H}%
^{n}\right) ^{\prime },$ 
\begin{equation*}
\sup_{\Omega }\left\{ {\frac{{1}}{{|\Omega |^{{\frac{{2}}{{p}}}-1}}}}\sum_{%
\mathcal{R}\in \mathsf{R}_{+}}\int_{\Omega }\sum_{\mathcal{R}\subset \Omega
}\sup_{\left( z,u\right) \in \mathcal{R}}\left\vert \psi _{\mathcal{R}}\ast
f\left( z,u\right) \right\vert ^{2}\left\vert \mathcal{R}\right\vert
\right\} ^{\frac{{1}}{{2}}}\approx
\end{equation*}%
\begin{equation*}
\sup_{\Omega }\left\{ {\frac{{1}}{{|\Omega |^{{\frac{{2}}{{p}}}-1}}}}\sum_{%
\mathcal{R}\in \mathsf{R}_{+}}\int_{\Omega }\sum_{\mathcal{R}\subset \Omega
}\inf_{\left( z,u\right) \in \mathcal{R}}\left\vert \phi _{\mathcal{R}}\ast
f\left( z,u\right) \right\vert ^{2}\left\vert \mathcal{R}\right\vert
\right\} ^{\frac{{1}}{{2}}},
\end{equation*}%
where $\Omega $ ranges over all open sets in $\mathbb{H}^{n}$ with finite
measure.
\end{theorem}

To show that $CMO_{flag}^{p}$ is the dual of $H_{flag}^{p}$, we introduce
appropriate sequence spaces.

\begin{definition}
Let $s^{p}$ be the collection of all sequences $s=\{s_{\mathcal{R}}\}_{%
\mathcal{R}\in \mathsf{R}_{+}}$ such that 
\begin{equation*}
\Vert s\Vert _{s^{p}}=\left\Vert \left\{ \sum_{\mathcal{R}\in \mathsf{R}%
_{+}}|s_{\mathcal{R}}|^{2}|\left\vert \mathcal{R}\right\vert ^{-1}\chi _{%
\mathcal{R}}\right\} ^{\frac{{1}}{{2}}}\right\Vert _{L^{p}\left( \mathbb{H}%
^{n}\right) }<\infty .
\end{equation*}%
Let $c^{p}$ be the collection of all sequences $s=\{s_{\mathcal{R}}\}$ such
that 
\begin{equation*}
\Vert s\Vert _{c^{p}}=\sup_{\Omega }\left\{ {\frac{{1}}{{|\Omega |^{{\frac{{2%
}}{{p}}}-1}}}}\sum_{\mathcal{R}\in \mathsf{R}_{+}}\sum_{\mathcal{R}\subset
\Omega }|s_{\mathcal{R}}|^{2}\right\} ^{\frac{{1}}{{2}}}<\infty ,
\end{equation*}%
where $\Omega $ ranges over all open sets in $\mathbb{H}^{n}$ with finite
measure.
\end{definition}

We point out that just certain of the dyadic rectangles are used in $s^{p}$
and $c^{p}$ and these choices reflect the implicit two-parameter structure.
Next, we obtain the following duality theorem for sequence spaces.

\begin{theorem}
\label{sequenceduality}Let $0<p\leq 1$. Then we have $(s^{p})^{\ast }=c^{p}$%
. More precisely, the map which sends $s=\{s_{\mathcal{R}}\}$ to $%
\left\langle s,t\right\rangle \equiv \sum\limits_{\mathcal{R}}s_{\mathcal{R}}%
\overline{t}_{\mathcal{R}}$ defines a continuous linear functional on $s^{p}$
with operator norm $\Vert t\Vert _{(s^{p})^{\ast }}\approx \Vert t\Vert
_{c^{p}}$, and moreover, every $\ell \in (s^{p})^{\ast }$ is of this form
for some $t\in c^{p}.$
\end{theorem}

When $p=1,$ this theorem in the one-parameter setting on $\mathbb{R}^{n}$
was proved in M. Frazier and B. Jawerth \cite{FJ}. The proof given in \cite%
{FJ} depends on estimates of certain distribution functions, which seem to
be difficult to apply to the two-parameter case. For all $0<p\leq 1$ we give
a simple and more constructive proof of Theorem \ref{sequenceduality}, which
uses a stopping time argument for sequence spaces. Theorem \ref%
{sequenceduality}, together with the discrete Calder\'{o}n reproducing
formula and the Plancherel-P\^{o}lya-type inequalities, yields the duality
theorem for $H_{flag}^{p}$.

\begin{theorem}
\label{flag duality}Let $0<p\leq 1$. Then%
\begin{equation*}
(H_{flag}^{p})^{\ast }=CMO_{F}^{p}.
\end{equation*}%
More precisely, if $g\in CMO_{flag}^{p}$, the map $\ell _{g}$ given by $\ell
_{g}(f)=\left\langle f,g\right\rangle $, defined initially for $f\in 
\mathcal{M}_{flag}^{M+\delta }\left( \mathbb{H}^{n}\right) $, extends to a
continuous linear functional on $H_{flag}^{p}$ with $\Vert \ell _{g}\Vert
\approx \Vert g\Vert _{CMO_{flag}^{p}}.$ Conversely, for every $\ell \in
(H_{flag}^{p})^{\ast }$ there exists some $g\in CMO_{flag}^{p}$ so that $%
\ell =\ell _{g}.$ In particular, $(H_{flag}^{1})^{\ast }=BMO_{flag}$.
\end{theorem}

As a consequence of the duality of $H_{flag}^{1}$ and $BMO_{flag}$, together
with the $H_{flag}^{1}$-boundedness of flag singular integrals, we obtain
the $BMO_{flag}$-boundedness of flag singular integrals. Furthermore, we
will see that $L^{\infty }\subseteq BMO_{flag}$ and hence the $L^{\infty
}\rightarrow BMO_{flag}$ boundedness of flag singular integrals is also
obtained. These provide the endpoint results of those in \cite{MRS} and \cite%
{NRS}, and can be summarized as follows:

\begin{theorem}
\label{BMO bound}Suppose that $T$ is a flag singular integral with kernel as
in Definition \ref{defflag}. Then $T$ is bounded on $BMO_{flag}.$ Moreover,
there exists a constant $C$ such that 
\begin{equation*}
\Vert T(f)\Vert _{BMO_{flag}}\leq C\Vert f\Vert _{BMO_{flag}}.
\end{equation*}
\end{theorem}

\subsection{C-Z decompositions and interpolation}

Now we give the Calder\'{o}n-Zygmund decomposition and interpolation
theorems for flag Hardy spaces. We note that $H_{flag}^{p}\left( \mathbb{H}%
^{n}\right) =L^{p}(\mathbb{R}^{2n+1})$ for $1<p<\infty $ by Theorem \ref{p>1}%
.

\begin{theorem}
\label{CZ decomposition}(Calder\'{o}n-Zygmund decomposition for flag Hardy
spaces) Let $0<p_{2}\leq 1$, $p_{2}<p<p_{1}<\infty $ and let $\alpha >0$ be
given and suppose $f\in H_{flag}^{p}\left( \mathbb{H}^{n}\right) $. Then we
can write%
\begin{equation*}
f=g+b,
\end{equation*}%
where $g\in H_{flag}^{p_{1}}(\mathbb{H}^{n})$ with $p<p_{1}<\infty $ and $%
b\in H_{flag}^{p_{2}}(\mathbb{H}^{n})$ with $0<p_{2}<p$ such that 
\begin{equation*}
||g||_{H_{flag}^{p_{1}}}^{p_{1}}\leq C\alpha
^{p_{1}-p}||f||_{H_{flag}^{p}}^{p}\text{ and }%
||b||_{H_{flag}^{p_{2}}}^{p_{2}}\leq C\alpha
^{p_{2}-p}||f||_{H_{flag}^{p}}^{p},
\end{equation*}%
where $C$ is an absolute constant.
\end{theorem}

\begin{theorem}
\label{interpolation}(Interpolation theorem on flag Hardy spaces) Let $%
0<p_{2}<p_{1}<\infty $ and let $T$ be a linear operator which is bounded
from $H_{flag}^{p_{2}}$ to $L^{p_{2}}$ and bounded from $H_{flag}^{p_{1}}$
to $L^{p_{1}}$. Then $T$ is bounded from $H_{flag}^{p}$ to $L^{p}$ for all $%
p_{2}<p<p_{1}$. Similarly, if $T$ is bounded on $H_{flag}^{p_{2}}$ and $%
H_{flag}^{p_{1}}$, then $T$ is bounded on $H_{flag}^{p}$ for all $%
p_{2}<p<p_{1}$.
\end{theorem}

\begin{remark}
Combining Theorem \ref{interpolation} with Corollary \ref{Marcin} recovers
the $L^{p}$ boundedness of Marcinkiewicz multipliers in \cite{MRS} (but not
the sharp versions in \cite{MRS2}).
\end{remark}

We point out that the Calder\'{o}n-Zygmund decomposition in pure product
domains for all $L^{p}$ functions ($1<p<2$) into $H^{1}$ and $L^{2}$
functions, as well as the corresponding interpolation theorem, was
established by A. Chang and R. Fefferman (\cite{CF1}, \cite{CF2}).

\part{Proofs of results}

The second part of this paper contains the proofs of the results stated in
the first part, and is organized as follows.

\begin{enumerate}
\item In Section \ref{s Lp estimates}, we establish $L^{p}$ estimates for
the multi-parameter Littlewood-Paley $g$-function when $1<p<\infty $, and
prove Theorems \ref{p>1} and \ref{flagp>1}.

\item In Section \ref{s developing} we show that the Calder\'{o}n
reproducing formula holds on the flag molecular test function space $%
\mathcal{M}_{flag}^{M+\delta }$ and its dual space $\left( \mathcal{M}%
_{flag}^{M+\delta }\right) ^{\prime }$, and then prove the almost
orthogonality estimates and establish the wavelet Calder\'{o}n reproducing
formula on $\mathcal{M}_{flag}^{M+\delta }$ and $\left( \mathcal{M}%
_{flag}^{M+\delta }\right) ^{\prime }$ in Theorem \ref{discretethm}. Some
estimates are established for the strong maximal function, and together with
the wavelet Calder\'{o}n reproducing formula, we then derive the Plancherel-P%
\^{o}lya-type inequalities in Theorem \ref{P-P}.

\item Section \ref{s boundedness} gives a general result for bounding the $%
L^{p}$ norm of the function by its $H_{flag}^{p}$ norm in Theorem \ref{L<H},
and then proves the $H_{flag}^{p}$ boundedness of flag singular integrals
for all $0<p\leq 1$ in Theorem \ref{flagintbounded}. The boundedness from $%
H_{flag}^{p}$ to $L^{p}$ for all $0<p\leq 1$ for the flag singular integral
operators, Theorem \ref{flagintHtoL}, is thus a consequence of Theorem \ref%
{flagintbounded} and Theorem \ref{L<H}.

\item Duality theory for the Hardy space $H_{flag}^{p}$ is then established
in Section \ref{s duality} along with the boundedness of flag singular
integral operators on $BMO_{flag}$. The proofs of Theorems \ref{P-PCar}, \ref%
{sequenceduality}, \ref{flag duality} and \ref{BMO bound} will all be given
in Section \ref{s duality}.

\item In Section \ref{s CZ decomposition}, we prove the Calder\'{o}n-Zygmund
decomposition in the flag two-parameter setting, Theorem \ref{CZ
decomposition}, and then derive an interpolation result, Theorem \ref%
{interpolation}.

\item In Section \ref{s embeddings} we prove the embeddings in Lemma \ref%
{containments} relating the flag molecular spaces $\mathcal{M}%
_{flag}^{M+\delta ,M_{1},M_{2}}$ and $\mathsf{M}_{F}^{M+\delta ,M_{1},M_{2}}$%
, and prove the embedding of quotient spaces.

\item In Section \ref{s counterexample} we show that flag singular integrals
are not in general bounded from the classical one-parameter Hardy space $%
H^{1}\left( \mathbb{H}^{n}\right) $ on the Heisenberg group to $L^{1}\left( 
\mathbb{H}^{n}\right) $.
\end{enumerate}

\section{L$^{p}$ estimates for the Littlewood-Paley square function\label{s
Lp estimates}}

The purpose of this section is to show that the $L^{p}$ norm of $f$ is
equivalent to the $L^{p}$ norm of $g_{flag}(f)$ when $1<p<\infty $. This was
shown in Proposition 4.1 of \cite{MRS2} for a function $g_{flag}(f)$ only
slightly different than that used here. Our proof is similar in spirit to
that in \cite{MRS2}.

\bigskip

\begin{proof}
\textbf{(}of Theorem \ref{p>1})\textbf{:} The proof is similar to that in
the pure product case given in \cite{FS} and follows from iteration and
standard vector-valued Littlewood-Paley inequalities. To see this, define $%
L^{p}\left( \mathbb{H}^{n}\right) \ni f\rightarrow F\in H=\ell ^{2}$ by $%
F(z,u)=\{\psi _{j}^{(1)}\ast f(z,u)\}$ so that 
\begin{equation*}
\Vert F\Vert _{H}=\{\sum\limits_{j}|\psi _{j}^{(1)}\ast f(z,u)|^{2}\}^{{%
\frac{{1}}{{2}}}}.
\end{equation*}%
For $z$ fixed, set 
\begin{equation*}
{\widetilde{g}}(F)(z,u)=\{\sum\limits_{k}\Vert \psi _{k}^{(2)}\ast
_{2}F(z,\cdot )(u)\Vert _{H}^{2}\}^{{\frac{{1}}{{2}}}}.
\end{equation*}%
It is then easy to see that ${\widetilde{g}}(F)(z,u)=g_{flag}(f)(z,u).$ For $%
z$ fixed, by the vector-valued Littlewood-Paley inequality, 
\begin{equation*}
\int\limits_{\mathbb{R}}{\widetilde{g}}(F)^{p}(z,u)du\leq C\int\limits_{%
\mathbb{R}}\Vert F\Vert _{H}^{p}du.
\end{equation*}%
However, $\Vert F\Vert _{H}^{p}=\{\sum\limits_{j}|\psi _{j}^{(1)}\ast
f(z,u)|^{2}\}^{{\frac{{p}}{{2}}}},$ so integrating with respect to $z$
together with the standard Littlewood-Paley inequality yields 
\begin{equation*}
\int\limits_{\mathbb{C}^{n}}\int\limits_{\mathbb{R}}g_{flag}(f)^{p}(z,u)dzdu%
\leq C\int\limits_{\mathbb{C}^{n}}\int\limits_{\mathbb{R}}\{\sum\limits_{j}|%
\psi _{j}^{(1)}\ast f(z,u)|^{2}\}^{{\frac{{p}}{{2}}}}dzdu\leq C\Vert f\Vert
_{L^{p}\left( \mathbb{H}^{n}\right) }^{p},
\end{equation*}%
which shows that $||g_{flag}(f)||_{p}\leq C||f||_{p}$.

The proof of the estimate $||f||_{p}\leq C||g_{flag}(f)||_{p}$ is a routine
duality argument using the Calder\'{o}n reproducing formula on $L^{2}(%
\mathbb{H}^{n})$, for all $f\in L^{2}\cap L^{p}$, $g\in L^{2}\cap
L^{p^{\prime }}$ and $\frac{1}{p}+\frac{1}{p^{\prime }}=1$, and the
inequality $||g_{flag}(f)||_{p}\leq C||f||_{p}$, which was just proved. This
completes the proof of Theorem \ref{p>1}.
\end{proof}

\bigskip

Let $\psi ^{(1)}\in \mathcal{S}(\mathbb{H}^{n})$ as in Theorem 1 of \cite{GM}
be supported in the unit ball in $\mathbb{H}^{n}$ and $\psi ^{(2)}\in 
\mathcal{S}(\mathbb{R})$ be supported in the unit ball of $\mathbb{R}$ and
satisfy 
\begin{equation*}
\int_{0}^{\infty }|\widehat{\psi ^{(2)}}(t\eta )|^{4}\frac{dt}{t}=1
\end{equation*}%
for all $\eta \in \mathbb{R}\backslash \{0\}.$ We define $\psi ^{\natural
}(z,u,v)=\psi ^{(1)}(z,u)\psi ^{(2)}(v).$ Set $\psi
_{s}^{(1)}(z,u)=s^{-n-2}\psi ^{(1)}(\frac{z}{s},\frac{u}{s^{2}})$ and $\psi
_{t}^{(2)}(v)=t^{-1}\psi (\frac{z}{t})$ and 
\begin{equation*}
\psi _{s,t}(z,u)=\int_{\mathbb{R}}\psi _{s}^{(1)}(z,u-v)\psi _{t}^{(2)}(v)dv.
\end{equation*}%
Repeating the proof of Theorem \ref{p>1}, we can get for $1<p<\infty $ 
\begin{equation*}
\Vert \{\int_{0}^{\infty }\int_{0}^{\infty }|\psi _{s,t}\ast f(z,u)|^{2}{%
\frac{{dt}}{{t}}}{\frac{{ds}}{{s}}}\}^{{\frac{{1}}{{2}}}}\Vert _{p}\leq
C\Vert f\Vert _{p},
\end{equation*}%
and 
\begin{equation}
\Vert f\Vert _{p}\approx \Vert \{\int_{0}^{\infty }\int_{0}^{\infty }|\psi
_{s,t}\ast \psi _{s,t}\ast f(z,y)|^{2}{\frac{{dt}}{{t}}}{\frac{{ds}}{{s}}}%
\}^{{\frac{{1}}{{2}}}}\Vert _{p}.  \label{square}
\end{equation}

The $L^{p}$ boundedness of flag singular integrals for $1<p<\infty $ is then
an easy consequence of Theorem \ref{p>1}. This theorem was originally
obtained in \cite{MRS} using a different proof that involved the method of
transference.

\begin{theorem}
\label{flagp>1}Suppose that $T$ is a flag singular integral defined on $%
\mathbb{H}^{n}$ with the flag kernel $K(z,u)$ as in Definition \ref{defflag}
above. Then $T$ is bounded on $L^{p}$ for $1<p<\infty $. Moreover, there
exists a constant $C$ depending on $p$ such that for $f\in L^{p}$, 
\begin{equation*}
\Vert Tf\Vert _{p}\leq C\Vert f\Vert _{p},\ \ \ \ \ 1<p<\infty .
\end{equation*}
\end{theorem}

\begin{proof}
\textbf{(}of Theorem \ref{flagp>1})\textbf{:} We may first assume that $K$
is integrable function and shall prove the $L^{p}$ boundedness of $T$ is
independent of the $L^{1}$ norm of $K$. The conclusion for general $K$ then
follows by an argument used in \cite{MRS}. For all $f\in L^{p}$, by (\ref%
{square}), 
\begin{equation}
\Vert T(f)\Vert _{p}\leq C\left\Vert \{\int_{0}^{\infty }\int_{0}^{\infty
}|\psi _{s,t}\ast \psi _{s,t}\ast K\ast f|^{2}{\frac{{dt}}{{t}}}{\frac{{ds}}{%
{s}}}\}^{{\frac{{1}}{{2}}}}\right\Vert _{p}.  \label{Tsquare}
\end{equation}

Now we claim the following estimate: for $f\in L^{p}$, 
\begin{equation}
|\psi _{s,t}\ast K\ast f(z,u)|\leq CM_{S}(f)(z,u),  \label{claimK}
\end{equation}%
where $C$ is a constant which is independent of the $L^{1}$ norm of $K$ and $%
M_{S}(f)$ is the strong maximal function of $f$ defined in (\ref{def strong
max}).

Assuming (\ref{claimK}) for the moment, we obtain from (\ref{Tsquare}) that 
\begin{equation*}
\Vert Tf\Vert _{p}\leq C\left\Vert \{\int_{0}^{\infty }\int_{0}^{\infty
}(M_{S}(\psi _{s,t}\ast f))^{2}{\frac{{dt}}{{t}}}{\frac{{ds}}{{s}}}\}^{{%
\frac{{1}}{{2}}}}\right\Vert _{p}\leq C\Vert f\Vert _{p},
\end{equation*}%
where the last inequality follows from the Fefferman-Stein vector-valued
maximal inequality.

We now turn to the claim (\ref{claimK}). This follows from dominating $%
\left\vert \psi _{s,t}\ast K\ast f\right\vert $ by a product Poisson
integral $\mathbb{P}_{prod}f$, and then dominating the product Poisson
integral $\mathbb{P}_{prod}f$ by the strong maximal function $M_{S}f$. The
arguments are familiar and we leave them to the reader.
\end{proof}

\section{Developing the wavelet Calder\'{o}n reproducing formula\label{s
developing}}

In this section, we develop the wavelet Calder\'{o}n reproducing formula and
prove the Plancherel-P\^{o}lya-type inequalities on test function spaces.
These are the main tools used in establishing the theory of Hardy spaces
associated with the flag dilation structure. In order to establish the
wavelet Calder\'{o}n reproducing formula and the Plancherel-P\^{o}lya-type
inequalities, we use the continuous version of the Calder\'{o}n reproducing
formula on test function spaces and certain almost orthogonality estimates.

\bigskip

We now start the relatively long proof of Theorem \ref{discretethm},
beginning with the Calder\'{o}n reproducing formula in (\ref{CRF st}) that
holds for $f\in L^{2}\left( \mathbb{H}^{n}\right) $ and converges in $%
L^{2}\left( \mathbb{H}^{n}\right) $. For any given $\alpha >0$ we discretize
it as follows:%
\begin{eqnarray*}
f\left( z,u\right) &=&\int_{0}^{\infty }\int_{0}^{\infty }\check{\psi}%
_{s,t}\ast _{\mathbb{H}^{n}}\psi _{s,t}\ast _{\mathbb{H}^{n}}f\left(
z,u\right) \frac{ds}{s}\frac{dt}{t} \\
&=&\sum_{j,k\in \mathbb{Z}}\int_{2^{-\alpha \left( j+1\right) }}^{2^{-\alpha
j}}\int_{2^{-2\alpha \left( k+1\right) }}^{2^{-2\alpha k}}\check{\psi}%
_{s,t}\ast \psi _{s,t}\ast f\left( z,u\right) \frac{dt}{t}\frac{ds}{s} \\
&=&c_{\alpha }\sum_{j\leq k}\check{\psi}_{j,k}\ast \psi _{j,k}\ast f\left(
z,u\right) +c_{\alpha }\sum_{j>k}\check{\psi}_{j,k}\ast \psi _{j,k}\ast
f\left( z,u\right) \\
&&+\sum_{j,k\in \mathbb{Z}}\int_{2^{-\alpha \left( j+1\right) }}^{2^{-\alpha
j}}\int_{2^{-2\alpha \left( k+1\right) }}^{2^{-2\alpha k}}\left\{ \check{\psi%
}_{s,t}\ast \psi _{s,t}-\check{\psi}_{j,k}\ast \psi _{j,k}\right\} \ast
f\left( z,u\right) \frac{dt}{t}\frac{ds}{s} \\
&=&T_{\alpha }^{\left( 1\right) }f\left( z,u\right) +T_{\alpha }^{\left(
2\right) }f\left( z,u\right) +R_{\alpha }f\left( z,u\right) ,
\end{eqnarray*}%
where%
\begin{eqnarray*}
\psi _{j,k} &=&\psi _{2^{-\alpha j},2^{-2\alpha k}}, \\
c_{\alpha } &=&\int_{2^{-\alpha \left( j+1\right) }}^{2^{-\alpha
j}}\int_{2^{-2\alpha \left( k+1\right) }}^{2^{-2\alpha k}}\frac{dt}{t}\frac{%
ds}{s}=\ln \frac{2^{-\alpha j}}{2^{-\alpha \left( j+1\right) }}\ln \frac{%
2^{-2\alpha k}}{2^{-2\alpha \left( k+1\right) }}=2\left( \alpha \ln 2\right)
^{2}.
\end{eqnarray*}

\begin{notation}
We have relabeled $\psi _{2^{-\alpha j},2^{-2\alpha k}}$ as simply $\psi
_{j,k}$ when we replace integrals $\int_{0}^{\infty }\int_{0}^{\infty }\frac{%
ds}{s}\frac{dt}{t}$ by sums $\sum_{j,k\in \mathbb{Z}}$. This abuse of
notation should not cause confusion as we will always use $j,k,j^{\prime
},k^{\prime }$ as subscripts for the discrete components $\psi _{j,k}$,
while we always use $s,t,s^{\prime },t^{\prime }$ as subscripts for the
continuous components $\psi _{s,t}$. Note however that directions are \emph{%
reversed} in passing from $s,t\in \left( 0,\infty \right) $ to $j,k\in 
\mathbb{Z}$, in the sense that $s=2^{-\alpha j}$ and $t=2^{-2\alpha k}$
decrease as $j$ and $k$ increase.
\end{notation}

To continue, we choose a large positive integer $N$ to be fixed later. We
decompose the first term $T_{\alpha }^{\left( 1\right) }f\left( z,u\right) $
by writing the Heisenberg group $\mathbb{H}^{n}$ as a pairwise disjoint
union of dyadic cubes $\mathcal{Q}$ of side length $2^{-\alpha \left(
j+N\right) }$, i.e. 
\begin{equation*}
\mathcal{Q}\in \mathcal{R}\left( 2^{-\alpha \left( j+N\right) }\times
2^{-2\alpha \left( j+N\right) }\right) .
\end{equation*}%
We decompose the second term $T_{\alpha }^{\left( 2\right) }f\left(
z,u\right) $ by writing the Heisenberg group $\mathbb{H}^{n}$ as a pairwise
disjoint union of dyadic rectangles $\mathcal{R}$ of dimension $2^{-\alpha
\left( j+N\right) }\times 2^{-2\alpha \left( k+N\right) }$, i.e. $\mathcal{R}%
\in \mathcal{R}\left( 2^{-\alpha \left( j+N\right) }\times 2^{-2\alpha
\left( k+N\right) }\right) $. Recall that%
\begin{eqnarray*}
\mathsf{R}\left( j,k\right) &\equiv &\mathcal{R}\left( 2^{-\alpha \left(
j+N\right) }\times 2^{-2\alpha \left( k+N\right) }\right) , \\
\mathsf{Q}\left( j\right) &\equiv &\mathcal{R}\left( 2^{-\alpha \left(
j+N\right) }\times 2^{-2\alpha \left( j+N\right) }\right) ,
\end{eqnarray*}%
and that $\left( z_{\mathcal{Q}},u_{\mathcal{Q}}\right) $ is any \emph{fixed}
point in the cube $\mathcal{Q}\in \mathsf{Q}\left( j\right) $; and that $%
\left( z_{\mathcal{R}},u_{\mathcal{R}}\right) $ is any \emph{fixed} point in
the rectangle $\mathcal{R}\in \mathsf{R}\left( j,k\right) $.

We further discretize the terms $T_{\alpha }^{\left( 1\right) }f\left(
z,u\right) $ and $T_{\alpha }^{\left( 2\right) }f\left( z,u\right) $ in
different ways, exploiting the one-parameter structure of the Heisenberg
group for $T_{\alpha }^{\left( 1\right) }$, and exploiting the implicit
product structure for $T_{\alpha }^{\left( 2\right) }$. We rewrite $%
T_{\alpha }^{\left( 1\right) }f\left( z,u\right) $ as%
\begin{eqnarray*}
T_{\alpha }^{\left( 1\right) }f\left( z,u\right) &=&c_{\alpha }\sum_{j\leq k}%
\check{\psi}_{j,k}\ast \psi _{j,k}\ast f\left( z,u\right) \\
&=&c_{\alpha }\sum_{j\leq k}\left( \check{\psi}_{j}^{(1)}\ast _{2}\check{\psi%
}_{k}^{(2)}\right) \ast \left( \psi _{j}^{(1)}\ast _{2}\psi
_{k}^{(2)}\right) \ast f\left( z,u\right) \\
&=&c_{\alpha }\sum_{j\leq k}\left( \check{\psi}_{j}^{(1)}\ast _{2}\check{\psi%
}_{k}^{(2)}\ast _{2}\psi _{k}^{(2)}\right) \ast \psi _{j}^{(1)}\ast f\left(
z,u\right) \\
&=&c_{\alpha }\sum_{j\in \mathbb{Z}}\left( \check{\psi}_{j}^{(1)}\ast
_{2}\left( \sum_{k\geq j}\check{\psi}_{k}^{(2)}\ast _{2}\psi
_{k}^{(2)}\right) \right) \ast \psi _{j}^{(1)}\ast f\left( z,u\right) \\
&=&c_{\alpha }\sum_{j\in \mathbb{Z}}\check{\psi}_{j}\ast \psi _{j}\ast
f\left( z,u\right) ,
\end{eqnarray*}%
where%
\begin{equation}
\psi _{j}\equiv \psi _{j}^{(1)}\text{ and }\check{\psi}_{j}\equiv \check{\psi%
}_{j}^{(1)}\ast _{2}\left( \sum_{k\geq j}\check{\psi}_{k}^{(2)}\ast _{2}\psi
_{k}^{(2)}\right) .  \label{new psi}
\end{equation}

\begin{remark}
\label{same est}It is a standard exercise to prove that $\check{\psi}_{j}$
satisfies the same type of estimates as does $\psi _{j}^{(1)}$ on the
Heisenberg group $\mathbb{H}^{n}$.
\end{remark}

Now we write 
\begin{eqnarray*}
T_{\alpha }^{\left( 1\right) }f\left( z,u\right) &=&\sum_{j\leq k}\sum_{%
\mathcal{Q}\in \mathsf{Q}\left( j\right) }f_{\mathcal{Q}}\psi _{\mathcal{Q}%
}\left( z,u\right) +R_{\alpha ,N}^{\left( 1\right) }f\left( z,u\right) , \\
T_{\alpha }^{\left( 2\right) }f\left( z,u\right) &=&\sum_{j>k}\sum_{\mathcal{%
R}\in \mathsf{R}\left( j,k\right) }f_{\mathcal{R}}\psi _{\mathcal{R}}\left(
z,u\right) +R_{\alpha ,N}^{\left( 2\right) }f\left( z,u\right) ,
\end{eqnarray*}%
where%
\begin{eqnarray*}
f_{\mathcal{Q}} &\equiv &c_{\alpha }\left\vert \mathcal{Q}\right\vert \ \psi
_{j,k}\ast f\left( z_{\mathcal{Q}},u_{\mathcal{Q}}\right) ,\ \ \ \ \ \text{%
for }\mathcal{Q}\in \mathsf{Q}\left( j\right) \text{ and }k\geq j, \\
f_{\mathcal{R}} &\equiv &c_{\alpha }\left\vert \mathcal{R}\right\vert \ \psi
_{j,k}\ast f\left( z_{\mathcal{R}},u_{\mathcal{R}}\right) ,\ \ \ \ \ \text{%
for }\mathcal{R}\in \mathsf{R}\left( j,k\right) \text{ and }k<j, \\
\psi _{\mathcal{Q}}\left( z,u\right) &=&\frac{1}{\left\vert \mathcal{Q}%
\right\vert }\int_{\mathcal{Q}}\check{\psi}_{j,k}\left( \left( z,u\right)
\circ \left( z^{\prime },u^{\prime }\right) ^{-1}\right) dz^{\prime
}du^{\prime },\ \ \ \ \ \text{for }\mathcal{Q}\in \mathsf{Q}\left( j\right) 
\text{ and }k\geq j, \\
\psi _{\mathcal{R}}\left( z,u\right) &=&\frac{1}{\left\vert \mathcal{R}%
\right\vert }\int_{\mathcal{R}}\check{\psi}_{j,k}\left( \left( z,u\right)
\circ \left( z^{\prime },u^{\prime }\right) ^{-1}\right) dz^{\prime
}du^{\prime },\ \ \ \ \ \text{for }\mathcal{R}\in \mathsf{R}\left(
j,k\right) \text{ and }k<j.
\end{eqnarray*}%
and%
\begin{eqnarray*}
R_{\alpha ,N}^{\left( 1\right) }f\left( z,u\right) &=&c_{\alpha }\sum_{j\leq
k}\sum_{\mathcal{Q}\in \mathsf{Q}\left( j\right) }\int_{\mathcal{Q}}\check{%
\psi}_{j,k}\left( \left( z,u\right) \circ \left( z^{\prime },u^{\prime
}\right) ^{-1}\right) \\
&&\times \left[ \psi _{j,k}\ast f\left( z^{\prime },u^{\prime }\right) -\psi
_{j,k}\ast f\left( z_{\mathcal{Q}},u_{\mathcal{Q}}\right) \right] dz^{\prime
}du^{\prime }, \\
R_{\alpha ,N}^{\left( 2\right) }f\left( z,u\right) &=&c_{\alpha
}\sum_{j>k}\sum_{\mathcal{R}\in \mathsf{R}\left( j,k\right) }\int_{\mathcal{R%
}}\check{\psi}_{j,k}\left( \left( z,u\right) \circ \left( z^{\prime
},u^{\prime }\right) ^{-1}\right) \\
&&\times \left[ \psi _{j,k}\ast f\left( z^{\prime },u^{\prime }\right) -\psi
_{j,k}\ast f\left( z_{\mathcal{R}},u_{\mathcal{R}}\right) \right] dz^{\prime
}du^{\prime }.
\end{eqnarray*}

Altogether we have%
\begin{eqnarray}
f\left( z,u\right) &=&\sum_{j\in \mathbb{Z}}\sum_{\mathcal{Q}\in \mathsf{Q}%
\left( j\right) }f_{\mathcal{Q}}\psi _{\mathcal{Q}}\left( z,u\right)
+\sum_{j>k}\sum_{\mathcal{R}\in \mathsf{R}\left( j,k\right) }f_{\mathcal{R}%
}\psi _{\mathcal{R}}\left( z,u\right)  \label{altogether} \\
&&+\left\{ R_{\alpha }f\left( z,u\right) +R_{\alpha ,N}^{\left( 1\right)
}f\left( z,u\right) +R_{\alpha ,N}^{\left( 2\right) }f\left( z,u\right)
\right\} .  \notag
\end{eqnarray}%
Recall that we denote by $\mathsf{Q}\equiv \bigcup_{j\in \mathbb{Z}}\mathsf{Q%
}\left( j\right) $ the collection of \emph{all} dyadic cubes, and by $%
\mathsf{R}_{vert}\equiv \bigcup_{j>k}\mathsf{R}\left( j,k\right) $ the
collection of \emph{all strictly vertical} dyadic rectangles. Then we can
rewrite (\ref{altogether}) as%
\begin{equation}
f\left( z,u\right) =\sum_{\mathcal{Q}\in \mathsf{Q}}f_{\mathcal{Q}}\psi _{%
\mathcal{Q}}\left( z,u\right) +\sum_{\mathcal{R}\in \mathsf{R}_{vert}}f_{%
\mathcal{R}}\psi _{\mathcal{R}}\left( z,u\right) +\left\{ R_{\alpha
}+R_{\alpha ,N}^{\left( 1\right) }+R_{\alpha ,N}^{\left( 2\right) }\right\}
f\left( z,u\right) ,  \label{altogether'}
\end{equation}%
which is a precursor to the \emph{wavelet} form of the Calder\'{o}n
reproducing formula given in the statement of Theorem \ref{discretethm}.

The following theorem is the analogue of Theorem 1.19 in \cite{H2} for the
operators $R_{\alpha }$, $R_{\alpha ,N}^{\left( 1\right) }$ and $R_{\alpha
,N}^{\left( 2\right) }$.

\begin{theorem}
\label{key boundedness}For fixed $M$ and $0<\delta <1$, we can choose $%
M^{\prime }$ and $0<\alpha <\varepsilon $ sufficiently small, and then
choose $N$ sufficiently large, so that the operators $R_{\alpha }$, $%
R_{\alpha }^{\left( 1\right) }$ and $R_{\alpha }^{\left( 2\right) }$ satisfy%
\begin{eqnarray}
&&\left\Vert R_{\alpha }f\right\Vert _{L^{p}\left( \mathbb{H}^{n}\right)
}+\left\Vert R_{\alpha ,N}^{\left( 1\right) }f\right\Vert _{L^{p}\left( 
\mathbb{H}^{n}\right) }+\left\Vert R_{\alpha ,N}^{\left( 2\right)
}f\right\Vert _{L^{p}\left( \mathbb{H}^{n}\right) }  \label{error op bounds}
\\
&&\ \ \ \ \ \ \ \ \ \ \leq \frac{1}{2}\left\Vert f\right\Vert _{L^{p}\left( 
\mathbb{H}^{n}\right) },\ \ \ \ \ f\in L^{p}\left( \mathbb{H}^{n}\right)
,1<p<\infty ,  \notag \\
&&\left\Vert R_{\alpha }f\right\Vert _{\mathcal{M}_{flag}^{M^{\prime
}+\delta }\left( \mathbb{H}^{n}\right) }+\left\Vert R_{\alpha ,N}^{\left(
1\right) }f\right\Vert _{\mathcal{M}_{flag}^{M^{\prime }+\delta }\left( 
\mathbb{H}^{n}\right) }+\left\Vert R_{\alpha ,N}^{\left( 2\right)
}f\right\Vert _{\mathcal{M}_{flag}^{M^{\prime }+\delta }\left( \mathbb{H}%
^{n}\right) }  \notag \\
&&\ \ \ \ \ \ \ \ \ \ \leq \frac{1}{2}\left\Vert f\right\Vert _{\mathcal{M}%
_{flag}^{M^{\prime }+\delta }\left( \mathbb{H}^{n}\right) },\ \ \ \ \ f\in 
\mathcal{M}_{flag}^{M^{\prime }+\delta }\left( \mathbb{H}^{n}\right) . 
\notag
\end{eqnarray}
\end{theorem}

With Theorem \ref{key boundedness} in hand we obtain that the operator 
\begin{equation*}
S_{\alpha ,N}\equiv I-R_{\alpha }-R_{\alpha ,N}^{\left( 1\right)
}f-R_{\alpha ,N}^{\left( 2\right) }
\end{equation*}%
is bounded and invertible on $\mathcal{M}_{flag}^{M^{\prime }+\delta }\left( 
\mathbb{H}^{n}\right) $. It follows that with $\widetilde{\psi }_{\mathcal{Q}%
}\equiv S_{\alpha ,N}^{-1}\psi _{\mathcal{Q}}$ and $\widetilde{\psi }_{%
\mathcal{R}}\equiv S_{\alpha ,N}^{-1}\psi _{\mathcal{R}}$, 
\begin{equation}
f\left( z,u\right) =\sum_{\mathcal{Q}\in \mathsf{Q}}f_{\mathcal{Q}}\ 
\widetilde{\psi }_{\mathcal{Q}}\left( z,u\right) +\sum_{\mathcal{R}\in 
\mathsf{R}_{vert}}f_{\mathcal{R}}\ \widetilde{\psi }_{\mathcal{R}}\left(
z,u\right) ,\ \ \ \ \ f\in \mathcal{M}_{flag}^{M^{\prime }+\delta }\left( 
\mathbb{H}^{n}\right) ,  \label{discrete formula}
\end{equation}%
where $\widetilde{\psi }_{\mathcal{Q}}$ and $\widetilde{\psi }_{\mathcal{R}}$
are in $\mathcal{M}_{flag}^{M^{\prime }+\delta }\left( \mathbb{H}^{n}\right) 
$, and the convergence in (\ref{discrete formula}) is in both $L^{p}\left( 
\mathbb{H}^{n}\right) $ and in the Banach space $\mathcal{M}%
_{flag}^{M^{\prime }+\delta }\left( \mathbb{H}^{n}\right) $. This finally is
the \emph{wavelet} form of the Calder\'{o}n reproducing formula given in the
statement of Theorem \ref{discretethm}. The same argument shows that (\ref%
{discrete formula}) holds for $f\in L^{p}\left( \mathbb{H}^{n}\right) $ with
convergence in $L^{p}\left( \mathbb{H}^{n}\right) $ provided $1<p<\infty $.
In fact we obtain that (\ref{discrete formula}) holds for $f$ in any Banach
space$\ \mathcal{X}\left( \mathbb{H}^{n}\right) $ with convergence in $%
\mathcal{X}\left( \mathbb{H}^{n}\right) $ provided we have operator bounds%
\begin{equation*}
\left\Vert R_{\alpha }f\right\Vert _{\mathcal{X}\left( \mathbb{H}^{n}\right)
}+\left\Vert R_{\alpha ,N}^{\left( 1\right) }f\right\Vert _{\mathcal{X}%
\left( \mathbb{H}^{n}\right) }+\left\Vert R_{\alpha ,N}^{\left( 2\right)
}f\right\Vert _{\mathcal{X}\left( \mathbb{H}^{n}\right) }\leq \frac{1}{2}%
\left\Vert f\right\Vert _{\mathcal{X}\left( \mathbb{H}^{n}\right) },\ \ \ \
\ f\in \mathcal{X}\left( \mathbb{H}^{n}\right) .
\end{equation*}

\bigskip

We turn first to proving the flag molecular estimates in (\ref{error op
bounds}), but only for $\left\Vert R_{\alpha ,N}^{\left( 1\right)
}f\right\Vert _{\mathcal{M}_{flag}^{M^{\prime }+\delta }\left( \mathbb{H}%
^{n}\right) }$ and $\left\Vert R_{\alpha ,N}^{\left( 2\right) }f\right\Vert
_{\mathcal{M}_{flag}^{M^{\prime }+\delta }\left( \mathbb{H}^{n}\right) }$,
as the estimate for $\left\Vert R_{\alpha }f\right\Vert _{\mathcal{M}%
_{flag}^{M^{\prime }+\delta }\left( \mathbb{H}^{n}\right) }$ is similar, but
easier. We will use the following special $T1$ type theorem on the
Heisenberg group $\mathbb{H}^{n}$ (see \cite{H1}, \cite{H2} for the
Euclidean case), to prove a corresponding product version below, which is
then used to obtain the aforementioned flag molecular estimates. Recall the
one-parameter molecular space $\mathcal{M}^{M^{\prime }+\delta }\left( 
\mathbb{H}^{n}\right) $ in Definition \ref{one para mol}.

\begin{theorem}
\label{standard molecular proof}Suppose that $T:L^{2}\left( \mathbb{H}%
^{n}\right) \rightarrow L^{2}\left( \mathbb{H}^{n}\right) $ is a bounded
linear operator with kernel $K\left( \left( z,u\right) ,\left( z^{\prime
},u^{\prime }\right) \right) $, i.e.%
\begin{equation*}
Tf\left( z,u\right) =\int_{\mathbb{H}^{n}}K\left( \left( z,u\right) ,\left(
z^{\prime },u^{\prime }\right) \right) f\left( z^{\prime },u^{\prime
}\right) dz^{\prime }du^{\prime }.
\end{equation*}%
Suppose furthermore that $K$ satisfies%
\begin{eqnarray*}
\int_{\mathbb{H}^{n}}z^{\alpha }u^{\beta }\ K\left( \left( z,u\right)
,\left( z^{\prime },u^{\prime }\right) \right) dzdu &=&0, \\
\int_{\mathbb{H}^{n}}\left( z^{\prime }\right) ^{\alpha }\left( u^{\prime
}\right) ^{\beta }\ K\left( \left( z,u\right) ,\left( z^{\prime },u^{\prime
}\right) \right) dz^{\prime }du^{\prime } &=&0,
\end{eqnarray*}%
for all $0\leq \left\vert \alpha \right\vert ,\beta $; and%
\begin{eqnarray*}
&&\left\vert \partial _{z}^{\alpha }\partial _{u}^{\beta }\partial
_{z^{\prime }}^{\alpha ^{\prime }}\partial _{u^{\prime }}^{\beta ^{\prime
}}K\left( \left( z,u\right) ,\left( z^{\prime },u^{\prime }\right) \right)
\right\vert \\
&\leq &A_{\alpha ,\beta ,\alpha ^{\prime },\beta ^{\prime }}\frac{1}{%
\left\vert \left( z,u\right) \circ \left( z^{\prime },u^{\prime }\right)
^{-1}\right\vert ^{Q+\left\vert \alpha \right\vert +2\beta +\left\vert
\alpha ^{\prime }\right\vert +2\beta ^{\prime }}},
\end{eqnarray*}%
for all $0\leq \left\vert \alpha \right\vert ,\beta ,\left\vert \alpha
^{\prime }\right\vert ,\beta ^{\prime }$. Then 
\begin{eqnarray*}
T &:&L^{p}\left( \mathbb{H}^{n}\right) \rightarrow L^{p}\left( \mathbb{H}%
^{n}\right) ,\ \ \ \ \ 1<p<\infty , \\
T &:&\mathcal{M}^{M^{\prime }+\delta }\left( \mathbb{H}^{n}\right)
\rightarrow \mathcal{M}^{M^{\prime }+\delta }\left( \mathbb{H}^{n}\right) ,\
\ \ \ \ \text{all }M^{\prime }\text{ and }0<\delta <1,
\end{eqnarray*}%
and moreover,%
\begin{equation*}
\left\Vert Tf\right\Vert _{L^{p}\left( \mathbb{H}^{n}\right) }\leq
C_{p}\left\Vert f\right\Vert _{L^{p}\left( \mathbb{H}^{n}\right) }\text{ and 
}\left\Vert Tf\right\Vert _{\mathcal{M}^{M^{\prime }+\delta }\left( \mathbb{H%
}^{n}\right) }\leq C_{M^{\prime },\delta }\left\Vert f\right\Vert _{\mathcal{%
M}^{M^{\prime }+\delta }\left( \mathbb{H}^{n}\right) },
\end{equation*}%
where the constants $C_{p}$ and $C_{M^{\prime },\delta }$ go to zero as $%
A_{\alpha ,\beta ,\alpha ^{\prime },\beta ^{\prime }}\rightarrow 0$ for
sufficiently many of the indices $\alpha ,\beta ,\alpha ^{\prime },\beta
^{\prime }$.
\end{theorem}

In order to obtain flag molecular estimates, we will use the technique of
lifting to the product space $\mathcal{M}_{product}^{M^{\prime }+\delta
}\left( \mathbb{H}^{n}\times \mathbb{R}\right) $ together with the following
special product $T1$ type theorem on the product group $\mathbb{H}^{n}\times 
\mathbb{R}$.

\begin{theorem}
\label{product molecular proof}Suppose that $T:L^{2}\left( \mathbb{H}%
^{n}\times \mathbb{R}\right) \rightarrow L^{2}\left( \mathbb{H}^{n}\times 
\mathbb{R}\right) $ is a bounded linear operator with kernel $K\left( \left[
\left( z,u\right) ,v\right] ,\left[ \left( z^{\prime },u^{\prime }\right)
,v^{\prime }\right] \right) $, i.e.%
\begin{equation*}
Tf\left( \left( z,u\right) ,v\right) =\int_{\mathbb{H}^{n}\times \mathbb{R}%
}K\left( \left[ \left( z,u\right) ,v\right] ,\left[ \left( z^{\prime
},u^{\prime }\right) ,v^{\prime }\right] \right) f\left( \left( z^{\prime
},u^{\prime }\right) ,v^{\prime }\right) dz^{\prime }du^{\prime }dv^{\prime
}.
\end{equation*}%
Suppose furthermore that $K$ satisfies%
\begin{eqnarray*}
\int_{\mathbb{H}^{n}}z^{\alpha }u^{\beta }\ K\left( \left[ \left( z,u\right)
,v\right] ,\left[ \left( z^{\prime },u^{\prime }\right) ,v^{\prime }\right]
\right) dzdu &=&0, \\
\int_{\mathbb{H}^{n}}\left( z^{\prime }\right) ^{\alpha }\left( u^{\prime
}\right) ^{\beta }\ K\left( \left[ \left( z,u\right) ,v\right] ,\left[
\left( z^{\prime },u^{\prime }\right) ,v^{\prime }\right] \right) dz^{\prime
}du^{\prime } &=&0, \\
\int_{\mathbb{R}}v^{\gamma }\ K\left( \left[ \left( z,u\right) ,v\right] ,%
\left[ \left( z^{\prime },u^{\prime }\right) ,v^{\prime }\right] \right) dv
&=&0, \\
\int_{\mathbb{R}}\left( v^{\prime }\right) ^{\gamma }\ K\left( \left[ \left(
z,u\right) ,v\right] ,\left[ \left( z^{\prime },u^{\prime }\right)
,v^{\prime }\right] \right) dv^{\prime } &=&0,
\end{eqnarray*}%
for all $0\leq \left\vert \alpha \right\vert ,\beta ,\gamma $; and%
\begin{eqnarray*}
&&\left\vert \partial _{z}^{\alpha }\partial _{u}^{\beta }\partial
_{v}^{\gamma }\partial _{z^{\prime }}^{\alpha ^{\prime }}\partial
_{u^{\prime }}^{\beta ^{\prime }}\partial _{v^{\prime }}^{\gamma ^{\prime
}}K\left( \left[ \left( z,u\right) ,v\right] ,\left[ \left( z^{\prime
},u^{\prime }\right) ,v^{\prime }\right] \right) \right\vert \\
&\leq &A_{\alpha ,\beta ,\gamma ,\alpha ^{\prime },\beta ^{\prime },\gamma
^{\prime }}\frac{1}{\left\vert \left( z,u\right) \circ \left( z^{\prime
},u^{\prime }\right) ^{-1}\right\vert ^{Q+\left\vert \alpha \right\vert
+2\beta +\left\vert \alpha ^{\prime }\right\vert +2\beta ^{\prime }}}\frac{1%
}{\left\vert v-v^{\prime }\right\vert ^{1+\gamma _{1}+\gamma _{2}}},
\end{eqnarray*}%
for all $0\leq \left\vert \alpha \right\vert ,\beta ,\gamma ,\left\vert
\alpha ^{\prime }\right\vert ,\beta ^{\prime },\gamma ^{\prime }$. Then 
\begin{eqnarray*}
T &:&L^{p}\left( \mathbb{H}^{n}\times \mathbb{R}\right) \rightarrow
L^{p}\left( \mathbb{H}^{n}\times \mathbb{R}\right) ,\ \ \ \ \ 1<p<\infty , \\
T &:&\mathcal{M}_{product}^{M^{\prime }+\delta }\left( \mathbb{H}^{n}\times 
\mathbb{R}\right) \rightarrow \mathcal{M}_{product}^{M^{\prime }+\delta
}\left( \mathbb{H}^{n}\times \mathbb{R}\right) ,\ \ \ \ \ \text{all }%
M^{\prime }\text{ and }0<\delta <1,
\end{eqnarray*}%
and moreover,%
\begin{equation*}
\left\Vert Tf\right\Vert _{L^{p}\left( \mathbb{H}^{n}\right) }\leq
C_{p}\left\Vert f\right\Vert _{L^{p}\left( \mathbb{H}^{n}\right) }\text{ and 
}\left\Vert Tf\right\Vert _{\mathcal{M}^{M^{\prime }+\delta }\left( \mathbb{H%
}^{n}\right) }\leq C_{M^{\prime },\delta }\left\Vert f\right\Vert _{\mathcal{%
M}^{M^{\prime }+\delta }\left( \mathbb{H}^{n}\right) },
\end{equation*}%
where the constants $C_{p}$ and $C_{M^{\prime },\delta }$ go to zero as $%
A_{\alpha ,\beta ,\gamma ,\alpha ^{\prime },\beta ^{\prime },\gamma ^{\prime
}}\rightarrow 0$ for sufficiently many of the indices $\alpha ,\beta ,\gamma
,\alpha ^{\prime },\beta ^{\prime },\gamma ^{\prime }$.
\end{theorem}

We postpone the proofs of these $T1$ type theorems, and turn now to using
them to complete the proof of Theorem \ref{key boundedness}, which in turn
completes the proof of Theorem \ref{discretethm}.

\subsection{Boundedness on the flag molecular space}

We prove the estimates for the operators $R_{\alpha ,N}^{\left( 1\right) }$
and $R_{\alpha ,N}^{\left( 2\right) }$ in Theorem \ref{key boundedness}
separately, beginning with $R_{\alpha ,N}^{\left( 2\right) }$.

\subsubsection{The operator $R_{\protect\alpha ,N}^{\left( 2\right) }$}

Here we prove the boundedness of the error operator%
\begin{eqnarray*}
R_{\alpha ,N}^{\left( 2\right) }f\left( z,u\right) &=&c_{\alpha
}\sum_{j>k}\sum_{\mathcal{R}\in \mathsf{R}\left( j,k\right) }\int_{\mathcal{R%
}}\check{\psi}_{j,k}\left( \left( z,u\right) \circ \left( z^{\prime
},u^{\prime }\right) ^{-1}\right) \\
&&\times \left[ \psi _{j,k}\ast f\left( z^{\prime },u^{\prime }\right) -\psi
_{j,k}\ast f\left( z_{\mathcal{R}},u_{\mathcal{R}}\right) \right] dz^{\prime
}du^{\prime }
\end{eqnarray*}%
on the flag molecular space $\mathcal{M}_{flag}^{M^{\prime }+\delta }\left( 
\mathbb{H}^{n}\right) $ where $M^{\prime }$ is taken sufficiently small
compared to $M$ as in the component functions. We begin by lifting the
desired inequality to the product group $\mathbb{H}^{n}\times \mathbb{R}$
and reducing matters to Theorem \ref{product molecular proof}. So we begin
by writing%
\begin{eqnarray*}
&&R_{\alpha ,N}^{\left( 2\right) }f\left( z,u\right) =c_{\alpha
}\sum_{j>k}\sum_{\mathcal{R}\in \mathsf{R}\left( j,k\right) }\int_{\mathcal{R%
}}\check{\psi}_{j,k}\left( \left( z,u\right) \circ \left( z^{\prime
},u^{\prime }\right) ^{-1}\right) \\
&&\times \int \left[ \psi _{j,k}\left( \left( z^{\prime },u^{\prime }\right)
\circ \left( z^{\prime \prime },u^{\prime \prime }\right) ^{-1}\right) -\psi
_{j,k}\ast f\left( \left( z_{\mathcal{R}},u_{\mathcal{R}}\right) \circ
\left( z^{\prime \prime },u^{\prime \prime }\right) ^{-1}\right) \right]
f\left( z^{\prime \prime },u^{\prime \prime }\right) dz^{\prime \prime
}du^{\prime \prime }dz^{\prime }du^{\prime } \\
&=&c_{\alpha }\sum_{j>k}\sum_{\mathcal{R}\in \mathsf{R}\left( j,k\right)
}\int_{\mathcal{R}}\left\{ \int \check{\psi}_{j}^{\left( 1\right) }\left(
z-z^{\prime },u-u^{\prime }+\func{Im}z\overline{z^{\prime }}-w\right) \check{%
\psi}_{k}^{\left( 2\right) }\left( w\right) dw\right\} \\
&&\times \int \left\{ \int \psi _{j}^{\left( 1\right) }\left( z^{\prime
}-z^{\prime \prime },u^{\prime }-u^{\prime \prime }+\func{Im}z^{\prime }%
\overline{z^{\prime \prime }}-w^{\prime }\right) \check{\psi}_{k}^{\left(
2\right) }\left( w^{\prime }\right) \right. \\
&&\ \ \ \ \ \left. -\int \psi _{j}^{\left( 1\right) }\left( z_{\mathcal{R}%
}-z^{\prime \prime },u_{\mathcal{R}}-u^{\prime \prime }+\func{Im}z_{\mathcal{%
R}}\overline{z^{\prime \prime }}-w^{\prime }\right) \psi _{k}^{\left(
2\right) }\left( w^{\prime }\right) \right\} dw^{\prime }\int F\left(
z^{\prime \prime },u^{\prime \prime }-w^{\prime \prime },w^{\prime \prime
}\right) dw^{\prime \prime },
\end{eqnarray*}%
where%
\begin{equation*}
f\left( z,u\right) =\pi F\left( z,u\right) =\int F\left( \left( z,u-w\right)
,w\right) dw
\end{equation*}%
and $F\left( \left( z,u\right) ,w\right) \in \mathcal{M}_{product}^{M^{%
\prime }+\delta }\left( \mathbb{H}^{n}\times \mathbb{R}\right) $. We
continue with%
\begin{eqnarray*}
R_{\alpha ,N}^{\left( 2\right) }f\left( z,u\right) &=&c_{\alpha
}\sum_{j>k}\sum_{\mathcal{R}\in \mathsf{R}\left( j,k\right) }\int_{\mathcal{R%
}}\iiiint \check{\psi}_{j}^{\left( 1\right) }\left( z-z^{\prime
},u-u^{\prime }+\func{Im}z\overline{z^{\prime }}-w\right) \check{\psi}%
_{k}^{\left( 2\right) }\left( w\right) \\
&&\times \left\{ \psi _{j}^{\left( 1\right) }\left( z^{\prime }-z^{\prime
\prime },u^{\prime }-u^{\prime \prime }+\func{Im}z^{\prime }\overline{%
z^{\prime \prime }}-w^{\prime }\right) \right. \\
&&\ \ \ \ \ \ \ \ \ \ \left. -\psi _{j}^{\left( 1\right) }\left( z_{\mathcal{%
R}}-z^{\prime \prime },u_{\mathcal{R}}-u^{\prime \prime }+\func{Im}z_{%
\mathcal{R}}\overline{z^{\prime \prime }}-w^{\prime }\right) \right\} \check{%
\psi}_{k}^{\left( 2\right) }\left( w^{\prime }\right) \\
&&\times F\left( z^{\prime \prime },u^{\prime \prime }-w^{\prime \prime
},w^{\prime \prime }\right) dz^{\prime \prime }du^{\prime \prime }dw^{\prime
\prime }dw^{\prime }dwdz^{\prime }du^{\prime }.
\end{eqnarray*}%
Now for fixed $w^{\prime \prime }$ make the change of variable $u^{\prime
\prime }\rightarrow u^{\prime \prime }+w^{\prime \prime }$ (in the sense
that $u^{\prime \prime }\rightarrow \widetilde{u}^{\prime \prime }+w^{\prime
\prime }$ and we then rewrite $\widetilde{u}^{\prime \prime }$ as $u^{\prime
\prime }$) to obtain%
\begin{eqnarray*}
R_{\alpha ,N}^{\left( 2\right) }f\left( z,u\right) &=&c_{\alpha
}\sum_{j>k}\sum_{\mathcal{R}\in \mathsf{R}\left( j,k\right) }\int_{\mathcal{R%
}}\iiiint \check{\psi}_{j}^{\left( 1\right) }\left( z-z^{\prime
},u-u^{\prime }+\func{Im}z\overline{z^{\prime }}-w\right) \check{\psi}%
_{k}^{\left( 2\right) }\left( w\right) \\
&&\times \left\{ \psi _{j}^{\left( 1\right) }\left( z^{\prime }-z^{\prime
\prime },u^{\prime }-u^{\prime \prime }-w^{\prime \prime }+\func{Im}%
z^{\prime }\overline{z^{\prime \prime }}-w^{\prime }\right) \right. \\
&&\ \ \ \ \ \ \ \ \ \ \left. -\psi _{j}^{\left( 1\right) }\left( z_{\mathcal{%
R}}-z^{\prime \prime },u_{\mathcal{R}}-u^{\prime \prime }+\func{Im}z_{%
\mathcal{R}}\overline{z^{\prime \prime }}-w^{\prime }-w^{\prime \prime
}\right) \right\} \check{\psi}_{k}^{\left( 2\right) }\left( w^{\prime
}\right) \\
&&\times F\left( z^{\prime \prime },u^{\prime \prime },w^{\prime \prime
}\right) dz^{\prime \prime }du^{\prime \prime }dw^{\prime \prime }dw^{\prime
}dwdz^{\prime }du^{\prime }.
\end{eqnarray*}%
Then making a change of variable $w^{\prime }\rightarrow w^{\prime
}-w^{\prime \prime }$ (in the sense of the previous change of variable) we
get%
\begin{eqnarray*}
R_{\alpha ,N}^{\left( 2\right) }f\left( z,u\right) &=&c_{\alpha
}\sum_{j>k}\sum_{\mathcal{R}\in \mathsf{R}\left( j,k\right) }\int_{\mathcal{R%
}}\iiiint \check{\psi}_{j}^{\left( 1\right) }\left( z-z^{\prime
},u-u^{\prime }+\func{Im}z\overline{z^{\prime }}-w\right) \check{\psi}%
_{k}^{\left( 2\right) }\left( w\right) \\
&&\times \left\{ \psi _{j}^{\left( 1\right) }\left( z^{\prime }-z^{\prime
\prime },u^{\prime }-u^{\prime \prime }+\func{Im}z^{\prime }\overline{%
z^{\prime \prime }}-w^{\prime }\right) \right. \\
&&\ \ \ \ \ \ \ \ \ \ \left. -\psi _{j}^{\left( 1\right) }\left( z_{\mathcal{%
R}}-z^{\prime \prime },u_{\mathcal{R}}-u^{\prime \prime }+\func{Im}z_{%
\mathcal{R}}\overline{z^{\prime \prime }}-w^{\prime }\right) \right\} \check{%
\psi}_{k}^{\left( 2\right) }\left( w^{\prime }-w^{\prime \prime }\right) \\
&&\times F\left( z^{\prime \prime },u^{\prime \prime },w^{\prime \prime
}\right) dz^{\prime \prime }du^{\prime \prime }dw^{\prime \prime }dw^{\prime
}dwdz^{\prime }du^{\prime }.
\end{eqnarray*}%
Finally, making the change of variable $w\rightarrow w-w^{\prime }$ we get%
\begin{eqnarray*}
R_{\alpha ,N}^{\left( 2\right) }f\left( z,u\right) &=&c_{\alpha
}\sum_{j>k}\sum_{\mathcal{R}\in \mathsf{R}\left( j,k\right) }\int_{\mathcal{R%
}}\iiiint \check{\psi}_{j}^{\left( 1\right) }\left( z-z^{\prime
},u-u^{\prime }+\func{Im}z\overline{z^{\prime }}-w+w^{\prime }\right) \check{%
\psi}_{k}^{\left( 2\right) }\left( w-w^{\prime }\right) \\
&&\times \left\{ \psi _{j}^{\left( 1\right) }\left( z^{\prime }-z^{\prime
\prime },u^{\prime }-u^{\prime \prime }+\func{Im}z^{\prime }\overline{%
z^{\prime \prime }}-w^{\prime }\right) \right. \\
&&\ \ \ \ \ \ \ \ \ \ \left. -\psi _{j}^{\left( 1\right) }\left( z_{\mathcal{%
R}}-z^{\prime \prime },u_{\mathcal{R}}-u^{\prime \prime }+\func{Im}z_{%
\mathcal{R}}\overline{z^{\prime \prime }}-w^{\prime }\right) \right\} \check{%
\psi}_{k}^{\left( 2\right) }\left( w^{\prime }-w^{\prime \prime }\right) \\
&&\times F\left( z^{\prime \prime },u^{\prime \prime },w^{\prime \prime
}\right) dz^{\prime \prime }du^{\prime \prime }dw^{\prime \prime }dw^{\prime
}dwdz^{\prime }du^{\prime } \\
&=&\int \widetilde{R}_{\alpha ,N}^{\left( 2\right) }F\left( \left(
z,u-w\right) ,w\right) dw,
\end{eqnarray*}%
where the kernel of $\widetilde{R}_{\alpha ,N}^{\left( 2\right) }$ is given
by%
\begin{eqnarray*}
&&\widetilde{R}_{\alpha ,N}^{\left( 2\right) }\left[ \left( \left(
z,u\right) ,w\right) ,\left( \left( z^{\prime \prime },u^{\prime \prime
}\right) ,w^{\prime \prime }\right) \right] \\
&=&c_{\alpha }\sum_{j>k}\sum_{\mathcal{R}\in \mathsf{R}\left( j,k\right)
}\int_{\mathcal{R}}\int \check{\psi}_{j}^{\left( 1\right) }\left(
z-z^{\prime },u-u^{\prime }+\func{Im}z\overline{z^{\prime }}+w^{\prime
}\right) \check{\psi}_{k}^{\left( 2\right) }\left( w-w^{\prime }\right) \\
&&\times \left\{ \psi _{j}^{\left( 1\right) }\left( z^{\prime }-z^{\prime
\prime },u^{\prime }-u^{\prime \prime }+\func{Im}z^{\prime }\overline{%
z^{\prime \prime }}-w^{\prime }\right) \right. \\
&&\ \ \ \ \ \ \ \ \ \ \left. -\psi _{j}^{\left( 1\right) }\left( z_{\mathcal{%
R}}-z^{\prime \prime },u_{\mathcal{R}}-u^{\prime \prime }+\func{Im}z_{%
\mathcal{R}}\overline{z^{\prime \prime }}-w^{\prime }\right) \right\} \check{%
\psi}_{k}^{\left( 2\right) }\left( w^{\prime }-w^{\prime \prime }\right)
dz^{\prime }du^{\prime }dw^{\prime }.
\end{eqnarray*}

Now it suffices to show that%
\begin{equation*}
\widetilde{R}_{\alpha ,N}^{\left( 2\right) }F\in \mathcal{M}%
_{product}^{M^{\prime }+\delta }\left( \mathbb{H}^{n}\times \mathbb{R}\right)
\end{equation*}%
with small norm, since we then conclude that%
\begin{equation*}
R_{\alpha ,N}^{\left( 2\right) }f\in \mathcal{M}_{flag}^{M^{\prime }+\delta
}\left( \mathbb{H}^{n}\right)
\end{equation*}%
with small norm. To do this we need only check that the kernel of $%
\widetilde{R}_{\alpha ,N}^{\left( 2\right) }$ satisfies the conditions of
Theorem \ref{product molecular proof} with small bounds.

For this we rewrite the kernel in terms of Heisenberg group multiplication as%
\begin{eqnarray*}
&&\widetilde{R}_{\alpha ,N}^{\left( 2\right) }\left[ \left( \left(
z,u\right) ,w\right) ,\left( \left( z^{\prime \prime },u^{\prime \prime
}\right) ,w^{\prime \prime }\right) \right] \\
&=&c_{\alpha }\sum_{j>k}\sum_{\mathcal{R}\in \mathsf{R}\left( j,k\right)
}\int_{\mathcal{R}}\int \check{\psi}_{j}^{\left( 1\right) }\left( \left(
z,u\right) \circ \left( z^{\prime },u^{\prime }-w^{\prime }\right)
^{-1}\right) \check{\psi}_{k}^{\left( 2\right) }\left( w-w^{\prime }\right)
\\
&&\times \left\{ \psi _{j}^{\left( 1\right) }\left( \left( z^{\prime
},u^{\prime }-w^{\prime }\right) \circ \left( z^{\prime \prime },u^{\prime
\prime }\right) ^{-1}\right) \right. \\
&&\ \ \ \ \ \ \ \ \ \ \left. -\psi _{j}^{\left( 1\right) }\left( \left( z_{%
\mathcal{R}},u_{\mathcal{R}}-w^{\prime }\right) \circ \left( z^{\prime
\prime },u^{\prime \prime }\right) ^{-1}\right) \right\} \psi _{k}^{\left(
2\right) }\left( w^{\prime }-w^{\prime \prime }\right) dz^{\prime
}du^{\prime }dw^{\prime }.
\end{eqnarray*}%
By construction we have%
\begin{eqnarray*}
&&\psi _{j}^{\left( 1\right) }\left( \left( z^{\prime },u^{\prime
}-w^{\prime }\right) \circ \left( z^{\prime \prime },u^{\prime \prime
}\right) ^{-1}\right) -\psi _{j}^{\left( 1\right) }\left( \left( z_{\mathcal{%
R}},u_{\mathcal{R}}-w^{\prime }\right) \circ \left( z^{\prime \prime
},u^{\prime \prime }\right) ^{-1}\right) \\
&&\ \ \ \ \ \ \ \ \ \ \ \ \ \ \ \ \ \ \ \ \sim 2^{-N}\psi _{j}^{\left(
1\right) }\left( \left( z^{\prime },u^{\prime }-w^{\prime }\right) \circ
\left( z^{\prime \prime },u^{\prime \prime }\right) ^{-1}\right) ,
\end{eqnarray*}%
in the sense that the left side satisfies the same moment, size and
smoothness conditions as does the right side. Thus we have%
\begin{eqnarray}
&&\ \ \ \ \ \ \ \ \ \ \ \ \ \ \ \ \ \ \ \ \sum_{\mathcal{R}\in \mathsf{R}%
\left( j,k\right) }\int_{\mathcal{R}}\int \check{\psi}_{j}^{\left( 1\right)
}\left( \left( z,u\right) \circ \left( z^{\prime },u^{\prime }-w^{\prime
}\right) ^{-1}\right)  \label{sim} \\
&&\ \ \ \times \left\{ \psi _{j}^{\left( 1\right) }\left( \left( z^{\prime
},u^{\prime }-w^{\prime }\right) \circ \left( z^{\prime \prime },u^{\prime
\prime }\right) ^{-1}\right) -\psi _{j}^{\left( 1\right) }\left( \left( z_{%
\mathcal{R}},u_{\mathcal{R}}-w^{\prime }\right) \circ \left( z^{\prime
\prime },u^{\prime \prime }\right) ^{-1}\right) \right\} dz^{\prime
}du^{\prime }  \notag \\
&\sim &\sum_{\mathcal{R}\in \mathsf{R}\left( j,k\right) }\int_{\mathcal{R}%
}\int \check{\psi}_{j}^{\left( 1\right) }\left( \left( z,u\right) \circ
\left( z^{\prime },u^{\prime }-w^{\prime }\right) ^{-1}\right) 2^{-N}\psi
_{j}^{\left( 1\right) }\left( \left( z^{\prime },u^{\prime }-w^{\prime
}\right) \circ \left( z^{\prime \prime },u^{\prime \prime }\right)
^{-1}\right) dz^{\prime }du^{\prime }  \notag \\
&\sim &2^{-N}\psi _{j}^{\left( 1\right) }\left( \left( z,u\right) \circ
\left( z^{\prime \prime },u^{\prime \prime }\right) ^{-1}\right) .  \notag
\end{eqnarray}%
We also have%
\begin{equation*}
\int \check{\psi}_{k}^{\left( 2\right) }\left( w-w^{\prime }\right) \psi
_{k}^{\left( 2\right) }\left( w^{\prime }-w^{\prime \prime }\right)
dw^{\prime }\sim \psi _{k}^{\left( 2\right) }\left( w-w^{\prime \prime
}\right) .
\end{equation*}%
So altogether we obtain%
\begin{eqnarray*}
&&\widetilde{R}_{\alpha ,N}^{\left( 2\right) }\left[ \left( \left(
z,u\right) ,w\right) ,\left( \left( z^{\prime \prime },u^{\prime \prime
}\right) ,w^{\prime \prime }\right) \right] \\
&\sim &2^{-N}\sum_{j>k}\psi _{j}^{\left( 1\right) }\left( \left( z,u\right)
\circ \left( z^{\prime \prime },u^{\prime \prime }\right) ^{-1}\right) \psi
_{k}^{\left( 2\right) }\left( w-w^{\prime \prime }\right) ,
\end{eqnarray*}%
which satisfies the hypotheses of Theorem \ref{product molecular proof} with
bounds roughly $2^{-N}$ since $\psi ^{\left( 1\right) }\in \mathcal{S}\left( 
\mathbb{H}^{n}\right) $ and $\psi ^{\left( 2\right) }\in \mathcal{S}\left( 
\mathbb{R}\right) $. Here we are using the well known fact that the partial
sums $\sum_{j<M}\psi _{j}$ of an approximate identity satisfy Calder\'{o}%
n-Zygmund kernel conditions of infinite order \emph{uniformly} in $M$.

\subsubsection{The operator $R_{\protect\alpha ,N}^{\left( 1\right) }$}

Now we turn to boundedness of the error operator 
\begin{eqnarray*}
R_{\alpha ,N}^{\left( 1\right) }f\left( z,u\right) &=&c_{\alpha }\sum_{j\leq
k}\sum_{\mathcal{Q}\in \mathsf{Q}\left( j\right) }\int_{\mathcal{Q}}\check{%
\psi}_{j,k}\left( \left( z,u\right) \circ \left( z^{\prime },u^{\prime
}\right) ^{-1}\right) \\
&&\times \left[ \psi _{j,k}\ast f\left( z^{\prime },u^{\prime }\right) -\psi
_{j,k}\ast f\left( z_{\mathcal{Q}},u_{\mathcal{Q}}\right) \right] dz^{\prime
}du^{\prime },
\end{eqnarray*}%
on the flag molecular space $\mathcal{M}_{flag}^{M^{\prime }+\delta }\left( 
\mathbb{H}^{n}\right) $ where $M^{\prime }$ is taken sufficiently small
compared to $M$ as in the component functions. Applying the calculation used
for term $R_{\alpha ,N}^{\left( 2\right) }$ above, we can obtain%
\begin{equation*}
R_{\alpha ,N}^{\left( 1\right) }f\left( z,u\right) =\int \widetilde{R}%
_{\alpha ,N}^{\left( 1\right) }F\left( \left( z,u-w\right) ,w\right) dw,
\end{equation*}%
where the kernel of $\widetilde{R}_{\alpha ,N}^{\left( 1\right) }$ is given
by%
\begin{eqnarray*}
&&\widetilde{R}_{\alpha ,N}^{\left( 1\right) }\left[ \left( \left(
z,u\right) ,w\right) ,\left( \left( z^{\prime \prime },u^{\prime \prime
}\right) ,w^{\prime \prime }\right) \right] \\
&=&c_{\alpha }\sum_{j\leq k}\sum_{\mathcal{Q}\in \mathsf{Q}\left( j\right)
}\int_{\mathcal{Q}}\int \check{\psi}_{j}^{\left( 1\right) }\left(
z-z^{\prime },u-u^{\prime }+\func{Im}z\overline{z^{\prime }}+w^{\prime
}\right) \check{\psi}_{k}^{\left( 2\right) }\left( w-w^{\prime }\right) \\
&&\times \left\{ \psi _{j}^{\left( 1\right) }\left( z^{\prime }-z^{\prime
\prime },u^{\prime }-u^{\prime \prime }+\func{Im}z^{\prime }\overline{%
z^{\prime \prime }}-w^{\prime }\right) \right. \\
&&\ \ \ \ \ \ \ \ \ \ \left. -\psi _{j}^{\left( 1\right) }\left( z_{\mathcal{%
R}}-z^{\prime \prime },u_{\mathcal{R}}-u^{\prime \prime }+\func{Im}z_{%
\mathcal{R}}\overline{z^{\prime \prime }}-w^{\prime }\right) \right\} \check{%
\psi}_{k}^{\left( 2\right) }\left( w^{\prime }-w^{\prime \prime }\right)
dz^{\prime }du^{\prime }dw^{\prime }.
\end{eqnarray*}%
By construction we have%
\begin{eqnarray*}
&&\psi _{j}^{\left( 1\right) }\left( \left( z^{\prime },u^{\prime
}-w^{\prime }\right) \circ \left( z^{\prime \prime },u^{\prime \prime
}\right) ^{-1}\right) -\psi _{j}^{\left( 1\right) }\left( \left( z_{\mathcal{%
R}},u_{\mathcal{R}}-w^{\prime }\right) \circ \left( z^{\prime \prime
},u^{\prime \prime }\right) ^{-1}\right) \\
&&\ \ \ \ \ \ \ \ \ \ \ \ \ \ \ \ \ \ \ \ \sim 2^{-N}\psi _{j}^{\left(
1\right) }\left( \left( z^{\prime },u^{\prime }-w^{\prime }\right) \circ
\left( z^{\prime \prime },u^{\prime \prime }\right) ^{-1}\right) ,
\end{eqnarray*}%
in the sense that the left side satisfies the same moment, size and
smoothness conditions as does the right side. Thus we have%
\begin{eqnarray*}
&&\ \ \ \ \ \ \ \ \ \ \ \ \ \ \ \ \ \ \ \ \sum_{\mathcal{Q}\in \mathsf{Q}%
\left( j\right) }\int_{\mathcal{Q}}\int \check{\psi}_{j}^{\left( 1\right)
}\left( \left( z,u\right) \circ \left( z^{\prime },u^{\prime }-w^{\prime
}\right) ^{-1}\right) \\
&&\ \ \ \times \left\{ \psi _{j}^{\left( 1\right) }\left( \left( z^{\prime
},u^{\prime }-w^{\prime }\right) \circ \left( z^{\prime \prime },u^{\prime
\prime }\right) ^{-1}\right) -\psi _{j}^{\left( 1\right) }\left( \left( z_{%
\mathcal{R}},u_{\mathcal{R}}-w^{\prime }\right) \circ \left( z^{\prime
\prime },u^{\prime \prime }\right) ^{-1}\right) \right\} dz^{\prime
}du^{\prime } \\
&\sim &\sum_{\mathcal{Q}\in \mathsf{Q}\left( j\right) }\int_{\mathcal{R}%
}\int \check{\psi}_{j}^{\left( 1\right) }\left( \left( z,u\right) \circ
\left( z^{\prime },u^{\prime }-w^{\prime }\right) ^{-1}\right) 2^{-N}\psi
_{j}^{\left( 1\right) }\left( \left( z^{\prime },u^{\prime }-w^{\prime
}\right) \circ \left( z^{\prime \prime },u^{\prime \prime }\right)
^{-1}\right) dz^{\prime }du^{\prime } \\
&\sim &2^{-N}\psi _{j}^{\left( 1\right) }\left( \left( z,u\right) \circ
\left( z^{\prime \prime },u^{\prime \prime }\right) ^{-1}\right) .
\end{eqnarray*}%
We also have%
\begin{equation*}
\int \check{\psi}_{k}^{\left( 2\right) }\left( w-w^{\prime }\right) \psi
_{k}^{\left( 2\right) }\left( w^{\prime }-w^{\prime \prime }\right)
dw^{\prime }\sim \psi _{k}^{\left( 2\right) }\left( w-w^{\prime \prime
}\right) .
\end{equation*}%
So altogether we obtain%
\begin{eqnarray*}
&&\widetilde{R}_{\alpha ,N}^{\left( 1\right) }\left[ \left( \left(
z,u\right) ,w\right) ,\left( \left( z^{\prime \prime },u^{\prime \prime
}\right) ,w^{\prime \prime }\right) \right] \\
&\sim &2^{-N}\sum_{j\leq k}\psi _{j}^{\left( 1\right) }\left( \left(
z,u\right) \circ \left( z^{\prime \prime },u^{\prime \prime }\right)
^{-1}\right) \psi _{k}^{\left( 2\right) }\left( w-w^{\prime \prime }\right) ,
\end{eqnarray*}%
which satisfies the hypotheses of Theorem \ref{product molecular proof} with
bounds roughly $2^{-N}$ since $\psi ^{\left( 1\right) }\in \mathcal{S}\left( 
\mathbb{H}^{n}\right) $ and $\psi ^{\left( 2\right) }\in \mathcal{S}\left( 
\mathbb{R}\right) $.

It now follows that the kernels of both $\widetilde{R}_{\alpha ,N}^{\left(
1\right) }$ and $\widetilde{R}_{\alpha ,N}^{\left( 2\right) }$ satisfy the
hypotheses of Theorem \ref{product molecular proof} with bounds roughly $%
2^{-N}$, and we conclude that%
\begin{equation*}
\left\Vert \widetilde{R}_{\alpha ,N}^{\left( i\right) }F\right\Vert _{%
\mathcal{M}_{product}^{M^{\prime }+\delta }\left( \mathbb{H}^{n}\times 
\mathbb{R}\right) }\lesssim 2^{-N}\left\Vert F\right\Vert _{\mathcal{M}%
_{product}^{M^{\prime }+\delta }\left( \mathbb{H}^{n}\times \mathbb{R}%
\right) },\ \ \ \ \ i=1,2.
\end{equation*}%
Thus we obtain for each $i=1,2$:%
\begin{eqnarray*}
\left\Vert R_{\alpha ,N}^{\left( i\right) }f\right\Vert _{\mathcal{M}%
_{flag}^{M^{\prime }+\delta }\left( \mathbb{H}^{n}\right) } &\leq
&\inf_{f=\pi F}\left\Vert \widetilde{R}_{\alpha ,N}^{\left( i\right)
}F\right\Vert _{\mathcal{M}_{product}^{M^{\prime }+\delta }\left( \mathbb{H}%
^{n}\times \mathbb{R}\right) } \\
&\lesssim &2^{-N}\inf_{f=\pi F}\left\Vert F\right\Vert _{\mathcal{M}%
_{product}^{M^{\prime }+\delta }\left( \mathbb{H}^{n}\times \mathbb{R}%
\right) }=2^{-N}\left\Vert f\right\Vert _{\mathcal{M}_{flag}^{M^{\prime
}+\delta }\left( \mathbb{H}^{n}\right) },
\end{eqnarray*}%
and taking $N$ sufficiently large completes the proof of the molecular
estimates in (\ref{error op bounds}).

\subsubsection{The $L^{p}$ estimates}

Finally, we turn to proving the $L^{p}$ estimates in (\ref{error op bounds})
for $1<p<\infty $,%
\begin{equation*}
\left\Vert R_{\alpha }f\right\Vert _{L^{p}\left( \mathbb{H}^{n}\right)
}+\left\Vert R_{\alpha ,N}^{\left( 1\right) }f\right\Vert _{L^{p}\left( 
\mathbb{H}^{n}\right) }+\left\Vert R_{\alpha ,N}^{\left( 2\right)
}f\right\Vert _{L^{p}\left( \mathbb{H}^{n}\right) }\leq \frac{1}{2}%
\left\Vert f\right\Vert _{L^{p}\left( \mathbb{H}^{n}\right) }.
\end{equation*}%
The estimates for $R_{\alpha ,N}^{\left( 1\right) }$ and $R_{\alpha
,N}^{\left( 2\right) }$ follow from the estimates established above for the
kernels of the lifted operators $\widetilde{R}_{\alpha ,N}^{\left( 1\right)
} $ and $\widetilde{R}_{\alpha ,N}^{\left( 2\right) }$. Indeed, for $f\in
L^{p}\left( \mathbb{H}^{n}\right) $, we can use a result in \cite{MRS} to
find $F\in L^{p}\left( \mathbb{H}^{n}\times \mathbb{R}\right) $ with $f=\pi
F $ and $\left\Vert F\right\Vert _{L^{p}\left( \mathbb{H}^{n}\times \mathbb{R%
}\right) }\leq C\left\Vert f\right\Vert _{L^{p}\left( \mathbb{H}^{n}\right)
} $. Then we have%
\begin{equation*}
\left\Vert R_{\alpha ,N}^{\left( i\right) }f\right\Vert _{L^{p}\left( 
\mathbb{H}^{n}\right) }\leq \left\Vert \widetilde{R}_{\alpha ,N}^{\left(
i\right) }F\right\Vert _{L^{p}\left( \mathbb{H}^{n}\times \mathbb{R}\right)
}\lesssim 2^{-N}\left\Vert F\right\Vert _{L^{p}\left( \mathbb{H}^{n}\times 
\mathbb{R}\right) }\leq C2^{-N}\left\Vert f\right\Vert _{L^{p}\left( \mathbb{%
H}^{n}\right) }.
\end{equation*}%
In similar fashion, the kernel of the lifted operator $\widetilde{R}_{\alpha
}$ can be shown to satisfy product kernel estimates with constant $A$ that
is a multiple of $\alpha $ for small $\alpha >0$, and so we obtain from
Theorem \ref{product molecular proof} that 
\begin{equation*}
\left\Vert \widetilde{R}_{\alpha }F\right\Vert _{L^{p}\left( \mathbb{H}%
^{n}\times \mathbb{R}\right) }\lesssim \alpha \left\Vert F\right\Vert
_{L^{p}\left( \mathbb{H}^{n}\times \mathbb{R}\right) }.
\end{equation*}%
and hence with $f=\pi F$ as above,%
\begin{equation*}
\left\Vert R_{\alpha }f\right\Vert _{L^{p}\left( \mathbb{H}^{n}\right) }\leq
\left\Vert \widetilde{R}_{\alpha }F\right\Vert _{L^{p}\left( \mathbb{H}%
^{n}\times \mathbb{R}\right) }\lesssim \alpha \left\Vert F\right\Vert
_{L^{p}\left( \mathbb{H}^{n}\times \mathbb{R}\right) }\leq C\alpha
\left\Vert f\right\Vert _{L^{p}\left( \mathbb{H}^{n}\right) }.
\end{equation*}%
If we now take $0<\alpha <1$ sufficiently small, and then $N$ sufficiently
large, we obtain the $L^{p}$ estimates in (\ref{error op bounds}).

This concludes our proof of Theorem \ref{key boundedness}.

\subsection{The $T1$ type theorems}

We begin with the proof of Theorem \ref{standard molecular proof} in the
one-parameter case. We follow the argument in \cite{H}, where the same
result is proved in the Euclidean setting. For this we will also need an
extension to the Heisenberg group of the generalization of Meyer's lemma by
Torres \cite{Tor}. We use notation adapted to our arguments below.

\begin{lemma}
\label{3}Suppose that $T:L^{2}\left( \mathbb{H}^{n}\right) \rightarrow
L^{2}\left( \mathbb{H}^{n}\right) $ is a bounded linear operator with kernel 
$K\left( \left( z,u\right) ,\left( z^{\prime },u^{\prime }\right) \right) $
satisfying the kernel conditions in the hypotheses of Theorem \ref{standard
molecular proof}. Suppose that $M\geq 0$ and that $T\left( \left( z,u\right)
^{\left( \alpha ^{\prime \prime },\beta ^{\prime \prime }\right) }\right) =0$
for all multi-indices $\left( \alpha ^{\prime \prime },\beta ^{\prime \prime
}\right) $ with $\left\vert \alpha ^{\prime \prime }\right\vert +2\beta
^{\prime \prime }\leq M$. Then for any two points $\left( z,u\right) ,\left(
z^{\prime \prime },u^{\prime \prime }\right) \in \mathbb{H}^{n}$ and any
smooth $\varphi $ on $\mathbb{H}^{n}$ with compact support, and for any
multi-index $\left( \alpha ^{\prime },\beta ^{\prime }\right) $ with $%
\left\vert \alpha ^{\prime }\right\vert +2\beta ^{\prime }=M^{\prime }\leq M$%
, we have the following identity:%
\begin{equation*}
\partial _{z}^{\alpha ^{\prime }}\partial _{u}^{\beta ^{\prime }}T\varphi
\left( z,u\right) -\partial _{z}^{\alpha ^{\prime }}\partial _{u}^{\beta
^{\prime }}T\varphi \left( z^{\prime \prime },u^{\prime \prime }\right)
\end{equation*}%
equals%
\begin{eqnarray*}
&&\int \partial _{z}^{\alpha ^{\prime }}\partial _{u}^{\beta ^{\prime
}}K\left( \left( z,u\right) ,\left( z^{\prime },u^{\prime }\right) \right) \\
&&\times \left\{ \varphi \left( z^{\prime },u^{\prime }\right)
-\sum_{\left\vert \alpha ^{\prime \prime }\right\vert +2\beta ^{\prime
\prime }\leq M^{\prime }}c_{\alpha ^{\prime \prime },\beta ^{\prime \prime
}}\partial _{z}^{\alpha ^{\prime \prime }}\partial _{u}^{\beta ^{\prime
\prime }}\varphi \left( z,u\right) \left[ \left( z^{\prime },u^{\prime
}\right) \circ \left( z,u\right) ^{-1}\right] ^{\left( \alpha ^{\prime
\prime },\beta ^{\prime \prime }\right) }\right\} \widetilde{\theta }\left(
z^{\prime },u^{\prime }\right) dz^{\prime }du^{\prime }
\end{eqnarray*}%
\begin{eqnarray*}
&&-\int \partial _{z}^{\alpha ^{\prime }}\partial _{u}^{\beta ^{\prime
}}K\left( \left( z^{\prime \prime },u^{\prime \prime }\right) ,\left(
z^{\prime },u^{\prime }\right) \right) \\
&&\times \left\{ \varphi \left( z^{\prime },u^{\prime }\right)
-\sum_{\left\vert \alpha ^{\prime \prime }\right\vert +2\beta ^{\prime
\prime }\leq M^{\prime }}c_{\alpha ^{\prime \prime },\beta ^{\prime \prime
}}\partial _{z}^{\alpha ^{\prime \prime }}\partial _{u}^{\beta ^{\prime
\prime }}\varphi \left( z^{\prime \prime },u^{\prime \prime }\right) \left[
\left( z^{\prime },u^{\prime }\right) \circ \left( z^{\prime \prime
},u^{\prime \prime }\right) ^{-1}\right] ^{\left( \alpha ^{\prime \prime
},\beta ^{\prime \prime }\right) }\right\} \widetilde{\theta }\left(
z^{\prime },u^{\prime }\right) dz^{\prime }du^{\prime }
\end{eqnarray*}%
\begin{eqnarray*}
&&+\int \left\{ \partial _{z}^{\alpha ^{\prime }}\partial _{u}^{\beta
^{\prime }}K\left( \left( z,u\right) ,\left( z^{\prime },u^{\prime }\right)
\right) -\partial _{z}^{\alpha ^{\prime }}\partial _{u}^{\beta ^{\prime
}}K\left( \left( z^{\prime \prime },u^{\prime \prime }\right) ,\left(
z^{\prime },u^{\prime }\right) \right) \right\} \\
&&\times \left\{ \varphi \left( z^{\prime },u^{\prime }\right)
-\sum_{\left\vert \alpha ^{\prime \prime }\right\vert +2\beta ^{\prime
\prime }\leq M^{\prime }}c_{\alpha ^{\prime \prime },\beta ^{\prime \prime
}}\partial _{z}^{\alpha ^{\prime \prime }}\partial _{u}^{\beta ^{\prime
\prime }}\varphi \left( z^{\prime \prime },u^{\prime \prime }\right) \left[
\left( z^{\prime },u^{\prime }\right) \circ \left( z^{\prime \prime
},u^{\prime \prime }\right) ^{-1}\right] ^{\left( \alpha ^{\prime \prime
},\beta ^{\prime \prime }\right) }\right\} \left( 1-\widetilde{\theta }%
\left( z^{\prime },u^{\prime }\right) \right) dz^{\prime }du^{\prime }
\end{eqnarray*}%
\begin{eqnarray*}
&&+\sum_{\left\vert \alpha ^{\prime \prime }\right\vert +2\beta ^{\prime
\prime }\leq M^{\prime }} \\
&&\left\{ c_{\alpha ^{\prime \prime },\beta ^{\prime \prime }}\partial
_{z}^{\alpha ^{\prime \prime }}\partial _{u}^{\beta ^{\prime \prime
}}\varphi \left( z,u\right) -\sum_{\substack{ \left\vert \alpha ^{\prime
\prime \prime }\right\vert +2\beta ^{\prime \prime \prime }  \\ \leq
M^{\prime }-\left\vert \alpha ^{\prime \prime }\right\vert -2\beta ^{\prime
\prime }}}c_{\alpha ^{\prime \prime \prime },\beta ^{\prime \prime \prime
}}\partial _{z}^{\alpha ^{\prime \prime \prime }+\alpha ^{\prime \prime
}}\partial _{u}^{2\beta ^{\prime \prime \prime }+2\beta ^{\prime \prime
}}\varphi \left( z^{\prime \prime },u^{\prime \prime }\right) \left[ \left(
z,u\right) \circ \left( z^{\prime \prime },u^{\prime \prime }\right) ^{-1}%
\right] ^{\left( \alpha ^{\prime \prime \prime },\beta ^{\prime \prime
\prime }\right) }\right\} \\
&&\ \ \ \ \ \ \ \ \ \ \ \ \ \ \ \ \ \ \ \ \ \ \ \ \ \ \ \ \ \ \times
T_{\left( \alpha ^{\prime \prime },\beta ^{\prime \prime }\right) ,\left(
\alpha ^{\prime },\beta ^{\prime }\right) }\widetilde{\theta }\left(
z,u\right) .
\end{eqnarray*}
\end{lemma}

The proof of this lemma follows verbatim that of Lemma 3.1.22 on page 62 of 
\cite{Tor}.

\textbf{Proof of Theorem \ref{standard molecular proof}}: Let $f\in \mathcal{%
M}^{M^{\prime }+\delta }\left( \mathbb{H}^{n}\right) $. In order to show
that $Tf\in \mathcal{M}^{M^{\prime }+\delta }\left( \mathbb{H}^{n}\right) $,
we split the proof into four cases, the first two dealing with the size
estimates, and the last two with the difference estimates. For multi-indices 
$\left( \alpha ,\beta \right) $, $\left( \alpha ^{\prime },\beta ^{\prime
}\right) $, $\left( \alpha ^{\prime \prime },\beta ^{\prime \prime }\right) $
and $\left( \alpha ^{\prime \prime \prime },\beta ^{\prime \prime \prime
}\right) $ that arise in the arguments below, we will consistently use the
notation%
\begin{equation}
\left\vert \alpha \right\vert +\beta =M,\ \ \ \left\vert \alpha ^{\prime
}\right\vert +2\beta ^{\prime }=M^{\prime },\ \ \ \left\vert \alpha ^{\prime
\prime }\right\vert +2\beta ^{\prime \prime }=M^{\prime \prime },\ \ \
\left\vert \alpha ^{\prime \prime \prime }\right\vert +2\beta ^{\prime
\prime \prime }=M^{\prime \prime \prime }.  \label{equals M'}
\end{equation}

\bigskip

\textbf{Case 1}: Suppose that $\left\vert \left( z,u\right) \right\vert \leq
1$ and $\left\vert \alpha \right\vert +\beta =M$. Here we are using $%
\left\vert \left( z,u\right) \right\vert $ to denote the usual norm on the
Heisenberg group $\mathbb{H}^{n}$. Let $\theta $ be a smooth cutoff function
supported in $B\left( 0,4\right) $ and identically one on $B\left(
0,2\right) $. Then we have%
\begin{eqnarray*}
&&\ \ \ \ \ \ \ \ \ \ \partial _{z}^{\alpha }\partial _{u}^{\beta }Tf\left(
z,u\right) \\
&=&\int \partial _{z}^{\alpha }\partial _{u}^{\beta }K\left( \left(
z,u\right) ,\left( z^{\prime },u^{\prime }\right) \right) f\left( z^{\prime
},u^{\prime }\right) dz^{\prime }du^{\prime } \\
&=&\int \partial _{z}^{\alpha }\partial _{u}^{\beta }K\left( \left(
z,u\right) ,\left( z^{\prime },u^{\prime }\right) \right) \theta \left(
z^{\prime },u^{\prime }\right) \\
&&\ \ \ \ \ \times \left\{ f\left( z^{\prime },u^{\prime }\right)
-\sum_{\left\vert \alpha ^{\prime }\right\vert +2\beta ^{\prime }\leq
M}c_{\alpha ^{\prime },\beta ^{\prime }}\left( \partial _{z}^{\alpha
^{\prime }}\partial _{u}^{\beta ^{\prime }}f\right) \left( z,u\right) \left[
\left( z^{\prime },u^{\prime }\right) \circ \left( z,u\right) ^{-1}\right]
^{\left( \alpha ^{\prime },\beta ^{\prime }\right) }\right\} dz^{\prime
}du^{\prime } \\
&&+\int \partial _{z}^{\alpha }\partial _{u}^{\beta }K\left( \left(
z,u\right) ,\left( z^{\prime },u^{\prime }\right) \right) f\left( z^{\prime
},u^{\prime }\right) \left( 1-\theta \left( z^{\prime },u^{\prime }\right)
\right) dz^{\prime }du^{\prime } \\
&&+\sum_{\left\vert \alpha ^{\prime }\right\vert +2\beta ^{\prime }\leq
M}c_{\alpha ^{\prime },\beta ^{\prime }}\int \partial _{z}^{\alpha }\partial
_{u}^{\beta }K\left( \left( z,u\right) ,\left( z^{\prime },u^{\prime
}\right) \right) \theta \left( z^{\prime },u^{\prime }\right) \\
&&\ \ \ \ \ \times \left( \partial _{z}^{\alpha ^{\prime }}\partial
_{u}^{\beta ^{\prime }}f\right) \left( z,u\right) \left[ \left( z^{\prime
},u^{\prime }\right) \circ \left( z,u\right) ^{-1}\right] ^{\left( \alpha
^{\prime },\beta ^{\prime }\right) }dz^{\prime }du^{\prime } \\
&&I+II+III.
\end{eqnarray*}%
From Taylor's theorem and the size condition on $K$ we obtain%
\begin{eqnarray*}
\left\vert I\right\vert &\lesssim &\left\Vert f\right\Vert _{\mathcal{M}%
^{M^{\prime }+\delta }\left( \mathbb{H}^{n}\right) }\int_{\left\{ \left\vert
\left( z^{\prime },u^{\prime }\right) \right\vert \leq 4\right\} }\left\vert
\left( z^{\prime },u^{\prime }\right) \circ \left( z,u\right) \right\vert
^{-Q-M}\left\vert \left( z^{\prime },u^{\prime }\right) \circ \left(
z,u\right) ^{-1}\right\vert ^{M+\delta }dz^{\prime }du^{\prime } \\
&\lesssim &C_{\delta }\left\Vert f\right\Vert _{\mathcal{M}^{M^{\prime
}+\delta }\left( \mathbb{H}^{n}\right) },
\end{eqnarray*}%
since $Q$ is the homogeneous dimension of $\mathbb{H}^{n}$.

Now in term $II$ we have $\left\vert \left( z,u\right) \right\vert \leq
1\leq \frac{1}{2}\left\vert \left( z^{\prime },u^{\prime }\right)
\right\vert $ by the support condition on $\theta \left( z^{\prime
},u^{\prime }\right) $, and hence%
\begin{equation*}
\left\vert \left( z^{\prime },u^{\prime }\right) \right\vert \leq \left\vert
\left( z,u\right) \right\vert +\left\vert \left( z^{\prime },u^{\prime
}\right) \circ \left( z,u\right) ^{-1}\right\vert \leq \frac{1}{2}\left\vert
\left( z^{\prime },u^{\prime }\right) \right\vert +\left\vert \left(
z^{\prime },u^{\prime }\right) \circ \left( z,u\right) ^{-1}\right\vert
\end{equation*}%
implies that%
\begin{equation*}
\left\vert \left( z^{\prime },u^{\prime }\right) \right\vert \leq
2\left\vert \left( z^{\prime },u^{\prime }\right) \circ \left( z,u\right)
^{-1}\right\vert .
\end{equation*}%
As a consequence we obtain%
\begin{eqnarray*}
\left\vert II\right\vert &\lesssim &\left\Vert f\right\Vert _{\mathcal{M}%
^{M^{\prime }+\delta }\left( \mathbb{H}^{n}\right) }\int_{\left\{ \left\vert
\left( z^{\prime },u^{\prime }\right) \right\vert >2\right\} }\left\vert
\left( z^{\prime },u^{\prime }\right) \circ \left( z,u\right) \right\vert
^{-Q-M}\left\vert \left( z^{\prime },u^{\prime }\right) \right\vert
^{-Q-M-\delta }dz^{\prime }du^{\prime } \\
&\lesssim &C_{\delta }\left\Vert f\right\Vert _{\mathcal{M}^{M^{\prime
}+\delta }\left( \mathbb{H}^{n}\right) }.
\end{eqnarray*}

For term $III$ we will use the fact that%
\begin{equation*}
\left\Vert T_{\left( \alpha ^{\prime },\beta ^{\prime }\right) ;\left(
\alpha ,\beta \right) }\theta \left( \frac{\cdot -z}{t}\right) \right\Vert
_{\infty }\lesssim t^{\left\vert \alpha ^{\prime }\right\vert +\beta
^{\prime }-\left\vert \alpha \right\vert -\beta }
\end{equation*}%
where%
\begin{eqnarray*}
&&T_{\left( \alpha ^{\prime },\beta ^{\prime }\right) ;\left( \alpha ,\beta
\right) }f\equiv \left\langle K_{\left( \alpha ^{\prime },\beta ^{\prime
}\right) ;\left( \alpha ,\beta \right) },f\right\rangle \\
&\equiv &\iint_{\left\{ \left( z,u\right) \neq \left( z^{\prime },u^{\prime
}\right) \right\} }\left[ \left( z^{\prime },u^{\prime }\right) \circ \left(
z,u\right) \right] ^{\left( \alpha ^{\prime },\beta ^{\prime }\right)
}\left\{ \partial _{z}^{\alpha }\partial _{u}^{\beta }K\left( \left(
z,u\right) ,\left( z^{\prime },u^{\prime }\right) \right) \right\} f\left(
z^{\prime },u^{\prime }\right) dz^{\prime }du^{\prime }.
\end{eqnarray*}%
We obtain from $\left\vert T_{\left( \alpha ^{\prime },\beta ^{\prime
}\right) ;\left( \alpha ,\beta \right) }\theta \left( z,u\right) \right\vert
\leq C$ that 
\begin{eqnarray*}
\left\vert III\right\vert &\leq &\sum_{\left\vert \alpha ^{\prime
}\right\vert +2\beta ^{\prime }\leq M}c_{\alpha ^{\prime },\beta ^{\prime
}}\left\vert \left( \partial _{z}^{\alpha ^{\prime }}\partial _{u}^{\beta
^{\prime }}f\right) \left( z,u\right) \right\vert \left\vert T_{\left(
\alpha ^{\prime },\beta ^{\prime }\right) ;\left( \alpha ,\beta \right)
}\theta \left( z,u\right) \right\vert \\
&\lesssim &\sum_{\left\vert \alpha ^{\prime }\right\vert +2\beta ^{\prime
}\leq M}\left( 1+\left\vert z\right\vert ^{2}+\left\vert u\right\vert
\right) ^{-\frac{Q+M+\left\vert \alpha ^{\prime }\right\vert +2\beta
^{\prime }+\delta }{2}}\left\Vert f\right\Vert _{\mathcal{M}^{M^{\prime
}+\delta }\left( \mathbb{H}^{n}\right) } \\
&\lesssim &\left\Vert f\right\Vert _{\mathcal{M}^{M^{\prime }+\delta }\left( 
\mathbb{H}^{n}\right) }.
\end{eqnarray*}%
This shows that when $\left\vert \left( z,u\right) \right\vert \leq 1$, $%
\partial _{z}^{\alpha }\partial _{u}^{\beta }Tf\left( z,u\right) $ satisfies
the size conditions for the one-parameter molecular space $\mathcal{M}%
^{M^{\prime }+\delta }\left( \mathbb{H}^{n}\right) $ in Definition \ref{one
para mol}.

\bigskip

\textbf{Case 2}: Now suppose that $\left\vert \left( z,u\right) \right\vert
>1$, and suppose again that $\left\vert \alpha \right\vert +\beta =M$.
Following Meyer \cite{Mey}, we set%
\begin{eqnarray*}
\mathsf{1} &=&I\left( z^{\prime },u^{\prime }\right) +J\left( z^{\prime
},u^{\prime }\right) +L\left( z^{\prime },u^{\prime }\right) , \\
I\left( z^{\prime },u^{\prime }\right) &=&\theta \left( \frac{\delta \left(
z,u\right) \circ \left( z^{\prime },u^{\prime }\right) ^{-1}}{\left\vert
\left( z,u\right) \right\vert }\right) , \\
J\left( z^{\prime },u^{\prime }\right) &=&\theta \left( \frac{\delta \left(
z^{\prime },u^{\prime }\right) }{\left\vert \left( z,u\right) \right\vert }%
\right) , \\
f &=&f_{I}+f_{J}+f_{L}\equiv fI+fJ+fL.
\end{eqnarray*}%
We first estimate $\partial _{z}^{\alpha }\partial _{u}^{\beta }Tf_{I}\left(
z,u\right) $ as follows:%
\begin{eqnarray*}
&&\partial _{z}^{\alpha }\partial _{u}^{\beta }Tf_{I}\left( z,u\right) \\
&=&\int \partial _{z}^{\alpha }\partial _{u}^{\beta }K\left( \left(
z,u\right) ,\left( z^{\prime },u^{\prime }\right) \right) f_{I}\left(
z^{\prime },u^{\prime }\right) dz^{\prime }du^{\prime } \\
&=&\int \partial _{z}^{\alpha }\partial _{u}^{\beta }K\left( \left(
z,u\right) ,\left( z^{\prime },u^{\prime }\right) \right) \\
&&\times \left\{ f\left( z^{\prime },u^{\prime }\right) -\sum_{\left\vert
\alpha ^{\prime }\right\vert +2\beta ^{\prime }\leq M}c_{\alpha ^{\prime
},\beta ^{\prime }}\left( \partial _{z}^{\alpha ^{\prime }}\partial
_{u}^{\beta ^{\prime }}f\right) \left( z,u\right) \left[ \left( z^{\prime
},u^{\prime }\right) \circ \left( z,u\right) ^{-1}\right] ^{\left( \alpha
^{\prime },\beta ^{\prime }\right) }\right\} I\left( z^{\prime },u^{\prime
}\right) dz^{\prime }du^{\prime } \\
&&+\int \partial _{z}^{\alpha }\partial _{u}^{\beta }K\left( ...\right)
\sum_{\left\vert \alpha ^{\prime }\right\vert +2\beta ^{\prime }\leq
M}c_{\alpha ^{\prime },\beta ^{\prime }}\left( \partial _{z}^{\alpha
^{\prime }}\partial _{u}^{\beta ^{\prime }}f\right) \left( z,u\right) \left[
\left( z^{\prime },u^{\prime }\right) \circ \left( z,u\right) ^{-1}\right]
^{\left( \alpha ^{\prime },\beta ^{\prime }\right) }I\left( z^{\prime
},u^{\prime }\right) dz^{\prime }du^{\prime } \\
&=&A+B.
\end{eqnarray*}%
Using Taylor's theorem and the smoothness conditions on $f$ we obtain%
\begin{eqnarray*}
&&\left\vert A\right\vert \lesssim \left\Vert f\right\Vert _{\mathcal{M}%
^{M^{\prime }+\delta }\left( \mathbb{H}^{n}\right) }\times \\
&&\int_{\left\{ \frac{\left\vert \left( z,u\right) \circ \left( z^{\prime
},u^{\prime }\right) ^{-1}\right\vert }{\left\vert \left( z,u\right)
\right\vert }\leq \frac{1}{2}\right\} }\frac{1}{\left\vert \left( z,u\right)
\circ \left( z^{\prime },u^{\prime }\right) ^{-1}\right\vert ^{Q+M}}\frac{%
\left\vert \left( z,u\right) \circ \left( z^{\prime },u^{\prime }\right)
^{-1}\right\vert ^{\delta }}{\left( 1+\left\vert z\right\vert
^{2}+\left\vert u\right\vert \right) ^{\frac{Q+2M+2\delta }{2}}}dz^{\prime
}du^{\prime } \\
&\lesssim &\left\Vert f\right\Vert _{\mathcal{M}^{M^{\prime }+\delta }\left( 
\mathbb{H}^{n}\right) }\frac{1}{\left( 1+\left\vert z\right\vert
^{2}+\left\vert u\right\vert \right) ^{\frac{Q+2M+\delta }{2}}},
\end{eqnarray*}%
where we have also used the fact that the support of $I\left( z^{\prime
},u^{\prime }\right) $ is contained in the set%
\begin{equation*}
\left\{ \frac{\left\vert \left( z,u\right) \circ \left( z^{\prime
},u^{\prime }\right) ^{-1}\right\vert }{\left\vert \left( z,u\right)
\right\vert }\leq \frac{1}{2}\right\} .
\end{equation*}%
To estimate term $B$ we note that%
\begin{equation*}
\left\Vert T_{\left( \alpha ^{\prime },\beta ^{\prime }\right) ;\left(
\alpha ,\beta \right) }\theta \left( \frac{\delta \left\vert \left(
z,u\right) \circ \left( z^{\prime },u^{\prime }\right) ^{-1}\right\vert }{%
\left\vert \left( z,u\right) \right\vert }\right) \right\Vert _{\infty }\leq
C\left\vert \left( z,u\right) \right\vert ^{\left\vert \alpha ^{\prime
}\right\vert +2\beta ^{\prime }-M}.
\end{equation*}%
Then we get%
\begin{eqnarray*}
\left\vert B\right\vert &\lesssim &\sum_{\left\vert \alpha ^{\prime
}\right\vert +2\beta ^{\prime }\leq M}c_{\alpha ^{\prime },\beta ^{\prime
}}\left\Vert T_{\left( \alpha ^{\prime },\beta ^{\prime }\right) ;\left(
\alpha ,\beta \right) }I\left( \cdot ,\cdot \right) \right\Vert _{\infty
}\left\Vert \left( \partial _{z}^{\alpha ^{\prime }}\partial _{u}^{\beta
^{\prime }}f\right) \left( z,u\right) \right\Vert _{\infty } \\
&\lesssim &\sum_{\left\vert \alpha ^{\prime }\right\vert +2\beta ^{\prime
}\leq M}\left\vert \left( z,u\right) \right\vert ^{\left\vert \alpha
^{\prime }\right\vert +2\beta ^{\prime }-M}\frac{1}{\left( 1+\left\vert
z\right\vert ^{2}+\left\vert u\right\vert \right) ^{\frac{Q+M+\left\vert
\alpha ^{\prime }\right\vert +2\beta +\delta }{2}}}\left\Vert f\right\Vert _{%
\mathcal{M}^{M^{\prime }+\delta }\left( \mathbb{H}^{n}\right) } \\
&\lesssim &\left\Vert f\right\Vert _{\mathcal{M}^{M^{\prime }+\delta }\left( 
\mathbb{H}^{n}\right) }\frac{1}{\left( 1+\left\vert z\right\vert
^{2}+\left\vert u\right\vert \right) ^{\frac{Q+2M+\delta }{2}}}.
\end{eqnarray*}

For the term $\partial _{z}^{\alpha }\partial _{u}^{\beta }Tf_{J}\left(
z,u\right) $ we use the smoothness conditions on the kernel $T$ and write%
\begin{eqnarray*}
&&\partial _{z}^{\alpha }\partial _{u}^{\beta }Tf_{J}\left( z,u\right) \\
&=&\int \left\{ \partial _{z}^{\alpha }\partial _{u}^{\beta }K\left( \left(
z,u\right) ,\left( z^{\prime },u^{\prime }\right) \right) -\sum_{\left\vert
\alpha ^{\prime }\right\vert +2\beta ^{\prime }\leq M}c_{\alpha ^{\prime
},\beta ^{\prime }}\partial _{z}^{\alpha ^{\prime }}\partial _{u}^{\beta
^{\prime }}\partial _{z}^{\alpha }\partial _{u}^{\beta }K\left( \left(
z,u\right) ,\left( 0,0\right) \right) \left( z^{\prime },u^{\prime }\right)
^{\left( \alpha ^{\prime },\beta ^{\prime }\right) }\right\} \\
&&\ \ \ \ \ \ \ \ \ \ \ \ \ \ \ \ \ \ \ \ \times f_{J}\left( z^{\prime
},u^{\prime }\right) dz^{\prime }du^{\prime } \\
&&+\int \sum_{\left\vert \alpha ^{\prime }\right\vert +2\beta ^{\prime }\leq
M}c_{\alpha ^{\prime },\beta ^{\prime }}\partial _{z}^{\alpha ^{\prime
}}\partial _{u}^{\beta ^{\prime }}\partial _{z}^{\alpha }\partial
_{u}^{\beta }K\left( \left( z,u\right) ,\left( 0,0\right) \right) \left(
z^{\prime },u^{\prime }\right) ^{\left( \alpha ^{\prime },\beta ^{\prime
}\right) }f_{J}\left( z^{\prime },u^{\prime }\right) dz^{\prime }du^{\prime }
\\
&=&B_{1}+B_{2}.
\end{eqnarray*}%
Now%
\begin{eqnarray*}
\left\vert B_{1}\right\vert &\lesssim &\int_{\frac{\left\vert \left(
z^{\prime },u^{\prime }\right) \right\vert }{\left\vert \left( z,u\right)
\right\vert }\leq \frac{1}{2}}\frac{\left\vert \left( z^{\prime },u^{\prime
}\right) \right\vert ^{M+1}}{\left\vert \left( z,u\right) \circ \left(
z^{\prime },u^{\prime }\right) ^{-1}\right\vert ^{\frac{Q+M+M+1}{2}}}\frac{%
\left\Vert f\right\Vert _{\mathcal{M}^{M^{\prime }+\delta }\left( \mathbb{H}%
^{n}\right) }}{\left( 1+\left\vert z^{\prime }\right\vert ^{2}+\left\vert
u^{\prime }\right\vert \right) ^{\frac{Q+M+\delta }{2}}}dz^{\prime
}du^{\prime } \\
&\lesssim &\left\Vert f\right\Vert _{\mathcal{M}^{M^{\prime }+\delta }\left( 
\mathbb{H}^{n}\right) }\frac{1}{\left( 1+\left\vert z\right\vert
^{2}+\left\vert u\right\vert \right) ^{\frac{Q+2M+\delta }{2}}},
\end{eqnarray*}%
where we have used the fact that $\left\vert \left( z,u\right) \circ \left(
z^{\prime },u^{\prime }\right) ^{-1}\right\vert \geq \frac{1}{2}\left\vert
\left( z,u\right) \right\vert $ when $\frac{\left\vert \left( z^{\prime
},u^{\prime }\right) \right\vert }{\left\vert \left( z,u\right) \right\vert }%
\leq \frac{1}{2}$.

To estimate $B_{2}$, we note that by the vanishing moments conditions on $f$%
, 
\begin{eqnarray*}
&&\left\vert \int \left( z^{\prime },u^{\prime }\right) ^{\left( \alpha
^{\prime },\beta ^{\prime }\right) }f_{J}\left( z^{\prime },u^{\prime
}\right) dz^{\prime }du^{\prime }\right\vert \\
&=&\left\vert \int \left( z^{\prime },u^{\prime }\right) ^{\left( \alpha
^{\prime },\beta ^{\prime }\right) }f_{I}\left( z^{\prime },u^{\prime
}\right) dz^{\prime }du^{\prime }+\int \left( z^{\prime },u^{\prime }\right)
^{\left( \alpha ^{\prime },\beta ^{\prime }\right) }f_{L}\left( z^{\prime
},u^{\prime }\right) dz^{\prime }du^{\prime }\right\vert \\
&\lesssim &\left\Vert f\right\Vert _{\mathcal{M}^{M^{\prime }+\delta }\left( 
\mathbb{H}^{n}\right) }\left\{ \int_{\left\{ \left\vert \left( z,u\right)
\circ \left( z^{\prime },u^{\prime }\right) ^{-1}\right\vert \leq \frac{1}{2}%
\left\vert \left( z^{\prime },u^{\prime }\right) \right\vert \right\} }\frac{%
\left\vert \left( z^{\prime },u^{\prime }\right) \right\vert ^{\left\vert
\alpha ^{\prime }\right\vert +2\beta ^{\prime }}}{\left( 1+\left\vert
z^{\prime }\right\vert ^{2}+\left\vert u^{\prime }\right\vert \right) ^{%
\frac{Q+M^{\prime }+\delta }{2}}}dz^{\prime }du^{\prime }\right. \\
&&\ \ \ \ \ \ \ \ \ \ \ \ \ \ \ \ \ \ \ \ \ \ \ \ \ \left. +\int_{\left\{
\left\vert \left( z,u\right) \circ \left( z^{\prime },u^{\prime }\right)
^{-1}\right\vert \geq \frac{1}{4}\left\vert \left( z,u\right) \right\vert
\right\} }\frac{\left\vert \left( z^{\prime },u^{\prime }\right) \right\vert
^{\left\vert \alpha ^{\prime }\right\vert +2\beta ^{\prime }}}{\left(
1+\left\vert z^{\prime }\right\vert ^{2}+\left\vert u^{\prime }\right\vert
\right) ^{\frac{Q+M^{\prime }+\delta }{2}}}dz^{\prime }du^{\prime }\right\}
\\
&\lesssim &\left\Vert f\right\Vert _{\mathcal{M}^{M^{\prime }+\delta }\left( 
\mathbb{H}^{n}\right) }\frac{1}{\left( 1+\left\vert z\right\vert
^{2}+\left\vert u\right\vert \right) ^{\frac{M^{\prime }+\delta -\left\vert
\alpha ^{\prime }\right\vert -2\beta ^{\prime }}{2}}}.
\end{eqnarray*}%
Therefore%
\begin{eqnarray*}
\left\vert B_{2}\right\vert &\lesssim &\left\Vert f\right\Vert _{\mathcal{M}%
^{M^{\prime }+\delta }\left( \mathbb{H}^{n}\right) }\sum_{\left\vert \alpha
^{\prime }\right\vert +2\beta ^{\prime }\leq M}\frac{1}{\left( \left\vert
z\right\vert ^{2}+\left\vert u\right\vert \right) ^{\frac{Q+\left\vert
\alpha \right\vert +2\beta +\left\vert \alpha ^{\prime }\right\vert +2\beta
^{\prime }}{2}}}\frac{1}{\left( 1+\left\vert z\right\vert ^{2}+\left\vert
u\right\vert \right) ^{\frac{M^{\prime }+\delta -\left\vert \alpha ^{\prime
}\right\vert -2\beta ^{\prime }}{2}}} \\
&\lesssim &\left\Vert f\right\Vert _{\mathcal{M}^{M^{\prime }+\delta }\left( 
\mathbb{H}^{n}\right) }\frac{1}{\left( 1+\left\vert z\right\vert
^{2}+\left\vert u\right\vert \right) ^{\frac{Q+2M^{\prime }+\delta }{2}}},
\end{eqnarray*}%
since $\left\vert z\right\vert ^{2}+\left\vert u\right\vert \geq 1$.

For the last term $\partial _{z}^{\alpha }\partial _{u}^{\beta }Tf_{L}\left(
z,u\right) $, the estimate is similar to that of the first term, but easier.
Indeed, note that the support of $f_{L}$ is contained in 
\begin{equation*}
\left\{ \left( z^{\prime },u^{\prime }\right) :\left\vert \left( z^{\prime
},u^{\prime }\right) \right\vert \geq \frac{1}{4}\left\vert \left(
z,u\right) \right\vert \text{ and }\left\vert \left( z^{\prime },u^{\prime
}\right) \circ \left( z,u\right) ^{-1}\right\vert \geq \frac{1}{4}\left\vert
\left( z,u\right) \right\vert \right\} .
\end{equation*}%
Thus%
\begin{eqnarray*}
\left\vert \partial _{z}^{\alpha }\partial _{u}^{\beta }Tf_{L}\left(
z,u\right) \right\vert &=&\left\vert \int \partial _{z}^{\alpha }\partial
_{u}^{\beta }K\left( \left( z,u\right) ,\left( z^{\prime },u^{\prime
}\right) \right) f_{L}\left( z^{\prime },u^{\prime }\right) dz^{\prime
}du^{\prime }\right\vert \\
&\lesssim &\left\Vert f\right\Vert _{\mathcal{M}^{M^{\prime }+\delta }\left( 
\mathbb{H}^{n}\right) }\int_{\left\{ \frac{\left\vert z^{\prime }\right\vert
^{2}+\left\vert u^{\prime }\right\vert }{\left\vert z\right\vert
^{2}+\left\vert u\right\vert }\geq \frac{1}{4}\right\} }\frac{1}{\left(
\left\vert z\right\vert ^{2}+\left\vert u\right\vert \right) ^{\frac{%
Q+\left\vert \alpha \right\vert +2\beta }{2}}}\frac{1}{\left( \left\vert
z^{\prime }\right\vert ^{2}+\left\vert u^{\prime }\right\vert \right) ^{%
\frac{Q+M^{\prime }+\delta }{2}}}dz^{\prime }du^{\prime } \\
&\lesssim &\left\Vert f\right\Vert _{\mathcal{M}^{M^{\prime }+\delta }\left( 
\mathbb{H}^{n}\right) }\frac{1}{\left( \left\vert z\right\vert
^{2}+\left\vert u\right\vert \right) ^{\frac{Q+2M^{\prime }+\delta }{2}}}.
\end{eqnarray*}

\bigskip

We now prove the smoothness conditions for $\partial _{z}^{\alpha }\partial
_{u}^{\beta }Tf\left( z,u\right) $ in two separate cases.

\bigskip

\textbf{Case 3}: We consider first the case $\left\vert \left( z,u\right)
\right\vert \leq 1$. For this we need the following technical fact from
Lemma \ref{3}:%
\begin{equation*}
\partial _{z}^{\alpha ^{\prime }}\partial _{u}^{\beta ^{\prime }}Tf\left(
z,u\right) -\partial _{z}^{\alpha ^{\prime }}\partial _{u}^{\beta ^{\prime
}}Tf\left( z^{\prime \prime },u^{\prime \prime }\right)
\end{equation*}%
equals%
\begin{eqnarray*}
&&\int \partial _{z}^{\alpha ^{\prime }}\partial _{u}^{\beta ^{\prime
}}K\left( \left( z,u\right) ,\left( z^{\prime },u^{\prime }\right) \right) \\
&&\times \left\{ f\left( z^{\prime },u^{\prime }\right) -\sum_{\left\vert
\alpha ^{\prime \prime }\right\vert +2\beta ^{\prime \prime }\leq M^{\prime
}}c_{\alpha ^{\prime \prime },\beta ^{\prime \prime }}\partial _{z}^{\alpha
^{\prime \prime }}\partial _{u}^{\beta ^{\prime \prime }}f\left( z,u\right) 
\left[ \left( z^{\prime },u^{\prime }\right) \circ \left( z,u\right) ^{-1}%
\right] ^{\left( \alpha ^{\prime \prime },\beta ^{\prime \prime }\right)
}\right\} \widetilde{\theta }\left( z^{\prime },u^{\prime }\right)
dz^{\prime }du^{\prime }
\end{eqnarray*}%
\begin{eqnarray*}
&&+\int \partial _{z}^{\alpha ^{\prime }}\partial _{u}^{\beta ^{\prime
}}K\left( \left( z^{\prime \prime },u^{\prime \prime }\right) ,\left(
z^{\prime },u^{\prime }\right) \right) \\
&&\times \left\{ f\left( z^{\prime },u^{\prime }\right) -\sum_{\left\vert
\alpha ^{\prime \prime }\right\vert +2\beta ^{\prime \prime }\leq M^{\prime
}}c_{\alpha ^{\prime \prime },\beta ^{\prime \prime }}\partial _{z}^{\alpha
^{\prime \prime }}\partial _{u}^{\beta ^{\prime \prime }}f\left( z^{\prime
\prime },u^{\prime \prime }\right) \left[ \left( z^{\prime },u^{\prime
}\right) \circ \left( z^{\prime \prime },u^{\prime \prime }\right) ^{-1}%
\right] ^{\left( \alpha ^{\prime \prime },\beta ^{\prime \prime }\right)
}\right\} \widetilde{\theta }\left( z^{\prime },u^{\prime }\right)
dz^{\prime }du^{\prime }
\end{eqnarray*}%
\begin{eqnarray*}
&&+\int \left\{ \partial _{z}^{\alpha ^{\prime }}\partial _{u}^{\beta
^{\prime }}K\left( \left( z,u\right) ,\left( z^{\prime },u^{\prime }\right)
\right) -\partial _{z}^{\alpha ^{\prime }}\partial _{u}^{\beta ^{\prime
}}K\left( \left( z^{\prime \prime },u^{\prime \prime }\right) ,\left(
z^{\prime },u^{\prime }\right) \right) \right\} \\
&&\times \left\{ f\left( z^{\prime },u^{\prime }\right) -\sum_{\left\vert
\alpha ^{\prime \prime }\right\vert +2\beta ^{\prime \prime }\leq M^{\prime
}}c_{\alpha ^{\prime \prime },\beta ^{\prime \prime }}\partial _{z}^{\alpha
^{\prime \prime }}\partial _{u}^{\beta ^{\prime \prime }}f\left( z^{\prime
\prime },u^{\prime \prime }\right) \left[ \left( z^{\prime },u^{\prime
}\right) \circ \left( z^{\prime \prime },u^{\prime \prime }\right) ^{-1}%
\right] ^{\left( \alpha ^{\prime \prime },\beta ^{\prime \prime }\right)
}\right\} \left( 1-\widetilde{\theta }\left( z^{\prime },u^{\prime }\right)
\right) dz^{\prime }du^{\prime }
\end{eqnarray*}%
\begin{eqnarray*}
&&+\sum_{\left\vert \alpha ^{\prime \prime }\right\vert +2\beta ^{\prime
\prime }\leq M^{\prime }}c_{\alpha ^{\prime \prime },\beta ^{\prime \prime }}
\\
&&\times \left\{ \partial _{z}^{\alpha ^{\prime \prime }}\partial
_{u}^{\beta ^{\prime \prime }}f\left( z,u\right) -\sum_{\substack{ %
\left\vert \alpha ^{\prime \prime \prime }\right\vert +2\beta ^{\prime
\prime \prime }  \\ \leq \left\vert \alpha ^{\prime }\right\vert +2\beta
^{\prime }-\left\vert \alpha ^{\prime \prime }\right\vert -2\beta ^{\prime
\prime }}}c_{\alpha ^{\prime \prime \prime },\beta ^{\prime \prime \prime
}}\partial _{z}^{\alpha ^{\prime \prime \prime }+\alpha ^{\prime \prime
}}\partial _{u}^{2\beta ^{\prime \prime \prime }+2\beta ^{\prime \prime
}}f\left( z^{\prime \prime },u^{\prime \prime }\right) \left[ \left(
z,u\right) \circ \left( z^{\prime \prime },u^{\prime \prime }\right) ^{-1}%
\right] ^{\left( \alpha ^{\prime \prime \prime },\beta ^{\prime \prime
\prime }\right) }\right\} \\
&&\ \ \ \ \ \ \ \ \ \ \ \ \ \ \ \ \ \ \ \ \ \ \ \ \ \ \ \ \ \ \times
T_{\left( \alpha ^{\prime \prime },\beta ^{\prime \prime }\right) ,\left(
\alpha ^{\prime },\beta ^{\prime }\right) }\widetilde{\theta }\left(
z,u\right) ,
\end{eqnarray*}%
where%
\begin{equation*}
\widetilde{\theta }\left( z^{\prime },u^{\prime }\right) =\theta \left( 
\frac{8\left( z^{\prime },u^{\prime }\right) \circ \left( z,u\right) ^{-1}}{%
\left\vert \left( z,u\right) \circ \left( z^{\prime \prime },u^{\prime
\prime }\right) ^{-1}\right\vert }\right) ,
\end{equation*}%
and where all of the integrals above are absolutely convergent.

Let $I,II,III$ denote the three integrals and let $IV\ $denote the operator
term above. We first notice that 
\begin{eqnarray*}
&&\left\vert f\left( z^{\prime },u^{\prime }\right) -\sum_{\left\vert \alpha
^{\prime }\right\vert +2\beta ^{\prime }\leq M^{\prime }}c_{\alpha ^{\prime
},\beta ^{\prime }}\partial _{z}^{\alpha ^{\prime }}\partial _{u}^{\beta
^{\prime }}f\left( z,u\right) \left[ \left( z^{\prime },u^{\prime }\right)
\circ \left( z,u\right) ^{-1}\right] ^{\left( \alpha ^{\prime },\beta
^{\prime }\right) }\right\vert \\
&&\ \ \ \ \ \ \ \ \ \ \ \ \ \ \ \lesssim \left\vert \left( z^{\prime
},u^{\prime }\right) \circ \left( z,u\right) ^{-1}\right\vert ^{M^{\prime
}+1},
\end{eqnarray*}%
for all $\left( z^{\prime },u^{\prime }\right) \in \mathbb{H}^{n}$. Indeed,
for $\left\vert \left( z^{\prime },u^{\prime }\right) \circ \left(
z,u\right) ^{-1}\right\vert \leq 1$, this follows from the Taylor remainder
estimate, and for $\left\vert \left( z^{\prime },u^{\prime }\right) \circ
\left( z,u\right) ^{-1}\right\vert >1$, it follows from the mean value
theorem. Thus we have%
\begin{eqnarray*}
\left\vert I\right\vert &\lesssim &\left\Vert f\right\Vert _{\mathcal{M}%
^{M^{\prime }+\delta }\left( \mathbb{H}^{n}\right) }\int_{\left\{ \frac{%
\left\vert \left( z^{\prime },u^{\prime }\right) \circ \left( z^{\prime
\prime },u^{\prime \prime }\right) ^{-1}\right\vert }{\left\vert \left(
z,u\right) \circ \left( z^{\prime \prime },u^{\prime \prime }\right)
^{-1}\right\vert }\leq 4\right\} }\frac{\left\vert \left( z,u\right) \circ
\left( z^{\prime \prime },u^{\prime \prime }\right) ^{-1}\right\vert
^{\delta }}{\left\vert \left( z,u\right) \circ \left( z^{\prime \prime
},u^{\prime \prime }\right) ^{-1}\right\vert ^{Q+\delta }\left\vert \left(
z,u\right) \circ \left( z^{\prime },u^{\prime }\right) ^{-1}\right\vert
^{\left\vert \alpha ^{\prime }\right\vert +2\beta ^{\prime }}}dz^{\prime
}du^{\prime } \\
&\lesssim &\left\Vert f\right\Vert _{\mathcal{M}^{M^{\prime }+\delta }\left( 
\mathbb{H}^{n}\right) }\left\vert \left( z,u\right) \circ \left( z^{\prime
\prime },u^{\prime \prime }\right) ^{-1}\right\vert ^{\delta },
\end{eqnarray*}%
which gives the correct bound since $\left\vert \left( z,u\right)
\right\vert \leq 1$. Integral $II$ is estimated similarly. For integral $III$
we have%
\begin{eqnarray*}
\left\vert III\right\vert &\lesssim &\left\Vert f\right\Vert _{\mathcal{M}%
^{M^{\prime }+\delta }\left( \mathbb{H}^{n}\right) }\int_{\left\{ \frac{%
\left\vert \left( z^{\prime },u^{\prime }\right) \circ \left( z^{\prime
\prime },u^{\prime \prime }\right) ^{-1}\right\vert }{\left\vert \left(
z,u\right) \circ \left( z^{\prime \prime },u^{\prime \prime }\right)
^{-1}\right\vert }>2\right\} }\left\vert \left( z,u\right) \circ \left(
z^{\prime \prime },u^{\prime \prime }\right) ^{-1}\right\vert ^{\delta } \\
&&\times \left\vert \left( z^{\prime },u^{\prime }\right) \circ \left(
z^{\prime \prime },u^{\prime \prime }\right) ^{-1}\right\vert ^{-Q-\delta
}\left\vert \left( z^{\prime },u^{\prime }\right) \circ \left( z^{\prime
\prime },u^{\prime \prime }\right) ^{-1}\right\vert ^{-\left\vert \alpha
^{\prime }\right\vert -2\beta ^{\prime }}dz^{\prime }du^{\prime } \\
&\lesssim &\left\Vert f\right\Vert _{\mathcal{M}^{M^{\prime }+\delta }\left( 
\mathbb{H}^{n}\right) }\left\vert \left( z,u\right) \circ \left( z^{\prime
\prime },u^{\prime \prime }\right) ^{-1}\right\vert ^{\delta }.
\end{eqnarray*}%
Finally, 
\begin{eqnarray*}
\left\vert IV\right\vert &\lesssim &\sum_{\substack{ \left\vert \alpha
^{\prime }\right\vert +2\beta ^{\prime }  \\ \leq \left\vert \alpha
\right\vert +2\beta }}\left\vert \partial _{z}^{\alpha ^{\prime }}\partial
_{u}^{\beta ^{\prime }}f\left( z,u\right) -\sum_{\substack{ \left\vert
\alpha ^{\prime \prime }\right\vert +2\beta ^{\prime \prime }  \\ \leq
\left\vert \alpha \right\vert +2\beta -\left\vert \alpha ^{\prime
}\right\vert -2\beta ^{\prime }}}c_{\alpha ^{\prime \prime },\beta ^{\prime
\prime }}\partial _{z}^{\alpha ^{\prime \prime }+\alpha ^{\prime }}\partial
_{u}^{\beta ^{\prime \prime }+\beta ^{\prime }}f\left( z^{\prime \prime
},u^{\prime \prime }\right) \left[ \left( z,u\right) \circ \left( z^{\prime
\prime },u^{\prime \prime }\right) ^{-1}\right] ^{\left( \alpha ^{\prime
\prime },\beta ^{\prime \prime }\right) }\right\vert \\
&&\times \left\Vert T_{\left( \alpha ^{\prime },\beta ^{\prime }\right)
,\left( \alpha ,\beta \right) }\widetilde{\theta }\right\Vert _{\infty } \\
&\lesssim &\sum_{\substack{ \left\vert \alpha ^{\prime \prime }\right\vert
+2\beta ^{\prime \prime }  \\ \leq \left\vert \alpha \right\vert +2\beta
-\left\vert \alpha ^{\prime }\right\vert -2\beta ^{\prime }}}\left\vert
\partial _{z}^{\alpha ^{\prime \prime }+\alpha ^{\prime }}\partial
_{u}^{\beta ^{\prime \prime }+\beta ^{\prime }}f\left( \xi \right) -\partial
_{z}^{\alpha ^{\prime \prime }+\alpha ^{\prime }}\partial _{u}^{\beta
^{\prime \prime }+\beta ^{\prime }}f\left( z^{\prime \prime },u^{\prime
\prime }\right) \right\vert \\
&&\times \left\vert \left( z,u\right) \circ \left( z^{\prime \prime
},u^{\prime \prime }\right) ^{-1}\right\vert ^{\left\vert \alpha \right\vert
+2\beta -\left\vert \alpha ^{\prime }\right\vert -2\beta ^{\prime
}}\left\vert \left( z,u\right) \circ \left( z^{\prime \prime },u^{\prime
\prime }\right) ^{-1}\right\vert ^{\left\vert \alpha ^{\prime }\right\vert
+2\beta ^{\prime }-\left\vert \alpha \right\vert -2\beta } \\
&\lesssim &\left\Vert f\right\Vert _{\mathcal{M}^{M^{\prime }+\delta }\left( 
\mathbb{H}^{n}\right) }\sum_{\substack{ \left\vert \alpha ^{\prime
}\right\vert +2\beta ^{\prime }+\left\vert \alpha ^{\prime \prime
}\right\vert +2\beta ^{\prime \prime }  \\ \leq \left\vert \alpha
\right\vert +2\beta }}\left\vert \left( z,u\right) \circ \left( z^{\prime
\prime },u^{\prime \prime }\right) ^{-1}\right\vert ^{\delta } \\
&\lesssim &\left\Vert f\right\Vert _{\mathcal{M}^{M^{\prime }+\delta }\left( 
\mathbb{H}^{n}\right) }\left\vert \left( z,u\right) \circ \left( z^{\prime
\prime },u^{\prime \prime }\right) ^{-1}\right\vert ^{\delta }.
\end{eqnarray*}

\bigskip

\textbf{Case 4}: We prove the smoothness condition for $\partial
_{z}^{\alpha }\partial _{u}^{\beta }f\left( z,u\right) $ only when 
\begin{equation*}
\left\vert \left( z,u\right) \right\vert \geq 1\text{ and }\left\vert \left(
z,u\right) \circ \left( z^{\prime \prime },u^{\prime \prime }\right)
^{-1}\right\vert \leq \frac{1}{2}\left( 1+\left\vert \left( z,u\right)
\right\vert \right) ,
\end{equation*}%
since the remaining cases are similar, but easier. First we claim that for $%
\left( z^{\prime \prime \prime },u^{\prime \prime \prime }\right) ,\left(
z^{\prime \prime },u^{\prime \prime }\right) \in \mathbb{H}^{n}$, andwe have%
\begin{eqnarray*}
&&\left\vert \partial _{z}^{\alpha ^{\prime }}\partial _{u}^{\beta ^{\prime
}}f_{I}\left( z^{\prime \prime \prime },u^{\prime \prime \prime }\right)
-\partial _{z}^{\alpha ^{\prime }}\partial _{u}^{\beta ^{\prime
}}f_{I}\left( z^{\prime \prime },u^{\prime \prime }\right) \right\vert \\
&\lesssim &\left\Vert f\right\Vert _{\mathcal{M}_{flag}^{M^{\prime }+\delta
}\left( \mathbb{H}^{n}\right) }\left\vert \left( z^{\prime \prime \prime
},u^{\prime \prime \prime }\right) \circ \left( z^{\prime \prime },u^{\prime
\prime }\right) ^{-1}\right\vert ^{\delta }\left\vert \left( z,u\right)
\right\vert ^{-\left( Q+2\left\vert \alpha ^{\prime }\right\vert +\beta
^{\prime }\right) }.
\end{eqnarray*}%
To see this, first consider the case $\left\vert \left( z^{\prime \prime
\prime },u^{\prime \prime \prime }\right) \circ \left( z^{\prime \prime
},u^{\prime \prime }\right) ^{-1}\right\vert >\frac{1}{3}\left\vert \left(
z,u\right) \right\vert $. If $\left( z^{\prime \prime \prime },u^{\prime
\prime \prime }\right) $ lies in the support of $f_{I}\left( z,u\right) $,
then $\left\vert \left( z^{\prime \prime },u^{\prime \prime }\right)
\right\vert \approx \left\vert \left( z,u\right) \right\vert $, while
otherwise, $\partial _{z}^{\alpha ^{\prime }}\partial _{u}^{\beta ^{\prime
}}f_{I}\left( z^{\prime \prime \prime },u^{\prime \prime \prime }\right) =0$%
. Similarly for $\left( z^{\prime \prime \prime },u^{\prime \prime \prime
}\right) $. Using 
\begin{equation*}
\left\vert \partial _{z}^{\alpha ^{\prime }}\partial _{u}^{\beta ^{\prime
}}f_{I}\left( z^{\prime \prime \prime },u^{\prime \prime \prime }\right)
-\partial _{z}^{\alpha ^{\prime }}\partial _{u}^{\beta ^{\prime
}}f_{I}\left( z^{\prime \prime },u^{\prime \prime }\right) \right\vert \leq
\left\vert \partial _{z}^{\alpha ^{\prime }}\partial _{u}^{\beta ^{\prime
}}f_{I}\left( z^{\prime \prime \prime },u^{\prime \prime \prime }\right)
\right\vert +\left\vert \partial _{z}^{\alpha ^{\prime }}\partial
_{u}^{\beta ^{\prime }}f_{I}\left( z^{\prime \prime },u^{\prime \prime
}\right) \right\vert ,
\end{equation*}%
and the symmetry in $\left( z^{\prime \prime \prime },u^{\prime \prime
\prime }\right) $ and $\left( z^{\prime \prime },u^{\prime \prime }\right) $%
, it suffices to estimate just the first term on the right:%
\begin{eqnarray*}
&&\left\vert \partial _{z}^{\alpha ^{\prime }}\partial _{u}^{\beta ^{\prime
}}f_{I}\left( z^{\prime \prime \prime },u^{\prime \prime \prime }\right)
\right\vert \\
&\lesssim &\sum_{\left\vert \alpha ^{\prime \prime \prime }\right\vert
+2\beta ^{\prime \prime \prime }+\left\vert \alpha ^{\prime \prime
}\right\vert +2\beta ^{\prime \prime }=M^{\prime }}\left\vert \partial
_{z}^{\alpha ^{\prime \prime \prime }}\partial _{u}^{\beta ^{\prime \prime
\prime }}f_{I}\left( z^{\prime \prime \prime },u^{\prime \prime \prime
}\right) \left( \partial _{z}^{\alpha ^{\prime \prime }}\partial _{u}^{\beta
^{\prime \prime }}f_{I}\left( z^{\prime \prime \prime },u^{\prime \prime
\prime }\right) \right) \right\vert \\
&\lesssim &\left\Vert f\right\Vert _{\mathcal{M}^{M^{\prime }+\delta }\left( 
\mathbb{H}^{n}\right) }\sum_{\left\vert \alpha ^{\prime \prime \prime
}\right\vert +2\beta ^{\prime \prime \prime }+\left\vert \alpha ^{\prime
\prime }\right\vert +2\beta ^{\prime \prime }=M^{\prime }}\left(
1+\left\vert z^{\prime \prime \prime }\right\vert ^{2}+\left\vert u^{\prime
\prime \prime }\right\vert \right) ^{-\frac{Q+M^{\prime }+\left\vert \alpha
^{\prime \prime \prime }\right\vert +2\beta ^{\prime \prime \prime }}{2}%
}\left( \left\vert z\right\vert ^{2}+\left\vert u\right\vert \right) ^{-%
\frac{\left\vert \alpha ^{\prime \prime }\right\vert +2\beta ^{\prime \prime
}}{2}} \\
&\lesssim &\left\Vert f\right\Vert _{\mathcal{M}^{M^{\prime }+\delta }\left( 
\mathbb{H}^{n}\right) }\sum_{\left\vert \alpha ^{\prime \prime \prime
}\right\vert +2\beta ^{\prime \prime \prime }+\left\vert \alpha ^{\prime
\prime }\right\vert +2\beta ^{\prime \prime }=M^{\prime }}\left( \frac{%
\left\vert \left( z^{\prime \prime \prime },u^{\prime \prime \prime }\right)
\circ \left( z^{\prime \prime },u^{\prime \prime }\right) ^{-1}\right\vert }{%
\left\vert \left( z,u\right) \right\vert }\right) ^{\delta }\left\vert
\left( z,u\right) \right\vert ^{-\left( Q+M^{\prime }+\left\vert \alpha
^{\prime \prime \prime }\right\vert +2\beta ^{\prime \prime \prime
}+\left\vert \alpha ^{\prime \prime }\right\vert +2\beta ^{\prime \prime
}\right) } \\
&\lesssim &\left\Vert f\right\Vert _{\mathcal{M}^{M^{\prime }+\delta }\left( 
\mathbb{H}^{n}\right) }\left\vert \left( z^{\prime \prime \prime },u^{\prime
\prime \prime }\right) \circ \left( z^{\prime \prime },u^{\prime \prime
}\right) ^{-1}\right\vert ^{\delta }\left\vert \left( z,u\right) \right\vert
^{-\left( Q+M^{\prime }\right) }.
\end{eqnarray*}

On the other hand, suppose that $\left\vert \left( z^{\prime \prime \prime
},u^{\prime \prime \prime }\right) \circ \left( z^{\prime \prime },u^{\prime
\prime }\right) ^{-1}\right\vert \leq \frac{1}{3}\left\vert \left(
z,u\right) \right\vert $. Without loss of generality we may assume that at
least one of $\left( z^{\prime \prime \prime },u^{\prime \prime \prime
}\right) $ or $\left( z^{\prime \prime },u^{\prime \prime }\right) $ lies in
the support of $I$. Then%
\begin{equation*}
\left\vert \left( z^{\prime \prime \prime },u^{\prime \prime \prime }\right)
\circ \left( z,u\right) ^{-1}\right\vert ,\left\vert \left( z^{\prime \prime
},u^{\prime \prime }\right) \circ \left( z,u\right) ^{-1}\right\vert \leq 
\frac{5}{6}\left\vert \left( z,u\right) \right\vert \text{,}
\end{equation*}%
and it follows that $\left\vert \left( z^{\prime \prime \prime },u^{\prime
\prime \prime }\right) \right\vert $ and $\left\vert \left( z^{\prime \prime
},u^{\prime \prime }\right) \right\vert $ are comparable to $\left\vert
\left( z,u\right) \right\vert $. Thus we have%
\begin{eqnarray*}
&&\left\vert \partial _{z}^{\alpha ^{\prime }}\partial _{u}^{\beta ^{\prime
}}f_{I}\left( z^{\prime \prime \prime },u^{\prime \prime \prime }\right)
-\partial _{z}^{\alpha ^{\prime }}\partial _{u}^{\beta ^{\prime
}}f_{I}\left( z^{\prime \prime },u^{\prime \prime }\right) \right\vert \\
&\lesssim &\sum_{\left\vert \alpha ^{\prime \prime \prime }\right\vert
+2\beta ^{\prime \prime \prime }+\left\vert \alpha ^{\prime \prime
}\right\vert +2\beta ^{\prime \prime }=M^{\prime }}\left\vert \partial
_{z}^{\alpha ^{\prime \prime \prime }}\partial _{u}^{\beta ^{\prime \prime
\prime }}f\left( z^{\prime \prime \prime },u^{\prime \prime \prime }\right)
\partial _{z}^{\alpha ^{\prime \prime }}\partial _{u}^{\beta ^{\prime \prime
}}I\left( z^{\prime \prime \prime },u^{\prime \prime \prime }\right)
-\partial _{z}^{\alpha ^{\prime \prime \prime }}\partial _{u}^{\beta
^{\prime \prime \prime }}f\left( z^{\prime \prime },u^{\prime \prime
}\right) \partial _{z}^{\alpha ^{\prime \prime }}\partial _{u}^{\beta
^{\prime \prime }}I\left( z^{\prime \prime },u^{\prime \prime }\right)
\right\vert \\
&\lesssim &\sum_{\left\vert \alpha ^{\prime \prime \prime }\right\vert
+2\beta ^{\prime \prime \prime }+\left\vert \alpha ^{\prime \prime
}\right\vert +2\beta ^{\prime \prime }=M^{\prime }}\left\vert \partial
_{z}^{\alpha ^{\prime \prime \prime }}\partial _{u}^{\beta ^{\prime \prime
\prime }}f\left( z^{\prime \prime \prime },u^{\prime \prime \prime }\right)
\left\{ \partial _{z}^{\alpha ^{\prime \prime }}\partial _{u}^{\beta
^{\prime \prime }}I\left( z^{\prime \prime \prime },u^{\prime \prime \prime
}\right) -\partial _{z}^{\alpha ^{\prime \prime }}\partial _{u}^{\beta
^{\prime \prime }}I\left( z^{\prime \prime },u^{\prime \prime }\right)
\right\} \right\vert \\
&&+\sum_{\left\vert \alpha ^{\prime \prime \prime }\right\vert +2\beta
^{\prime \prime \prime }+\left\vert \alpha ^{\prime \prime }\right\vert
+2\beta ^{\prime \prime }=M^{\prime }}\left\vert \left\{ \partial
_{z}^{\alpha ^{\prime \prime \prime }}\partial _{u}^{\beta ^{\prime \prime
\prime }}f\left( z^{\prime \prime \prime },u^{\prime \prime \prime }\right)
-\partial _{z}^{\alpha ^{\prime \prime \prime }}\partial _{u}^{\beta
^{\prime \prime \prime }}f\left( z^{\prime \prime },u^{\prime \prime
}\right) \right\} \partial _{z}^{\alpha ^{\prime \prime }}\partial
_{u}^{\beta ^{\prime \prime }}I\left( z^{\prime \prime },u^{\prime \prime
}\right) \right\vert ,
\end{eqnarray*}%
which by the definition of decay for $f\in \mathcal{M}^{M^{\prime }+\delta
}\left( \mathbb{H}^{n}\right) $, and the estimate 
\begin{equation*}
\left\vert \partial _{z}^{\alpha ^{\prime \prime }}\partial _{u}^{\beta
^{\prime \prime }}I\left( w,v\right) \right\vert \leq C_{M^{\prime \prime
}}\left\vert \partial _{z}^{\alpha ^{\prime \prime }}\partial _{u}^{\beta
^{\prime \prime }}\widetilde{\theta }\left( \frac{8\left( z,u\right) \circ
\left( w,v\right) ^{-1}}{\left\vert \left( z,u\right) \right\vert }\right)
\right\vert \left\vert \left( z,u\right) \right\vert ^{-M^{\prime \prime }},
\end{equation*}%
is dominated by%
\begin{eqnarray*}
&&\left\Vert f\right\Vert _{\mathcal{M}^{M^{\prime }+\delta }\left( \mathbb{H%
}^{n}\right) }\sum_{\left\vert \alpha ^{\prime \prime \prime }\right\vert
+2\beta ^{\prime \prime \prime }+\left\vert \alpha ^{\prime \prime
}\right\vert +2\beta ^{\prime \prime }=M^{\prime }}\left\vert \left(
z,u\right) \right\vert ^{-\left( Q+M^{\prime }+\left\vert \alpha ^{\prime
\prime \prime }\right\vert +2\beta ^{\prime \prime \prime }\right) } \\
&&\times \left\vert \partial _{z}^{\alpha ^{\prime \prime }}\partial
_{u}^{\beta ^{\prime \prime }}\widetilde{\theta }\left( \frac{\left\vert
\left( z^{\prime \prime \prime },u^{\prime \prime \prime }\right) \circ
\left( z,u\right) ^{-1}\right\vert }{\left\vert \left( z,u\right)
\right\vert }\right) -\partial _{z}^{\alpha ^{\prime \prime }}\partial
_{u}^{\beta ^{\prime \prime }}\widetilde{\theta }\left( \frac{\left\vert
\left( z^{\prime \prime },u^{\prime \prime }\right) \circ \left( z,u\right)
^{-1}\right\vert }{\left\vert \left( z,u\right) \right\vert }\right)
\right\vert \left\vert \left( z,u\right) \right\vert ^{-\left( \left\vert
\alpha ^{\prime \prime }\right\vert +2\beta ^{\prime \prime }\right) } \\
&&+\left\Vert f\right\Vert _{\mathcal{M}^{M^{\prime }+\delta }\left( \mathbb{%
H}^{n}\right) }\sum_{\left\vert \alpha ^{\prime \prime \prime }\right\vert
+2\beta ^{\prime \prime \prime }+\left\vert \alpha ^{\prime \prime
}\right\vert +2\beta ^{\prime \prime }=M^{\prime }}\left\vert \left(
z^{\prime \prime \prime },u^{\prime \prime \prime }\right) \circ \left(
z^{\prime \prime },u^{\prime \prime }\right) ^{-1}\right\vert ^{\delta } \\
&&\ \ \ \ \ \ \ \ \ \ \ \ \ \ \ \ \ \ \ \ \ \ \ \ \ \ \ \ \ \ \ \ \ \ \ \ \
\ \ \ \times \left\vert \left( z,u\right) \right\vert ^{-\left( Q+M^{\prime
}+\left\vert \alpha ^{\prime \prime \prime }\right\vert +2\beta ^{\prime
\prime \prime }+\delta \right) }\left\vert \left( z,u\right) \right\vert
^{-\left( \left\vert \alpha ^{\prime \prime }\right\vert +2\beta ^{\prime
\prime }\right) } \\
&\lesssim &\left\Vert f\right\Vert _{\mathcal{M}^{M^{\prime }+\delta }\left( 
\mathbb{H}^{n}\right) }\left\vert \left( z^{\prime \prime \prime },u^{\prime
\prime \prime }\right) \circ \left( z^{\prime \prime },u^{\prime \prime
}\right) ^{-1}\right\vert ^{\delta }\left\vert \left( z,u\right) \right\vert
^{-\left( Q+2M^{\prime }+\delta \right) }.
\end{eqnarray*}

We now return to the estimate for the difference%
\begin{equation*}
\partial _{z}^{\alpha ^{\prime }}\partial _{u}^{\beta ^{\prime
}}Tf_{I}\left( z,u\right) -\partial _{z}^{\alpha ^{\prime }}\partial
_{u}^{\beta ^{\prime }}Tf_{I}\left( z^{\prime \prime },u^{\prime \prime
}\right) =I+II+III+IV,
\end{equation*}%
where we are using the decomposition introduced in Case 3, but applied to
the function $f_{I}$ instead of $f$. To estimate the first term we begin by
writing%
\begin{eqnarray*}
&&f\left( z^{\prime },u^{\prime }\right) -\sum_{\left\vert \alpha ^{\prime
\prime }\right\vert +2\beta ^{\prime \prime }\leq M^{\prime }}c_{\alpha
^{\prime \prime },\beta ^{\prime \prime }}\partial _{z}^{\alpha ^{\prime
\prime }}\partial _{u}^{\beta ^{\prime \prime }}f\left( z,u\right) \left[
\left( z^{\prime },u^{\prime }\right) \circ \left( z,u\right) ^{-1}\right]
^{\left( \alpha ^{\prime \prime },\beta ^{\prime \prime }\right) } \\
&=&\sum_{\left\vert \alpha ^{\prime \prime }\right\vert +2\beta ^{\prime
\prime }=M^{\prime }}c_{\alpha ^{\prime \prime },\beta ^{\prime \prime
}}\partial _{z}^{\alpha ^{\prime \prime }}\partial _{u}^{\beta ^{\prime
\prime }}f\left( \zeta ,\eta \right) \left[ \left( z^{\prime },u^{\prime
}\right) \circ \left( z,u\right) ^{-1}\right] ^{\left( \alpha ^{\prime
\prime },\beta ^{\prime \prime }\right) } \\
&&-\sum_{\left\vert \alpha ^{\prime \prime }\right\vert +2\beta ^{\prime
\prime }=M^{\prime }}c_{\alpha ^{\prime \prime },\beta ^{\prime \prime
}}\partial _{z}^{\alpha ^{\prime \prime }}\partial _{u}^{\beta ^{\prime
\prime }}f\left( z,u\right) \left[ \left( z^{\prime },u^{\prime }\right)
\circ \left( z,u\right) ^{-1}\right] ^{\left( \alpha ^{\prime \prime },\beta
^{\prime \prime }\right) } \\
&&\sum_{\left\vert \alpha ^{\prime \prime }\right\vert +2\beta ^{\prime
\prime }=M^{\prime }}c_{\alpha ^{\prime \prime },\beta ^{\prime \prime }} 
\left[ \partial _{z}^{\alpha ^{\prime \prime }}\partial _{u}^{\beta ^{\prime
\prime }}f\left( \zeta ,\eta \right) -\partial _{z}^{\alpha ^{\prime \prime
}}\partial _{u}^{\beta ^{\prime \prime }}f\left( z,u\right) \right] \left[
\left( z^{\prime },u^{\prime }\right) \circ \left( z,u\right) ^{-1}\right]
^{\left( \alpha ^{\prime \prime },\beta ^{\prime \prime }\right) },
\end{eqnarray*}%
where $\left( \zeta ,\eta \right) $ lies on line segment joining $\left(
z,u\right) $ and $\left( z^{\prime \prime },u^{\prime \prime }\right) $, in
order to obtain 
\begin{eqnarray*}
I &=&\int \partial _{z}^{\alpha ^{\prime }}\partial _{u}^{\beta ^{\prime
}}K\left( \left( z,u\right) ,\left( z^{\prime },u^{\prime }\right) \right) \\
&&\times \left\{ f_{I}\left( z^{\prime },u^{\prime }\right)
-\sum_{\left\vert \alpha ^{\prime \prime }\right\vert +2\beta ^{\prime
\prime }\leq M^{\prime }}c_{\alpha ^{\prime \prime },\beta ^{\prime \prime
}}\partial _{z}^{\alpha ^{\prime \prime }}\partial _{u}^{\beta ^{\prime
\prime }}f_{I}\left( z,u\right) \left[ \left( z^{\prime },u^{\prime }\right)
\circ \left( z,u\right) ^{-1}\right] ^{\left( \alpha ^{\prime \prime },\beta
^{\prime \prime }\right) }\right\} \widetilde{\theta }\left( z^{\prime
},u^{\prime }\right) dz^{\prime }du^{\prime } \\
&=&\int \partial _{z}^{\alpha ^{\prime }}\partial _{u}^{\beta ^{\prime
}}K\left( \left( z,u\right) ,\left( z^{\prime },u^{\prime }\right) \right) \\
&&\times \left\{ \sum_{\left\vert \alpha ^{\prime \prime }\right\vert
+2\beta ^{\prime \prime }=M^{\prime }}c_{\alpha ^{\prime \prime },\beta
^{\prime \prime }}\left[ \partial _{z}^{\alpha ^{\prime \prime }}\partial
_{u}^{\beta ^{\prime \prime }}f_{I}\left( \zeta ,\eta \right) -\partial
_{z}^{\alpha ^{\prime \prime }}\partial _{u}^{\beta ^{\prime \prime
}}f_{I}\left( z,u\right) \right] \left[ \left( z^{\prime },u^{\prime
}\right) \circ \left( z,u\right) ^{-1}\right] ^{\left( \alpha ^{\prime
\prime },\beta ^{\prime \prime }\right) }\right\} \widetilde{\theta }\left(
z^{\prime },u^{\prime }\right) dz^{\prime }du^{\prime }.
\end{eqnarray*}%
Hence, using the kernel estimate for $K$, we obtain%
\begin{eqnarray*}
&\lesssim &\int_{\left\{ \frac{\left\vert \left( z^{\prime },u^{\prime
}\right) \circ \left( z^{\prime \prime },u^{\prime \prime }\right)
^{-1}\right\vert }{\left\vert \left( z,u\right) \circ \left( z^{\prime
\prime },u^{\prime \prime }\right) ^{-1}\right\vert }\leq 4\right\}
}\left\vert \left( z,u\right) \circ \left( z^{\prime },u^{\prime }\right)
^{-1}\right\vert ^{-\left( Q+M^{\prime }\right) } \\
&&\times \sum_{\left\vert \alpha ^{\prime \prime }\right\vert +2\beta
^{\prime \prime }=M^{\prime }}\left\vert \partial _{z}^{\alpha ^{\prime
\prime }}\partial _{u}^{\beta ^{\prime \prime }}f_{I}\left( \zeta ,\eta
\right) -\partial _{z}^{\alpha ^{\prime \prime }}\partial _{u}^{\beta
^{\prime \prime }}f_{I}\left( z,u\right) \right\vert dz^{\prime }du^{\prime }
\\
&\lesssim &\left\Vert f\right\Vert _{\mathcal{M}^{M^{\prime }+\delta }\left( 
\mathbb{H}^{n}\right) }^{M^{\prime }+\delta }\int_{\left\{ \frac{\left\vert
\left( z^{\prime },u^{\prime }\right) \circ \left( z,u\right)
^{-1}\right\vert }{\left\vert \left( z,u\right) \circ \left( z^{\prime
\prime },u^{\prime \prime }\right) ^{-1}\right\vert }\leq 5\right\}
}\left\vert \left( z,u\right) \circ \left( z^{\prime },u^{\prime }\right)
^{-1}\right\vert ^{\delta }\left\vert \left( z,u\right) \right\vert
^{-\left( Q+2\left\vert \alpha ^{\prime }\right\vert +4\beta ^{\prime
}+\delta \right) }dz^{\prime }du^{\prime } \\
&\lesssim &\left\Vert f\right\Vert _{\mathcal{M}^{M^{\prime }+\delta }\left( 
\mathbb{H}^{n}\right) }^{M^{\prime }+\delta }\left\vert \left( z,u\right)
\circ \left( z^{\prime \prime },u^{\prime \prime }\right) ^{-1}\right\vert
^{\delta }\left\vert \left( z,u\right) \right\vert ^{-\left( 2M^{\prime
}+\delta \right) },
\end{eqnarray*}%
which is the correct estimate since $M^{\prime }=\left\vert \alpha ^{\prime
}\right\vert +2\beta ^{\prime }$. Term $II$ is estimated in similar fashion.
For term $III$ we have%
\begin{eqnarray*}
III &=&\int \left\{ \partial _{z}^{\alpha ^{\prime }}\partial _{u}^{\beta
^{\prime }}K\left( \left( z,u\right) ,\left( z^{\prime },u^{\prime }\right)
\right) -\partial _{z}^{\alpha ^{\prime }}\partial _{u}^{\beta ^{\prime
}}K\left( \left( z^{\prime \prime },u^{\prime \prime }\right) ,\left(
z^{\prime },u^{\prime }\right) \right) \right\} \\
&&\times \left\{ \sum_{\left\vert \alpha ^{\prime \prime }\right\vert
+2\beta ^{\prime \prime }=M^{\prime }}c_{\alpha ^{\prime \prime },\beta
^{\prime \prime }}\left[ \partial _{z}^{\alpha ^{\prime \prime }}\partial
_{u}^{\beta ^{\prime \prime }}f_{I}\left( \xi ,\eta \right) -\partial
_{z}^{\alpha ^{\prime \prime }}\partial _{u}^{\beta ^{\prime \prime
}}f_{I}\left( z^{\prime \prime },u^{\prime \prime }\right) \right] \left[
\left( z^{\prime },u^{\prime }\right) \circ \left( z^{\prime \prime
},u^{\prime \prime }\right) ^{-1}\right] ^{\left( \alpha ^{\prime \prime
},\beta ^{\prime \prime }\right) }\right\} \\
&&\times \left( 1-\widetilde{\theta }\left( z^{\prime },u^{\prime }\right)
\right) dz^{\prime }du^{\prime },
\end{eqnarray*}%
which gives%
\begin{eqnarray*}
\left\vert III\right\vert &\lesssim &\left\Vert f\right\Vert _{\mathcal{M}%
^{M^{\prime }+\delta }\left( \mathbb{H}^{n}\right) }\left\vert \left(
z,u\right) \circ \left( z^{\prime \prime },u^{\prime \prime }\right)
^{-1}\right\vert ^{1+\delta }\left\vert \left( z,u\right) \right\vert
^{-\left( Q+\left\vert \alpha ^{\prime }\right\vert +2\beta ^{\prime
}+\delta \right) } \\
&&\times \int_{\left\{ \frac{\left\vert \left( z,u\right) \circ \left(
z^{\prime \prime },u^{\prime \prime }\right) ^{-1}\right\vert }{\left\vert
\left( z^{\prime \prime },u^{\prime \prime }\right) \circ \left( z^{\prime
},u^{\prime }\right) ^{-1}\right\vert }\leq 2\right\} }\left\vert \left(
z^{\prime },u^{\prime }\right) \circ \left( z^{\prime \prime },u^{\prime
\prime }\right) ^{-1}\right\vert ^{-\left( Q+M^{\prime }+1\right)
}\left\vert \left( z^{\prime },u^{\prime }\right) \circ \left( z^{\prime
\prime },u^{\prime \prime }\right) ^{-1}\right\vert ^{M^{\prime }}dz^{\prime
}du^{\prime } \\
&\lesssim &\left\Vert f\right\Vert _{\mathcal{M}^{M^{\prime }+\delta }\left( 
\mathbb{H}^{n}\right) }\left\vert \left( z,u\right) \circ \left( z^{\prime
\prime },u^{\prime \prime }\right) ^{-1}\right\vert ^{\delta }\left\vert
\left( z,u\right) \right\vert ^{-\left( Q+2M^{\prime }+\delta \right) }.
\end{eqnarray*}%
Finally, using what we proved in Case 1, we have%
\begin{eqnarray*}
&&IV=\sum_{\left\vert \alpha ^{\prime \prime }\right\vert +2\beta ^{\prime
\prime }\leq M^{\prime }}c_{\alpha ^{\prime \prime },\beta ^{\prime \prime
}}\left\{ \partial _{z}^{\alpha ^{\prime \prime }}\partial _{u}^{\beta
^{\prime \prime }}f_{I}\left( z,u\right) -\sum_{\substack{ \left\vert \alpha
^{\prime \prime \prime }\right\vert +2\beta ^{\prime \prime \prime }  \\ %
\leq \left\vert \alpha ^{\prime }\right\vert +2\beta ^{\prime }-\left\vert
\alpha ^{\prime \prime }\right\vert -2\beta ^{\prime \prime }}}c_{\alpha
^{\prime \prime \prime },\beta ^{\prime \prime \prime }}\partial
_{z}^{\alpha ^{\prime \prime \prime }+\alpha ^{\prime \prime }}\partial
_{u}^{2\beta ^{\prime \prime \prime }+2\beta ^{\prime \prime }}f_{I}\left(
z^{\prime \prime },u^{\prime \prime }\right) \right. \\
&&\ \ \ \ \ \ \ \ \ \ \ \ \ \ \ \ \ \ \ \ \ \ \ \ \ \left. \times \left[
\left( z,u\right) \circ \left( z^{\prime \prime },u^{\prime \prime }\right)
^{-1}\right] ^{\left( \alpha ^{\prime \prime \prime },\beta ^{\prime \prime
\prime }\right) }\right\} T_{\left( \alpha ^{\prime \prime },\beta ^{\prime
\prime }\right) ,\left( \alpha ^{\prime },\beta ^{\prime }\right) }%
\widetilde{\theta }\left( z,u\right)
\end{eqnarray*}%
\begin{eqnarray*}
&=&\sum_{\left\vert \alpha ^{\prime \prime }\right\vert +2\beta ^{\prime
\prime }\leq M^{\prime }}c_{\alpha ^{\prime \prime },\beta ^{\prime \prime
}}\left\{ \sum_{\substack{ \left\vert \alpha ^{\prime \prime \prime
}\right\vert +2\beta ^{\prime \prime \prime }  \\ =\left\vert \alpha
^{\prime }\right\vert +2\beta ^{\prime }-\left\vert \alpha ^{\prime \prime
}\right\vert -2\beta ^{\prime \prime }}}c_{\alpha ^{\prime \prime \prime
},\beta ^{\prime \prime \prime }}\left[ \partial _{z}^{\alpha ^{\prime
\prime \prime }+\alpha ^{\prime \prime }}\partial _{u}^{2\beta ^{\prime
\prime \prime }+2\beta ^{\prime \prime }}f_{I}\left( \xi ,\eta \right)
-\partial _{z}^{\alpha ^{\prime \prime \prime }+\alpha ^{\prime \prime
}}\partial _{u}^{2\beta ^{\prime \prime \prime }+2\beta ^{\prime \prime
}}f_{I}\left( z^{\prime \prime },u^{\prime \prime }\right) \right] \right. \\
&&\ \ \ \ \ \ \ \ \ \ \ \ \ \ \ \ \ \ \ \ \left. \times \left[ \left(
z,u\right) \circ \left( z^{\prime \prime },u^{\prime \prime }\right) ^{-1}%
\right] ^{\left( \alpha ^{\prime \prime \prime },\beta ^{\prime \prime
\prime }\right) }\right\} T_{\left( \alpha ^{\prime \prime },\beta ^{\prime
\prime }\right) ,\left( \alpha ^{\prime },\beta ^{\prime }\right) }%
\widetilde{\theta }\left( z,u\right) ,
\end{eqnarray*}%
and so%
\begin{eqnarray*}
\left\vert IV\right\vert &\lesssim &\sum_{\substack{ \left\vert \alpha
^{\prime \prime \prime }\right\vert +2\beta ^{\prime \prime \prime
}+\left\vert \alpha ^{\prime \prime }\right\vert +2\beta ^{\prime \prime } 
\\ =\left\vert \alpha ^{\prime }\right\vert +2\beta ^{\prime }}}\left\vert
\partial _{z}^{\alpha ^{\prime \prime \prime }+\alpha ^{\prime \prime
}}\partial _{u}^{\beta ^{\prime \prime \prime }+\beta ^{\prime \prime
}}f_{I}\left( \xi ,\eta \right) -\partial _{z}^{\alpha ^{\prime \prime
\prime }+\alpha ^{\prime \prime }}\partial _{u}^{\beta ^{\prime \prime
\prime }+\beta ^{\prime \prime }}f_{I}\left( z^{\prime \prime },u^{\prime
\prime }\right) \right\vert \\
&&\times \left\vert \left( z,u\right) \circ \left( z^{\prime \prime
},u^{\prime \prime }\right) ^{-1}\right\vert ^{\left\vert \alpha ^{\prime
\prime \prime }\right\vert +2\beta ^{\prime \prime \prime }}\left\vert
\left( z,u\right) \circ \left( z^{\prime \prime },u^{\prime \prime }\right)
^{-1}\right\vert ^{\left\vert \alpha ^{\prime }\right\vert +2\beta ^{\prime
}-\left( \left\vert \alpha ^{\prime \prime \prime }\right\vert +2\beta
^{\prime \prime \prime }\right) } \\
&\lesssim &\left\Vert f\right\Vert _{\mathcal{M}^{M^{\prime }+\delta }\left( 
\mathbb{H}^{n}\right) }\left\vert \left( z,u\right) \circ \left( z^{\prime
\prime },u^{\prime \prime }\right) ^{-1}\right\vert ^{\delta }\left\vert
\left( z,u\right) \right\vert ^{-\left( Q+2M^{\prime }+\delta \right) }.
\end{eqnarray*}

Now we estimate the smoothness for the next term $\partial _{z}^{\alpha
^{\prime }}\partial _{u}^{\beta ^{\prime }}Tf_{J}\left( z,u\right) $. We
write%
\begin{eqnarray*}
&&\left\vert \partial _{z}^{\alpha ^{\prime }}\partial _{u}^{\beta ^{\prime
}}Tf_{J}\left( z,u\right) -\partial _{z}^{\alpha ^{\prime }}\partial
_{u}^{\beta ^{\prime }}Tf_{J}\left( z^{\prime \prime },u^{\prime \prime
}\right) \right\vert \\
&=&\left\vert \int \left\{ \partial _{z}^{\alpha ^{\prime }}\partial
_{u}^{\beta ^{\prime }}K\left( \left( z,u\right) ,\left( z^{\prime
},u^{\prime }\right) \right) -\partial _{z}^{\alpha ^{\prime }}\partial
_{u}^{\beta ^{\prime }}K\left( \left( z^{\prime \prime },u^{\prime \prime
}\right) ,\left( z^{\prime },u^{\prime }\right) \right) \right\} f_{J}\left(
z^{\prime },u^{\prime }\right) dz^{\prime }du^{\prime }\right\vert
\end{eqnarray*}%
and adding and subtracting a Taylor polynomial, we bound this by%
\begin{eqnarray*}
&&IV_{1}\equiv \left\vert \int \left[ \left\{ \partial _{z}^{\alpha ^{\prime
}}\partial _{u}^{\beta ^{\prime }}K\left( \left( z,u\right) ,\left(
z^{\prime },u^{\prime }\right) \right) -\partial _{z}^{\alpha ^{\prime
}}\partial _{u}^{\beta ^{\prime }}K\left( \left( z^{\prime \prime
},u^{\prime \prime }\right) ,\left( z^{\prime },u^{\prime }\right) \right)
\right\} \right. \right. \\
&&\left. -\sum_{\substack{ \left\vert \alpha ^{\prime \prime }\right\vert
+2\beta ^{\prime \prime }  \\ \leq \left\vert \alpha ^{\prime }\right\vert
+2\beta ^{\prime }}}c_{\alpha ^{\prime \prime },\beta ^{\prime \prime
}}\left\{ \partial _{z}^{\alpha ^{\prime \prime }+\alpha ^{\prime }}\partial
_{u}^{\beta ^{\prime \prime }+\beta ^{\prime }}K\left( \left( z,u\right)
,\left( 0,0\right) \right) -\partial _{z}^{\alpha ^{\prime \prime }+\alpha
^{\prime }}\partial _{u}^{\beta ^{\prime \prime }+\beta ^{\prime }}K\left(
\left( z^{\prime \prime },u^{\prime \prime }\right) ,\left( 0,0\right)
\right) \right\} \left( z^{\prime },u^{\prime }\right) ^{\left( \alpha
^{\prime \prime },\beta ^{\prime \prime }\right) }\right] \\
&&\ \ \ \ \ \ \ \ \ \ \left. \times f_{J}\left( z^{\prime },u^{\prime
}\right) dz^{\prime }du^{\prime }\right\vert
\end{eqnarray*}%
plus%
\begin{eqnarray*}
&&IV_{2}\equiv \\
&&\left\vert \int \left[ \sum_{\substack{ \left\vert \alpha ^{\prime \prime
}\right\vert +2\beta ^{\prime \prime }  \\ \leq \left\vert \alpha ^{\prime
}\right\vert +2\beta ^{\prime }}}c_{\alpha ^{\prime \prime },\beta ^{\prime
\prime }}\left\{ \partial _{z}^{\alpha ^{\prime \prime }+\alpha ^{\prime
}}\partial _{u}^{\beta ^{\prime \prime }+\beta ^{\prime }}K\left( \left(
z,u\right) ,\left( 0,0\right) \right) -\partial _{z}^{\alpha ^{\prime \prime
}+\alpha ^{\prime }}\partial _{u}^{\beta ^{\prime \prime }+\beta ^{\prime
}}K\left( \left( z^{\prime \prime },u^{\prime \prime }\right) ,\left(
0,0\right) \right) \right\} \left( z^{\prime },u^{\prime }\right) ^{\left(
\alpha ^{\prime \prime },\beta ^{\prime \prime }\right) }\right] \right. \\
&&\ \ \ \ \ \ \ \ \ \ \left. \times f_{J}\left( z^{\prime },u^{\prime
}\right) dz^{\prime }du^{\prime }\right\vert .
\end{eqnarray*}%
Now we use the double difference estimate for $\partial _{z}^{\alpha
^{\prime \prime }+\alpha ^{\prime }}\partial _{u}^{\beta ^{\prime \prime
}+\beta ^{\prime }}K\left( \left( z,u\right) ,\left( z^{\prime },u^{\prime
}\right) \right) $ to obtain%
\begin{eqnarray*}
\left\vert IV_{1}\right\vert &\lesssim &\int \sum_{\substack{ \left\vert
\alpha ^{\prime \prime }\right\vert +2\beta ^{\prime \prime }  \\ %
=\left\vert \alpha ^{\prime }\right\vert +2\beta ^{\prime }}}\left\vert
\left( \left\{ \partial _{z}^{\alpha ^{\prime \prime }+\alpha ^{\prime
}}\partial _{u}^{\beta ^{\prime \prime }+\beta ^{\prime }}K\left( \left(
z,u\right) ,\left( \xi ,\eta \right) \right) -\partial _{z}^{\alpha ^{\prime
\prime }+\alpha ^{\prime }}\partial _{u}^{\beta ^{\prime \prime }+\beta
^{\prime }}K\left( \left( z,u\right) ,\left( 0,0\right) \right) \right\}
\right. \right. \\
&&\left. \left. -\left\{ \partial _{z}^{\alpha ^{\prime \prime }+\alpha
^{\prime }}\partial _{u}^{\beta ^{\prime \prime }+\beta ^{\prime }}K\left(
\left( z^{\prime \prime },u^{\prime \prime }\right) ,\left( \xi ,\eta
\right) \right) -\partial _{z}^{\alpha ^{\prime \prime }+\alpha ^{\prime
}}\partial _{u}^{\beta ^{\prime \prime }+\beta ^{\prime }}K\left( \left(
z^{\prime \prime },u^{\prime \prime }\right) ,\left( 0,0\right) \right)
\right\} \right) \left( z^{\prime },u^{\prime }\right) ^{\left( \alpha
^{\prime \prime },\beta ^{\prime \prime }\right) }\right\vert \\
&&\ \ \ \ \ \ \ \ \ \ \times \left\vert f_{J}\left( z^{\prime },u^{\prime
}\right) \right\vert dz^{\prime }du^{\prime } \\
&\lesssim &\left\Vert f\right\Vert _{\mathcal{M}^{M^{\prime }+\delta }\left( 
\mathbb{H}^{n}\right) }\int_{\left\{ \frac{\left\vert \left( z^{\prime
},u^{\prime }\right) \right\vert }{\left\vert \left( z,u\right) \right\vert }%
\leq \frac{1}{2}\right\} }\left\vert \left( z,u\right) \circ \left(
z^{\prime \prime },u^{\prime \prime }\right) ^{-1}\right\vert \\
&&\ \ \ \ \ \ \ \ \ \ \ \ \ \ \ \ \ \ \ \ \times \left\vert \left(
z,u\right) \right\vert ^{-\left( Q+2\left\vert \alpha ^{\prime }\right\vert
+4\beta ^{\prime }+1\right) }\left( 1+\left\vert \left( z^{\prime
},u^{\prime }\right) \right\vert \right) ^{-\left( Q+M^{\prime }\right)
}dz^{\prime }du^{\prime } \\
&\lesssim &\left\Vert f\right\Vert _{\mathcal{M}^{M^{\prime }+\delta }\left( 
\mathbb{H}^{n}\right) }\left\vert \left( z,u\right) \circ \left( z^{\prime
\prime },u^{\prime \prime }\right) ^{-1}\right\vert \left\vert \left(
z,u\right) \right\vert ^{-\left( Q+2M^{\prime }+1\right) }.
\end{eqnarray*}%
For term $IV_{2}$ we use the estimate derived in Case 2 for the moments of $%
f_{J}$ to obtain%
\begin{eqnarray*}
\left\vert IV_{2}\right\vert &\lesssim &\sum_{\substack{ \left\vert \alpha
^{\prime \prime }\right\vert +2\beta ^{\prime \prime }  \\ \leq \left\vert
\alpha ^{\prime }\right\vert +2\beta ^{\prime }}}\left\vert \left(
z,u\right) \circ \left( z^{\prime \prime },u^{\prime \prime }\right)
^{-1}\right\vert \left\vert \left( z,u\right) \right\vert ^{-\left(
Q+2M^{\prime }+1\right) }\left\vert \int f_{J}\left( z^{\prime },u^{\prime
}\right) \left( z^{\prime },u^{\prime }\right) ^{\left( \alpha ^{\prime
\prime },\beta ^{\prime \prime }\right) }dz^{\prime }du^{\prime }\right\vert
\\
&\lesssim &\left\Vert f\right\Vert _{\mathcal{M}^{M^{\prime }+\delta }\left( 
\mathbb{H}^{n}\right) }\left\vert \left( z,u\right) \circ \left( z^{\prime
\prime },u^{\prime \prime }\right) ^{-1}\right\vert \left\vert \left(
z,u\right) \right\vert ^{-\left( Q+M^{\prime }+1\right) } \\
&&\ \ \ \ \ \ \ \ \ \ \ \ \ \ \ \ \ \ \ \ \times \sum_{\substack{ \left\vert
\alpha ^{\prime \prime }\right\vert +2\beta ^{\prime \prime }  \\ \leq
\left\vert \alpha ^{\prime }\right\vert +2\beta ^{\prime }}}\left\vert
\left( z,u\right) \right\vert ^{-\left( \left\vert \alpha ^{\prime \prime
}\right\vert +2\beta ^{\prime \prime }\right) }\left\vert \left( z,u\right)
\right\vert ^{\left\vert \alpha ^{\prime \prime }\right\vert +2\beta
^{\prime \prime }-\left( \left\vert \alpha ^{\prime }\right\vert +2\beta
^{\prime }\right) } \\
&\lesssim &\left\Vert f\right\Vert _{\mathcal{M}^{M^{\prime }+\delta }\left( 
\mathbb{H}^{n}\right) }\left\vert \left( z,u\right) \circ \left( z^{\prime
\prime },u^{\prime \prime }\right) ^{-1}\right\vert \left\vert \left(
z,u\right) \right\vert ^{-\left( Q+2M^{\prime \prime }+1\right) },
\end{eqnarray*}%
where we have used%
\begin{equation*}
\int f_{J}\left( z^{\prime },u^{\prime }\right) \left( z^{\prime },u^{\prime
}\right) ^{\left( \alpha ^{\prime \prime },\beta ^{\prime \prime }\right)
}dz^{\prime }du^{\prime }=-\int f_{I}\left( z^{\prime },u^{\prime }\right)
\left( z^{\prime },u^{\prime }\right) ^{\left( \alpha ^{\prime \prime
},\beta ^{\prime \prime }\right) }dz^{\prime }du^{\prime }-\int f_{L}\left(
z^{\prime },u^{\prime }\right) \left( z^{\prime },u^{\prime }\right)
^{\left( \alpha ^{\prime \prime },\beta ^{\prime \prime }\right) }dz^{\prime
}du^{\prime },
\end{equation*}%
together with the estimate%
\begin{equation*}
\left\vert \int f_{I}\left( z^{\prime },u^{\prime }\right) \left( z^{\prime
},u^{\prime }\right) ^{\left( \alpha ^{\prime \prime },\beta ^{\prime \prime
}\right) }dz^{\prime }du^{\prime }\right\vert \lesssim \left\vert \left(
z,u\right) \right\vert ^{\left\vert \alpha ^{\prime \prime }\right\vert
+2\beta ^{\prime \prime }-M^{\prime }}\leq C.
\end{equation*}

Finally, we estimate%
\begin{eqnarray*}
&&\left\vert \partial _{z}^{\alpha ^{\prime }}\partial _{u}^{\beta ^{\prime
}}Tf_{L}\left( z,u\right) -\partial _{z}^{\alpha ^{\prime }}\partial
_{u}^{\beta ^{\prime }}Tf_{L}\left( z^{\prime \prime },u^{\prime \prime
}\right) \right\vert \\
&\lesssim &\left\Vert f\right\Vert _{\mathcal{M}^{M^{\prime }+\delta }\left( 
\mathbb{H}^{n}\right) }\int_{\left\{ \frac{\left\vert \left( z^{\prime
},u^{\prime }\right) \circ \left( z,u\right) ^{-1}\right\vert }{\left\vert
\left( z,u\right) \right\vert }\geq \frac{1}{4}\text{ and }\frac{\left\vert
\left( z^{\prime },u^{\prime }\right) \right\vert }{\left\vert \left(
z,u\right) \right\vert }\geq \frac{1}{4}\right\} }\left\vert \left(
z,u\right) \circ \left( z^{\prime \prime },u^{\prime \prime }\right)
^{-1}\right\vert ^{\delta } \\
&&\times \left\vert \left( z^{\prime },u^{\prime }\right) \circ \left(
z,u\right) ^{-1}\right\vert ^{-Q-M^{\prime }-\delta }\left\vert \left(
z^{\prime },u^{\prime }\right) \right\vert ^{-\left( Q+\left\vert \alpha
^{\prime }\right\vert +2\beta ^{\prime }\right) }dz^{\prime }du^{\prime } \\
&\lesssim &\left\Vert f\right\Vert _{\mathcal{M}^{M^{\prime }+\delta }\left( 
\mathbb{H}^{n}\right) }\left\vert \left( z,u\right) \circ \left( z^{\prime
\prime },u^{\prime \prime }\right) ^{-1}\right\vert ^{\delta }\left\vert
\left( z,u\right) \right\vert ^{-Q-M^{\prime }-\delta } \\
&&\times \int_{\left\{ \frac{\left\vert \left( z^{\prime },u^{\prime
}\right) \circ \left( z,u\right) ^{-1}\right\vert }{\left\vert \left(
z,u\right) \right\vert }\geq \frac{1}{4}\text{ and }\frac{\left\vert \left(
z^{\prime },u^{\prime }\right) \right\vert }{\left\vert \left( z,u\right)
\right\vert }\geq \frac{1}{4}\right\} }\left\vert \left( z^{\prime
},u^{\prime }\right) \right\vert ^{-\left( Q+\left\vert \alpha ^{\prime
}\right\vert +2\beta ^{\prime }\right) }dz^{\prime }du^{\prime } \\
&\lesssim &\left\Vert f\right\Vert _{\mathcal{M}^{M^{\prime }+\delta }\left( 
\mathbb{H}^{n}\right) }\left\vert \left( z,u\right) \circ \left( z^{\prime
\prime },u^{\prime \prime }\right) ^{-1}\right\vert ^{\delta }\left\vert
\left( z,u\right) \right\vert ^{-\left( Q+2M^{\prime }+\delta \right) }.
\end{eqnarray*}

\bigskip

\textbf{Proof of Theorem \ref{product molecular proof}}: To prove the
product version we note that the above one-parameter proof extends virtually
verbatim to establish a \emph{vector-valued} version in a Banach space.
Indeed, all the main tools, such as integration, differentiation and
Taylor's formula, carry over to the Banach space setting. First we will
define the $X$-valued molecular space $\mathcal{M}^{M+\delta
,M_{1},M_{2}}\left( \mathbb{H}^{n};X\right) $, and then we will give the
extension of Theorem \ref{standard molecular proof} to this space.

\begin{definition}
\label{one para mol X}Let $X$ be a Banach space with norm $\left\vert
x\right\vert $ for $x\in X$. Let $M,M_{1},M_{2}\in \mathbb{N}$ be positive
integers, $0<\delta \leq 1$, and let $Q=2n+2$ denote the homogeneous
dimension of $\mathbb{H}^{n}$. The \emph{one-parameter} molecular space $%
\mathcal{M}^{M+\delta ,M_{1},M_{2}}\left( \mathbb{H}^{n};X\right) $ consists
of all $X$-valued functions $f:\mathbb{H}^{n}\rightarrow X$ satisfying the
moment conditions%
\begin{equation*}
\int_{\mathbb{H}^{n}}z^{\alpha }u^{\beta }f\left( z,u\right) dzdu=0\text{
for all }\left\vert \alpha \right\vert +2\left\vert \beta \right\vert \leq
M_{1},
\end{equation*}%
and such that there is a nonnegative constant $A$ satisfying the following
two differential inequalities:%
\begin{equation*}
\left\vert \partial _{z}^{\alpha }\partial _{u}^{\beta }f\left( z,u\right)
\right\vert _{X}\leq A\frac{1}{\left( 1+\left\vert z\right\vert
^{2}+\left\vert u\right\vert \right) ^{\frac{Q+M+\left\vert \alpha
\right\vert +2\left\vert \beta \right\vert +\delta }{2}}}\text{ for all }%
\left\vert \alpha \right\vert +2\left\vert \beta \right\vert \leq M_{2},
\end{equation*}%
\begin{eqnarray*}
&&\left\vert \partial _{z}^{\alpha }\partial _{u}^{\beta }f\left( z,u\right)
-\partial _{z}^{\alpha }\partial _{u}^{\beta }f\left( z^{\prime },u^{\prime
}\right) \right\vert _{X}\leq A\frac{\left\vert \left( z,u\right) \circ
\left( z^{\prime },u^{\prime }\right) ^{-1}\right\vert ^{\delta }}{\left(
1+\left\vert z\right\vert ^{2}+\left\vert u\right\vert \right) ^{\frac{%
Q+M+\delta +M_{2}+2\delta }{2}}} \\
&&\ \ \ \ \ \text{for all }\left\vert \alpha \right\vert +2\left\vert \beta
\right\vert =M_{2}\text{ and }\left\vert \left( z,u\right) \circ \left(
z^{\prime },u^{\prime }\right) ^{-1}\right\vert \leq \frac{1}{2}\left(
1+\left\vert z\right\vert ^{2}+\left\vert u\right\vert \right) ^{\frac{1}{2}%
}.
\end{eqnarray*}
\end{definition}

We have the following extension of Theorem \ref{standard molecular proof} to 
$X$-valued functions for an arbitrary Banach space $X$.

\begin{theorem}
\label{standard molecular proof X}Suppose that $T:L^{2}\left( \mathbb{H}%
^{n}\right) \rightarrow L^{2}\left( \mathbb{H}^{n}\right) $ is a bounded
linear operator with kernel $K\left( \left( z,u\right) ,\left( z^{\prime
},u^{\prime }\right) \right) $, i.e.%
\begin{equation*}
Tf\left( z,u\right) =\int_{\mathbb{H}^{n}}K\left( \left( z,u\right) ,\left(
z^{\prime },u^{\prime }\right) \right) f\left( z^{\prime },u^{\prime
}\right) dz^{\prime }du^{\prime },\ \ \ \ \ f\in L^{2}\left( \mathbb{H}%
^{n}\right) .
\end{equation*}%
Suppose furthermore that $K$ satisfies%
\begin{eqnarray*}
\int_{\mathbb{H}^{n}}z^{\alpha }u^{\beta }\ K\left( \left( z,u\right)
,\left( z^{\prime },u^{\prime }\right) \right) dzdu &=&0, \\
\int_{\mathbb{H}^{n}}\left( z^{\prime }\right) ^{\alpha }\left( u^{\prime
}\right) ^{\beta }\ K\left( \left( z,u\right) ,\left( z^{\prime },u^{\prime
}\right) \right) dz^{\prime }du^{\prime } &=&0,
\end{eqnarray*}%
for all $0\leq \left\vert \alpha \right\vert ,\beta $; and%
\begin{eqnarray*}
&&\left\vert \partial _{z}^{\alpha }\partial _{u}^{\beta }\partial
_{z^{\prime }}^{\alpha ^{\prime }}\partial _{u^{\prime }}^{\beta ^{\prime
}}K\left( \left( z,u\right) ,\left( z^{\prime },u^{\prime }\right) \right)
\right\vert \\
&\leq &A_{\alpha ,\beta ,\alpha ^{\prime },\beta ^{\prime }}\frac{1}{%
\left\vert \left( z,u\right) \circ \left( z^{\prime },u^{\prime }\right)
^{-1}\right\vert ^{Q+\left\vert \alpha \right\vert +2\beta +\left\vert
\alpha ^{\prime }\right\vert +2\beta ^{\prime }}},
\end{eqnarray*}%
for all $0\leq \left\vert \alpha \right\vert ,\beta ,\left\vert \alpha
^{\prime }\right\vert ,\beta ^{\prime }$. For $f:\mathbb{H}^{n}\rightarrow X$
we define $Tf$ by the Banach space valued integral 
\begin{equation*}
Tf\left( z,u\right) =\int_{\mathbb{H}^{n}}K\left( \left( z,u\right) ,\left(
z^{\prime },u^{\prime }\right) \right) f\left( z^{\prime },u^{\prime
}\right) dz^{\prime }du^{\prime }.
\end{equation*}%
Then 
\begin{equation*}
T:\mathcal{M}^{M^{\prime }+\delta }\left( \mathbb{H}^{n};X\right)
\rightarrow \mathcal{M}^{M^{\prime }+\delta }\left( \mathbb{H}^{n};X\right)
\end{equation*}%
is bounded for all $M^{\prime }$ and $0<\alpha <1$, and moreover, the
operator norm satisfies%
\begin{equation*}
\left\Vert T\right\Vert _{\mathcal{M}^{M^{\prime }+\delta }\left( \mathbb{H}%
^{n};X\right) }\leq C_{M^{\prime },\alpha },
\end{equation*}%
where the constant $C_{M^{\prime },\alpha }$ goes to zero as $A_{\alpha
,\beta ,\alpha ^{\prime },\beta ^{\prime }}\rightarrow 0$ for sufficiently
many of the indices $\alpha ,\beta ,\alpha ^{\prime },\beta ^{\prime }$.
\end{theorem}

\textbf{Proof}: We simply repeat the scalar proof of Theorem \ref{standard
molecular proof} but replace $\left\vert \partial _{z}^{\alpha }\partial
_{u}^{\beta }f\left( z,u\right) \right\vert $ by $\left\vert \partial
_{z}^{\alpha }\partial _{u}^{\beta }f\left( z,u\right) \right\vert _{X}$
throughout and use Banach space analogues of Taylor's theorem and the
identities of Torres \cite{Tor}.

\bigskip

Now we can quickly finish the proof of the product Theorem \ref{product
molecular proof}. We take $X=\mathcal{M}^{M^{\prime }+\delta }\left( \mathbb{%
R}\right) $ and note that the identification of product and iterated
molecular spaces, namely 
\begin{equation}
\mathcal{M}_{product}^{M^{\prime }+\delta }\left( \mathbb{H}^{n}\times 
\mathbb{R}\right) =\mathcal{M}^{M^{\prime }+\delta }\left( \mathbb{H}^{n};%
\mathcal{M}^{M^{\prime }+\delta }\left( \mathbb{R}\right) \right) =\mathcal{M%
}^{M^{\prime }+\delta }\left( \mathbb{H}^{n};X\right) ,  \label{equalities}
\end{equation}%
follows immediately from the definitions of the spaces involved; see
Definition \ref{one para mol}, Definition \ref{product molecular}\ and the
definition of $\mathcal{M}^{M+\delta ,M_{1},M_{2}}\left( \mathbb{R}\right) $
which we recall here.

\begin{definition}
Let $M\in \mathbb{N}$ be a positive integer and $0<\delta \leq 1$. The \emph{%
one-parameter} molecular space $\mathcal{M}^{M+\delta ,M_{1},M_{2}}\left( 
\mathbb{R}\right) $ consists of all functions $f\left( v\right) $ on $%
\mathbb{R}$ satisfying the moment conditions%
\begin{equation*}
\int_{\mathbb{R}}v^{\gamma }f\left( v\right) dv=0\text{ for all }2\gamma
\leq M_{1},
\end{equation*}%
and such that there is a nonnegative constant $A$ satisfying the following
two differential inequalities:%
\begin{equation*}
\left\vert \partial _{v}^{\gamma }f\left( v\right) \right\vert \leq A\frac{1%
}{\left( 1+\left\vert v\right\vert \right) ^{1+M+\gamma +\delta }}\ \ \ \ \ 
\text{ for all }2\gamma \leq M_{2},
\end{equation*}%
\begin{equation*}
\left\vert \partial _{v}^{M_{2}}f\left( v\right) -\partial
_{v}^{M_{2}}f\left( v^{\prime }\right) \right\vert \leq A\frac{\left\vert
v-v^{\prime }\right\vert ^{\delta }}{\left( 1+\left\vert v\right\vert
\right) ^{1+\frac{3}{2}M+\gamma +2\delta }}\ \ \ \ \ \text{for all }%
\left\vert v-v^{\prime }\right\vert \leq \frac{1}{2}\left( 1+\left\vert
v\right\vert \right) .
\end{equation*}
\end{definition}

For $f\in \mathcal{M}_{product}^{M^{\prime }+\delta }\left( \mathbb{H}%
^{n}\times \mathbb{R}\right) $, denote the realization of $f$ as an $X$%
-valued map by $\widetilde{f}:$ $\mathbb{H}^{n}\rightarrow \mathcal{M}%
_{product}^{M^{\prime }+\delta }\left( \mathbb{R}\right) $. Then from (\ref%
{equalities}) and Theorem \ref{standard molecular proof X} we have%
\begin{eqnarray*}
\left\Vert Tf\right\Vert _{\mathcal{M}_{product}^{M^{\prime }+\delta }\left( 
\mathbb{H}^{n}\times \mathbb{R}\right) } &=&\left\Vert T\widetilde{f}%
\right\Vert _{\mathcal{M}^{M^{\prime }+\delta }\left( \mathbb{H}^{n};%
\mathcal{M}^{M^{\prime }+\delta }\left( \mathbb{R}\right) \right) } \\
&\leq &C\left\Vert \widetilde{f}\right\Vert _{\mathcal{M}^{M^{\prime
}+\delta }\left( \mathbb{H}^{n};\mathcal{M}^{M^{\prime }+\delta }\left( 
\mathbb{R}\right) \right) }=C\left\Vert f\right\Vert _{\mathcal{M}%
_{product}^{M^{\prime }+\delta }\left( \mathbb{H}^{n}\times \mathbb{R}%
\right) }.
\end{eqnarray*}%
This completes the proof of Theorem \ref{product molecular proof}.

\subsection{Orthogonality estimates and the proof of the Plancherel-P\^{o}%
lya inequalities\label{5.3}}

We will need almost \emph{orthogonality estimates} in order to prove both
the Plancherel-P\^{o}lya inequalities and the boundedness of flag singular
integrals on $H_{flag}^{p}\left( \mathbb{H}^{n}\right) $. Recall from (\ref%
{defpsits}) the definition of the components $\psi _{t,s}$ of the continuous
decomposition of the identity adapted to the Heisenberg group:%
\begin{equation*}
\psi \left( z,u\right) =\psi ^{(1)}\ast _{2}\psi ^{(2)}\left( z,u\right)
=\int_{\mathbb{R}}\psi ^{(1)}(z,u-v)\psi ^{(2)}(v)dv,\ \ \ \ \ \left(
z,u\right) \in \mathbb{C}^{n}\times \mathbb{R},
\end{equation*}%
and%
\begin{eqnarray*}
\psi _{t,s}\left( z,u\right) &=&\psi _{t}^{\left( 1\right) }\ast _{2}\psi
_{s}^{\left( 2\right) }\left( z,u\right) =\int_{\mathbb{R}}\psi _{t}^{\left(
1\right) }\left( z,u-v\right) \psi _{s}^{\left( 2\right) }\left( v\right) dv
\\
&=&\int_{\mathbb{R}}t^{-2n-2}\psi ^{\left( 1\right) }\left( \frac{z}{t},%
\frac{u-v}{t^{2}}\right) s^{-1}\psi ^{\left( 2\right) }\left( \frac{v}{s}%
\right) dv.
\end{eqnarray*}%
Here $\psi ^{(1)}\in \mathcal{S}(\mathbb{H}^{n}\mathbb{)}$ is as in Theorem %
\ref{GM}, and $\psi ^{\left( 2\right) }\in \mathcal{S}\left( \mathbb{R}%
\right) $ satisfies 
\begin{equation*}
\int_{0}^{\infty }|\widehat{{\psi }^{(2)}}(t\eta )|^{2}\frac{dt}{t}=1
\end{equation*}%
for all $\eta \in \mathbb{R}\backslash \{0\}$, along with the moment
conditions 
\begin{eqnarray*}
\int\limits_{\mathbb{H}^{n}}z^{\alpha }u^{\beta }\psi ^{(1)}(z,u)dzdu &=&0,\
\ \ \ \ \left\vert \alpha \right\vert +2\beta \leq M, \\
\int\limits_{\mathbb{R}}v^{\gamma }\psi ^{(2)}(v)dv &=&0,\ \ \ \ \ \gamma
\geq 0,
\end{eqnarray*}%
where $M$ may be fixed arbitrarily large.

In particular the collection of component functions $\left\{ \psi
_{t,s}\right\} _{t,s>0}$ satisfies 
\begin{eqnarray}
\psi _{t,s} &=&\psi _{t}^{\left( 1\right) }\ast _{2}\psi _{s}^{\left(
2\right) },  \label{psi conditions} \\
\psi _{t}^{\left( 1\right) }\left( z,u\right) &=&t^{-2n-2}\psi ^{\left(
1\right) }\left( \frac{z}{t},\frac{u-v}{t^{2}}\right) ,  \notag \\
\psi _{s}^{\left( 2\right) }\left( v\right) &=&s^{-1}\psi ^{\left( 2\right)
}\left( \frac{v}{s}\right) ,  \notag \\
\psi ^{(1)}\left( z,u\right) \psi ^{(2)}\left( v\right) &\in &\mathcal{M}%
_{product}^{M+\delta }\left( \mathbb{H}^{n}\times \mathbb{R}\right) .  \notag
\end{eqnarray}%
Of course the conditions in (\ref{psi conditions}) imply that $\psi
_{t,s}\in \mathcal{M}_{flag}^{M}\left( \mathbb{H}^{n}\right) $ for all $%
t,s>0 $, but (\ref{psi conditions}) also contains the implicit dilation
information that cannot be expressed solely in terms of $\psi _{1,1}$.
Motivated by these considerations we make the following definition that
encompasses the identity (\ref{Psi}).

\begin{definition}
\label{10}To each function $\Psi \in \mathcal{M}_{product}^{M+\delta }\left( 
\mathbb{H}^{n}\times \mathbb{R}\right) $ we associate a collection of \emph{%
product dilations} $\left\{ \Psi _{t,s}\right\} _{t,s>0}$ defined by%
\begin{equation*}
\Psi _{t,s}\left( z,u,v\right) =t^{-2n-2}s^{-1}\Psi \left( \left( \frac{z}{t}%
,\frac{u}{t^{2}}\right) ,\frac{v}{s}\right) ,
\end{equation*}%
and a collection of \emph{component functions} $\left\{ \psi _{t,s}\right\}
_{t,s>0}$ defined by%
\begin{equation*}
\psi _{t,s}\left( z,u\right) =\pi \Psi _{t,s}\left( z,u\right) =\int_{%
\mathbb{R}}t^{-2n-2}s^{-1}\Psi \left( \left( \frac{z}{t},\frac{u-v}{t^{2}}%
\right) ,\frac{v}{s}\right) dv,\ \ \ \ \ t,s>0.
\end{equation*}
\end{definition}

Given two functions in $\mathcal{M}_{product}^{M+\delta }\left( \mathbb{H}%
^{n}\times \mathbb{R}\right) $ and their corresponding collections of
component functions we have the \emph{almost orthogonality} estimates given
below. We \ use $\ast _{\mathbb{H}^{n}}$ to denote convolution on the
Heisenberg group $\mathbb{H}^{n}$, and $\ast _{\mathbb{H}^{n}\times \mathbb{R%
}}$ to denote convolution on the product group $\mathbb{H}^{n}\times \mathbb{%
R}$. From Lemma \ref{intertwine} we obtain that $\pi $ intertwines these two
convolutions, which we record here.

\begin{lemma}
\label{project}For $\psi _{t,s},\Psi _{t,s},\phi _{t^{\prime },s^{\prime
}},\Phi _{t^{\prime },s^{\prime }}$ as above we have%
\begin{equation}
\psi _{t,s}\ast _{\mathbb{H}^{n}}\phi _{t^{\prime },s^{\prime }}=\pi \left\{
\Psi _{t,s}\ast _{\mathbb{H}^{n}\times \mathbb{R}}\Phi _{t^{\prime
},s^{\prime }}\right\} .  \label{projection commutes}
\end{equation}
\end{lemma}

We now give the orthogonality estimates, first in the product case and then
in the flag case. The product case in Lemma \ref{almostorthoproduct} will
prove crucial in establishing Theorem \ref{key boundedness} for the flag
molecular space $\mathcal{M}_{flag}^{M^{\prime }+\delta }\left( \mathbb{H}%
^{n}\right) $.

For convenience we give the almost orthogonal estimates only for the case $%
\mathcal{M}_{product}^{4M+2,2M,2M}\left( \mathbb{H}^{n}\times \mathbb{R}%
\right) $.

\begin{lemma}
\label{almostorthoproduct}Suppose $\Psi ,\Phi \in \mathcal{M}%
_{product}^{4M+2,2M,2M}\left( \mathbb{H}^{n}\times \mathbb{R}\right) $. Then
there exists a constant $C=C(M)$ depending only on $M$ such that%
\begin{eqnarray}
&&\ \ \ \ \ \ \ \ \ \ \ \ \ \ \ \ \ \ \ \ \ \ \ \ \ \ \ \ \ \ \left\vert
\Psi _{t,s}\ast _{\mathbb{H}^{n}\times \mathbb{R}}\Phi _{t^{\prime
},s^{\prime }}\left( \left( z,u\right) ,v\right) \right\vert
\label{product est} \\
&\leq &C\left( {\frac{{t}}{{t^{\prime }}}}\wedge {\frac{{t^{\prime }}}{{t}}}%
\right) ^{2M+1}\left( {\frac{{s}}{{s^{\prime }}}}\wedge {\frac{{s^{\prime }}%
}{{s}}}\right) ^{M+1}\frac{\left( {{t}\vee {t^{\prime }}}\right) ^{2\frac{%
4M+2}{2}}}{\left( \left( {{t}\vee {t^{\prime }}}\right) ^{2}{+}\left\vert {z}%
\right\vert ^{2}+\left\vert u\right\vert \right) {^{\frac{Q+4M+2}{2}}}}\frac{%
\left( {{s}\vee {s^{\prime }}}\right) ^{4M+2}}{\left( {{s}\vee {s^{\prime }}+%
}\left\vert {v}\right\vert \right) {^{1+4M+2}}}.  \notag
\end{eqnarray}
\end{lemma}

\textbf{Proof of Lemma \ref{almostorthoproduct}}: The proof uses a standard
orthogonality argument on the integral%
\begin{eqnarray}
&&\Psi _{t,s}\ast _{\mathbb{H}^{n}\times \mathbb{R}}\Phi _{t^{\prime
},s^{\prime }}\left( \left( z,u\right) ,v\right)  \label{integral'} \\
&=&\int_{\mathbb{H}^{n}\times \mathbb{R}}\Psi _{t,s}\left( \left( z,u\right)
\circ \left( z^{\prime },u^{\prime }\right) ^{-1},v-v^{\prime }\right) \Phi
_{t^{\prime },s^{\prime }}\left( \left( z^{\prime },u^{\prime }\right)
,v^{\prime }\right) dz^{\prime }du^{\prime }dv^{\prime }.  \notag
\end{eqnarray}%
Consider the following four cases:

\begin{enumerate}
\item $t\geq t^{\prime }$ and $s\geq s^{\prime }$,

\item $t\geq t^{\prime }$ and $s<s^{\prime }$,

\item $t<t^{\prime }$ and $s\geq s^{\prime }$,

\item $t<t^{\prime }$ and $s<s^{\prime }$.
\end{enumerate}

We prove Cases (1) and (2) since (4) is similar to (1) and (3) is similar to
(2).

\bigskip

\textbf{Case (1)}: In this case we can exploit the smoothness of $\Psi
_{t,s} $ and the vanishing moments of $\Phi _{t^{\prime },s^{\prime }}$ to
write%
\begin{eqnarray*}
&&\Psi _{t,s}\ast _{\mathbb{H}^{n}\times \mathbb{R}}\Phi _{t^{\prime
},s^{\prime }}\left( \left( z,u\right) ,v\right) \\
&=&\int_{\mathbb{H}^{n}\times \mathbb{R}}\left\{ \left[ \Psi _{t,s}\left(
\left( z,u\right) \circ \left( z^{\prime },u^{\prime }\right)
^{-1},v-v^{\prime }\right) -\sum_{\left\vert \alpha \right\vert +2\beta \leq
2M}c_{\alpha ,\beta }\partial _{z}^{\alpha }\partial _{u}^{\beta }\Psi
_{t,s}\left( \left( z,u\right) ,v-v^{\prime }\right) \left[ \left( z^{\prime
},u^{\prime }\right) ^{-1}\right] ^{\left( \alpha ,\beta \right) }\right]
\right. \\
&&\left. -\sum_{\gamma \leq M}c_{\gamma }\partial _{v}^{\gamma }\left[ \Psi
_{t,s}\left( \left( z,u\right) \circ \left( z^{\prime },u^{\prime }\right)
^{-1},v\right) -\sum_{\left\vert \alpha \right\vert +2\beta \leq
2M}c_{\alpha ,\beta }\partial _{z}^{\alpha }\partial _{u}^{\beta }\Psi
_{t,s}\left( \left( z,u\right) ,v\right) \left[ \left( z^{\prime },u^{\prime
}\right) ^{-1}\right] ^{\left( \alpha ,\beta \right) }\right] \left(
-v^{\prime }\right) ^{\gamma }\right\} \\
&&\ \ \ \ \ \ \ \ \ \ \ \ \ \ \ \ \ \ \ \ \times \Phi _{t^{\prime
},s^{\prime }}\left( \left( z^{\prime },u^{\prime }\right) ,v^{\prime
}\right) dz^{\prime }du^{\prime }dv^{\prime }.
\end{eqnarray*}%
Indeed, we only used the moment conditions%
\begin{equation*}
\int_{\mathbb{H}^{n}}\left[ \left( z^{\prime },u^{\prime }\right) ^{-1}%
\right] ^{\left( \alpha ,\beta \right) }\Phi _{t^{\prime },s^{\prime
}}\left( \left( z^{\prime },u^{\prime }\right) ,v^{\prime }\right)
dz^{\prime }du^{\prime }=0,\ \ \ \ \ \left\vert \alpha \right\vert +2\beta
\leq 2M,
\end{equation*}%
and%
\begin{equation*}
\int_{\mathbb{R}}\left( -v^{\prime }\right) ^{\gamma }\Phi _{t^{\prime
},s^{\prime }}\left( \left( z^{\prime },u^{\prime }\right) ,v^{\prime
}\right) dv^{\prime }=0,\ \ \ \ \ \gamma \leq M.
\end{equation*}

We now write%
\begin{equation*}
\Psi _{t,s}\ast _{\mathbb{H}^{n}\times \mathbb{R}}\Phi _{t^{\prime
},s^{\prime }}\left( \left( z,u\right) ,v\right) \equiv
I_{1}+I_{2}+I_{3}+I_{4}
\end{equation*}%
where $I_{j}$ denotes that part of the integral over $\mathbb{H}^{n}\times 
\mathbb{R}$ where integration is taken over the set $W_{j}$ given by%
\begin{eqnarray}
&&  \label{def W} \\
W_{1} &=&\left\{ \left( \left( z^{\prime },u^{\prime }\right) ,v^{\prime
}\right) \in \mathbb{H}^{n}\times \mathbb{R}:\left\vert \left( z^{\prime
},u^{\prime }\right) \right\vert \leq \frac{1}{2}\left( t^{2}+\left\vert
z\right\vert ^{2}+\left\vert u\right\vert \right) ^{\frac{1}{2}}\text{ and }%
\left\vert v^{\prime }\right\vert \leq \frac{1}{2}\left( s+\left\vert
v\right\vert \right) \right\} ,  \notag \\
W_{2} &=&\left\{ \left( \left( z^{\prime },u^{\prime }\right) ,v^{\prime
}\right) \in \mathbb{H}^{n}\times \mathbb{R}:\left\vert \left( z^{\prime
},u^{\prime }\right) \right\vert \leq \frac{1}{2}\left( t^{2}+\left\vert
z\right\vert ^{2}+\left\vert u\right\vert \right) ^{\frac{1}{2}}\text{ and }%
\left\vert v^{\prime }\right\vert >\frac{1}{2}\left( s+\left\vert
v\right\vert \right) \right\} ,  \notag \\
W_{3} &=&\left\{ \left( \left( z^{\prime },u^{\prime }\right) ,v^{\prime
}\right) \in \mathbb{H}^{n}\times \mathbb{R}:\left\vert \left( z^{\prime
},u^{\prime }\right) \right\vert >\frac{1}{2}\left( t^{2}+\left\vert
z\right\vert ^{2}+\left\vert u\right\vert \right) ^{\frac{1}{2}}\text{ and }%
\left\vert v^{\prime }\right\vert \leq \frac{1}{2}\left( s+\left\vert
v\right\vert \right) \right\} ,  \notag \\
W_{4} &=&\left\{ \left( \left( z^{\prime },u^{\prime }\right) ,v^{\prime
}\right) \in \mathbb{H}^{n}\times \mathbb{R}:\left\vert \left( z^{\prime
},u^{\prime }\right) \right\vert >\frac{1}{2}\left( t^{2}+\left\vert
z\right\vert ^{2}+\left\vert u\right\vert \right) ^{\frac{1}{2}}\text{ and }%
\left\vert v^{\prime }\right\vert >\frac{1}{2}\left( s+\left\vert
v\right\vert \right) \right\} .  \notag
\end{eqnarray}%
Note that $\mathbb{H}^{n}\times \mathbb{R}=W_{1}+W_{2}+W_{3}+W_{4}$.

\bigskip

To estimate term $I_{1}$ we first use Taylor's theorem in the factor $%
\mathbb{H}^{n}$ to write%
\begin{eqnarray*}
&&\Psi _{t,s}\left( \left( z,u\right) \circ \left( z^{\prime },u^{\prime
}\right) ^{-1},v-v^{\prime }\right) -\sum_{\left\vert \alpha \right\vert
+2\beta \leq 2M}c_{\alpha ,\beta }\partial _{z}^{\alpha }\partial
_{u}^{\beta }\Psi _{t,s}\left( \left( z,u\right) ,v-v^{\prime }\right) \left[
\left( z^{\prime },u^{\prime }\right) ^{-1}\right] ^{\left( \alpha ,\beta
\right) } \\
&=&\sum_{\left\vert \alpha \right\vert +2\beta =2M}c_{\alpha ,\beta }\left\{
\partial _{z}^{\alpha }\partial _{u}^{\beta }\Psi _{t,s}\left( \left(
z^{\prime \prime },u^{\prime \prime }\right) ,v-v^{\prime }\right) -\partial
_{z}^{\alpha }\partial _{u}^{\beta }\Psi _{t,s}\left( \left( z,u\right)
,v-v^{\prime }\right) \right\} \left[ \left( z^{\prime },u^{\prime }\right)
^{-1}\right] ^{\left( \alpha ,\beta \right) },
\end{eqnarray*}%
where $\left( z^{\prime \prime },u^{\prime \prime }\right) =\left(
z,u\right) \circ \delta _{\theta }\left( z^{\prime },u^{\prime }\right)
^{-1} $ for some $0<\theta <1$ (here $\delta _{\theta }$ is the usual
automorphic dilation of $\mathbb{H}^{n}$), and then use Taylor's theorem
again but in the factor $\mathbb{R}$ to continue with%
\begin{eqnarray*}
&&\Psi _{t,s}\left( \left( z,u\right) \circ \left( z^{\prime },u^{\prime
}\right) ^{-1},v-v^{\prime }\right) -\sum_{\left\vert \alpha \right\vert
+2\beta \leq 2M}c_{\alpha ,\beta }\partial _{z}^{\alpha }\partial
_{u}^{\beta }\Psi _{t,s}\left( \left( z,u\right) ,v-v^{\prime }\right) \left[
\left( z^{\prime },u^{\prime }\right) ^{-1}\right] ^{\left( \alpha ,\beta
\right) } \\
&&-\sum_{\gamma \leq M}c_{\gamma }\partial _{v}^{\gamma }\left[ \Psi
_{t,s}\left( \left( z,u\right) \circ \left( z^{\prime },u^{\prime }\right)
^{-1},v\right) -\sum_{\left\vert \alpha \right\vert +2\beta \leq
2M}c_{\alpha ,\beta }\partial _{z}^{\alpha }\partial _{u}^{\beta }\Psi
_{t,s}\left( \left( z,u\right) ,v\right) \left[ \left( z^{\prime },u^{\prime
}\right) ^{-1}\right] ^{\left( \alpha ,\beta \right) }\right] \left(
-v^{\prime }\right) ^{\gamma } \\
&=&\sum_{\left\vert \alpha \right\vert +2\beta =2M}c_{\alpha ,\beta }\left\{
\partial _{z}^{\alpha }\partial _{u}^{\beta }\Psi _{t,s}\left( \left(
z^{\prime \prime },u^{\prime \prime }\right) ,v-v^{\prime }\right) -\partial
_{z}^{\alpha }\partial _{u}^{\beta }\Psi _{t,s}\left( \left( z,u\right)
,v-v^{\prime }\right) \right\} \left[ \left( z^{\prime },u^{\prime }\right)
^{-1}\right] ^{\left( \alpha ,\beta \right) } \\
&&-\sum_{\gamma \leq M}c_{\gamma }\partial _{v}^{\gamma }\left\{
\sum_{\left\vert \alpha \right\vert +2\beta =2M}c_{\alpha ,\beta }\left\{
\partial _{z}^{\alpha }\partial _{u}^{\beta }\Psi _{t,s}\left( \left(
z^{\prime \prime },u^{\prime \prime }\right) ,v\right) -\partial
_{z}^{\alpha }\partial _{u}^{\beta }\Psi _{t,s}\left( \left( z,u\right)
,v-v^{\prime }\right) \right\} \left[ \left( z^{\prime },u^{\prime }\right)
^{-1}\right] ^{\left( \alpha ,\beta \right) }\right\} \left( -v^{\prime
}\right) ^{\gamma } \\
&=&\sum_{\substack{ \left\vert \alpha \right\vert +2\beta =2M  \\ \gamma =M}}%
c_{\gamma }c_{\alpha ,\beta }\left\{ \partial _{v}^{\gamma }\partial
_{z}^{\alpha }\partial _{u}^{\beta }\Psi _{t,s}\left( \left( z^{\prime
\prime },u^{\prime \prime }\right) ,v^{\prime \prime }\right) -\partial
_{v}^{\gamma }\partial _{z}^{\alpha }\partial _{u}^{\beta }\Psi _{t,s}\left(
\left( z^{\prime \prime },u^{\prime \prime }\right) ,v\right) \right. \\
&&\ \ \ \ \ \ \ \ \ \ \ \ \ \ \ \ \ \ \ \ \ \ \ \ \ \ \ \ \ \ \left.
-\partial _{z}^{\alpha }\partial _{u}^{\beta }\Psi _{t,s}\left( \left(
z,u\right) ,v^{\prime \prime }\right) +\partial _{v}^{\gamma }\partial
_{z}^{\alpha }\partial _{u}^{\beta }\Psi _{t,s}\left( \left( z,u\right)
,v\right) \right\} \left[ \left( z^{\prime },u^{\prime }\right) ^{-1}\right]
^{\left( \alpha ,\beta \right) }\left( -v^{\prime }\right) ^{\gamma }
\end{eqnarray*}%
where $v^{\prime \prime }=v-\theta v^{\prime }$ for some $0<\theta <1$ (not
necessarily the same $\theta $ as before). Now we apply the double
difference hypothesis in Definition \ref{product molecular} for elements of $%
\mathcal{M}_{product}^{4M+2,2M,2M}\left( \mathbb{H}^{n}\times \mathbb{R}%
\right) $, and note that we are using $2M+1$ in place of $M$, to obtain the
estimate that the modulus of the final sum above is dominated by%
\begin{equation*}
A\frac{\left\vert \left( z^{\prime },u^{\prime }\right) \right\vert ^{\frac{%
2M+1}{2}}}{\left( t^{2}+\left\vert z\right\vert ^{2}+\left\vert u\right\vert
\right) ^{\frac{2M+1}{2}}}\frac{\left( t^{2}\right) ^{\frac{4M+2}{2}}}{%
\left( t^{2}+\left\vert z\right\vert ^{2}+\left\vert u\right\vert \right) ^{%
\frac{Q+4M+2}{2}}}\frac{\left\vert v^{\prime }\right\vert ^{M+1}}{\left(
s+\left\vert v\right\vert \right) ^{M+1}}\frac{s^{4M+2}}{\left( s+\left\vert
v\right\vert \right) ^{1+4M+2}}
\end{equation*}%
since%
\begin{eqnarray*}
\left\vert \left( z,u\right) \circ \left( z^{\prime \prime },u^{\prime
\prime }\right) ^{-1}\right\vert &\lesssim &\left\vert \left( z^{\prime
},u^{\prime }\right) \right\vert \lesssim \frac{1}{2}\left( t^{2}+\left\vert
z\right\vert ^{2}+\left\vert u\right\vert \right) ^{\frac{1}{2}}, \\
\left\vert v-v^{\prime \prime }\right\vert &\lesssim &\left\vert v^{\prime
}\right\vert \lesssim \frac{1}{2}\left( s+\left\vert v\right\vert \right) .
\end{eqnarray*}%
This together with the size condition on $\Phi _{t^{\prime },s^{\prime }}$
yields 
\begin{eqnarray*}
\left\vert I_{1}\right\vert &\lesssim &A\int_{W_{1}}\frac{\left\vert \left(
z^{\prime },u^{\prime }\right) \right\vert ^{\frac{2M+1}{2}}}{\left(
t^{2}+\left\vert z\right\vert ^{2}+\left\vert u\right\vert \right) ^{\frac{%
2M+1}{2}}}\frac{\left( t^{2}\right) ^{\frac{4M+2}{2}}}{\left(
t^{2}+\left\vert z\right\vert ^{2}+\left\vert u\right\vert \right) ^{\frac{%
Q+4M+2}{2}}}\frac{\left\vert v^{\prime }\right\vert ^{M+1}}{\left(
s+\left\vert v\right\vert \right) ^{M+1}}\frac{s^{4M+2}}{\left( s+\left\vert
v\right\vert \right) ^{1+4M+2}} \\
&&\times \frac{\left( \left( {t}^{\prime }\right) ^{2}\right) ^{\frac{4M+2}{2%
}}}{\left( \left( {t}^{\prime }\right) ^{2}{+}\left\vert {z}^{\prime
}\right\vert ^{2}+\left\vert u^{\prime }\right\vert \right) {^{\frac{Q+4M+2}{%
2}}}}\frac{\left( s^{\prime }\right) ^{4M+2}}{\left( {s}^{\prime }{+}%
\left\vert {v}^{\prime }\right\vert \right) {^{1+4M+2}}}dz^{\prime
}du^{\prime }dv^{\prime } \\
&\lesssim &A\left( {\frac{{t^{\prime }}}{{t}}}\right) ^{2\frac{2M+1}{2}%
}\left( {\frac{{s^{\prime }}}{{s}}}\right) ^{M+1}\frac{\left( {t}^{2}\right)
^{2M+1}}{\left( {t}^{2}{+}\left\vert {z}\right\vert ^{2}+\left\vert
u\right\vert \right) {^{n+1+2M+1}}}\frac{{s}^{4M+2}}{\left( {{s}+}\left\vert 
{v}\right\vert \right) {^{1+4M+2}}},
\end{eqnarray*}%
which is dominated by the right side of (\ref{product est}) as required.
Recall that $Q=2n+2$.

\bigskip

In order to estimate $I_{2}$ we rewrite it as follows:%
\begin{eqnarray*}
I_{2} &=&\int_{W_{2}}\left\{ \sum_{\left\vert \alpha \right\vert +2\beta
=2M}c_{\alpha ,\beta }\left\{ \partial _{z}^{\alpha }\partial _{u}^{\beta
}\Psi _{t,s}\left( \left( z^{\prime \prime },u^{\prime \prime }\right)
,v-v^{\prime }\right) -\partial _{z}^{\alpha }\partial _{u}^{\beta }\Psi
_{t,s}\left( \left( z,u\right) ,v-v^{\prime }\right) \right\} \right. \\
&&-\left. \sum_{\left\vert \alpha \right\vert +2\beta =2M}\sum_{\gamma \leq
M}c_{\gamma }c_{\alpha ,\beta }\left[ \partial _{v}^{\gamma }\partial
_{z}^{\alpha }\partial _{u}^{\beta }\Psi _{t,s}\left( \left( z^{\prime
\prime },u^{\prime \prime }\right) ,v\right) -\partial _{v}^{\gamma
}\partial _{z}^{\alpha }\partial _{u}^{\beta }\Psi _{t,s}\left( \left(
z,u\right) ,v-v^{\prime }\right) \right] \left( -v^{\prime }\right) ^{\gamma
}\right\} \\
&&\ \ \ \ \ \ \ \ \ \ \ \ \ \ \ \ \ \ \ \ \ \ \ \ \ \ \ \ \ \ \times \left[
\left( z^{\prime },u^{\prime }\right) ^{-1}\right] ^{\left( \alpha ,\beta
\right) }\Phi _{t^{\prime },s^{\prime }}\left( \left( z^{\prime },u^{\prime
}\right) ,v^{\prime }\right) dz^{\prime }du^{\prime }dv^{\prime } \\
&=&I_{2}^{1}+I_{2}^{2}.
\end{eqnarray*}%
For $I_{2}^{1}$ we use%
\begin{equation*}
\left\vert \left( z,u\right) \circ \left( z^{\prime \prime },u^{\prime
\prime }\right) ^{-1}\right\vert \lesssim \left\vert \left( z^{\prime
},u^{\prime }\right) \right\vert \lesssim \frac{1}{2}\left( t^{2}+\left\vert
z\right\vert ^{2}+\left\vert u\right\vert \right) ^{\frac{1}{2}}
\end{equation*}%
to get%
\begin{eqnarray*}
\left\vert I_{2}^{1}\right\vert &\leq &A\int_{W_{2}}\frac{\left\vert \left(
z^{\prime },u^{\prime }\right) \right\vert ^{\frac{2M+1}{2}}}{\left(
t^{2}+\left\vert z\right\vert ^{2}+\left\vert u\right\vert \right) ^{\frac{%
2M+1}{2}}}\frac{\left( t^{2}\right) ^{\frac{4M+2}{2}}}{\left(
t^{2}+\left\vert z\right\vert ^{2}+\left\vert u\right\vert \right) ^{\frac{%
Q+4M+2}{2}}}\frac{s^{4M+2}}{\left( s+\left\vert v-v^{\prime }\right\vert
\right) ^{1+4M+2}} \\
&&\times \frac{\left( \left( {t}^{\prime }\right) ^{2}\right) ^{\frac{4M+2}{2%
}}}{\left( \left( {t}^{\prime }\right) ^{2}{+}\left\vert {z}^{\prime
}\right\vert ^{2}+\left\vert u^{\prime }\right\vert \right) {^{\frac{Q+4M+2}{%
2}}}}\frac{\left( {s}^{\prime }\right) ^{4M+2}}{\left( {s}^{\prime }{+}%
\left\vert {v}^{\prime }\right\vert \right) {^{1+4M+2}}}dz^{\prime
}du^{\prime }dv^{\prime }.
\end{eqnarray*}%
Now on $W_{2}$ we have $\left\vert {v}^{\prime }\right\vert \geq \frac{1}{2}%
\left( {{s}+}\left\vert {v}\right\vert \right) $, which gives%
\begin{equation*}
\frac{\left( {s}^{\prime }\right) ^{4M+2}}{\left( {s}^{\prime }{+}\left\vert 
{v}^{\prime }\right\vert \right) {^{1+4M+2}}}\leq C\frac{\left( {s}^{\prime
}\right) ^{4M+2}}{\left( {{s}+}\left\vert {v}\right\vert \right) {^{1+4M+2}}}%
=C\left( \frac{s^{\prime }}{s}\right) ^{4M+2}\frac{s^{4M+2}}{\left(
s+\left\vert v\right\vert \right) ^{1+4M+2}},
\end{equation*}%
and hence%
\begin{equation*}
\left\vert I_{2}^{1}\right\vert \leq A\left( {\frac{{t^{\prime }}}{{t}}}%
\right) ^{2\frac{2M+1}{2}}\left( {\frac{{s^{\prime }}}{{s}}}\right) ^{4M+2}%
\frac{\left( t^{2}\right) ^{\frac{4M+2}{2}}}{\left( t^{2}+\left\vert
z\right\vert ^{2}+\left\vert u\right\vert \right) ^{\frac{Q+4M+2}{2}}}\frac{%
s^{4M+2}}{\left( s+\left\vert v\right\vert \right) ^{1+4M+2}}.
\end{equation*}

Similarly, for $I_{2}^{2}$ we have%
\begin{eqnarray*}
\left\vert I_{2}^{2}\right\vert &\leq &A\int_{W_{2}}\sum_{\gamma \leq M}%
\frac{\left\vert \left( z^{\prime },u^{\prime }\right) \right\vert ^{\frac{%
2M+1}{2}}}{\left( t^{2}+\left\vert z\right\vert ^{2}+\left\vert u\right\vert
\right) ^{\frac{2M+1}{2}}}\frac{\left( t^{2}\right) ^{\frac{4M+2}{2}}}{%
\left( t^{2}+\left\vert z\right\vert ^{2}+\left\vert u\right\vert \right) ^{%
\frac{Q+4M+2}{2}}}\frac{s^{4M+2}\left\vert v^{\prime }\right\vert ^{\gamma }%
}{\left( s+\left\vert v\right\vert \right) ^{1+4M+2+\gamma }} \\
&&\times \frac{\left( \left( {t}^{\prime }\right) ^{2}\right) ^{\frac{4M+2}{2%
}}}{\left( \left( {t}^{\prime }\right) ^{2}{+}\left\vert {z}^{\prime
}\right\vert ^{2}+\left\vert u^{\prime }\right\vert \right) {^{\frac{Q+4M+2}{%
2}}}}\frac{\left( {s}^{\prime }\right) ^{4M+2}}{\left( {s}^{\prime }{+}%
\left\vert {v}^{\prime }\right\vert \right) {^{1+4M+2}}}dz^{\prime
}du^{\prime }dv^{\prime }.
\end{eqnarray*}%
Now we will use the inequality,%
\begin{equation*}
\int_{\left\{ \left\vert {v}^{\prime }\right\vert \geq \frac{1}{2}\left( {{s}%
+}\left\vert {v}\right\vert \right) \geq \frac{1}{2}s\right\} }\frac{1}{%
\left( {s}^{\prime }{+}\left\vert {v}^{\prime }\right\vert \right) {%
^{1+4M+2-\gamma }}}dv^{\prime }\leq C\frac{1}{s^{4M+2-\gamma }},
\end{equation*}%
and take integration first in $v^{\prime }$ and then in $\left( z^{\prime
},u^{\prime }\right) $ to obtain%
\begin{equation*}
\left\vert I_{2}^{2}\right\vert \leq A\left( {\frac{{t^{\prime }}}{{t}}}%
\right) ^{2\frac{2M+1}{2}}\left( {\frac{{s^{\prime }}}{{s}}}\right) ^{4M+2}%
\frac{\left( t^{2}\right) ^{\frac{4M+2}{2}}}{\left( t^{2}+\left\vert
z\right\vert ^{2}+\left\vert u\right\vert \right) ^{\frac{Q+4M+2}{2}}}\frac{%
s^{4M+2}}{\left( s+\left\vert v\right\vert \right) ^{1+4M+2}}.
\end{equation*}

\bigskip

In order to estimate $I_{3}$ we rewrite it as follows:%
\begin{eqnarray*}
I_{3} &=&\int_{W_{3}}\left\{ \left[ \Psi _{t,s}\left( \left( z,u\right)
\circ \left( z^{\prime },u^{\prime }\right) ^{-1},v-v^{\prime }\right)
-\sum_{\gamma \leq M}c_{\gamma }\partial _{v}^{\gamma }\Psi _{t,s}\left(
\left( z,u\right) \circ \left( z^{\prime },u^{\prime }\right) ^{-1},v\right)
\left( -v^{\prime }\right) ^{\gamma }\right] \right. \\
&&\ \ \ \ \ -\left[ \sum_{\left\vert \alpha \right\vert +2\beta \leq
2M}c_{\alpha ,\beta }\partial _{z}^{\alpha }\partial _{u}^{\beta }\Psi
_{t,s}\left( \left( z,u\right) ,v-v^{\prime }\right) \left[ \left( z^{\prime
},u^{\prime }\right) ^{-1}\right] ^{\left( \alpha ,\beta \right) }\right. \\
&&\ \ \ \ \ \ \ \ \ \ \ \ \ \ \ \ \ \ \ \ \left. -\left. \sum_{\gamma \leq
M}c_{\gamma }\partial _{v}^{\gamma }\sum_{\left\vert \alpha \right\vert
+2\beta \leq 2M}c_{\alpha ,\beta }\partial _{z}^{\alpha }\partial
_{u}^{\beta }\Psi _{t,s}\left( \left( z,u\right) ,v\right) \left[ \left(
z^{\prime },u^{\prime }\right) ^{-1}\right] ^{\left( \alpha ,\beta \right)
}\left( -v^{\prime }\right) ^{\gamma }\right] \right\} \\
&&\times \Phi _{t^{\prime },s^{\prime }}\left( \left( z^{\prime },u^{\prime
}\right) ,v^{\prime }\right) dz^{\prime }du^{\prime }dv^{\prime } \\
&=&I_{3}^{1}+I_{3}^{2}.
\end{eqnarray*}%
We have%
\begin{eqnarray*}
\left\vert I_{3}^{1}\right\vert &=&\left\vert \int_{W_{3}}\left[ \Psi
_{t,s}\left( \left( z,u\right) \circ \left( z^{\prime },u^{\prime }\right)
^{-1},v-v^{\prime }\right) -\sum_{\gamma \leq M}c_{\gamma }\partial
_{v}^{\gamma }\Psi _{t,s}\left( \left( z,u\right) \circ \left( z^{\prime
},u^{\prime }\right) ^{-1},v\right) \left( -v^{\prime }\right) ^{\gamma }%
\right] \right. \\
&&\ \ \ \ \ \ \ \ \ \ \ \ \ \ \ \left. \times \Phi _{t^{\prime },s^{\prime
}}\left( \left( z^{\prime },u^{\prime }\right) ,v^{\prime }\right)
dz^{\prime }du^{\prime }dv^{\prime }\right\vert \\
&=&\left\vert \int_{W_{3}}c_{M}\left[ \partial _{v}^{M}\Psi _{t,s}\left(
\left( z,u\right) \circ \left( z^{\prime },u^{\prime }\right)
^{-1},v^{\prime \prime }\right) -\partial _{v}^{M}\Psi _{t,s}\left( \left(
z,u\right) \circ \left( z^{\prime },u^{\prime }\right) ^{-1},v\right) \left(
-v^{\prime }\right) ^{M}\right] \right. \\
&&\ \ \ \ \ \ \ \ \ \ \ \ \ \ \ \left. \times \Phi _{t^{\prime },s^{\prime
}}\left( \left( z^{\prime },u^{\prime }\right) ,v^{\prime }\right)
dz^{\prime }du^{\prime }dv^{\prime }\right\vert \\
&\leq &A\int_{W_{3}}\frac{\left( {t}^{2}\right) ^{\frac{4M+2}{2}}}{\left( {t}%
^{2}{+}\left\vert \left( z,u\right) \circ \left( z^{\prime },u^{\prime
}\right) ^{-1}\right\vert ^{2}\right) {^{\frac{Q+4M+2}{2}}}}\frac{\left\vert
v^{\prime }\right\vert ^{M+1}}{\left( s+\left\vert v\right\vert \right)
^{M+1}}\frac{s^{4M+2}}{\left( s+\left\vert v\right\vert \right) ^{1+4M+2}} \\
&&\times \frac{\left( \left( {t}^{\prime }\right) ^{2}\right) ^{\frac{4M+2}{2%
}}}{\left( \left( {t}^{\prime }\right) ^{2}{+}\left\vert {z}^{\prime
}\right\vert ^{2}+\left\vert u^{\prime }\right\vert \right) {^{\frac{Q+4M+2}{%
2}}}}\frac{\left( {s}^{\prime }\right) ^{4M+2}}{\left( {s}^{\prime }{+}%
\left\vert {v}^{\prime }\right\vert \right) {^{1+4M+2}}}dz^{\prime
}du^{\prime }dv^{\prime }.
\end{eqnarray*}%
Now on $W_{3}$ we have $\left\vert {v}^{\prime }\right\vert \leq \frac{1}{2}%
\left( {{s}+}\left\vert {v}\right\vert \right) $ and $\left\vert \left(
z^{\prime },u^{\prime }\right) \right\vert \geq \frac{1}{2}\left( {t}^{2}{+}%
\left\vert {z}\right\vert ^{2}+\left\vert u\right\vert \right) ^{\frac{1}{2}%
} $ so that%
\begin{equation*}
\frac{\left( \left( {t}^{\prime }\right) ^{2}\right) ^{\frac{4M+2}{2}}}{%
\left( \left( {t}^{\prime }\right) ^{2}{+}\left\vert {z}^{\prime
}\right\vert ^{2}+\left\vert u^{\prime }\right\vert \right) {^{\frac{Q+4M+2}{%
2}}}}\leq C\left( {\frac{\left( {t^{\prime }}\right) ^{2}}{{t}^{2}}}\right)
^{\frac{4M+2}{2}}\frac{\left( {t}^{2}\right) ^{\frac{4M+2}{2}}}{\left( {t}%
^{2}{+}\left\vert {z}\right\vert ^{2}+\left\vert u\right\vert \right) {^{%
\frac{Q+4M+2}{2}}}}.
\end{equation*}%
Thus we have%
\begin{eqnarray*}
\left\vert I_{3}^{1}\right\vert &\leq &A\left( {\frac{\left( {t^{\prime }}%
\right) ^{2}}{{t}^{2}}}\right) ^{\frac{4M+2}{2}}\frac{\left( {t}^{2}\right)
^{\frac{4M+2}{2}}}{\left( {t}^{2}{+}\left\vert {z}\right\vert
^{2}+\left\vert u\right\vert \right) {^{\frac{Q+4M+2}{2}}}}\frac{s^{4M+2}}{%
\left( s+\left\vert v\right\vert \right) ^{1+4M+2}} \\
&&\times \int_{W_{3}}\frac{\left( {t}^{2}\right) ^{\frac{4M+2}{2}}}{\left( {t%
}^{2}{+}\left\vert \left( z,u\right) \circ \left( z^{\prime },u^{\prime
}\right) ^{-1}\right\vert ^{2}\right) {^{\frac{Q+4M+2}{2}}}}\frac{\left\vert
v^{\prime }\right\vert ^{M+1}}{\left( s+\left\vert v\right\vert \right)
^{M+1}}\frac{\left( {s}^{\prime }\right) ^{4M+2}}{\left( {s}^{\prime }{+}%
\left\vert {v}^{\prime }\right\vert \right) {^{1+4M+2}}}dz^{\prime
}du^{\prime }dv^{\prime } \\
&\leq &A\left( {\frac{\left( {t^{\prime }}\right) ^{2}}{{t}^{2}}}\right) ^{%
\frac{4M+2}{2}}\left( {\frac{{s^{\prime }}}{{s}}}\right) ^{M+1}\frac{\left( {%
t}^{2}\right) ^{\frac{4M+2}{2}}}{\left( {t}^{2}{+}\left\vert {z}\right\vert
^{2}+\left\vert u\right\vert \right) {^{\frac{Q+4M+2}{2}}}}\frac{s^{4M+2}}{%
\left( s+\left\vert v\right\vert \right) ^{1+4M+2}}.
\end{eqnarray*}

Similarly, since $\left\vert {v}^{\prime }\right\vert \leq \frac{1}{2}\left( 
{{s}+}\left\vert {v}\right\vert \right) $ on $W_{3}$, we have%
\begin{eqnarray*}
\left\vert I_{3}^{2}\right\vert &=&\left\vert \int_{W_{3}}\left[
\sum_{\left\vert \alpha \right\vert +2\beta \leq 2M}c_{\alpha ,\beta
}\partial _{z}^{\alpha }\partial _{u}^{\beta }\Psi _{t,s}\left( \left(
z,u\right) ,v-v^{\prime }\right) \left[ \left( z^{\prime },u^{\prime
}\right) ^{-1}\right] ^{\left( \alpha ,\beta \right) }\right. \right. \\
&&-\left. \sum_{\gamma \leq M}c_{\gamma }\partial _{v}^{\gamma
}\sum_{\left\vert \alpha \right\vert +2\beta \leq 2M}c_{\alpha ,\beta
}\partial _{z}^{\alpha }\partial _{u}^{\beta }\Psi _{t,s}\left( \left(
z,u\right) ,v\right) \left[ \left( z^{\prime },u^{\prime }\right) ^{-1}%
\right] ^{\left( \alpha ,\beta \right) }\left( -v^{\prime }\right) ^{\gamma }%
\right] \\
&&\ \ \ \ \ \ \ \ \ \ \ \ \ \ \ \ \ \ \ \ \left. \times \Phi _{t^{\prime
},s^{\prime }}\left( \left( z^{\prime },u^{\prime }\right) ,v^{\prime
}\right) dz^{\prime }du^{\prime }dv^{\prime }\right\vert \\
&=&\left\vert \int_{W_{3}}c_{M}\left[ \sum_{\left\vert \alpha \right\vert
+2\beta \leq 2M}c_{\alpha ,\beta }\partial _{v}^{M}\partial _{z}^{\alpha
}\partial _{u}^{\beta }\Psi _{t,s}\left( \left( z,u\right) ,v^{\prime \prime
}\right) \left[ \left( z^{\prime },u^{\prime }\right) ^{-1}\right] ^{\left(
\alpha ,\beta \right) }\right. \right. \\
&&-\left. \partial _{v}^{M}\partial _{z}^{\alpha }\partial _{u}^{\beta }\Psi
_{t,s}\left( \left( z,u\right) ,v\right) \left[ \left( z^{\prime },u^{\prime
}\right) ^{-1}\right] ^{\left( \alpha ,\beta \right) }\left( -v^{\prime
}\right) ^{M}\right] \\
&&\ \ \ \ \ \ \ \ \ \ \ \ \ \ \ \ \ \ \ \ \left. \times \Phi _{t^{\prime
},s^{\prime }}\left( \left( z^{\prime },u^{\prime }\right) ,v^{\prime
}\right) dz^{\prime }du^{\prime }dv^{\prime }\right\vert \\
&\leq &A\int_{W_{3}}\sum_{\left\vert \alpha \right\vert +2\beta \leq 2M}%
\frac{\left( {t}^{2}\right) ^{\frac{4M+2}{2}}\left\vert \left( z^{\prime
},u^{\prime }\right) \right\vert ^{\left\vert \alpha \right\vert +2\beta }}{%
\left( {t}^{2}{+}\left\vert {z}\right\vert ^{2}+\left\vert u\right\vert
\right) {^{\frac{Q+4M+2+\left\vert \alpha \right\vert +2\beta }{2}}}}\frac{%
\left\vert v^{\prime }\right\vert ^{M+1}}{\left( s+\left\vert v\right\vert
\right) ^{M+1}}\frac{s^{4M+2}}{\left( s+\left\vert v\right\vert \right)
^{1+4M+2}} \\
&&\times \frac{\left( \left( {t}^{\prime }\right) ^{2}\right) ^{\frac{4M+2}{2%
}}}{\left( \left( {t}^{\prime }\right) ^{2}{+}\left\vert {z}^{\prime
}\right\vert ^{2}+\left\vert u^{\prime }\right\vert \right) {^{\frac{Q+4M+2}{%
2}}}}\frac{\left( {s}^{\prime }\right) ^{4M+2}}{\left( {s}^{\prime }{+}%
\left\vert {v}^{\prime }\right\vert \right) {^{1+4M+2}}}dz^{\prime
}du^{\prime }dv^{\prime }.
\end{eqnarray*}%
Using the fact that%
\begin{equation*}
\left\vert \left( z^{\prime },u^{\prime }\right) \right\vert \geq \frac{1}{2}%
\left( {t}^{2}{+}\left\vert {z}\right\vert ^{2}+\left\vert u\right\vert
\right) ^{\frac{1}{2}}\geq \frac{1}{2}t
\end{equation*}%
on $W_{3}$ we have%
\begin{eqnarray*}
&&\int_{W_{3}}\frac{\left\vert \left( z^{\prime },u^{\prime }\right)
\right\vert ^{\left\vert \alpha \right\vert +2\beta }}{\left( \left( {t}%
^{\prime }\right) ^{2}{+}\left\vert {z}^{\prime }\right\vert ^{2}+\left\vert
u^{\prime }\right\vert \right) {^{\frac{Q+4M+2}{2}}}}dz^{\prime }du^{\prime }
\\
&\leq &\int_{W_{3}}\frac{1}{\left( \left( {t}^{\prime }\right) ^{2}{+}%
\left\vert {z}^{\prime }\right\vert ^{2}+\left\vert u^{\prime }\right\vert
\right) {^{\frac{Q+4M+2-\left\vert \alpha \right\vert -2\beta }{2}}}}%
dz^{\prime }du^{\prime } \\
&\leq &C\frac{1}{\left( {t}^{2}\right) {^{\frac{4M+2-\left\vert \alpha
\right\vert -2\beta }{2}}}},
\end{eqnarray*}%
and so%
\begin{eqnarray*}
\left\vert I_{3}^{2}\right\vert &\leq &A\left( {\frac{\left( {t^{\prime }}%
\right) ^{2}}{{t}^{2}}}\right) ^{\frac{4M+2}{2}}\left( t^{2}\right) ^{\frac{%
\left\vert \alpha \right\vert +2\beta }{2}}\frac{\left( {t}^{2}\right) ^{%
\frac{4M+2}{2}}}{\left( {t}^{2}{+}\left\vert {z}\right\vert ^{2}+\left\vert
u\right\vert \right) {^{\frac{Q+4M+2+\left\vert \alpha \right\vert +2\beta }{%
2}}}}\left( {\frac{{s^{\prime }}}{{s}}}\right) ^{M+1}\frac{s^{4M+2}}{\left(
s+\left\vert v\right\vert \right) ^{1+4M+2}} \\
&\leq &A\left( {\frac{\left( {t^{\prime }}\right) ^{2}}{{t}^{2}}}\right) ^{%
\frac{4M+2}{2}}\left( {\frac{{s^{\prime }}}{{s}}}\right) ^{M+1}\frac{\left( {%
t}^{2}\right) ^{\frac{4M+2}{2}}}{\left( {t}^{2}{+}\left\vert {z}\right\vert
^{2}+\left\vert u\right\vert \right) {^{\frac{Q+4M+2}{2}}}}\frac{s^{4M+2}}{%
\left( s+\left\vert v\right\vert \right) ^{1+4M+2}}.
\end{eqnarray*}

\bigskip

For the final term $I_{4}$, we use only the size conditions on $\Psi $ to
obtain%
\begin{eqnarray*}
\left\vert I_{4}\right\vert &\leq &A\int_{W_{4}}\left\{ \frac{\left( {t}%
^{2}\right) ^{\frac{4M+2}{2}}}{\left( {t}^{2}{+}\left\vert \left( z,u\right)
\circ \left( z^{\prime },u^{\prime }\right) ^{-1}\right\vert \right) {^{%
\frac{Q+4M+2}{2}}}}\frac{s^{4M+2}}{\left( s+\left\vert v-v^{\prime
}\right\vert \right) ^{1+4M+2}}\right. \\
&&+\sum_{\left\vert \alpha \right\vert +2\beta \leq 2M}c_{\alpha ,\beta }%
\frac{\left( {t}^{2}\right) ^{\frac{4M+2}{2}}\left\vert \left( z^{\prime
},u^{\prime }\right) \right\vert ^{\left\vert \alpha \right\vert +2\beta }}{%
\left( {t}^{2}{+}\left\vert {z}\right\vert ^{2}+\left\vert u\right\vert
\right) {^{\frac{Q+4M+2+\left\vert \alpha \right\vert +2\beta }{2}}}}\frac{%
s^{4M+2}}{\left( s+\left\vert v-v^{\prime }\right\vert \right) ^{1+4M+2}} \\
&&+\sum_{\gamma \leq M}c_{\gamma }\frac{\left( {t}^{2}\right) ^{\frac{4M+2}{2%
}}}{\left( {t}^{2}{+}\left\vert \left( z,u\right) \circ \left( z^{\prime
},u^{\prime }\right) ^{-1}\right\vert \right) {^{\frac{Q+4M+2}{2}}}}\frac{%
s^{4M+2}\left\vert v^{\prime }\right\vert ^{\gamma }}{\left( s+\left\vert
v\right\vert \right) ^{1+4M+2+\gamma }} \\
&&\left. +\sum_{\gamma \leq M}\sum_{\left\vert \alpha \right\vert +2\beta
\leq 2M}c_{\gamma }c_{\alpha ,\beta }\frac{\left( {t}^{2}\right) ^{\frac{4M+2%
}{2}}\left\vert \left( z^{\prime },u^{\prime }\right) \right\vert
^{\left\vert \alpha \right\vert +2\beta }}{\left( {t}^{2}{+}\left\vert {z}%
\right\vert ^{2}+\left\vert u\right\vert \right) {^{\frac{Q+4M+2+\left\vert
\alpha \right\vert +2\beta }{2}}}}\frac{s^{4M+2}\left\vert v^{\prime
}\right\vert ^{\gamma }}{\left( s+\left\vert v\right\vert \right)
^{1+4M+2+\gamma }}\right\} \\
&&\times \frac{\left( \left( {t}^{\prime }\right) ^{2}\right) ^{\frac{4M+2}{2%
}}}{\left( \left( {t}^{\prime }\right) ^{2}{+}\left\vert {z}^{\prime
}\right\vert ^{2}+\left\vert u^{\prime }\right\vert \right) {^{\frac{Q+4M+2}{%
2}}}}\frac{\left( {s}^{\prime }\right) ^{4M+2}}{\left( {s}^{\prime }{+}%
\left\vert {v}^{\prime }\right\vert \right) {^{1+4M+2}}}dz^{\prime
}du^{\prime }dv^{\prime }.
\end{eqnarray*}%
Now on $W_{4}$ we have 
\begin{equation*}
\left\vert \left( z^{\prime },u^{\prime }\right) \right\vert \geq \frac{1}{2}%
\left( {t}^{2}{+}\left\vert {z}\right\vert ^{2}+\left\vert u\right\vert
\right) ^{\frac{1}{2}}\text{ and }\left\vert {v}^{\prime }\right\vert \geq 
\frac{1}{2}\left( {{s}+}\left\vert {v}\right\vert \right) ,
\end{equation*}%
and so

\begin{enumerate}
\item 
\begin{eqnarray*}
&&\int_{W_{4}}\frac{\left( {t}^{2}\right) ^{\frac{4M+2}{2}}}{\left( {t}^{2}{+%
}\left\vert \left( z,u\right) \circ \left( z^{\prime },u^{\prime }\right)
^{-1}\right\vert \right) {^{\frac{Q+4M+2}{2}}}}\frac{s^{4M+2}}{\left(
s+\left\vert v-v^{\prime }\right\vert \right) ^{1+4M+2}} \\
&&\times \frac{\left( \left( {t}^{\prime }\right) ^{2}\right) ^{\frac{4M+2}{2%
}}}{\left( \left( {t}^{\prime }\right) ^{2}{+}\left\vert {z}^{\prime
}\right\vert ^{2}+\left\vert u^{\prime }\right\vert \right) {^{\frac{Q+4M+2}{%
2}}}}\frac{\left( {s}^{\prime }\right) ^{4M+2}}{\left( {s}^{\prime }{+}%
\left\vert {v}^{\prime }\right\vert \right) {^{1+4M+2}}}dz^{\prime
}du^{\prime }dv^{\prime } \\
&\leq &A\left( {\frac{\left( {t^{\prime }}\right) ^{2}}{{t}^{2}}}\right) ^{%
\frac{4M+2}{2}}\frac{\left( {t}^{2}\right) ^{\frac{4M+2}{2}}}{\left( {t}^{2}{%
+}\left\vert {z}\right\vert ^{2}+\left\vert u\right\vert \right) {^{\frac{%
Q+4M+2}{2}}}}\left( {\frac{{s^{\prime }}}{{s}}}\right) ^{4M+2}\frac{s^{4M+2}%
}{\left( s+\left\vert v\right\vert \right) ^{1+4M+2}};
\end{eqnarray*}

\item 
\begin{eqnarray*}
&&\int_{W_{4}}\sum_{\left\vert \alpha \right\vert +2\beta \leq 2M}c_{\alpha
,\beta }\frac{\left( {t}^{2}\right) ^{\frac{4M+2}{2}}\left\vert \left(
z^{\prime },u^{\prime }\right) \right\vert ^{\left\vert \alpha \right\vert
+2\beta }}{\left( {t}^{2}{+}\left\vert {z}\right\vert ^{2}+\left\vert
u\right\vert \right) {^{\frac{Q+4M+2+\left\vert \alpha \right\vert +2\beta }{%
2}}}}\frac{s^{4M+2}}{\left( s+\left\vert v-v^{\prime }\right\vert \right)
^{1+4M+2}} \\
&&\times \frac{\left( \left( {t}^{\prime }\right) ^{2}\right) ^{\frac{4M+2}{2%
}}}{\left( \left( {t}^{\prime }\right) ^{2}{+}\left\vert {z}^{\prime
}\right\vert ^{2}+\left\vert u^{\prime }\right\vert \right) {^{\frac{Q+4M+2}{%
2}}}}\frac{\left( {s}^{\prime }\right) ^{4M+2}}{\left( {s}^{\prime }{+}%
\left\vert {v}^{\prime }\right\vert \right) {^{1+4M+2}}}dz^{\prime
}du^{\prime }dv^{\prime } \\
&\leq &A{\frac{\left( \left( {t^{\prime }}\right) ^{2}\right) ^{\frac{4M+2}{2%
}}}{\left( {t}^{2}\right) ^{\frac{4M+2-\left\vert \alpha \right\vert -2\beta 
}{2}}}}\frac{\left( {t}^{2}\right) ^{\frac{4M+2}{2}}}{\left( {t}^{2}{+}%
\left\vert {z}\right\vert ^{2}+\left\vert u\right\vert \right) {^{\frac{%
Q+4M+2+\left\vert \alpha \right\vert +2\beta }{2}}}}\left( {\frac{{s^{\prime
}}}{{s}}}\right) ^{4M+2}\frac{s^{4M+2}}{\left( s+\left\vert v\right\vert
\right) ^{1+4M+2}} \\
&\leq &A\left( {\frac{\left( {t^{\prime }}\right) ^{2}}{{t}^{2}}}\right) ^{%
\frac{4M+2}{2}}\left( {\frac{{s^{\prime }}}{{s}}}\right) ^{4M+2}\frac{\left( 
{t}^{2}\right) ^{\frac{4M+2}{2}}}{\left( {t}^{2}{+}\left\vert {z}\right\vert
^{2}+\left\vert u\right\vert \right) {^{\frac{Q+4M+2}{2}}}}\frac{s^{4M+2}}{%
\left( s+\left\vert v\right\vert \right) ^{1+4M+2}};
\end{eqnarray*}

\item 
\begin{eqnarray*}
&&\int_{W_{4}}\sum_{\gamma \leq M}c_{\gamma }\frac{\left( {t}^{2}\right) ^{%
\frac{4M+2}{2}}}{\left( {t}^{2}{+}\left\vert \left( z,u\right) \circ \left(
z^{\prime },u^{\prime }\right) ^{-1}\right\vert \right) {^{\frac{Q+4M+2}{2}}}%
}\frac{s^{4M+2}\left\vert v^{\prime }\right\vert ^{\gamma }}{\left(
s+\left\vert v\right\vert \right) ^{1+4M+2+\gamma }} \\
&&\times \frac{\left( \left( {t}^{\prime }\right) ^{2}\right) ^{\frac{4M+2}{2%
}}}{\left( \left( {t}^{\prime }\right) ^{2}{+}\left\vert {z}^{\prime
}\right\vert ^{2}+\left\vert u^{\prime }\right\vert \right) {^{\frac{Q+4M+2}{%
2}}}}\frac{\left( {s}^{\prime }\right) ^{4M+2}}{\left( {s}^{\prime }{+}%
\left\vert {v}^{\prime }\right\vert \right) {^{1+4M+2}}}dz^{\prime
}du^{\prime }dv^{\prime } \\
&\leq &A{\frac{\left( \left( {t^{\prime }}\right) ^{2}\right) ^{\frac{4M+2}{2%
}}}{\left( {t}^{2}{+}\left\vert {z}\right\vert ^{2}+\left\vert u\right\vert
\right) {^{\frac{Q+4M+2}{2}}}}}\frac{\left( {s^{\prime }}\right) ^{4M+2}}{%
s^{4M+2-\gamma }}\frac{s^{4M+2}}{\left( s+\left\vert v\right\vert \right)
^{1+4M+2+\gamma }} \\
&\leq &A\left( {\frac{\left( {t^{\prime }}\right) ^{2}}{{t}^{2}}}\right) ^{%
\frac{4M+2}{2}}\left( {\frac{{s^{\prime }}}{{s}}}\right) ^{4M+2}\frac{\left( 
{t}^{2}\right) ^{\frac{4M+2}{2}}}{\left( {t}^{2}{+}\left\vert {z}\right\vert
^{2}+\left\vert u\right\vert \right) {^{\frac{Q+4M+2}{2}}}}\frac{s^{4M+2}}{%
\left( s+\left\vert v\right\vert \right) ^{1+4M+2}};
\end{eqnarray*}

\item 
\begin{eqnarray*}
&&\int_{W_{4}}\sum_{\gamma \leq M}\sum_{\left\vert \alpha \right\vert
+2\beta \leq 2M}c_{\gamma }c_{\alpha ,\beta }\frac{\left( {t}^{2}\right) ^{%
\frac{4M+2}{2}}\left\vert \left( z^{\prime },u^{\prime }\right) \right\vert
^{\left\vert \alpha \right\vert +2\beta }}{\left( {t}^{2}{+}\left\vert {z}%
\right\vert ^{2}+\left\vert u\right\vert \right) {^{\frac{Q+4M+2+\left\vert
\alpha \right\vert +2\beta }{2}}}}\frac{s^{4M+2}\left\vert v^{\prime
}\right\vert ^{\gamma }}{\left( s+\left\vert v\right\vert \right)
^{1+4M+2+\gamma }} \\
&&\times \frac{\left( \left( {t}^{\prime }\right) ^{2}\right) ^{\frac{4M+2}{2%
}}}{\left( \left( {t}^{\prime }\right) ^{2}{+}\left\vert {z}^{\prime
}\right\vert ^{2}+\left\vert u^{\prime }\right\vert \right) {^{\frac{Q+4M+2}{%
2}}}}\frac{\left( {s}^{\prime }\right) ^{4M+2}}{\left( {s}^{\prime }{+}%
\left\vert {v}^{\prime }\right\vert \right) {^{1+4M+2}}}dz^{\prime
}du^{\prime }dv^{\prime } \\
&\leq &A\frac{\left( {t}^{2}\right) ^{\frac{4M+2}{2}}}{\left( {t}^{2}{+}%
\left\vert {z}\right\vert ^{2}+\left\vert u\right\vert \right) {^{\frac{%
Q+4M+2+\left\vert \alpha \right\vert +2\beta }{2}}}}\frac{s^{4M+2}}{\left(
s+\left\vert v\right\vert \right) ^{1+4M+2+\gamma }}\frac{\left( \left( {t}%
^{\prime }\right) ^{2}\right) ^{\frac{4M+2}{2}}}{\left( {t}^{2}\right) {^{%
\frac{4M+2-\left\vert \alpha \right\vert -2\beta }{2}}}}\frac{\left( {s}%
^{\prime }\right) ^{4M+2}}{{{s}^{4M+2-\gamma }}} \\
&\leq &A\left( {\frac{\left( {t^{\prime }}\right) ^{2}}{{t}^{2}}}\right) ^{%
\frac{4M+2}{2}}\left( {\frac{{s^{\prime }}}{{s}}}\right) ^{4M+2}\frac{\left( 
{t}^{2}\right) ^{\frac{4M+2}{2}}}{\left( {t}^{2}{+}\left\vert {z}\right\vert
^{2}+\left\vert u\right\vert \right) {^{\frac{Q+4M+2}{2}}}}\frac{s^{4M+2}}{%
\left( s+\left\vert v\right\vert \right) ^{1+4M+2}}.
\end{eqnarray*}
\end{enumerate}

\bigskip

\textbf{Case (2)}: In this case we have $t\geq t^{\prime }$ and $s<s^{\prime
}$ and we use vanishing moments for both $\Psi $ and $\Phi $ to obtain%
\begin{eqnarray*}
&&\Psi _{t,s}\ast _{\mathbb{H}^{n}\times \mathbb{R}}\Phi _{t^{\prime
},s^{\prime }}\left( \left( z,u\right) ,v\right) \\
&=&\int_{\mathbb{H}^{n}\times \mathbb{R}}\left\{ \left[ \Psi _{t,s}\left(
\left( z,u\right) \circ \left( z^{\prime },u^{\prime }\right)
^{-1},v-v^{\prime }\right) -\sum_{\left\vert \alpha \right\vert +2\beta \leq
2M}c_{\alpha ,\beta }\partial _{z}^{\alpha }\partial _{u}^{\beta }\Psi
_{t,s}\left( \left( z,u\right) ,v-v^{\prime }\right) \left[ \left( z^{\prime
},u^{\prime }\right) ^{-1}\right] ^{\left( \alpha ,\beta \right) }\right]
\right. \\
&&\ \ \ \ \ \ \ \ \ \ \ \ \ \ \ \ \ \ \ \ \left. \times \left[ \Phi
_{t^{\prime },s^{\prime }}\left( \left( z^{\prime },u^{\prime }\right)
,v^{\prime }\right) -\sum_{\gamma \leq M}c_{\gamma }\partial _{v}^{\gamma
}\Phi _{t^{\prime },s^{\prime }}\left( \left( z^{\prime },u^{\prime }\right)
,v\right) \left( v^{\prime }-v\right) ^{\gamma }\right] \right\} dz^{\prime
}du^{\prime }dv^{\prime }.
\end{eqnarray*}%
As in Case (1) we write the integral above as $I_{1}+I_{2}+I_{3}+I_{4}$
where $I_{j}$ denotes integration taken over the set $W_{j}$ as in (\ref{def
W}). Recall that $\mathbb{H}^{n}\times \mathbb{R}=\dbigcup%
\limits_{j=1}^{4}W_{j}$.

Using the smoothness conditions on both $\Psi $ and $\Phi $ we obtain
similar to Case (1) that%
\begin{eqnarray*}
\left\vert I_{1}\right\vert &\lesssim &A\int_{W_{1}}\frac{\left\vert \left(
z^{\prime },u^{\prime }\right) \right\vert ^{\frac{2M+1}{2}}}{\left(
t^{2}+\left\vert z\right\vert ^{2}+\left\vert u\right\vert \right) ^{\frac{%
2M+1}{2}}}\frac{\left( t^{2}\right) ^{\frac{4M+2}{2}}}{\left(
t^{2}+\left\vert z\right\vert ^{2}+\left\vert u\right\vert \right) ^{\frac{%
Q+4M+2}{2}}}\frac{s^{4M+2}}{\left( s+\left\vert v-v^{\prime }\right\vert
\right) ^{1+4M+2}} \\
&&\times \frac{\left( \left( {t}^{\prime }\right) ^{2}\right) ^{\frac{4M+1}{2%
}}}{\left( \left( {t}^{\prime }\right) ^{2}{+}\left\vert {z}^{\prime
}\right\vert ^{2}+\left\vert u^{\prime }\right\vert \right) {^{\frac{Q+4M+2}{%
2}}}}\frac{\left\vert v^{\prime }-v\right\vert ^{M+1}}{\left( s+\left\vert
v\right\vert \right) ^{M+1}}\frac{\left( s^{\prime }\right) ^{4M+2}}{\left( {%
s}^{\prime }{+}\left\vert {v}\right\vert \right) {^{1+4M+2}}}dz^{\prime
}du^{\prime }dv^{\prime } \\
&\lesssim &A\left( {\frac{{t^{\prime }}}{{t}}}\right) ^{\frac{4M+1}{2}%
}\left( {\frac{{s^{\prime }}}{{s}}}\right) ^{M+1}\frac{\left( {t}^{2}\right)
^{2M+1}}{\left( {t}^{2}{+}\left\vert {z}\right\vert ^{2}+\left\vert
u\right\vert \right) {^{n+1+2M+1}}}\frac{\left( {s}^{\prime }\right) ^{4M+2}%
}{\left( {s}^{\prime }{+}\left\vert {v}\right\vert \right) {^{1+4M+2}}},
\end{eqnarray*}%
which is dominated by the right side of (\ref{product est}) as required.

For $I_{2}$ we have%
\begin{eqnarray*}
\left\vert I_{2}\right\vert &\leq &\int_{W_{2}}\left\vert \Psi _{t,s}\left(
\left( z,u\right) \circ \left( z^{\prime },u^{\prime }\right)
^{-1},v-v^{\prime }\right) -\sum_{\left\vert \alpha \right\vert +2\beta \leq
2M}c_{\alpha ,\beta }\partial _{z}^{\alpha }\partial _{u}^{\beta }\Psi
_{t,s}\left( \left( z,u\right) ,v-v^{\prime }\right) \left[ \left( z^{\prime
},u^{\prime }\right) ^{-1}\right] ^{\left( \alpha ,\beta \right) }\right\vert
\\
&&\times \left\{ \left\vert \Phi _{t^{\prime },s^{\prime }}\left( \left(
z^{\prime },u^{\prime }\right) ,v^{\prime }\right) \right\vert +\left\vert
\sum_{\gamma \leq M}c_{\gamma }\partial _{v}^{\gamma }\Phi _{t^{\prime
},s^{\prime }}\left( \left( z^{\prime },u^{\prime }\right) ,v\right) \left(
v^{\prime }-v\right) ^{\gamma }\right\vert \right\} dz^{\prime }du^{\prime
}dv^{\prime } \\
&\leq &A\int_{W_{2}}\frac{\left\vert \left( z^{\prime },u^{\prime }\right)
\right\vert ^{\frac{2M+1}{2}}}{\left( t^{2}+\left\vert z\right\vert
^{2}+\left\vert u\right\vert \right) ^{\frac{Q+4M+2}{2}}}\frac{\left(
t^{2}\right) ^{\frac{4M+2}{2}}}{\left( t^{2}+\left\vert z\right\vert
^{2}+\left\vert u\right\vert \right) ^{\frac{Q+4M+2}{2}}}\frac{s^{4M+2}}{%
\left( s+\left\vert v-v^{\prime }\right\vert \right) ^{1+4M+2}} \\
&&\times \frac{\left( \left( {t}^{\prime }\right) ^{2}\right) ^{\frac{4M+2}{2%
}}}{\left( \left( {t}^{\prime }\right) ^{2}{+}\left\vert {z}^{\prime
}\right\vert ^{2}+\left\vert u^{\prime }\right\vert \right) {^{\frac{Q+4M+2}{%
2}}}}\frac{\left( s^{\prime }\right) ^{4M+2}}{\left( {s}^{\prime }{+}%
\left\vert {v}\right\vert \right) {^{1+4M+2}}}dz^{\prime }du^{\prime
}dv^{\prime } \\
&&+A\int_{W_{2}}\frac{\left\vert \left( z^{\prime },u^{\prime }\right)
\right\vert ^{\frac{2M+1}{2}}}{\left( t^{2}+\left\vert z\right\vert
^{2}+\left\vert u\right\vert \right) ^{\frac{Q+4M+2}{2}}}\frac{\left(
t^{2}\right) ^{\frac{4M+2}{2}}}{\left( t^{2}+\left\vert z\right\vert
^{2}+\left\vert u\right\vert \right) ^{\frac{Q+4M+2}{2}}}\frac{s^{4M+2}}{%
\left( s+\left\vert v-v^{\prime }\right\vert \right) ^{1+4M+2}} \\
&&\times \sum_{\gamma \leq M}c_{\gamma }\frac{\left( \left( {t}^{\prime
}\right) ^{2}\right) ^{\frac{4M+2}{2}}}{\left( \left( {t}^{\prime }\right)
^{2}{+}\left\vert {z}^{\prime }\right\vert ^{2}+\left\vert u^{\prime
}\right\vert \right) {^{\frac{Q+4M+2}{2}}}}\frac{\left( s^{\prime }\right)
^{4M+2}\left\vert v-v^{\prime }\right\vert ^{\gamma }}{\left( {s}^{\prime }{+%
}\left\vert {v}\right\vert \right) {^{1+4M+2+\gamma }}}dz^{\prime
}du^{\prime }dv^{\prime },
\end{eqnarray*}%
and since $\left\vert v-v^{\prime }\right\vert \geq \frac{1}{2}\left( {s}%
^{\prime }{+}\left\vert {v}\right\vert \right) $, this is dominated by%
\begin{eqnarray*}
&&A\int_{W_{2}}\left( {\frac{{t^{\prime }}}{{t}}}\right) ^{4M+2}\frac{\left( 
{t}^{2}\right) ^{2M+1}}{\left( {t}^{2}{+}\left\vert {z}\right\vert
^{2}+\left\vert u\right\vert \right) {^{n+1+2M+1}}}\frac{{s}^{4M+2}}{\left( {%
s}^{\prime }{+}\left\vert {v}\right\vert \right) {^{1+4M+2}}} \\
&&+A\sum_{\gamma \leq M}c_{\gamma }\left( {\frac{{t^{\prime }}}{{t}}}\right)
^{4M+2}\frac{\left( {t}^{2}\right) ^{2M+1}}{\left( {t}^{2}{+}\left\vert {z}%
\right\vert ^{2}+\left\vert u\right\vert \right) {^{n+1+2M+1}}}\frac{{s}%
^{4M+2}}{\left( {s}^{\prime }{+}\left\vert {v}\right\vert \right) {^{1+4M+2}}%
} \\
&\leq &A\sum_{\gamma \leq M}c_{\gamma }\left( {\frac{{t^{\prime }}}{{t}}}%
\right) ^{4M+2}\left( {\frac{{s}}{{s^{\prime }}}}\right) ^{4M+2}\frac{\left( 
{t}^{2}\right) ^{2M+1}}{\left( {t}^{2}{+}\left\vert {z}\right\vert
^{2}+\left\vert u\right\vert \right) {^{n+1+2M+1}}}\frac{\left( {s}^{\prime
}\right) ^{4M+2}}{\left( {s}^{\prime }{+}\left\vert {v}\right\vert \right) {%
^{1+4M+2}}} \\
&\leq &A\left( {\frac{{t^{\prime }}}{{t}}}\right) ^{4M+2}\left( {\frac{{s}}{{%
s^{\prime }}}}\right) ^{3M+2}\frac{\left( {t}^{2}\right) ^{2M+1}}{\left( {t}%
^{2}{+}\left\vert {z}\right\vert ^{2}+\left\vert u\right\vert \right) {%
^{n+1+2M+1}}}\frac{\left( {s}^{\prime }\right) ^{3M+2}}{\left( {s}^{\prime }{%
+}\left\vert {v}\right\vert \right) {^{1+3M+2}}}.
\end{eqnarray*}

For $I_{3}$ we write%
\begin{eqnarray*}
\left\vert I_{3}\right\vert &\leq &\int_{W_{3}}\left\{ \left\vert \Psi
_{t,s}\left( \left( z,u\right) \circ \left( z^{\prime },u^{\prime }\right)
^{-1},v-v^{\prime }\right) \right\vert +\left\vert \sum_{\left\vert \alpha
\right\vert +2\beta \leq 2M}c_{\alpha ,\beta }\partial _{z}^{\alpha
}\partial _{u}^{\beta }\Psi _{t,s}\left( \left( z,u\right) ,v-v^{\prime
}\right) \left[ \left( z^{\prime },u^{\prime }\right) ^{-1}\right] ^{\left(
\alpha ,\beta \right) }\right\vert \right\} \\
&&\times \left\vert \Phi _{t^{\prime },s^{\prime }}\left( \left( z^{\prime
},u^{\prime }\right) ,v^{\prime }\right) -\sum_{\gamma \leq M}c_{\gamma
}\partial _{v}^{\gamma }\Phi _{t^{\prime },s^{\prime }}\left( \left(
z^{\prime },u^{\prime }\right) ,v\right) \left( v^{\prime }-v\right)
^{\gamma }\right\vert dz^{\prime }du^{\prime }dv^{\prime },
\end{eqnarray*}%
and using the size condition on $\Psi $ and the smoothness condition on $%
\Phi $ we get%
\begin{eqnarray*}
\left\vert I_{3}\right\vert &\leq &A\int_{W_{3}}\frac{\left( t^{2}\right) ^{%
\frac{4M+2}{2}}}{\left( t^{2}+\left\vert \left( z,u\right) \circ \left(
z^{\prime },u^{\prime }\right) ^{-1}\right\vert \right) ^{\frac{Q+4M+2}{2}}}%
\frac{s^{4M+2}}{\left( s+\left\vert v-v^{\prime }\right\vert \right)
^{1+4M+2}} \\
&&\times \frac{\left( \left( {t}^{\prime }\right) ^{2}\right) ^{\frac{4M+2}{2%
}}}{\left( \left( {t}^{\prime }\right) ^{2}{+}\left\vert {z}^{\prime
}\right\vert ^{2}+\left\vert u^{\prime }\right\vert \right) {^{\frac{Q+4M+2}{%
2}}}}\frac{\left\vert v-v^{\prime }\right\vert ^{M+1}}{\left( {s}^{\prime }{+%
}\left\vert {v}\right\vert \right) ^{M+1}}\frac{\left( s^{\prime }\right)
^{4M+2}}{\left( {s}^{\prime }{+}\left\vert {v}\right\vert \right) {^{1+4M+2}}%
}dz^{\prime }du^{\prime }dv^{\prime } \\
&&+A\int_{W_{3}}\sum_{\left\vert \alpha \right\vert +2\beta \leq
2M}c_{\alpha ,\beta }\frac{\left( t^{2}\right) ^{\frac{4M+2}{2}}\left\vert
\left( z^{\prime },u^{\prime }\right) \right\vert ^{\left\vert \alpha
\right\vert +2\beta }}{\left( {t}^{2}{+}\left\vert {z}\right\vert
^{2}+\left\vert u\right\vert \right) ^{\frac{Q+4M+2+\left\vert \alpha
\right\vert +2\beta }{2}}}\frac{s^{4M+2}}{\left( s+\left\vert v-v^{\prime
}\right\vert \right) ^{1+4M+2}} \\
&&\frac{\left( \left( {t}^{\prime }\right) ^{2}\right) ^{\frac{4M+2}{2}}}{%
\left( \left( {t}^{\prime }\right) ^{2}{+}\left\vert {z}^{\prime
}\right\vert ^{2}+\left\vert u^{\prime }\right\vert \right) {^{\frac{Q+4M+2}{%
2}}}}\frac{\left\vert v-v^{\prime }\right\vert ^{M+1}}{\left( {s}^{\prime }{+%
}\left\vert {v}\right\vert \right) ^{M+1}}\frac{\left( s^{\prime }\right)
^{4M+2}}{\left( {s}^{\prime }{+}\left\vert {v}\right\vert \right) {^{1+4M+2}}%
}dz^{\prime }du^{\prime }dv^{\prime }.
\end{eqnarray*}%
Using the fact that on\thinspace $W_{3}$ we have%
\begin{equation*}
\left\vert \left( z^{\prime },u^{\prime }\right) \right\vert \geq \frac{1}{2}%
\left( {t}^{2}{+}\left\vert {z}\right\vert ^{2}+\left\vert u\right\vert
\right) ^{\frac{1}{2}},
\end{equation*}%
we have%
\begin{eqnarray*}
\left\vert I_{3}\right\vert &\leq &A\frac{\left( \left( {t}^{\prime }\right)
^{2}\right) ^{\frac{4M+2}{2}}}{\left( \left( {t}^{\prime }\right) ^{2}{+}%
\left\vert {z}^{\prime }\right\vert ^{2}+\left\vert u^{\prime }\right\vert
\right) {^{\frac{Q+4M+2}{2}}}}\left( \frac{s}{s^{\prime }}\right) ^{M+1}%
\frac{\left( s^{\prime }\right) ^{4M+2}}{\left( {s}^{\prime }{+}\left\vert {v%
}\right\vert \right) {^{1+4M+2}}} \\
&&+A\frac{\left( t^{2}\right) ^{\frac{4M+2}{2}}}{\left( {t}^{2}{+}\left\vert 
{z}\right\vert ^{2}+\left\vert u\right\vert \right) ^{\frac{%
Q+4M+2+\left\vert \alpha \right\vert +2\beta }{2}}}\left( \frac{s}{s^{\prime
}}\right) ^{M+1}\frac{\left( s^{\prime }\right) ^{4M+2}}{\left( {s}^{\prime }%
{+}\left\vert {v}\right\vert \right) {^{1+4M+2}}}\frac{\left( \left( {t}%
^{\prime }\right) ^{2}\right) ^{\frac{4M+2}{2}}}{\left( \left( {t}^{\prime
}\right) ^{2}{+}\left\vert {z}^{\prime }\right\vert ^{2}+\left\vert
u^{\prime }\right\vert \right) {^{\frac{Q+4M+2-\left\vert \alpha \right\vert
-2\beta }{2}}}} \\
&\leq &A\left( {\frac{{t^{\prime }}}{{t}}}\right) ^{4M+2}\left( {\frac{{s}}{{%
s^{\prime }}}}\right) ^{M+1}\frac{\left( {t}^{2}\right) ^{2M+1}}{\left( {t}%
^{2}{+}\left\vert {z}\right\vert ^{2}+\left\vert u\right\vert \right) {%
^{n+1+2M+1}}}\frac{\left( {s}^{\prime }\right) ^{4M+2}}{\left( {s}^{\prime }{%
+}\left\vert {v}\right\vert \right) {^{1+4M+2}}}.
\end{eqnarray*}

Finally, for term $I_{4}$, we only use size conditions on both $\Psi $ and $%
\Phi $ to get%
\begin{eqnarray*}
\left\vert I_{4}\right\vert &\leq &\int_{W_{4}}\left\{ \left\vert \Psi
_{t,s}\left( \left( z,u\right) \circ \left( z^{\prime },u^{\prime }\right)
^{-1},v-v^{\prime }\right) \right\vert +\left\vert \sum_{\left\vert \alpha
\right\vert +2\beta \leq 2M}c_{\alpha ,\beta }\partial _{z}^{\alpha
}\partial _{u}^{\beta }\Psi _{t,s}\left( \left( z,u\right) ,v-v^{\prime
}\right) \left[ \left( z^{\prime },u^{\prime }\right) ^{-1}\right] ^{\left(
\alpha ,\beta \right) }\right\vert \right\} \\
&&\times \left\{ \left\vert \Phi _{t^{\prime },s^{\prime }}\left( \left(
z^{\prime },u^{\prime }\right) ,v^{\prime }\right) \right\vert +\left\vert
\sum_{\gamma \leq M}c_{\gamma }\partial _{v}^{\gamma }\Phi _{t^{\prime
},s^{\prime }}\left( \left( z^{\prime },u^{\prime }\right) ,v\right) \left(
v^{\prime }-v\right) ^{\gamma }\right\vert \right\} dz^{\prime }du^{\prime
}dv^{\prime }.
\end{eqnarray*}%
There are thus four terms after multiplying out, and each of these four
terms is handled in the same way that the corresponding four terms for $%
I_{4} $ in \emph{Case (1)} were handled.

\bigskip

There are corresponding orthogonality estimates that can now be obtained for
component functions on $\mathbb{H}^{n}$.

\begin{lemma}
\label{almostortho}Suppose $\Psi ,\Phi \in \mathcal{M}_{product}^{2M}\left( 
\mathbb{H}^{n}\times \mathbb{R}\right) $ and let $\left\{ \psi
_{t,s}\right\} _{t,s>0}$ and $\left\{ \phi _{t,s}\right\} _{t,s>0}$ be the
associated collections of component functions as defined in Definition \ref%
{10} and (\ref{Psi}) above. Then there exists a constant $C=C(M)$ depending
only on $M$ such that if $\left( {t}\vee {t^{\prime }}\right) ^{2}\leq {s}%
\vee {s^{\prime }},$ then 
\begin{eqnarray}
\left\vert \psi _{t,s}\ast _{\mathbb{H}^{n}}\phi _{t^{\prime },s^{\prime
}}(z,u)\right\vert &\leq &C\left( {\frac{{t}}{{t^{\prime }}}}\wedge {\frac{{%
t^{\prime }}}{{t}}}\right) ^{2M}\left( {\frac{{s}}{{s^{\prime }}}}\wedge {%
\frac{{s^{\prime }}}{{s}}}\right) ^{M}  \label{prod est} \\
&&\times {\frac{\left( {{t}\vee {t^{\prime }}}\right) ^{2M}}{\left( {{t}\vee 
{t^{\prime }}+}\left\vert {z}\right\vert \right) {^{2n+2M}}}}{\frac{\left( {{%
s}\vee {s^{\prime }}}\right) ^{M}}{\left( {{s}\vee {s^{\prime }}+}\left\vert 
{u}\right\vert \right) {^{1+M}}}},  \notag
\end{eqnarray}%
and if $\left( {t}\vee {t^{\prime }}\right) ^{2}\geq {s}\vee {s^{\prime }},$
then 
\begin{eqnarray}
\left\vert \psi _{t,s}\ast \phi _{t^{\prime },s^{\prime }}(z,u)\right\vert
&\leq &C\left( {\frac{{t}}{{t^{\prime }}}}\wedge {\frac{{t^{\prime }}}{{t}}}%
\right) ^{M}\left( {\frac{{s}}{{s^{\prime }}}}\wedge {\frac{{s^{\prime }}}{{s%
}}}\right) ^{M}  \label{one est} \\
&&\times {\frac{\left( {{t}\vee {t^{\prime }}}\right) ^{M}}{\left( {{t}\vee {%
t^{\prime }}+}\left\vert {z}\right\vert \right) {^{2n+M}}}}{\frac{\left( {{t}%
\vee {t^{\prime }}}\right) ^{M}}{\left( {{t}\vee {t^{\prime }}+}\sqrt{%
\left\vert {u}\right\vert }\right) {^{2+2M}}}}.  \notag
\end{eqnarray}
\end{lemma}

Roughly speaking, $\psi _{t,s}\ast \phi _{t^{\prime },s^{\prime }}(z,u)$
satisfies the \emph{product multi-parameter} almost orthogonality when $%
\left( {t}\vee {t^{\prime }}\right) ^{2}\leq {s}\vee {s^{\prime }}$ and the 
\emph{one-parameter} almost orthogonality when $\left( {t}\vee {t^{\prime }}%
\right) ^{2}\geq {s}\vee {s^{\prime }}$.

\bigskip

\textbf{Proof of Lemma \ref{almostortho}}: We will use the projection Lemma %
\ref{project} to pass from the orthogonality estimates for the product
dilations $\left\{ \Psi _{t,s}\right\} _{t,s>0}$ and $\left\{ \Phi
_{t,s}\right\} _{t,s>0}$ in Lemma \ref{almostorthoproduct} to the estimates
for the component functions $\left\{ \psi _{t,s}\right\} _{t,s>0}$ and $%
\left\{ \phi _{t,s}\right\} _{t,s>0}$\ in Lemma \ref{almostortho}.

From\ (\ref{product est}) and (\ref{projection commutes}) we obtain%
\begin{eqnarray}
&&\left\vert \psi _{t,s}\ast \phi _{t^{\prime },s^{\prime }}(z,u)\right\vert
=\left\vert \int_{\mathbb{R}}\Psi _{t,s}\ast _{\mathbb{H}^{n}\times \mathbb{R%
}}\Phi _{t^{\prime },s^{\prime }}\left( \left( z,u-v\right) ,v\right)
dv\right\vert  \label{integral} \\
&\lesssim &C\left( {\frac{{t}}{{t^{\prime }}}}\wedge {\frac{{t^{\prime }}}{{t%
}}}\right) ^{2M}\left( {\frac{{s}}{{s^{\prime }}}}\wedge {\frac{{s^{\prime }}%
}{{s}}}\right) ^{M}  \notag \\
&&\times \int_{\mathbb{R}}\frac{\left( {{t}\vee {t^{\prime }}}\right) ^{4M}}{%
\left( \left( {{t}\vee {t^{\prime }}}\right) ^{2}{+}\left\vert {z}%
\right\vert ^{2}+\left\vert u-v\right\vert \right) {^{n+1+2M}}}\frac{\left( {%
{s}\vee {s^{\prime }}}\right) ^{2M}}{\left( {{s}\vee {s^{\prime }}+}%
\left\vert {v}\right\vert \right) {^{1+2M}}}dv.  \notag
\end{eqnarray}%
Now we consider four cases separately.

\bigskip

\textbf{Case 1}: $\left( {{t}\vee {t^{\prime }}}\right) ^{2}\leq {{s}\vee {%
s^{\prime }}}$ and $\left\vert u\right\vert \geq {{s}\vee {s^{\prime }}}$.
In this case we use that 
\begin{equation}
\frac{\left( {{s}\vee {s^{\prime }}}\right) ^{2M}}{\left( {{s}\vee {%
s^{\prime }}+}\left\vert {v}\right\vert \right) {^{1+2M}}}=\frac{1}{{{s}\vee 
{s^{\prime }}}}\frac{1}{\left( {{1}+}\left\vert \frac{{v}}{{{s}\vee {%
s^{\prime }}}}\right\vert \right) {^{1+2M}}}  \label{s integral}
\end{equation}%
has integral roughly $1$, with essential support $\left[ -{{s}\vee {%
s^{\prime },s}\vee {s^{\prime }}}\right] $, to obtain%
\begin{eqnarray*}
&&\int_{\mathbb{R}}\frac{\left( {{t}\vee {t^{\prime }}}\right) ^{4M}}{\left(
\left( {{t}\vee {t^{\prime }}}\right) ^{2}{+}\left\vert {z}\right\vert
^{2}+\left\vert u-v\right\vert \right) {^{n+1+2M}}}\frac{\left( {{s}\vee {%
s^{\prime }}}\right) ^{2M}}{\left( {{s}\vee {s^{\prime }}+}\left\vert {v}%
\right\vert \right) {^{1+2M}}}dv \\
&\approx &\frac{\left( {{t}\vee {t^{\prime }}}\right) ^{4M}}{\left( \left( {{%
t}\vee {t^{\prime }}}\right) ^{2}{+}\left\vert {z}\right\vert
^{2}+\left\vert u\right\vert \right) {^{n+1+2M}}}\leq \frac{\left( {{t}\vee {%
t^{\prime }}}\right) ^{2M}}{\left( \left( {{t}\vee {t^{\prime }}}\right) ^{2}%
{+}\left\vert {z}\right\vert ^{2}\right) {^{n+M}}}\frac{\left( {{t}\vee {%
t^{\prime }}}\right) ^{2M}}{\left( \left( {{t}\vee {t^{\prime }}}\right)
^{2}+\left\vert u\right\vert \right) {^{1+M}}} \\
&\leq &\frac{\left( {{t}\vee {t^{\prime }}}\right) ^{2M}}{\left( \left( {{t}%
\vee {t^{\prime }}}\right) {+}\left\vert {z}\right\vert \right) {^{2n+2M}}}%
\frac{\left( {{s}\vee s{^{\prime }}}\right) ^{M}}{\left( {{s}\vee s{^{\prime
}}}+\left\vert u\right\vert \right) {^{1+M}}}.
\end{eqnarray*}%
Plugging this estimate into the right side of (\ref{integral}) leads to the
correct product estimate (\ref{prod est}) for this case.

\bigskip

\textbf{Case 2}: $\left( {{t}\vee {t^{\prime }}}\right) ^{2}\leq {{s}\vee {%
s^{\prime }}}$ and $\left\vert u\right\vert \leq {{s}\vee {s^{\prime }}}$.
In this case we bound the left side of (\ref{s integral}) by $\frac{1}{{{s}%
\vee {s^{\prime }}}}$ to obtain%
\begin{eqnarray*}
&&\int_{\mathbb{R}}\frac{\left( {{t}\vee {t^{\prime }}}\right) ^{4M}}{\left(
\left( {{t}\vee {t^{\prime }}}\right) ^{2}{+}\left\vert {z}\right\vert
^{2}+\left\vert u-v\right\vert \right) {^{n+1+2M}}}\frac{\left( {{s}\vee {%
s^{\prime }}}\right) ^{2M}}{\left( {{s}\vee {s^{\prime }}+}\left\vert {v}%
\right\vert \right) {^{1+2M}}}dv \\
&\lesssim &\frac{1}{{{s}\vee {s^{\prime }}}}\int_{\mathbb{R}}\frac{\left( {{t%
}\vee {t^{\prime }}}\right) ^{4M}}{\left( \left( {{t}\vee {t^{\prime }}}%
\right) ^{2}{+}\left\vert {z}\right\vert ^{2}+\left\vert u-v\right\vert
\right) {^{n+1+2M}}}dv \\
&\lesssim &\frac{1}{{{s}\vee {s^{\prime }}}}\frac{\left( {{t}\vee {t^{\prime
}}}\right) ^{4M}}{\left( \left( {{t}\vee {t^{\prime }}}\right) ^{2}{+}%
\left\vert {z}\right\vert ^{2}\right) {^{n+2M}}}\leq \frac{\left( {{t}\vee {%
t^{\prime }}}\right) ^{2M}}{\left( \left( {{t}\vee {t^{\prime }}}\right) {+}%
\left\vert {z}\right\vert \right) {^{2n+2M}}}\frac{\left( {{s}\vee s{%
^{\prime }}}\right) ^{M}}{\left( {{s}\vee s{^{\prime }}}+\left\vert
u\right\vert \right) {^{1+M}}}.
\end{eqnarray*}%
Plugging this estimate into the right side of (\ref{integral}) again leads
to the correct product estimate (\ref{prod est}) for this case.

\bigskip

\textbf{Case 3}: $\left( {{t}\vee {t^{\prime }}}\right) ^{2}\geq {{s}\vee {%
s^{\prime }}}$ and $\left\vert u\right\vert \leq \left( {{t}\vee t{^{\prime }%
}}\right) ^{2}$. In this case we have%
\begin{eqnarray*}
&&\int_{\mathbb{R}}\frac{\left( {{t}\vee {t^{\prime }}}\right) ^{4M}}{\left(
\left( {{t}\vee {t^{\prime }}}\right) ^{2}{+}\left\vert {z}\right\vert
^{2}+\left\vert u-v\right\vert \right) {^{n+1+2M}}}\frac{\left( {{s}\vee {%
s^{\prime }}}\right) ^{2M}}{\left( {{s}\vee {s^{\prime }}+}\left\vert {v}%
\right\vert \right) {^{1+2M}}}dv \\
&\lesssim &\frac{\left( {{t}\vee {t^{\prime }}}\right) ^{4M}}{\left( \left( {%
{t}\vee {t^{\prime }}}\right) ^{2}{+}\left\vert {z}\right\vert ^{2}\right) {%
^{n+1+2M}}}\lesssim \frac{\left( {{t}\vee {t^{\prime }}}\right) ^{2M}}{%
\left( \left( {{t}\vee {t^{\prime }}}\right) ^{2}{+}\left\vert {z}%
\right\vert ^{2}\right) {^{n+M}}}\frac{\left( {{t}\vee {t^{\prime }}}\right)
^{2M}}{\left( \left( {{t}\vee {t^{\prime }}}\right) ^{2}{+}\left\vert {u}%
\right\vert \right) {^{1+M}}} \\
&\approx &\frac{\left( {{t}\vee {t^{\prime }}}\right) ^{2M}}{\left( {{t}\vee 
{t^{\prime }}+}\left\vert {z}\right\vert \right) {^{2n+2M}}}\frac{\left( {{t}%
\vee {t^{\prime }}}\right) ^{2M}}{\left( {{t}\vee {t^{\prime }}+}\sqrt{%
\left\vert {u}\right\vert }\right) {^{2+2M}}}.
\end{eqnarray*}%
Plugging this estimate into the right side of (\ref{integral}) leads to the
correct one-parameter estimate (\ref{one est}) for this case.

\bigskip

\textbf{Case 4}: $\left( {{t}\vee {t^{\prime }}}\right) ^{2}\geq {{s}\vee {%
s^{\prime }}}$ and $\left\vert u\right\vert \geq \left( {{t}\vee t{^{\prime }%
}}\right) ^{2}$. In this case we have%
\begin{eqnarray*}
&&\int_{\mathbb{R}}\frac{\left( {{t}\vee {t^{\prime }}}\right) ^{4M}}{\left(
\left( {{t}\vee {t^{\prime }}}\right) ^{2}{+}\left\vert {z}\right\vert
^{2}+\left\vert u-v\right\vert \right) {^{n+1+2M}}}\frac{\left( {{s}\vee {%
s^{\prime }}}\right) ^{2M}}{\left( {{s}\vee {s^{\prime }}+}\left\vert {v}%
\right\vert \right) {^{1+2M}}}dv \\
&\lesssim &\frac{\left( {{t}\vee {t^{\prime }}}\right) ^{4M}}{\left( \left( {%
{t}\vee {t^{\prime }}}\right) ^{2}{+}\left\vert {z}\right\vert
^{2}+\left\vert u\right\vert \right) {^{n+1+2M}}}\lesssim \frac{\left( {{t}%
\vee {t^{\prime }}}\right) ^{2M}}{\left( \left( {{t}\vee {t^{\prime }}}%
\right) ^{2}{+}\left\vert {z}\right\vert ^{2}\right) {^{n+M}}}\frac{\left( {{%
t}\vee {t^{\prime }}}\right) ^{2M}}{\left( \left( {{t}\vee {t^{\prime }}}%
\right) ^{2}{+}\left\vert {u}\right\vert \right) {^{1+M}}} \\
&\approx &\frac{\left( {{t}\vee {t^{\prime }}}\right) ^{2M}}{\left( {{t}\vee 
{t^{\prime }}+}\left\vert {z}\right\vert \right) {^{2n+2M}}}\frac{\left( {{t}%
\vee {t^{\prime }}}\right) ^{2M}}{\left( {{t}\vee {t^{\prime }}+}\sqrt{%
\left\vert {u}\right\vert }\right) {^{2+2M}}}.
\end{eqnarray*}

Plugging this estimate into the right side of (\ref{integral}) again leads
to the correct one-parameter estimate (\ref{one est}).

\subsubsection{Proof of the Plancherel-P\^{o}lya inequalities}

Before we prove the Plancherel-P\^{o}lya-type inequality in Theorem \ref{P-P}%
, we first prove the following lemma. We will often use the notation $\left(
x_{I},y_{J}\right) $ in place of $\left( z_{I},u_{J}\right) $ for the center
of the dyadic rectangle $I\times J$ in $\mathbb{H}^{n}$, i.e. we write $x$
in place of $z$, and $y$ in place of $u$.

\begin{lemma}
\label{7}Let $I\times J$ and $I^{\prime }\times J^{\prime }$ be dyadic
rectangles in $\mathbb{H}^{n}$ such that $\ell (I)=2^{-j-N}$, $\ell
(J)=2^{-j-N}+2^{-k-N}$, $\ell (I^{\prime })=2^{-j^{\prime }-N}$ and $\ell
(J^{\prime })=2^{-j^{\prime }-N}+2^{-k^{\prime }-N}$. Thus for any $\left(
z,u\right) $ and $\left( z^{\ast },u^{\ast }\right) $ in $\mathbb{H}^{n}$,
we have when $j\wedge j^{\prime }\geq k\wedge k^{\prime }$ 
\begin{eqnarray*}
&&\sum\limits_{I^{\prime },J^{\prime }}{\frac{{2^{-|j-{j^{\prime }}|L_{1}-|k-%
{k^{\prime }}|L_{2}}2^{-(j\wedge j^{\prime })K_{1}-(k\wedge k^{\prime
})K_{2}}|I^{\prime }||J^{\prime }|}}{{\left( 2^{-j\wedge j^{\prime
}}+|z-x_{I^{\prime }}|\right) ^{2n+K_{1}}\left( 2^{-k\wedge k^{\prime
}}+|u-y_{J^{\prime }}|\right) ^{1+K_{2}}}}}\cdot |\phi _{{j^{\prime }},{%
k^{\prime }}}\ast f\left( x_{I^{\prime }},y_{J^{\prime }}\right) | \\
&\leq &C_{1}(N,r,j,j^{\prime },k,k^{\prime })2^{-|j-{j^{\prime }}|L_{1}} \\
&&\times 2^{-|k-{k^{\prime }}|L_{2}}\left\{ M_{S}\left[ \left(
\sum\limits_{J^{\prime }}\sum\limits_{I^{\prime }}|\phi _{{j^{\prime }},{%
k^{\prime }}}\ast f(x_{I^{\prime }},y_{J^{\prime }})|\chi _{J^{\prime }}\chi
_{I^{\prime }}\right) ^{r}\right] \right\} ^{{\frac{{1}}{{r}}}}(z^{\ast
},u^{\ast })
\end{eqnarray*}%
and when $j\wedge j^{\prime }\leq k\wedge k^{\prime }$ 
\begin{eqnarray*}
&&\sum\limits_{I^{\prime },J^{\prime }}{\frac{{2^{-|j-{j^{\prime }}|L_{1}-|k-%
{k^{\prime }}|L_{2}}2^{-(j\wedge j^{\prime })K_{1}-(j\wedge j^{\prime
})K_{2}}|I^{\prime }||J^{\prime }|}}{{\left( 2^{-j\wedge j^{\prime
}}+|z-x_{I^{\prime }}|\right) ^{2n+K_{1}}\left( 2^{-j\wedge j^{\prime
}}+|u-y_{J^{\prime }}|\right) ^{1+K_{2}}}}}|\phi _{{j^{\prime }},{k^{\prime }%
}}\ast f(x_{I^{\prime }},y_{J^{\prime }})| \\
&\leq &C_{2}(N,r,j,j^{\prime },k,k^{\prime })2^{-|j-{j^{\prime }}%
|L_{1}}2^{-|k-{k^{\prime }}|L_{2}} \\
&&\times \left\{ M\left[ \left( \sum\limits_{J^{\prime
}}\sum\limits_{I^{\prime }}|\phi _{{j^{\prime }},{k^{\prime }}}\ast
f(x_{I^{\prime }},y_{J^{\prime }})|\chi _{J^{\prime }}\chi _{I^{\prime
}}\right) ^{r}\right] \right\} ^{{\frac{{1}}{{r}}}}(z^{\ast },u^{\ast })
\end{eqnarray*}%
where $M$ is the Hardy-Littlewood maximal function on $\mathbb{H}^{n}$, $%
M_{S}$ is the strong maximal function on $\mathbb{H}^{n}$ as defined in
(1.1), and $\max \left\{ \frac{2n}{2n+K_{1}},\frac{1}{1+K_{2}}\right\} <r$
and 
\begin{equation*}
C_{1}(N,r,j,j^{\prime },k,k^{\prime })=2^{(\frac{1}{r}-1)N(2n+1)}\cdot
2^{[2n(j\wedge j^{\prime }-j^{\prime })+(k\wedge k^{\prime }-k^{\prime })](1-%
\frac{1}{r})}
\end{equation*}%
\begin{equation*}
C_{2}(N,r,j,j^{\prime },k,k^{\prime })=2^{(\frac{1}{r}-1)N(2n+1)}\cdot
2^{[2n(j\wedge j^{\prime }-j^{\prime })+(j\wedge j^{\prime }-j^{\prime
}\wedge k^{\prime })](1-\frac{1}{r})}.
\end{equation*}
\end{lemma}

\begin{proof}
We set 
\begin{equation*}
A_{0}=\{I^{\prime }:\ell (I^{\prime })=2^{-j^{\prime }-N},\,\,\frac{%
|z-x_{I^{\prime }}|}{2^{-j\wedge j^{\prime }}}\leq 1\}
\end{equation*}%
\begin{equation*}
B_{0}=\{J^{\prime }:\ell (J^{\prime })=2^{-j^{\prime }-N}+2^{-k^{\prime
}-N},\,\,\frac{|u-y_{J^{\prime }}|}{2^{-k\wedge k^{\prime }}}\leq 1\}
\end{equation*}%
where $x_{I^{\prime }}\in I^{\prime }$ and $y_{J^{\prime }}\in J^{\prime }$,
and where for $\ell \geq 1$, $i\geq 1$ 
\begin{equation*}
A_{\ell }=\{I^{\prime }:\ell (I^{\prime })=2^{-j^{\prime }-N},\,\,2^{\ell
-1}<\frac{|z-x_{I^{\prime }}|}{2^{-j\wedge j^{\prime }}}\leq 2^{\ell }\}.
\end{equation*}%
\begin{equation*}
B_{i}=\{J^{\prime }:\ell (J^{\prime })=2^{-j^{\prime }-N}+2^{-k^{\prime
}-N},\,\,2^{i-1}<\frac{|u-y_{J^{\prime }}|}{2^{-k\wedge k^{\prime }}}\leq
2^{i}\}.
\end{equation*}%
We first consider the case when $j\wedge j^{\prime }\geq k\wedge k^{\prime }$%
, and let 
\begin{equation*}
\tau =[2n(j\wedge j^{\prime }-j^{\prime })+(k\wedge k^{\prime }-k^{\prime
})](1-\frac{1}{r}).
\end{equation*}%
Then 
\begin{eqnarray*}
&&\sum\limits_{I^{\prime },J^{\prime }}{\frac{{2^{-(j\wedge j^{\prime
})K_{1}-(k\wedge k^{\prime })K_{2}}|I^{\prime }||J^{\prime }|}}{{\left(
2^{-j\wedge j^{\prime }}+|z-x_{I^{\prime }}|\right) ^{2n+K_{1}}\left(
2^{-k\wedge k^{\prime }}+|u-y_{J^{\prime }}|\right) ^{1+K_{2}}}}}\cdot |\phi
_{j^{\prime },k^{\prime }}\ast f(x_{I^{\prime }},y_{J^{\prime }})| \\
&\leq &\sum\limits_{\ell ,i\geq 0}2^{-\ell
(2n+K_{1})}2^{-i(1+K_{2})}2^{-N(2n+1)}2^{(j\wedge j^{\prime }-j^{\prime
})n+(k\wedge k^{\prime }-k^{\prime })m}\sum\limits_{I^{\prime }\in A_{\ell
},J^{\prime }\in B_{i}}|\phi _{{j^{\prime }},{k^{\prime }}}\ast
f(x_{I^{\prime }},y_{J^{\prime }})| \\
&\leq &\sum\limits_{\ell ,i\geq 0}2^{-\ell
(n+K_{1})}2^{-i(m+K_{2})}2^{-N(n+m)}2^{(j\wedge j^{\prime }-j^{\prime
})2n+(k\wedge k^{\prime }-k^{\prime })}\left( \sum\limits_{I^{\prime }\in
A_{\ell },J^{\prime }\in B_{i}}\left\vert \phi _{{j^{\prime }},{k^{\prime }}%
}\ast f(x_{I^{\prime }},y_{J^{\prime }})\right\vert ^{r}\right) ^{\frac{1}{r}%
} \\
&=&\sum\limits_{\ell ,i\geq 0}2^{-\ell
(2n+K_{1})-i(1+K_{2})-N(2n+1)}2^{(j\wedge j^{\prime }-j^{\prime
})2n+(k\wedge k^{\prime }-k^{\prime })} \\
&&\times \left( \int_{\mathbb{H}^{n}}|I^{\prime }|^{-1}|J^{\prime
}|^{-1}\sum\limits_{I^{\prime }\in A_{\ell },J^{\prime }\in B_{i}}|\phi _{{%
j^{\prime }},{k^{\prime }}}\ast f(x_{I^{\prime }},y_{J^{\prime }})|^{r}\chi
_{I^{\prime }}\chi _{J^{\prime }}\right) ^{\frac{1}{r}} \\
&\leq &\sum\limits_{\ell ,i\geq 0}2^{-\ell (2n+K_{1}-\frac{2n}{r})-i(1+K_{2}-%
\frac{1}{r})+(\frac{1}{r}-1)N(2n+1)} \\
&&\times 2^{\tau }\left( M_{S}\left( \sum\limits_{I^{\prime }\in A_{\ell
},J^{\prime }\in B_{i}}|\phi _{{j^{\prime }},{k^{\prime }}}\ast
f(x_{I^{\prime }},y_{J^{\prime }})|^{r}\chi _{I^{\prime }}\chi _{J^{\prime
}}\right) (z^{\ast },u^{\ast })\right) ^{\frac{1}{r}} \\
&\leq &C_{1}(N,r,j,k,j^{\prime },k^{\prime })\left( M_{S}\left(
\sum\limits_{I^{\prime },J^{\prime }}|\phi _{{j^{\prime }},{k^{\prime }}%
}\ast f(x_{I^{\prime }},y_{J^{\prime }})|^{r}\chi _{I^{\prime }}\chi
_{J^{\prime }}\right) (z^{\ast },u^{\ast })\right) ^{\frac{1}{r}}
\end{eqnarray*}%
The last inequality follows from the assumption that $r>\max \{\frac{2n}{%
2n+K_{1}},\frac{1}{1+K_{2}}\}$ which can be achieved by choosing $%
K_{1},K_{2} $ large enough. The second inequality can be proved similarly.
\end{proof}

We now are ready to give the proof of the Plancherel-P\^{o}lya inequality.

\begin{proof}
\textbf{(of Theorem \ref{P-P}):} By Theorem \ref{discretethm}, $f\in 
\mathcal{M}_{flag}^{M+\delta }\left( \mathbb{H}^{n}\right) ^{\prime }$ can
be represented by 
\begin{equation*}
f(z,u)=\sum\limits_{j^{\prime }}\sum\limits_{k^{\prime
}}\sum\limits_{J^{\prime }}\sum\limits_{I^{\prime }}|J^{\prime }||I^{\prime
}|{\widetilde{\phi }}_{{j^{\prime }},{k^{\prime }}}\left( \left( z,u\right)
\circ \left( x_{I^{\prime }},y_{J^{\prime }}\right) ^{-1}\right) \left( \phi
_{{j^{\prime }},{k^{\prime }}}\ast f\right) (x_{I^{\prime }},y_{J^{\prime
}}).
\end{equation*}%
We write 
\begin{eqnarray*}
&&\left( \psi _{j,k}\ast f\right) (u,v) \\
&=&\sum\limits_{j^{\prime }}\sum\limits_{k^{\prime }}\sum\limits_{J^{\prime
}}\sum\limits_{I^{\prime }}|I^{\prime }||J^{\prime }|\left( \psi _{j,k}\ast {%
\widetilde{\phi }}_{{j^{\prime }},{k^{\prime }}}\left( \left( \cdot ,\cdot
\right) \circ \left( x_{I^{\prime }},y_{J^{\prime }}\right) ^{-1}\right)
\right) (z,u)\left( \phi _{{j^{\prime }},{k^{\prime }}}\ast f\right)
(x_{I^{\prime }},y_{J^{\prime }}).
\end{eqnarray*}

By the almost orthogonality estimates in Lemma \ref{almostortho}, and by
choosing $t=2^{-j}$, $s=2^{-k}$, $t^{\prime }=2^{-j^{\prime }}$, $s^{\prime
}=2^{-k^{\prime }}$, and for any given positive integers $%
L_{1},L_{2},K_{1},K_{2}$, we have if $j\wedge j^{\prime }\geq k\wedge
k^{\prime },$ 
\begin{eqnarray*}
&&\left\vert \left( \psi _{j,k}\ast {\widetilde{\phi }}_{{j^{\prime }},{%
k^{\prime }}}\left( \left( \cdot ,\cdot \right) \circ \left( x_{I^{\prime
}},y_{J^{\prime }}\right) ^{-1}\right) \right) (z,u)\right\vert \\
&\leq &{\frac{{2^{-|j-{j^{\prime }}|L_{1}-|k-{k^{\prime }}%
|L_{2}}2^{-(j\wedge j^{\prime })K_{1}-(k\wedge k^{\prime })K_{2}}|I^{\prime
}||J^{\prime }|}}{{\left( 2^{-j\wedge j^{\prime }}+|z-x_{I^{\prime
}}|\right) ^{2n+K_{1}}\left( 2^{-k\wedge k^{\prime }}+|u-y_{J^{\prime
}}|\right) ^{1+K_{2}}}}}|\phi _{{j^{\prime }},{k^{\prime }}}\ast
f(x_{I^{\prime }},y_{J^{\prime }})|,
\end{eqnarray*}%
and when $j\wedge j^{\prime }\leq k\wedge k^{\prime }$, we have 
\begin{eqnarray*}
&&\left\vert \left( \psi _{j,k}\ast {\widetilde{\phi }}_{{j^{\prime }},{%
k^{\prime }}}(\left( \cdot ,\cdot \right) \circ \left( x_{I^{\prime
}},y_{J^{\prime }}\right) ^{-1})\right) (z,u)\right\vert \\
&\leq &{\frac{{2^{-|j-{j^{\prime }}|L_{1}-|k-{k^{\prime }}%
|L_{2}}2^{-(j\wedge j^{\prime })K_{1}-(j\wedge j^{\prime })K_{2}}|I^{\prime
}||J^{\prime }|}}{{\left( 2^{-j\wedge j^{\prime }}+|z-x_{I^{\prime
}}|\right) ^{2n+K_{1}}\left( 2^{-j\wedge j^{\prime }}+|u-y_{J^{\prime
}}|\right) ^{1+K_{2}}}}}|\phi _{{j^{\prime }},{k^{\prime }}}\ast
f(x_{I^{\prime }},y_{J^{\prime }})|.
\end{eqnarray*}%
Using Lemma \ref{7}, for any $z,z^{\ast }\in I$, $x_{I^{\prime }}\in
I^{\prime }$, $u,u^{\ast }\in J$ and $y_{J^{\prime }}\in J^{\prime },$ we
have 
\begin{eqnarray*}
|\psi _{j,k}\ast f(z,u)| &\leq &C_{1}\sum_{j^{\prime },k^{\prime }:j\wedge
j^{\prime }\geq k\wedge k^{\prime }}2^{-|j-{j^{\prime }}|L_{1}}\cdot 2^{-|k-{%
k^{\prime }}|L_{2}} \\
&&\times \left\{ M_{S}\left[ \left( \sum\limits_{J^{\prime
}}\sum\limits_{I^{\prime }}|\phi _{{j^{\prime }},{k^{\prime }}}\ast
f(x_{I^{\prime }},y_{J^{\prime }})|\chi _{J^{\prime }}\chi _{I^{\prime
}}\right) ^{r}\right] \right\} ^{{\frac{{1}}{{r}}}}(z^{\ast },u^{\ast }) \\
&&+C_{2}\sum_{j^{\prime },k^{\prime }:j\wedge j^{\prime }\leq k\wedge
k^{\prime }}2^{-|j-{j^{\prime }}|L_{1}}\cdot 2^{-|k-{k^{\prime }}|L_{2}} \\
&&\times \left\{ M\left[ \left( \sum\limits_{J^{\prime
}}\sum\limits_{I^{\prime }}|\phi _{{j^{\prime }},{k^{\prime }}}\ast
f(x_{I^{\prime }},y_{J^{\prime }})|\chi _{J^{\prime }}\chi _{I^{\prime
}}\right) ^{r}\right] \right\} ^{{\frac{{1}}{{r}}}}(z^{\ast },u^{\ast }) \\
&\leq &C\sum_{j^{\prime },k^{\prime }}2^{-|j-{j^{\prime }}|L_{1}}\cdot
2^{-|k-{k^{\prime }}|L_{2}} \\
&&\times \left\{ M_{S}\left[ \left( \sum\limits_{J^{\prime
}}\sum\limits_{I^{\prime }}|\phi _{{j^{\prime }},{k^{\prime }}}\ast
f(x_{I^{\prime }},y_{J^{\prime }})|\chi _{J^{\prime }}\chi _{I^{\prime
}}\right) ^{r}\right] \right\} ^{{\frac{{1}}{{r}}}}(z^{\ast },u^{\ast })
\end{eqnarray*}%
where $M$ is the Hardy-Littlewood maximal function on $\mathbb{H}^{n}$, $%
M_{S}$ is the strong maximal function on $\mathbb{H}^{n}$, and $\max \{\frac{%
2n}{2n+K_{1}},\frac{1}{1+K_{2}}\}<r<p.$

Applying H\"{o}lder's inequality and summing over $j,k,I,J$ yields 
\begin{eqnarray*}
&&\left\{ \sum\limits_{j,k}\sum\limits_{I,J}\sup_{u\in I,v\in J}|\psi
_{j,k}\ast f(z,u)|^{2}\chi _{I}\chi _{J}\right\} ^{{\frac{{1}}{{2}}}} \\
&\leq &C\left\{ \sum\limits_{j^{\prime },k^{\prime }}\left\{ M_{S}\left[
\left( \sum\limits_{I^{\prime },J^{\prime }}|\phi _{{j^{\prime }},{k^{\prime
}}}\ast f(x_{I^{\prime }},y_{J^{\prime }})|\chi _{I^{\prime }}\chi
_{J^{\prime }}\right) ^{r}\right] (z^{\ast },u^{\ast })\right\} ^{{\frac{{2}%
}{{r}}}}\right\} ^{{\frac{{1}}{{2}}}}
\end{eqnarray*}%
Since $x_{I^{\prime }}$ and $y_{J^{\prime }}$ are arbitrary points in $%
I^{\prime }$ and $J^{\prime },$ respectively, we have 
\begin{eqnarray*}
&&\left\{ \sum\limits_{j,k}\sum\limits_{I,J}\sup_{u\in I,v\in J}|\psi
_{j,k}\ast f(z,u)|^{2}\chi _{I}\chi _{J}\right\} ^{{\frac{{1}}{{2}}}} \\
&\leq &C\left\{ \sum\limits_{j^{\prime },k^{\prime }}\left\{
M_{S}(\sum\limits_{I^{\prime },J^{\prime }}\inf_{u\in I^{\prime },v\in
J^{\prime }}|\phi _{{j^{\prime }},{k^{\prime }}}\ast f(z,u)|\chi _{I^{\prime
}}\chi _{J^{\prime }})^{r}\right\} ^{{\frac{{2}}{{r}}}}\right\} ^{{\frac{{1}%
}{{2}}}},
\end{eqnarray*}%
and hence, by the Fefferman-Stein vector-valued maximal function inequality
[FS] with $r<p,$ we get 
\begin{eqnarray*}
&&\left\Vert \left\{
\sum\limits_{j}\sum\limits_{k}\sum\limits_{J}\sum\limits_{I}\sup_{u\in
I,v\in J}|\psi _{j,k}\ast f(z,u)|^{2}\chi _{I}\chi _{J}\right\} ^{{\frac{{1}%
}{{2}}}}\right\Vert _{p} \\
&\leq &C\left\Vert \left\{ \sum\limits_{j^{\prime }}\sum\limits_{k^{\prime
}}\sum\limits_{J^{\prime }}\sum\limits_{I^{\prime }}\inf_{u\in I^{\prime
},v\in J^{\prime }}|\phi _{{j^{\prime }},{k^{\prime }}}\ast f(z,u)|^{2}\chi
_{I^{\prime }}\chi _{J^{\prime }}\right\} ^{{\frac{{1}}{{2}}}}\right\Vert
_{p}.
\end{eqnarray*}%
This completes the proof of Theorem \ref{P-P}.
\end{proof}

\section{Boundedness of flag singular integrals\label{s boundedness}}

As a consequence of Theorem \ref{P-P}, it is easy to see that the Hardy
space $H_{flag}^{p}$ is independent of the choice of the functions $\psi .$
Moreover, we have the following characterization of $H_{flag}^{p}$ using the
wavelet norm.

\begin{proposition}
\label{1}Let $0<p\leq 1$. Then we have 
\begin{equation*}
\Vert f\Vert _{H_{flag}^{p}}\approx \left\Vert \left\{
\sum\limits_{j}\sum\limits_{k}\sum\limits_{J}\sum\limits_{I}|\psi _{j,k}\ast
f(x_{I},y_{J})|^{2}\chi _{I}(z)\chi _{J}(u)\right\} ^{{\frac{{1}}{{2}}}%
}\right\Vert _{p}
\end{equation*}%
where $j,k,\psi ,\chi _{I},\chi _{J},x_{I},y_{J}$ are as in Theorem \ref{P-P}%
.
\end{proposition}

Before we give the proof of the boundedness of flag singular integrals on $%
H_{flag}^{p},$ we show several properties of $H_{flag}^{p}.$

\begin{proposition}
\label{molecules dense}$\mathcal{M}_{flag}^{M+\delta }(\mathbb{H}^{n})$ is
dense in $H_{flag}^{p}(\mathbb{H}^{n})$ for $M$ large enough.
\end{proposition}

\begin{proof}
Suppose $f\in H_{flag}^{p},$ and set $W=\{(j,k,I,J):|j|\leq L,|k|\leq
M,I\times J\subseteq B(0,r)\},$ where $I\times J$ is a dyadic rectangle in $%
\mathbb{H}^{n}$ with $\ell \left( I\right) =2^{-j-N}$ and $\ell \left(
J\right) =2^{-k-N}+2^{-j-N}$, and where $B(0,r)$ is the ball in $\mathbb{H}%
^{n}$ centered at the origin with radius $r$. It is easy to see that 
\begin{equation*}
\sum\limits_{(j,k,I,J)\in W}|I||J|{\widetilde{\psi }}_{j,k}(\left(
z,u\right) \circ \left( x_{I},y_{J}\right) ^{-1})\psi _{j,k}\ast
f(x_{I},y_{J})
\end{equation*}%
is a test function in $\mathcal{M}_{flag}^{M+\delta }(\mathbb{H}^{n})$ for
any fixed $L,M,r.$ To obtain the proposition, it suffices to prove 
\begin{equation*}
\sum\limits_{(j,k,I,J)\in W^{c}}|I||J|{\widetilde{\psi }}_{j,k}(\left(
z,u\right) \circ \left( x_{I},y_{J}\right) ^{-1})\psi _{j,k}\ast
f(x_{I},y_{J})
\end{equation*}%
tends to zero in the $H_{flag}^{p}$ norm as $L,M,r$ tend to infinity. This
follows from an argument similar to that in the proof of Theorem \ref{P-P}.
In fact, repeating the argument in Theorem \ref{P-P}, yields 
\begin{equation*}
\left\Vert \sum\limits_{(j,k,I,J)\in W^{c}}|I||J|{\widetilde{\psi }}%
_{j,k}(\left( z,u\right) \circ \left( x_{I},y_{J}\right) ^{-1})\psi
_{j,k}\ast f(x_{I},y_{J})\right\Vert _{H_{flag}^{p}}
\end{equation*}%
\begin{equation*}
\leq C\left\Vert \left\{ \sum\limits_{(j,k,I,J)\in W^{c}}|\psi _{j,k}\ast
f(x_{I},y_{J})|^{2}\chi _{I}\chi _{J}\right\} ^{{\frac{{1}}{{2}}}%
}\right\Vert _{p},
\end{equation*}%
where the last term tends to zero as $L,M,r$ tend to infinity whenever $f\in
H_{flag}^{p}$.
\end{proof}

As a consequence of Proposition \ref{molecules dense}, $L^{2}(\mathbb{H}%
^{n})\cap H_{flag}^{p}\left( \mathbb{H}^{n}\right) $ is dense in $%
H_{flag}^{p}(\mathbb{H}^{n}).$ Furthermore, we have this theorem.

\begin{theorem}
\label{L<H}If $f\in L^{2}(\mathbb{H}^{n})\cap H_{flag}^{p}\left( \mathbb{H}%
^{n}\right) ,0<p\leq 1,$ then $f\in L^{p}(\mathbb{H}^{n})$ and there is a
constant $C_{p}>0$ which is independent of the $L^{2}$ norm of $f$ such that 
\begin{equation*}
\Vert f\Vert _{p}\leq C\Vert f\Vert _{H_{flag}^{p}}.
\end{equation*}
\end{theorem}

To prove Theorem \ref{L<H}, we need a discrete Calder\'{o}n reproducing
formula on $L^{2}(\mathbb{H}^{n})$. To be more precise, take $\phi ^{(1)}\in
C_{0}^{\infty }(\mathbb{H}^{n})$ as in Theorem \ref{GM} with 
\begin{equation*}
\int_{\mathbb{H}^{n}}\phi ^{(1)}(z,u)z^{\alpha }u^{\beta }dzdu=0,\,\,\text{%
for all}\,\alpha ,\beta \,\,\text{satisfying}\,0\leq |\alpha |\leq
M_{0},\,0\leq |\beta |\leq M_{0},
\end{equation*}%
and take $\phi ^{(2)}\in C_{0}^{\infty }(\mathbb{R})$ with 
\begin{equation*}
\int_{\mathbb{R}}\phi ^{(2)}(v)z^{\gamma }dv=0\,\,\text{for all}\,\,0\leq
|\gamma |\leq M_{0},
\end{equation*}%
and $\sum_{k}|\widehat{\phi ^{(2)}}(2^{-k}\xi _{2})^{2}=1$ for all $\xi
_{2}\in \mathbb{R}\backslash \{0\}$.

Furthermore, we may assume that $\phi ^{(1)}$ and $\phi ^{(2)}$ are radial
functions and supported in the unit balls of $\mathbb{H}^{n}$ and $\mathbb{R}
$ respectively. Set 
\begin{equation*}
\phi _{jk}(z,u)=\int_{\mathbb{R}}\phi _{j}^{(1)}(z,u-v)\phi _{k}^{(2)}(v)dv.
\end{equation*}%
By Theorem \ref{GM} we have the following continuous version of the Calder%
\'{o}n reproducing formula on $L^{2}$: for $f\in L^{2}(\mathbb{H}^{n})$, 
\begin{equation*}
f(z,u)=\sum\limits_{j}\sum\limits_{k}\phi _{jk}\ast \phi _{jk}\ast f(z,u).
\end{equation*}%
For our purposes, we need a discrete version of the above reproducing
formula.

\begin{theorem}
\label{19}There exist functions $\widetilde{\phi }_{jk}$ and an operator $%
T_{N}^{-1}$ such that 
\begin{equation*}
f(x,y)=\sum\limits_{j}\sum\limits_{k}\sum\limits_{J}\sum\limits_{I}|I||J|%
\widetilde{\phi }_{j,k}(\left( z,u\right) \circ \left( x_{I},y_{J}\right)
^{-1})\phi _{j,k}\ast \left( T_{N}^{-1}(f)\right) (x_{I},y_{J})
\end{equation*}%
where functions $\widetilde{\phi }_{jk}(\left( z,u\right) \circ \left(
x_{I},y_{J}\right) ^{-1})$ satisfy the conditions in Theorem \ref%
{discretethm} with $\alpha _{1},\beta _{1},\gamma _{1},N,M$ depending on $%
M_{0}$. Moreover, $T_{N}^{-1}$ is bounded on both $L^{2}(\mathbb{H}^{n})$
and $H_{flag}^{p}\left( \mathbb{H}^{n}\right) ,$ and the series converges in 
$L^{2}(\mathbb{H}^{n})$.
\end{theorem}

\begin{remark}
The difference between Theorem \ref{19} and Theorem \ref{discretethm} is
that the ${\widetilde{\phi }}_{jk}$ in Theorem \ref{19} have compact
support. The price we pay here is that ${\widetilde{\phi }}_{jk}$ only
satisfies moment conditions of finite order, unlike in Theorem \ref%
{discretethm} where moment conditions of infinite order are satisfied.
Moreover, the formula in Theorem \ref{19} only holds on $L^{2}(\mathbb{H}%
^{n})$ while the formula in Theorem \ref{discretethm} holds in both the test
function space $\mathcal{M}_{flag}^{M+\delta }$ and its dual space $(%
\mathcal{M}_{flag}^{M+\delta })^{\prime }$.
\end{remark}

\textbf{Proof of Theorem }\ref{19}\textbf{:} Following the proof of Theorem %
\ref{discretethm}, we have 
\begin{equation*}
f(z,u)=\sum\limits_{j}\sum\limits_{k}\sum\limits_{J}\sum\limits_{I}\left[
\int\limits_{J}\int\limits_{I}\phi _{j,k}(\left( z,u\right) \circ \left(
u,v\right) ^{-1})dudv\right] \left( \phi _{j,k}\ast f\right) (x_{I},y_{J})+%
\mathcal{R}f(z,u).
\end{equation*}%
where $I,J,j,k$ and $\mathcal{R}$ are as in Theorem \ref{discretethm}.

We need the following lemma to handle the remainder term $\mathcal{R}$.

\begin{lemma}
\label{small norm}Let $0<p\leq 1$. Then the operator $\mathcal{R}$ is
bounded on $L^{2}(\mathbb{H}^{n})$ and $H_{flag}^{p}(\mathbb{H}^{n})$
whenever $M_{0}$ is chosen to be a large positive integer. Moreover, there
exists a constant $C>0$ such that 
\begin{equation*}
||\mathcal{R}f||_{2}\leq C2^{-N}||f||_{2}
\end{equation*}%
and 
\begin{equation*}
||\mathcal{R}f||_{H_{flag}^{p}\left( \mathbb{H}^{n}\right) }\leq
C2^{-N}||f||_{H_{flag}^{p}\left( \mathbb{H}^{n}\right) }.
\end{equation*}
\end{lemma}

\begin{proof}
Following the proofs of Theorems \ref{discretethm} and \ref{P-P} and using
the wavelet Calder\'{o}n reproducing formula for $f\in L^{2}(\mathbb{H}%
^{n}), $ we have 
\begin{eqnarray*}
&&||g_{flag}(\mathcal{R}f)||_{p} \\
&\leq &\left\Vert \left\{
\sum\limits_{j}\sum\limits_{k}\sum\limits_{J}\sum\limits_{I}|\left( \psi
_{j,k}\ast \mathcal{R}f\right) |^{2}\chi _{I}\chi _{J}\right\} ^{{\frac{{1}}{%
{2}}}}\right\Vert _{p} \\
&=&\left\Vert \left\{ \sum\limits_{j,k,J,I}\sum\limits_{{j^{\prime }}%
,k^{\prime },J^{\prime },I^{\prime }}|J^{\prime }||I^{\prime }||\left( \psi
_{j,k}\ast \mathcal{R}\widetilde{\psi _{j^{\prime },k^{\prime }}}\left(
\left( \cdot ,\cdot \right) \circ \left( x_{I^{\prime }},y_{J^{\prime
}}\right) ^{-1}\right) \cdot \psi _{j^{\prime }k^{\prime }}\ast
f(x_{I^{\prime }},y_{J^{\prime }})\right) |^{2}\chi _{I}\chi _{J}\right\} ^{{%
\frac{{1}}{{2}}}}\right\Vert _{p}
\end{eqnarray*}%
where $j,k,\psi ,\chi _{I},\chi _{J},x_{I},y_{J}$ are as in Theorem \ref{P-P}%
.

\medskip

\textbf{Claim}: We have%
\begin{eqnarray*}
&&\left\vert \left( \psi _{j,k}\ast \mathcal{R}\left( {\widetilde{\psi }}_{{%
j^{\prime }},{k^{\prime }}}\left( \left( \cdot ,\cdot \right) \circ \left(
x_{I^{\prime }},y_{J^{\prime }}\right) ^{-1}\right) \right) \right)
(z,u)\right\vert \\
&\leq &C2^{-N}2^{-|j-{j^{\prime }}|K}2^{-|k-{k^{\prime }}|K} \\
&&\times \int\limits_{\mathbb{R}}{\frac{{2^{-(j\wedge {j^{\prime }})K}}}{{%
(2^{-(j\wedge {j^{\prime }})}+|z-x_{I^{\prime }}|+|v-u-y_{J^{\prime
}})^{2n+1+K}}}}{\frac{{2^{-(k\wedge {k^{\prime }})K}}}{{(2^{-(k\wedge {%
k^{\prime }})}+|v|)^{1+K}}}}dv,
\end{eqnarray*}%
where we have chosen for simplicity 
\begin{equation*}
L_{1}=L_{2}=K_{1}=K_{2}=K<M_{0},max({\frac{2{n}}{2{n+K}}},{\frac{{1}}{{1+K}}}%
)<p,
\end{equation*}%
and $M_{0}$ is chosen to be a larger integer later.

Assuming the claim for the moment, we can repeat an argument used in Lemma %
\ref{7}, and then use Theorem \ref{P-P}, to obtain 
\begin{eqnarray*}
&&\left\Vert g_{flag}\left( \mathcal{R}f\right) \right\Vert _{p} \\
&\leq &C2^{-N}\left\Vert \left\{ \sum\limits_{j^{\prime
}}\sum\limits_{k^{\prime }}\left[ M_{S}(\sum\limits_{J^{\prime
}}\sum\limits_{I^{\prime }}|\psi _{{j^{\prime }},{k^{\prime }}}\ast
f(x_{I^{\prime }},y_{J^{\prime }})|\chi _{J^{\prime }}\chi _{I^{\prime
}})^{r}\right] ^{{\frac{{2}}{{r}}}}\right\} ^{{\frac{{1}}{{2}}}}\right\Vert
_{p} \\
&\leq &C2^{-N}\left\Vert \left\{ \sum\limits_{j^{\prime
}}\sum\limits_{k^{\prime }}\sum\limits_{J^{\prime }}\sum\limits_{I^{\prime
}}|\psi _{{j^{\prime }},{k^{\prime }}}\ast f(x_{I^{\prime }},y_{J^{\prime
}})|^{2}\chi _{I^{\prime }}\chi _{J^{\prime }}\right\} ^{{\frac{{1}}{{2}}}%
}\right\Vert _{p}\leq C2^{-N}\Vert f\Vert _{H_{flag}^{p}\left( \mathbb{H}%
^{n}\right) }.
\end{eqnarray*}%
It is clear that the above estimates continue to hold when $p$ is replaced
by $2$. This completes the proof of Lemma \ref{small norm} modula the claim.

In order to prove the Claim made above, we note that Theorem \ref{key
boundedness} shows that the functions 
\begin{equation*}
\mathcal{R}\left( {\widetilde{\psi }}_{{j^{\prime }},{k^{\prime }}}(\left(
\cdot ,\cdot \right) \circ \left( x_{I^{\prime }},y_{J^{\prime }}\right)
^{-1})\right) (z,u)
\end{equation*}%
are flag molecules. Then the claim follows from Lemma \ref{7}, and this
completes the proof of Lemma \ref{small norm}.
\end{proof}

We now return to the proof of Theorem \ref{19}. Let $(T_{N})^{-1}=%
\sum_{i=1}^{\infty }\mathcal{R}^{i},$ where 
\begin{equation*}
T_{N}f=\sum\limits_{j}\sum\limits_{k}\sum\limits_{J}\sum\limits_{I}\left( 
\frac{1}{|I||J|}\int\limits_{J}\int\limits_{I}\phi _{j,k}\left( \left(
z,u\right) \circ \left( w,v\right) ^{-1}\right) dwdv\right) I||J|\left( \phi
_{j,k}\ast f\right) (x_{I},y_{J}).
\end{equation*}%
Lemma \ref{small norm} shows that if $N$ is large enough, then both of $%
T_{N} $ and $(T_{N})^{-1}$ are bounded on $L^{2}(\mathbb{H}^{n})\cap
H_{flag}^{p}\left( \mathbb{H}^{n}\right) $. Hence, we can get the following
reproducing formula 
\begin{equation*}
f(x,y)=\sum\limits_{j}\sum\limits_{k}\sum\limits_{J}\sum\limits_{I}|I||J|{%
\widetilde{\phi }}_{j,k}(\left( z,u\right) \circ \left( x_{I},y_{J}\right)
^{-1})\phi _{j,k}\ast \left( T_{N}^{-1}f\right) (x_{I},y_{J})
\end{equation*}%
where the functions ${\widetilde{\phi }}_{jk}(\left( z,u\right) \circ \left(
x_{I},y_{J}\right) ^{-1})$ are flag molecules, and the series converges in $%
L^{2}(\mathbb{H}^{n})$. This completes the proof of Theorem \ref{19}.

\bigskip

As a consequence of Theorem \ref{19}, we obtain the following corollary.

\begin{corollary}
\label{L2 dense}If $f\in L^{2}(\mathbb{H}^{n})\cap H_{flag}^{p}\left( 
\mathbb{H}^{n}\right) $ and $0<p\leq 1$, then 
\begin{equation*}
\Vert f\Vert _{H_{flag}^{p}}\approx \left\Vert \left\{
\sum\limits_{j}\sum\limits_{k}\sum\limits_{J}\sum\limits_{I}\left\vert \phi
_{jk}\ast \left( T_{N}^{-1}f\right) (x_{I},y_{J})\right\vert ^{2}\chi
_{I}(z)\chi _{J}(u)\right\} ^{\frac{1}{2}}\right\Vert _{p}
\end{equation*}%
where the constants are independent of the $L^{2}$ norm of $f.$
\end{corollary}

\begin{proof}
\textbf{(of Corollary \ref{L2 dense})}: Note that if $f\in L^{2}(\mathbb{H}%
^{n})$, we can apply the Caldero\'{n} reproducing formula in Theorem \ref{19}
and then repeat the proof of Theorem \ref{P-P}. We leave the details to the
reader, and this completes our proof of Corollary \ref{L2 dense}.
\end{proof}

\bigskip

We now start the proof of Theorem\textbf{\ }\ref{L<H}\textbf{. }We define a
square function by 
\begin{equation*}
\widetilde{g}(f)(z,u)=\left\{
\sum\limits_{j}\sum\limits_{k}\sum\limits_{J}\sum\limits_{I}|\phi _{j,k}\ast
\left( T_{N}^{-1}(f)\right) (x_{I},y_{J})|^{2}\chi _{I}(z)\chi
_{J}(u)\right\} ^{{\frac{{1}}{{2}}}},
\end{equation*}%
where $\phi _{jk}$ are as in Theorem \ref{19}. By Corollary \ref{L2 dense},
for $f\in L^{2}(\mathbb{H}^{n})\cap H_{flag}^{p}\left( \mathbb{H}^{n}\right) 
$ we have, 
\begin{equation*}
||\widetilde{g}(f)||_{L^{p}(\mathbb{H}^{n})}\leq C||f||_{H_{flag}^{p}\left( 
\mathbb{H}^{n}\right) }.
\end{equation*}%
To complete the proof of Theorem \ref{L<H}, let $f\in L^{2}(\mathbb{H}%
^{n})\cap H_{flag}^{p}\left( \mathbb{H}^{n}\right) $. Set 
\begin{equation*}
\Omega _{i}=\{(z,u)\in \mathbb{H}^{n}:\widetilde{g}(f)(z,u)>2^{i}\}.
\end{equation*}%
Let 
\begin{equation*}
\mathcal{B}_{i}=\{(j,k,I,J):|(I\times J)\cap \Omega _{i}|>{\frac{{1}}{{2}}}%
|I\times J|,|(I\times J)\cap \Omega _{i+1}|\leq {\frac{{1}}{{2}}}|I\times
J|\},
\end{equation*}%
where $I\times J$ are rectangles in $\mathbb{H}^{n}$ with side lengths $\ell
\left( I\right) =2^{-j-N}$ and $\ell \left( J\right) =2^{-k-N}+2^{-j-N}$.
Since $f\in L^{2}(\mathbb{H}^{n})$, the discrete Calder\'{o}n reproducing
formula in Theorem \ref{19} gives, 
\begin{eqnarray*}
f(z,u) &=&\sum\limits_{j}\sum\limits_{k}\sum\limits_{J}\sum\limits_{I}{%
\widetilde{\phi }}_{j,k}(\left( z,u\right) \circ \left( x_{I},y_{J}\right)
^{-1})|I||J|\phi _{j,k}\ast \left( T_{N}^{-1}(f)\right) (x_{I},y_{J}) \\
&=&\sum\limits_{i}\sum\limits_{(j,k,I,J)\in \mathcal{B}_{i}}|I||J|{%
\widetilde{\phi }}_{j,k}(\left( z,u\right) \circ \left( x_{I},y_{J}\right)
^{-1})\phi _{j,k}\ast \left( T_{N}^{-1}(f)\right) (x_{I},y_{J}),
\end{eqnarray*}%
where the series converges rapidly in $L^{2}$ norm, and hence almost
everywhere.

\bigskip

\textbf{Claim:} We have%
\begin{equation*}
\left\Vert \sum\limits_{(j,k,I,J)\in \mathcal{B}_{i}}|I||J|{\widetilde{\phi }%
}_{j,k}(\left( z,u\right) \circ \left( x_{I},y_{J}\right) ^{-1})\phi
_{j,k}\ast \left( T_{N}^{-1}(f)\right) (x_{I},y_{J})\right\Vert _{p}^{p}\leq
C2^{ip}|\Omega _{i}|,
\end{equation*}%
which together with the fact $0<p\leq 1$ yields 
\begin{eqnarray*}
||f||_{p}^{p} &\leq &\sum_{i}\left\Vert \sum\limits_{(j,k,I,J)\in \mathcal{B}%
_{i}}|I||J|{\widetilde{\phi }}_{j,k}(\left( z,u\right) \circ \left(
x_{I},y_{J}\right) ^{-1})\phi _{j,k}\ast \left( T_{N}^{-1}(f)\right)
(x_{I},y_{J})\right\Vert _{p}^{p} \\
&\leq &C\sum_{i}2^{ip}|\Omega _{i}|\leq C\left\Vert \widetilde{g}%
(f)\right\Vert _{p}^{p}\leq C||f||_{H_{flag}^{p}}^{p}.
\end{eqnarray*}

To obtain the claim, note that $\phi ^{(1)}$ and $\psi ^{(2)}$ are radial
functions supported in unit balls in $\mathbb{H}^{n}$ and $\mathbb{R}$
respectively. Hence, if $(j,k,I,J)\in \mathcal{B}_{i}$, then $\phi
_{j,k}(\left( z,u\right) \circ \left( x_{I},y_{J}\right) ^{-1})$ is
supported in 
\begin{equation*}
{\widetilde{\Omega _{i}}}=\{(z,u):M_{S}(\chi _{\Omega _{i}})(z,u)>{\frac{{1}%
}{{100}}}\}.
\end{equation*}%
Thus, by H\"{o}lder's inequality, 
\begin{eqnarray*}
&&\left\Vert \sum\limits_{(j,k,I,J)\in \mathcal{B}_{i}}|J||I|{\widetilde{%
\phi }}_{j,k}(\left( z,u\right) \circ \left( x_{I},y_{J}\right) ^{-1})\phi
_{j,k}\ast \left( T_{N}^{-1}(f)\right) (x_{I},y_{J})\right\Vert _{p}^{p} \\
&\leq &|{\widetilde{\Omega _{i}}}|^{1-{\frac{{p}}{{2}}}}\left\Vert
\sum\limits_{(j,k,I,J)\in \mathcal{B}_{i}}|J||I|{\widetilde{\phi }}%
_{j,k}(\left( z,u\right) \circ \left( x_{I},y_{J}\right) ^{-1})\phi
_{j,k}\ast \left( T_{N}^{-1}(f)\right) (x_{I},y_{J})\right\Vert _{2}^{p}.
\end{eqnarray*}%
By duality, for all $g\in L^{2}$ with $\Vert g\Vert _{2}\leq 1,$ 
\begin{eqnarray*}
&&\left\vert \left\langle \sum\limits_{(j,k,I,J)\in \mathcal{B}_{i}}|J||I|{%
\widetilde{\phi }}_{j,k}(\left( z,u\right) \circ \left( x_{I},y_{J}\right)
^{-1})\phi _{j,k}\ast \left( T_{N}^{-1}(f)\right)
(x_{I},y_{J}),g\right\rangle \right\vert \\
&=&\left\vert \sum\limits_{(j,k,I,J)\in \mathcal{B}_{i}}|J||I|{\widetilde{%
\phi }}_{j,k}\ast g(x_{I},y_{J})\phi _{j,k}\ast \left( T_{N}^{-1}(f)\right)
(x_{I},y_{J})\right\vert \\
&\leq &C\left( \sum\limits_{(j,k,I,J)\in \mathcal{B}_{i}}|I||J||\phi
_{j,k}\ast \left( T_{N}^{-1}(f)\right) (x_{I},y_{J})|^{2}\right) ^{\frac{1}{2%
}}\cdot \left( \sum\limits_{(j,k,I,J)\in \mathcal{B}_{i}}|I||J||{\widetilde{%
\phi }}_{j,k}\ast g(x_{I},y_{J})|^{2}\right) ^{\frac{1}{2}}.
\end{eqnarray*}%
Since 
\begin{eqnarray*}
&&\left( \sum\limits_{(j,k,I,J)\in \mathcal{B}_{i}}|I||J||{\widetilde{\phi }}%
_{j,k}\ast g(x_{I},y_{J})|^{2}\right) ^{\frac{1}{2}} \\
&\leq &\left( \sum\limits_{(j,k,I,J)\in \mathcal{B}_{i}}|I||J|\left(
M_{S}\left( {\widetilde{\phi }}_{j,k}\ast g\right) (z,u)\chi _{I}(z)\chi
_{J}(u)\right) ^{2}\right) ^{\frac{1}{2}} \\
&\leq &C\left( \sum_{j,k}\int_{\mathbb{C}^{n}}\int_{\mathbb{R}}\left(
M_{S}\left( {\widetilde{\phi }}_{j,k}\ast g\right) ^{2}(z,u)dzdu\right)
\right) ^{\frac{1}{2}}\leq C||g||_{2},
\end{eqnarray*}%
the claim now follows from the fact that $|{\widetilde{\Omega _{i}}}|\leq C|{%
\Omega _{i}}|$ and the following estimate: 
\begin{eqnarray*}
C2^{2i}|\Omega _{i}| &\geq &\int\limits_{{\widetilde{\Omega _{i}}}\backslash 
{\Omega _{i+1}}}\widetilde{g}^{2}(f)(z,u)dzdu \\
&\geq &\sum\limits_{(j,k,I,J)\in \mathcal{B}_{i}}|\phi _{j,k}\ast \left(
T_{N}^{-1}(f)\right) (x_{I},y_{J})|^{2}|(I\times J)\cap {{\widetilde{\Omega
_{i}}}\backslash {\Omega _{i+1}}}| \\
&\geq &{\frac{{1}}{{2}}}\sum\limits_{(j,k,I,J)\in \mathcal{B}%
_{i}}|I||J||\phi _{j,k}\ast \left( T_{N}^{-1}(f)\right) (x_{I},y_{J})|^{2},
\end{eqnarray*}%
where the fact that $|(I\times J)\cap {{\widetilde{\Omega _{i}}}\backslash {%
\Omega _{i+1}}}|>{\frac{{1}}{{2}}}|I\times J|$ when $(j,k,I,J)\in \mathcal{B}%
{_{i}}$ is used in the last inequality. This finishes the proof of Theorem %
\ref{L<H}.

As a consequence of Theorem \ref{L<H}, we have the following corollary.

\begin{corollary}
\label{4}$H_{flag}^{1}(\mathbb{H}^{n})$ is a subspace of $L^{1}(\mathbb{H}%
^{n})$.
\end{corollary}

\begin{proof}
Given $f\in H_{flag}^{1}(\mathbb{H}^{n}),$ by Proposition \ref{molecules
dense}, there is a sequence $\{f_{n}\}$ such that $f_{n}\in {L^{2}(\mathbb{H}%
^{n})}\cap {H_{flag}^{1}(\mathbb{H}^{n})}$ and $f_{n}$ converges to $f$ in
the norm of $H_{flag}^{1}(\mathbb{H}^{n}).$ By Theorem \ref{L<H}, $f_{n}$
converges to $g$ in $L^{1}(\mathbb{H}^{n})$ for some $g\in L^{1}(\mathbb{H}%
^{n}).$ Therefore, $f=g$ in $(\mathcal{M}_{flag}^{M+\delta })^{\prime }$.
\end{proof}

We now turn to the proof of Theorem \ref{flagintbounded}.

\begin{proof}
\textbf{(of Theorem \ref{flagintbounded}):} We assume that $K$ is the kernel
of $T$. Applying the discrete Calder\'{o}n reproducing formula in Theorem %
\ref{19} implies that for $f\in L^{2}(\mathbb{H}^{n})\cap H_{flag}^{p}\left( 
\mathbb{H}^{n}\right) $, 
\begin{eqnarray*}
&&\left\Vert \left\{ \sum\limits_{j,k}\sum\limits_{I,J}|\phi _{j,k}\ast
K\ast f(z,u)|^{2}\chi _{I}(z)\chi _{J}(u)\right\} ^{{\frac{{1}}{{2}}}%
}\right\Vert _{p} \\
&=&\left\Vert \left\{
\sum\limits_{j,k}\sum\limits_{I,J}|\sum\limits_{j^{\prime },k^{\prime
}}\sum\limits_{I^{\prime },J^{\prime }}|J^{\prime }||I^{\prime }|\phi
_{j,k}\ast K\ast {\widetilde{\phi }}_{{j^{\prime }},{k^{\prime }}}(\left(
\cdot ,\cdot \right) \circ \left( x_{I},y_{J}\right) ^{-1})(z,u)\right.
\right. \\
&&\left. \left. \times \phi _{{j^{\prime }},{k^{\prime }}}\ast \left(
T_{N}^{-1}(f)\right) (x_{I^{\prime }},y_{J^{\prime }})|^{2}\chi _{I}(z)\chi
_{J}(u)\right\} ^{\frac{1}{2}}\right\Vert _{p},
\end{eqnarray*}%
where the discrete Calder\'{o}n reproducing formula in $L^{2}(\mathbb{H}%
^{n}) $ is used.

Note that $\phi _{jk}$ are dilations of bump functions, and by estimates
similar to the those in (\ref{product est}), one can easily check that 
\begin{equation*}
|\phi _{j,k}\ast K\ast {\widetilde{\phi }}_{j^{\prime },k^{\prime }}(\left(
\cdot ,\cdot \right) \circ \left( x_{I},y_{J}\right) ^{-1})(z,u)|\leq
C2^{-|j-{j^{\prime }}|K}2^{-|k-{k^{\prime }}|K}
\end{equation*}%
\begin{equation*}
\int\limits_{R^{m}}{\frac{{2^{-(j\wedge {j^{\prime }})K}}}{{(2^{-(j\wedge {%
j^{\prime }})}+|z-x_{I^{\prime }}|+|v-u-y_{J^{\prime }}|)^{2n+1+K}}}}\cdot {%
\frac{{2^{-(k\wedge {k^{\prime }})K}}}{{(2^{-(k\wedge {k^{\prime }}%
)}+|v|)^{1+K}}}}dv,
\end{equation*}%
where $K$ depends on $M_{0}$ given in Theorem \ref{flagintbounded}, and $%
M_{0}$ is chosen large enough.

Repeating an argument similar to that in the proof of Theorem \ref{P-P},
together with Corollary \ref{L2 dense}, we obtain 
\begin{eqnarray*}
&&\Vert Tf\Vert _{H_{flag}^{p}} \\
&\leq &C\left\Vert \left\{ \sum\limits_{j^{\prime }}\sum\limits_{k^{\prime
}}\left( M_{S}\left[ \left( \sum\limits_{J^{\prime }}\sum\limits_{I^{\prime
}}|\phi _{{j^{\prime }},{k^{\prime }}}\ast \left( T_{N}^{-1}(f)\right)
(x_{I^{\prime }},y_{J^{\prime }})|\chi _{J^{\prime }}\chi _{I^{\prime
}}\right) ^{r}\right] \right) ^{{\frac{{2}}{{r}}}}(z,u)\right\} ^{{\frac{{1}%
}{{2}}}}\right\Vert _{p}
\end{eqnarray*}%
\begin{equation*}
\leq C\left\Vert \left\{ \sum\limits_{j^{\prime }}\sum\limits_{k^{\prime
}}\sum\limits_{J^{\prime }}\sum\limits_{I^{\prime }}|\phi _{{j^{\prime }},{%
k^{\prime }}}\ast \left( T_{N}^{-1}(f)\right) (x_{I^{\prime }},y_{J^{\prime
}})|^{2}\chi _{J^{\prime }}(u)\chi _{I^{\prime }}(z)\right\} ^{{\frac{{1}}{{2%
}}}}\right\Vert _{p}\leq C\Vert f\Vert _{H_{flag}^{p}},
\end{equation*}%
where the last inequality follows from Corollary \ref{L2 dense}.

Since $L^{2}(\mathbb{H}^{n})\cap H_{flag}^{p}\left( \mathbb{H}^{n}\right) $
is dense in $H_{flag}^{p}\left( \mathbb{H}^{n}\right) ,$ $T$ can be extended
to a bounded operator on $H_{flag}^{p}\left( \mathbb{H}^{n}\right) $, and
this ends the proof of Theorem \ref{flagintbounded}.
\end{proof}

We now immediately obtain the proof of Theorem \ref{flagintHtoL}.

\begin{proof}
\textbf{(Theorem \ref{flagintHtoL}):} We note that $H_{flag}^{p}\cap L^{2}$
is dense in $H_{flag}^{p}$, so we only have to obtain the required
inequality for $f\in H_{flag}^{p}\cap L^{2}$. Thus Theorem \ref{flagintHtoL}
follows immediately from Theorems \ref{flagintbounded} and \ref{L<H}.
\end{proof}

\section{Duality of Hardy spaces $H_{flag}^{p}$\label{s duality}}

A. Chang and R. Fefferman established that the dual space of $H^{1}(\mathbb{R%
}_{+}^{2}\times \mathbb{R}_{+}^{2})$ is $BMO(\mathbb{R}_{+}^{2}\times 
\mathbb{R}_{+}^{2})$ (\cite{CF1}) by using the bi-Hilbert transform, and
consequently, their method is not directly applicable to the implicit
two-parameter structure associated to flag singular integrals. In order to
deal with the duality theory of $H_{flag}^{p}\left( \mathbb{H}^{n}\right) $
for all $0<p\leq 1$, we proceed differently, and first prove Theorem \ref%
{P-PCar}, the Plancherel-P\^{o}lya inequalities for the Carleson space $%
CMO_{flag}^{p}$. This theorem implies that the function space $%
CMO_{flag}^{p} $ is well defined.

\textbf{Proof of Theorem \ref{P-PCar}:} The idea of the proof of this
theorem is, as in the proof of Theorem \ref{P-P}, to use the wavelet Calder%
\'{o}n reproducing formula and almost orthogonality estimates.

For convenience, we prove Theorem \ref{P-PCar} for the smallest Heisenberg
group $\mathbb{H}^{1}=\mathbb{C}\times \mathbb{R}$. However, it will be
clear from the proof that its extension to general $\mathbb{H}^{n}$ is
straightforward. Moreover, to simplify notation, we denote $f_{j,k}=f_{R}$
when $R=I\times J$ is a dyadic rectangle contained in $\mathbb{H}^{1}$ and $%
\ell \left( I\right) =2^{-j-N}$, $\ell \left( J\right) =2^{-k-N}+2^{-j-N}$
are dyadic cubes and intervals respectively. Thus $\left\vert I\right\vert
=2^{-2\left( j+N\right) }$ and $\left\vert J\right\vert =2^{-k-N}+2^{-j-N}$
in this case. Here $N$ is the same as in Theorem \ref{discretethm}. We also
denote by $dist(I,I^{\prime })$ the distance between intervals $I$ and $%
I^{\prime }$, 
\begin{equation*}
S_{R}=\sup\limits_{u\in I,v\in J}|\psi _{R}\ast
f(z,u)|^{2},\,\,\,T_{R}=\inf\limits_{u\in I,v\in J}|\phi _{R}\ast
f(z,u)|^{2}.
\end{equation*}

With this notation, we can rewrite the wavelet Calder\'{o}n reproducing
formula in Theorem \ref{discretethm} as 
\begin{equation*}
f(z,u)=\sum\limits_{R=I\times J}|I||J|{\widetilde{\phi }}_{R}(z,u)\phi
_{R}\ast f(x_{I},y_{J}),
\end{equation*}%
where the sum runs over all rectangles $R=I\times J$. Let 
\begin{equation*}
R^{\prime }=I^{\prime }\times J^{\prime },\ \ \ \left\vert I^{\prime
}\right\vert =2^{-2\left( j^{\prime }+N\right) },\ \ \ \left\vert J^{\prime
}\right\vert =2^{-j^{\prime }-N}+2^{-k^{\prime }-N},\ \ \ j^{\prime
}>k^{\prime }.
\end{equation*}%
Applying the above wavelet Calder\'{o}n reproducing formula, and the
orthogonality estimates in Subsection \ref{5.3}, yields for all $(z,u)\in R,$
\begin{equation*}
|\psi _{R}\ast f(z,u)|^{2}\leq C\sum\limits_{R^{\prime }=I^{\prime }\times
J^{\prime },j^{\prime }>k^{\prime }}({\frac{{|I|}}{{|I^{\prime }|}}}\wedge {%
\frac{{|I^{\prime }|}}{{|I|}}})^{L}({\frac{{|J|}}{{|J^{\prime }|}}}\wedge {%
\frac{{|J^{\prime }|}}{{|J|}}})^{L}\times
\end{equation*}%
\begin{equation*}
{\frac{{|I^{\prime }|}^{K}}{{({|I^{\prime }|}+|z-x_{I^{\prime }}|)^{(2+K)}}}}%
{\frac{{|J^{\prime }|}^{K}}{{({\ |J^{\prime }|}+|u-y_{J^{\prime }}|)^{(1+K)}}%
}}|I^{\prime }||J^{\prime }||\phi _{R^{\prime }}\ast f(x_{I^{\prime
}},y_{J^{\prime }})|^{2}
\end{equation*}%
\begin{equation*}
+C\sum\limits_{R^{\prime }=I^{\prime }\times J^{\prime },j^{\prime }\leq
k^{\prime }}({\frac{{|I|}}{{|I^{\prime }|}}}\wedge {\frac{{|I^{\prime }|}}{{%
|I|}}})^{L}({\frac{{|J|}}{{|J^{\prime }|}}}\wedge {\frac{{|J^{\prime }|}}{{%
|J|}}})^{L}\times
\end{equation*}%
\begin{equation*}
{\frac{{|I^{\prime }|}^{K}}{{({|I^{\prime }|}+|z-x_{I^{\prime }}|)^{(2+K)}}}}%
{\frac{{|I^{\prime }|}^{K}}{{({|I^{\prime }|}+|u-y_{J^{\prime }}|)^{(1+K)}}}}%
|I^{\prime }||J^{\prime }||\phi _{R^{\prime }}\ast f(x_{I^{\prime
}},y_{J^{\prime }})|^{2},
\end{equation*}%
where $K,L$ are any positive integers which can be chosen by $L,K>4\left( {%
\frac{{2}}{{p}}}-1\right) $ (for general $\mathbb{H}^{n}$, $K$ can be chosen
greater than $(2n+2)({\frac{{2}}{{p}}}-1)$), the constant $C$ depends only
on $K,L$ and functions $\psi $ and $\phi $, here $x_{I^{\prime }}$ and $%
y_{J^{\prime }}$ are any fixed points in $I^{\prime },J^{\prime },$
respectively.

Adding up over $R\subseteq \Omega ,$ we obtain 
\begin{equation}
\sum\limits_{R\subseteq \Omega }|I||J|S_{R}\leq C\sum\limits_{R\subseteq
\Omega }\sum\limits_{R^{\prime }}|I^{\prime }||J^{\prime }|r(R,R^{\prime
})P(R,R^{\prime })T_{R^{\prime }},  \label{7.1}
\end{equation}%
where 
\begin{equation*}
r(R,R^{\prime })=({\frac{{|I|}}{{|I^{\prime }|}}}\wedge {\frac{{|I^{\prime }|%
}}{{|I|}}})^{L-1}({\frac{{|J|}}{{|J^{\prime }|}}}\wedge {\frac{{|J^{\prime }|%
}}{{|J|}}})^{L-1}
\end{equation*}%
and 
\begin{equation*}
P(R,R^{\prime })={\frac{{1}}{{(1+{\frac{{dist(I,I^{\prime })}}{{|I^{\prime }|%
}}})^{2+K}(1+{\frac{{dist(J,J^{\prime })}}{{|J^{\prime }|}}})^{1+K}}}}
\end{equation*}%
if $j^{\prime }>k^{\prime }$, and 
\begin{equation*}
P(R,R^{\prime })={\frac{{1}}{{(1+{\frac{{dist(I,I^{\prime })}}{{|I^{\prime }|%
}}})^{2+K}(1+{\frac{{dist(J,J^{\prime })}}{{|I^{\prime }|}}})^{1+K}}}}
\end{equation*}%
if $j^{\prime }\leq k^{\prime }$.

We estimate the right-hand side in the above inequality, where we first
consider 
\begin{equation*}
R^{\prime }=I^{\prime }\times J^{\prime },\ \ \ \left\vert I^{\prime
}\right\vert =2^{-\left( j^{\prime }+N\right) },\ \ \ \left\vert J^{\prime
}\right\vert =2^{-j^{\prime }-N}+2^{-k^{\prime }-N},\ \ \ j^{\prime
}>k^{\prime }.
\end{equation*}%
Define 
\begin{equation*}
\Omega ^{i,\ell }=\bigcup_{I\times J\subset \Omega }3(2^{i}I\times 2^{\ell
}J)\,\,\text{for}\,\,i,\ell \geq 0.
\end{equation*}%
Let $B_{i,\ell }$ be a collection of dyadic rectangles $R^{\prime }$ so that
for $i,\ell \geq 1$ 
\begin{equation*}
B_{i,\ell }=\{R^{\prime }=I^{\prime }\times J^{\prime i}I^{\prime \ell
}J^{\prime })\bigcap \Omega ^{i,\ell }\neq \emptyset \,\,\text{and}%
\,\,3(2^{i-1}I^{\prime }\times 2^{\ell -1}J^{\prime i,\ell }=\emptyset \},
\end{equation*}%
and 
\begin{equation*}
B_{0,\ell }=\{R^{\prime }=I^{\prime }\times J^{\prime },3(I^{\prime }\times
2^{\ell }J^{\prime 0,\ell }\neq \emptyset \,\,\text{and}\,\,3(I^{\prime
}\times 2^{\ell -1}J^{\prime 0,\ell }=\emptyset \}\,\,\text{for}\,\,\ell
\geq 1,
\end{equation*}%
and 
\begin{equation*}
B_{i,0}=\{R^{\prime }=I^{\prime }\times J^{\prime i}I^{\prime }\times
J^{\prime i,0}\neq \emptyset \,\,\text{and}\,\,3(2^{i-1}I^{\prime }\times
J^{\prime i,0}=\emptyset \}\,\,\text{for}\,\,i\geq 1,
\end{equation*}%
and 
\begin{equation*}
B_{0,0}=\{R^{\prime }:R^{\prime }=I^{\prime }\times J^{\prime },3(I^{\prime
}\times J^{\prime })\bigcap \Omega ^{0,0}\neq \emptyset \}.
\end{equation*}%
We write 
\begin{equation*}
\sum\limits_{R\subseteq \Omega }\sum\limits_{R^{\prime }}|I^{\prime
}||J^{\prime }|r(R,R^{\prime })P(R,R^{\prime })T_{R^{\prime
}}=\sum\limits_{i\geq 0,\ell \geq 0}\sum\limits_{R^{\prime }\in {B_{i,\ell }}%
}\sum\limits_{R\subseteq \Omega }|I^{\prime }||J^{\prime }|r(R,R^{\prime
})P(R,R^{\prime })T_{R^{\prime }}.
\end{equation*}

To estimate the right-hand side of the above equality, we first consider the
case when $i=\ell =0.$ Note that when $R^{\prime }\in B_{0,0}$, $3R^{\prime
0,0}\neq \emptyset .$ For each integer $h\geq 1,$ let 
\begin{equation*}
\mathcal{F}_{h}=\{R^{\prime }=I^{\prime }\times J^{\prime }\in
B_{0,0},|(3I^{\prime }\times 3J^{\prime 0,0}|\geq ({\frac{{1}}{{2^{h}}}}%
)|3I^{\prime }\times 3J^{\prime }|\}.
\end{equation*}%
Let $\mathcal{D}_{h}=\mathcal{F}_{h}\backslash \mathcal{F}_{h-1},$ and $%
\Omega _{h}=\bigcup_{R^{\prime }\in \mathcal{D}_{h}}R^{\prime }.$ Finally,
assume that for any open set $\Omega \subset R^{2},$ 
\begin{equation*}
\sum\limits_{R=I\times J\subseteq \Omega }|I||J|T_{R}\leq C{|\Omega |}^{{%
\frac{{2}}{{p}}}-1}.
\end{equation*}%
Since $B_{0,0}=\bigcup_{h\geq 1\mathcal{D}h}$ and for each $R^{\prime }\in
B_{0,0},P(R,R^{\prime })\leq 1,$ thus, 
\begin{equation*}
\sum\limits_{R^{\prime }\in {B_{0,0}}}\sum\limits_{R\subseteq \Omega
}|I^{\prime }||J^{\prime }|r(R,R^{\prime })P(R,R^{\prime })T_{R^{\prime }}
\end{equation*}%
\begin{equation*}
\leq \sum\limits_{h\geq 1}\sum\limits_{R^{\prime }\subseteq \Omega
_{h}}\sum\limits_{R\subseteq \Omega }|I^{\prime }||J^{\prime }|r(R,R^{\prime
})T_{R^{\prime }}
\end{equation*}

For each $h\geq 1$ and $R^{\prime }\subseteq \Omega _{h}$, we decompose $%
\{R:R\subseteq \Omega \}$ into 
\begin{equation*}
A_{0,0}(R^{\prime })=R\subseteq \Omega :\,\,\text{dist}(I,I^{\prime })\leq
|I|\vee |I^{\prime }|,\,\text{dist}(J,J^{\prime })\leq |J|\vee |J^{\prime }|;
\end{equation*}%
\begin{equation*}
A_{i^{\prime },0}(R^{\prime })=R\subseteq \Omega :\,2^{i^{\prime
}-1}(|I|\vee |I^{\prime }|)<\,\text{dist}(I,I^{\prime })\leq 2^{i^{\prime
}}(|I|\vee |I^{\prime }|),\,\text{dist}(J,J^{\prime })\leq |J|\vee
|J^{\prime }|;
\end{equation*}%
\begin{equation*}
A_{0,\ell ^{\prime }}(R^{\prime })=R\subseteq \Omega :\,\text{dist}%
(I,I^{\prime })\leq |I|\vee |I^{\prime }|,\ 2^{\ell ^{\prime }-1}(|J|\vee
|J^{\prime }|)<\text{dist}(J,J^{\prime })\leq 2^{\ell ^{\prime }}(|J|\vee
|J^{\prime }|);
\end{equation*}%
\begin{equation*}
A_{i^{\prime },\ell ^{\prime }}(R^{\prime })=R\subseteq \Omega :\
2^{i^{\prime }-1}(|I|\vee |I^{\prime }|)<\text{dist}(I,I^{\prime })\leq
2^{i^{\prime }}(|I|\vee |I^{\prime }|),
\end{equation*}%
\begin{equation*}
2^{\ell ^{\prime }-1}(|J|\vee |J^{\prime }|)<\text{dist}(J,J^{\prime \ell
^{\prime }}(|J|\vee |J^{\prime }|),
\end{equation*}%
where $i^{\prime },\ell ^{\prime }\geq 1$.

Now we split $\sum\limits_{h\geq 1}\sum\limits_{R^{\prime }\subseteq \Omega
_{h}}\sum\limits_{R\subseteq \Omega }|I^{\prime }||J^{\prime }|r(R,R^{\prime
})P(R,R^{\prime })T_{R^{\prime }}$ into 
\begin{equation*}
\sum\limits_{h\geq 1}\sum\limits_{R^{\prime }\in \Omega
_{h}}\sum\limits_{R\in A_{0,0}(R^{\prime })}+\sum\limits_{i^{\prime }\geq
1}\sum\limits_{R\in A_{i^{\prime },0}(R^{\prime })}+\sum\limits_{\ell
^{\prime }\geq 1}\sum\limits_{R\in A_{0,\ell ^{\prime }}(R^{\prime
})}+\sum\limits_{i^{\prime },\ell ^{\prime }\geq 1}\sum\limits_{R\in
A_{i^{\prime },\ell ^{\prime }}(R^{\prime })}|I^{\prime }||J^{\prime }|
\end{equation*}%
\begin{equation*}
\times r(R,R^{\prime })P(R,R^{\prime })T_{R^{\prime
}}=:I_{1}+I_{2}+I_{3}+I_{4}.
\end{equation*}

To estimate term $I_{1}$, we only need to estimate $\sum\limits_{R\in
A_{0,0}(R^{\prime })}r(R,R^{\prime })$ since $P(R,R^{\prime })\leq 1$ in
this case.

Note that $R\in A_{0,0}(R^{\prime })$ implies $3R\bigcap3R^{\prime }\ne
\emptyset$. For such $R$, there are four cases:

\textbf{Case 1}: $|I^{\prime }|\geq |I|$, $|J^{\prime }|\leq |J|$;

\textbf{Case 2}: $|I^{\prime }|\leq |I|$, $|J^{\prime }|\geq |J|$;

\textbf{Case 3}: $|I^{\prime }|\geq |I|$, $|J^{\prime }|\geq |J|$;

\textbf{Case 4}: $|I^{\prime }|\leq |I|$, $|J^{\prime }|\leq |J|$.

In each case, we can estimate $\sum_{R\in A_{0,0}}r\left( R,R^{\prime
-hL}\right) $ by using a simple geometric argument similar to that of
Chang-R. Fefferman \cite{CF3}. This, together with (\ref{7.1}), implies that 
$I_{1}$ is bounded by 
\begin{equation*}
\sum\limits_{h\geq 1}2^{-hL}|\Omega _{h}|^{{\frac{{2}}{{p}}}-1}\leq
C\sum\limits_{h\geq 1}h^{{\frac{{2}}{{p}}}-1}2^{-h(L-{\frac{{2}}{{p}}}%
+1)}|\Omega ^{0,0}|^{{\frac{{2}}{{p}}}-1}\leq C|\Omega |^{{\frac{{2}}{{p}}}%
-1},
\end{equation*}%
since $|\Omega _{h}|\leq Ch2^{h}|\Omega ^{0,0}|$ and $|\Omega ^{0,0}|\leq
C|\Omega |.$

Thus it remains to estimate term $I_{4}$, since estimates of $I_{2}$ and $%
I_{3}$ can be derived using the same techniques as in $I_{1}$ and $I_{4}$.
The estimate for this term is more complicated than that for term $I_{1}$.

As in estimating term $I_{1}$, we only need to estimate the sum $\sum_{R\in
A_{i^{\prime },\ell ^{\prime }}}r(R,R^{\prime })$. Note that $R\in
A_{i^{\prime },\ell ^{\prime }}(R^{\prime })$ implies $3(2^{i^{\prime
}}I\times 2^{\ell ^{\prime }}J)\cap 3(2^{i^{\prime }}I^{\prime \ell ^{\prime
}}J^{\prime })\neq \emptyset $. We also split our estimate into four cases.

\textbf{Case 1}: $|2^{i^{\prime }}I^{\prime }|\geq |2^{i^{\prime }}I|$, $%
|2^{\ell ^{\prime }}J^{\prime }|\geq |2^{\ell ^{\prime }}J|$. Then 
\begin{equation*}
{\frac{{|2^{i^{\prime }}I|}}{{|3\cdot 2^{i^{\prime }}I^{\prime }|}}}%
\left\vert 3\left( 2^{i^{\prime }}I^{\prime }\times 2^{\ell ^{\prime
}}J^{\prime }\right) \right\vert \leq \left\vert 3\left( 2^{i^{\prime
}}I^{\prime }\times 2^{\ell ^{\prime }}J^{\prime }\right) \cap 3\left(
2^{i^{\prime }}I\times 2^{\ell ^{\prime }}J\right) \right\vert
\end{equation*}%
\begin{equation*}
\leq C2^{i^{\prime }}2^{\ell ^{\prime }}|3R^{\prime }\cap \Omega ^{0,0}|\leq
C2^{i^{\prime }}2^{\ell ^{\prime }}{\frac{{1}}{{2^{h-1}}}}|3R^{\prime }|\leq
C{\frac{{1}}{{2^{h-1}}}}|3(2^{i^{\prime }}I^{\prime }\times 2^{\ell ^{\prime
}}J^{\prime })|.
\end{equation*}%
Thus $\left\vert 2^{i^{\prime }}I^{\prime }\right\vert ={2^{h-1+n}}%
\left\vert 2^{i^{\prime }}I\right\vert $ for some $n\geq 0$. For each fixed $%
n$, the number of such $2^{i^{\prime }}I$ must be $\leq 2^{n}\cdot 5$.
Similarly $|2^{\ell ^{\prime }}J|=2^{m}|2^{\ell ^{\prime }}J^{\prime }|$ for
some $m\geq 0$, and for each fixed $m$, $3\cdot 2^{\ell ^{\prime }}J\cap
3\cdot 2^{\ell ^{\prime }}J^{\prime }\neq \emptyset $ implies that the
number of such $2^{\ell ^{\prime }}J^{\prime }$ is less than $5$. Thus 
\begin{equation*}
\sum_{R\in case1}r(R,R^{\prime })\leq \sum_{m,n\geq 0}r\left( R,R^{\prime
}\right) \left( {\frac{1}{{2^{n+m+h-1}}}}\right) ^{L}2^{n}\cdot 5^{2}\leq
C2^{-hL}.
\end{equation*}

Similarly, we can handle the other three cases. Combining the four cases, we
have 
\begin{equation*}
\sum_{R\in A_{i^{\prime },\ell ^{\prime }}(R^{\prime })}r\left( R,R^{\prime
}\right) \leq C2^{-hL},
\end{equation*}%
which, together with the estimate for $P(R,R^{\prime }),$ implies that 
\begin{equation*}
I_{4}\leq C\sum\limits_{h\geq 1}\sum\limits_{i^{\prime },\ell ^{\prime }\geq
1}\sum\limits_{R^{\prime }\subseteq \Omega _{h}}2^{-hL}2^{-i^{\prime
}(2+K)}2^{-\ell ^{\prime }(1+K)}|I^{\prime }||J^{\prime }|T_{R^{\prime }}.
\end{equation*}%
Hence $I_{4}$ is bounded by 
\begin{equation*}
\sum\limits_{h\geq 1}2^{-hL}|\Omega _{h}|^{{\frac{{2}}{{p}}}-1}\leq
C\sum\limits_{h\geq 1}h^{{\frac{{2}}{{p}}}-1}2^{-h(L-{\frac{{2}}{{p}}}%
+1)}|\Omega ^{0,0}|^{{\frac{{2}}{{p}}}-1}\leq C|\Omega |^{{\frac{{2}}{{p}}}%
-1},
\end{equation*}%
since $|\Omega _{h}|\leq Ch2^{h}|\Omega ^{0,0}|$ and $|\Omega ^{0,0}|\leq
C|\Omega |$. Combining $I_{1}$, $I_{2}$, $I_{3}$ and $I_{4}$, we have 
\begin{equation*}
{\frac{1}{{|\Omega |^{{\frac{{2}}{{p}}}-1}}}}\sum_{R^{\prime }\in {B_{0,0}}%
}\sum\limits_{R\subseteq \Omega }|I^{\prime }||J^{\prime }|r(R,R^{\prime
})P(R,R^{\prime })T_{R^{\prime }}\leq C\sup_{\bar{\Omega}}{\frac{1}{{|\bar{%
\Omega}|^{{\frac{{2}}{{p}}}-1}}}}\sum_{R^{\prime }\subseteq \bar{\Omega}%
}|I^{\prime }||J^{\prime }|T_{R^{\prime }}.
\end{equation*}

Now we consider 
\begin{equation*}
\sum\limits_{i,\ell \geq 1}\sum\limits_{R^{\prime }\in {B_{i,\ell }}%
}\sum\limits_{R\subseteq \Omega }|I^{\prime }||J^{\prime }|r(R,R^{\prime
})P(R,R^{\prime })T_{R^{\prime }}.
\end{equation*}%
Note that for $R^{\prime }\in {B_{i,\ell }}$, $3\left( 2^{i}I^{\prime
}\times 2^{\ell }J^{\prime }\right) \cap \Omega ^{i,\ell }\neq \emptyset $.
Let 
\begin{equation*}
\mathcal{F}_{h}^{i,\ell }=\left\{ R^{\prime }\in B_{i,\ell }:\left\vert
3\left( 2^{i}I^{\prime }\times 2^{\ell }J^{\prime }\right) \cap \Omega
^{i,\ell }\right\vert \geq {\frac{1}{{2^{h}}}}\left\vert 3\left(
2^{i}I^{\prime }\times 2^{\ell }J^{\prime }\right) \right\vert \right\} ,
\end{equation*}%
\begin{equation*}
\mathcal{D}_{h}^{i,\ell }=\mathcal{F}_{h}^{i,\ell }\setminus \mathcal{F}%
_{h-1}^{i,\ell }
\end{equation*}%
and 
\begin{equation*}
\Omega _{h}^{i,\ell }=\bigcup_{R^{\prime }\in \mathcal{D}_{h}^{i,\ell
}}R^{\prime }.
\end{equation*}%
Since $B_{i,\ell }=\bigcup_{h\geq 1}\mathcal{D}_{h}^{i,\ell }$, we first
estimate 
\begin{equation*}
\sum_{R^{\prime }\in \mathcal{D}_{h}^{i,\ell }}\sum_{R\subseteq \Omega
}|I^{\prime }||J^{\prime }|r(R,R^{\prime })P(R,R^{\prime })T_{R^{\prime }}
\end{equation*}%
for some $i,\ell ,h\geq 1$.

Note that for each $R^{\prime }\in \mathcal{D}_{h}^{i,\ell }$, $3\left(
2^{i}I^{\prime }\times 2^{\ell }J^{\prime }\right) \cap \Omega ^{i-1,\ell
-1}=\emptyset $. So for any $R\subseteq \Omega $, we have $2^{i}(|I|\vee
|I^{\prime }|)\leq \,\text{dist}(I,I^{\prime })$ and $2^{\ell }(|J|\vee
|J^{\prime }|)\leq \,\,\text{dist}(J,J^{\prime })$. We decompose $\{R:\
R\subseteq \Omega \}$ by 
\begin{equation*}
A_{i^{\prime },\ell ^{\prime }}(R^{\prime i^{\prime }-1}\cdot 2^{i}(|I|\vee
|I^{\prime }|)\leq \,\text{dist}(I,I^{\prime i^{\prime }}\cdot 2^{i}(|I|\vee
|I^{\prime }|),
\end{equation*}%
\begin{equation*}
2^{\ell ^{\prime }-1}\cdot 2^{\ell }(|J|\vee |J^{\prime }|)\leq \,\text{dist}%
(J,J^{\prime \ell ^{\prime }}\cdot 2^{\ell }(|J|\vee |J^{\prime }|)\},
\end{equation*}%
where $i^{\prime },\ell ^{\prime }\geq 1$. Then we write 
\begin{equation*}
\sum_{R^{\prime }\in \mathcal{D}_{h}^{i,\ell }}\sum_{R\subseteq \Omega
}|I^{\prime }||J^{\prime }|r(R,R^{\prime })P(R,R^{\prime })T_{R^{\prime
}}=\sum_{i^{\prime },\ell ^{\prime }\geq 1}\sum_{R^{\prime }\in \mathcal{D}%
_{h}^{i,\ell }}\sum_{R\in A_{i^{\prime },\ell ^{\prime }}(R^{\prime
})}|I^{\prime }||J^{\prime }|r(R,R^{\prime })P(R,R^{\prime })T_{R^{\prime }}
\end{equation*}

Since $P\left( R,R^{\prime }\right) \leq 2^{-i\left( 2+K\right) }2^{-\ell
(1+K)}2^{-i^{\prime }(2+K)}$ for $R^{\prime }\in B_{i,\ell }$ and $R\in
A_{i^{\prime },\ell ^{\prime }}(R^{\prime })$, we can repeat the same proof
with $B_{0,0}$ replaced by $B_{i,\ell }$, with with the necessary
modifications, to obtain 
\begin{equation*}
\sum_{R^{\prime }\in \mathcal{D}_{h}^{i,\ell }}\sum_{R\in A_{i^{\prime
},\ell ^{\prime }}(R^{\prime })}|I^{\prime }||J^{\prime }|r(R,R^{\prime
})P(R,R^{\prime })T_{R^{\prime }}\leq C2^{-i(2+K)}2^{-\ell
(1+K)}2^{-i^{\prime }(2+K)}2^{-\ell ^{\prime }(1+K)}\times
\end{equation*}%
\begin{equation*}
i^{{\frac{2}{p}}-1}2^{i\left( {\frac{2}{p}}-1\right) }{\ell }^{{\frac{2}{p}}%
-1}2^{\ell \left( {\frac{2}{p}}-1\right) }h^{{\frac{2}{p}}-1}2^{-h\left( L-{%
\frac{2}{p}}+1\right) }\sup_{\bar{\Omega}}{\frac{1}{{|\bar{\Omega}|^{{\frac{{%
2}}{{p}}}-1}}}}\sum_{R^{\prime }\subseteq \bar{\Omega}}|I^{\prime
}||J^{\prime }|T_{R^{\prime }}.
\end{equation*}%
Adding over all $i,\ell ,i^{\prime },\ell ^{\prime },h\geq 1,$ we have 
\begin{equation*}
{\frac{1}{{|\Omega |^{{\frac{{2}}{{p}}}-1}}}}\sum\limits_{i,\ell \geq
1}\sum\limits_{R^{\prime }\in {B_{i,\ell }}}\sum\limits_{R\subseteq \Omega
}|I^{\prime }||J^{\prime }|r(R,R^{\prime })P(R,R^{\prime })T_{R^{\prime
}}\leq C\sup_{\bar{\Omega}}{\frac{1}{{|\bar{\Omega}|^{{\frac{{2}}{{p}}}-1}}}}%
\sum_{R^{\prime }\subseteq \bar{\Omega}}|I^{\prime }||J^{\prime
}|T_{R^{\prime }}.
\end{equation*}

Similar estimates, which we leave to the reader, hold for 
\begin{equation*}
\sum\limits_{i\geq 1}\sum\limits_{R^{\prime }\in {B_{i,0}}%
}\sum\limits_{R\subseteq \Omega }|I^{\prime }||J^{\prime }|r(R,R^{\prime
})P(R,R^{\prime })T_{R^{\prime }}
\end{equation*}%
and 
\begin{equation*}
\sum\limits_{\ell \geq 1}\sum\limits_{R^{\prime }\in {B_{0,\ell }}%
}\sum\limits_{R\subseteq \Omega }|I^{\prime }||J^{\prime }|r(R,R^{\prime
})P(R,R^{\prime })T_{R^{\prime }},
\end{equation*}%
which, after adding over all $i,\ell \geq 0$, complete the proof of Theorem %
\ref{P-PCar}.

\bigskip

As a consequence of Theorem \ref{P-PCar}, it is easy to see that the space $%
CMO_{F}^{p}$ is well defined. In particular, we have

\begin{corollary}
We have 
\begin{equation*}
\Vert f\Vert _{CMO_{F}^{p}}\approx \sup_{\Omega }\left\{ {\frac{{1}}{{%
|\Omega |}^{{\frac{{2}}{{p}}}-1}}}\sum\limits_{j}\sum\limits_{k}\sum%
\limits_{I\times J\subseteq \Omega }|\psi _{j,k}\ast
f(x_{I},y_{J})|^{2}|I||J|\right\} ^{\frac{{1}}{{2}}},
\end{equation*}%
where $I\times J$ is a dyadic rectangle in $\mathbb{H}^{n}$ with $\ell
\left( I\right) =2^{-j-N}$ and $\ell \left( J\right) =2^{-j-N}+2^{-k-N}$,
and where $x_{I},y_{J}$ are any fixed points in $I,J,$ respectively.
\end{corollary}

We are now ready to give the proof of Theorem \ref{sequenceduality}.

\begin{proof}
\textbf{(of Theorem \ref{sequenceduality}):} We first prove $c^{p}\subseteq
(s^{p})^{\ast }.$ Applying the proof in Theorem \ref{L<H}, set 
\begin{equation*}
s(z,u)=\{\sum\limits_{I\times J}|s_{I\times J}|^{2}|I|^{-1}|J|^{-1}\chi
_{I}(z)\chi _{J}(u)\}^{\frac{{1}}{{2}}}
\end{equation*}%
and 
\begin{equation*}
\Omega _{i}=\{(z,u)\in \mathbb{H}^{n}:s(z,u)>2^{i}\}.
\end{equation*}%
Let 
\begin{equation*}
\mathcal{B}_{i}=\{(I\times J):|(I\times J)\cap \Omega _{i}|>{\frac{{1}}{{2}}}%
|I\times J|,|(I\times J)\cap \Omega _{i+1}|\leq {\frac{{1}}{{2}}}|I\times
J|\},
\end{equation*}%
where $I\times J$ is a dyadic rectangle in $\mathbb{H}^{n}$ with $\ell
\left( I\right) =2^{-j-N}$ and $\ell \left( J\right) =2^{-j-N}+2^{-k-N}$.
Suppose $t=\{t_{I\times J}\}\in c^{p}$ and write 
\begin{eqnarray}
\left\vert \sum\limits_{I\times J}s_{I\times J}\overline{t}_{I\times
J}\right\vert &=&\left\vert \sum\limits_{i}\sum\limits_{(I\times J)\in 
\mathcal{B}_{i}}s_{I\times J}\overline{t}_{I\times J}\right\vert  \label{7.2}
\\
&\leq &\left\{ \sum\limits_{i}\left[ \sum\limits_{(I\times J)\in \mathcal{B}%
_{i}}|s_{I\times J}|^{2}\right] ^{\frac{{p}}{{2}}}\left[ \sum\limits_{(I%
\times J)\in \mathcal{B}_{i}}|t_{I\times J}|^{2}\right] ^{\frac{{p}}{{2}}%
}\right\} ^{\frac{{1}}{{p}}}  \notag \\
&\leq &C\Vert t\Vert _{c^{p}}\left\{ \sum\limits_{i}|\Omega _{i}|^{1-{\frac{{%
p}}{{2}}}}\left[ \sum\limits_{(I\times J)\in \mathcal{B}_{i}}|s_{I\times
J}|^{2}\right] ^{\frac{{p}}{{2}}}\right\} ^{\frac{{1}}{{p}}}  \notag
\end{eqnarray}%
since if $I\times J\in \mathcal{B}{_{i}}$, then 
\begin{eqnarray*}
I\times J &\subseteq &{\widetilde{\Omega _{i}}}=\left\{ (z,u):M_{S}(\chi
_{\Omega _{i}})(z,u)>{\frac{{1}}{{2}}}\right\} , \\
|{\widetilde{\Omega _{i}}}| &\leq &C|{\Omega _{i}}|,
\end{eqnarray*}%
and $\{t_{I\times J}\}\in c^{p}$ yield 
\begin{equation*}
\{\sum\limits_{(I\times J)\in \mathcal{B}_{i}}|t_{I\times J}|^{2}\}^{\frac{{1%
}}{{2}}}\leq C\Vert t\Vert _{c^{p}}|{\Omega _{i}}|^{{\frac{{1}}{{p}}}-{\frac{%
{1}}{{2}}}}.
\end{equation*}

The same proof as in the claim of Theorem 4.4 implies 
\begin{equation*}
\sum\limits_{(I\times J)\in \mathcal{B}_{i}}|s_{I\times J}|^{2}\leq C2^{2i}|{%
\Omega _{i}}|.
\end{equation*}%
Substituting the above term back into the last term in (\ref{7.2}) gives $%
c^{p}\subseteq (s^{p})^{\ast }$.

The proof of the converse is simple and is similar to the one given in \cite%
{FJ} for $p=1$ in the one-parameter setting on $R^{n}$. If $\ell \in
(s^{p})^{\ast },$ then it is clear that $\ell (s)=\sum\limits_{I\times
J}s_{I\times J}{\overline{t}}_{I\times J}$ for some $t=\{t_{I\times J}\}$.
Now fix an open set $\Omega \subset \mathbb{H}^{n}$ and let $S$ be the
sequence space of all $s=\{s_{I\times J}\}$ such that $I\times J\subseteq
\Omega .$ Finally, let $\mu $ be a measure on $S$ so that the $\mu -$measure
of the \textquotedblleft point\textquotedblright\ $I\times J$ is ${\frac{{1}%
}{{{|\Omega |}^{{\frac{{2}}{{p}}}-1}}}}.$ Then, 
\begin{equation*}
\{{\frac{{1}}{{|\Omega |^{{\frac{{2}}{{p}}}-1}}}}\sum_{I\times J\subseteq
\Omega }|t_{I\times J}|^{2}\}^{\frac{{1}}{{2}}}=\Vert t_{I\times J}\Vert
_{\ell ^{2}(S,d\mu )}
\end{equation*}%
\begin{equation*}
=\sup_{\Vert s\Vert _{\ell ^{2}(S,d\mu )}\leq 1}|{\frac{{1}}{{|\Omega |^{{%
\frac{{2}}{{p}}}-1}}}}\sum\limits_{I\times J\subseteq \Omega }s_{I\times J}{%
\overline{t}}_{I\times J}|
\end{equation*}%
\begin{equation*}
\leq \Vert t\Vert _{(s^{p})^{\ast }}\sup_{\Vert s\Vert _{\ell ^{2}(S,d\mu
)}\leq 1}\Vert s_{I\times J}{\frac{{1}}{{|\Omega |^{{\frac{{2}}{{p}}}-1}}}}%
\Vert _{s^{p}}.
\end{equation*}%
By H\"{o}lder's inequality, 
\begin{eqnarray*}
\Vert s_{I\times J}{\frac{{1}}{{|\Omega |^{{\frac{{2}}{{p}}}-1}}}}\Vert
_{s^{p}} &=&{\frac{{1}}{{|\Omega |^{{\frac{{2}}{{p}}}-1}}}}%
\{\int\limits_{\Omega }(\sum\limits_{I\times J\subseteq \Omega }|s_{I\times
J}|^{2}|I\times J|^{-1}\chi _{I}(x)\chi _{J}(y))^{\frac{{p}}{{2}}}dzdu\}^{%
\frac{{1}}{{p}}} \\
&\leq &\{{\frac{{1}}{{|\Omega |^{{\frac{{2}}{{p}}}-1}}}}\int\limits_{\Omega
}\sum\limits_{I\times J\subseteq \Omega }|s_{I\times J}|^{2}{|I\times J|}%
^{-1}\chi _{I}(x)\chi _{J}(y)dzdu\}^{\frac{{1}}{{2}}}=\Vert s\Vert _{\ell
^{2}(S,d\mu )}\leq 1,
\end{eqnarray*}%
which shows $\Vert t\Vert _{c^{p}}\leq \Vert t\Vert _{(s^{p})^{\ast }}$%
\textbf{.}
\end{proof}

In order to use Theorem \ref{sequenceduality} to obtain Theorem \ref{flag
duality}, we introduce a map $S$ which takes $f\in (\mathcal{M}%
_{flag}^{M+\delta })^{\prime }$ to the sequence of coefficients 
\begin{equation*}
Sf\equiv \{s_{I\times J}\}=\left\{ |I|^{\frac{{1}}{{2}}}|J|^{\frac{{1}}{{2}}%
}\psi _{j,k}\ast f(x_{I},y_{J})\right\} ,
\end{equation*}%
where $I\times J$ is a dyadic rectangle in $\mathbb{H}^{n}$ with $\ell
\left( I\right) =2^{-j-N}$ and $\ell \left( J\right) =2^{-j-N}+2^{-k-N}$,
and where $x_{I},y_{J}$ are any fixed points in $I,J,$ respectively. For any
sequence $s=\{s_{I\times J}\},$ we define a map $T$ which takes $s$ to 
\begin{equation*}
T(s)=\sum\limits_{j}\sum\limits_{k}\sum\limits_{J}\sum\limits_{I}|I|^{\frac{{%
1}}{{2}}}|J|^{\frac{{1}}{{2}}}{\widetilde{\psi }}_{j,k}(z,u)s_{I\times J},
\end{equation*}%
where ${\widetilde{\psi }}_{j,k}$ are as in (\ref{defdiscreteCRF}).

The following result together with Theorem \ref{sequenceduality} will give
Theorem \ref{flag duality}.

\begin{theorem}
The maps $S:H_{flag}^{p}\rightarrow s^{p}$ and $S:CMO_{flag}^{p}\rightarrow
c^{p}$, as well as the maps $T:s^{p}\rightarrow H_{flag}^{p}$ and $%
T:c^{p}\rightarrow CMO_{flag}^{p}$ are bounded. Moreover, $T\circ S$ is the
identity on both $H_{flag}^{p}$ and $CMO_{flag}^{p}$.
\end{theorem}

\begin{proof}
The boundedness of $S$ on $H_{flag}^{p}$ and $CMO_{flag}^{p}$ follows
directly from the Plancherel-P\^{o}lya inequalities, Theorem \ref{P-P} and
Theorem \ref{P-PCar}. The boundedness of $T$ also follows from the arguments
in Theorem \ref{P-P} and \ref{P-PCar}. Indeed, to see that $T$ is bounded
from $s^{p}$ to $H_{flag}^{p}$, let $s=\{s_{I\times J}\}.$ Then, by
Proposition \ref{1}, 
\begin{equation*}
\Vert T(s)\Vert _{H_{flag}^{p}}\leq C\left\Vert \left\{
\sum\limits_{j}\sum\limits_{k}\sum\limits_{J}\sum\limits_{I}|\psi _{j,k}\ast
T(s)(z,u)|^{2}\chi _{I}(z)\chi _{J}(u)\right\} ^{{\frac{{1}}{{2}}}%
}\right\Vert _{p}.
\end{equation*}%
By adapting an argument similar to that in the proof of Theorem \ref{P-P},
we have for some $0<r<p$ 
\begin{eqnarray*}
&&|\psi _{j,k}\ast T(s)(z,u)\chi _{I}(z)\chi _{J}(u)|^{2} \\
&=&|\sum\limits_{j^{\prime },k^{\prime }}\sum\limits_{I^{\prime },J^{\prime
}}|I^{\prime }||I^{\prime }|\psi _{j,k}\ast {\widetilde{\psi }}_{{j^{\prime }%
},{k^{\prime }}}(\cdot ,\cdot )(z,u)s_{I^{\prime }\times J^{\prime
}}|I^{\prime }|^{-{\frac{{1}}{{}2}}}|J^{\prime }|^{-{\frac{{1}}{{}2}}}\chi
_{I}(z)\chi _{J}(u)|^{2} \\
&\leq &C\sum\limits_{k\wedge k^{\prime }\leq j\wedge j^{\prime }}2^{-|j-{%
j^{\prime }}|K}2^{-|k-{k^{\prime }}|K}\{M_{S}(\sum\limits_{I^{\prime
},J^{\prime }}|s_{I^{\prime }\times J^{\prime }}||I^{\prime
}|^{-1}|J^{\prime }|^{-1}\chi _{J^{\prime {\chi _{I^{\prime }}}}})^{r}\}^{{%
\frac{{2}}{{r}}}}(z,u)\chi _{I}(z)\chi _{J}(u) \\
&&+\sum\limits_{k\wedge k^{\prime }>j\wedge j^{\prime }}2^{-|j-{j^{\prime }}%
|K}2^{-|k-{k^{\prime }}|K}\{M(\sum\limits_{I^{\prime },J^{\prime
}}|s_{I^{\prime }\times J^{\prime }}||I^{\prime }|^{-1}|J^{\prime
}|^{-1}\chi _{J^{\prime {\chi _{I^{\prime }}}}})^{r}\}^{{\frac{{2}}{{r}}}%
}(z,u)\chi _{I}(z)\chi _{J}(u).
\end{eqnarray*}%
Repeating the argument in Theorem \ref{P-P} gives the boundedness of $T$
from $s^{p}$ to $H_{flag}^{p}$. A similar adaptation of the argument in the
proof of Theorem \ref{P-PCar} applies to yield the boundedness of $T$ from $%
c^{p}$ to $CMO_{flag}^{p}$. We leave the details to the reader. The discrete
Calder\'{o}n reproducing formula, Theorem \ref{discretethm}, and Theorem \ref%
{P-PCar} show that $T\circ S$ is the identity on both $H_{flag}^{p}$ and $%
CMO_{flag}^{p}$.
\end{proof}

We are now ready to give the proofs of Theorems \ref{flag duality} and \ref%
{BMO bound}.

\begin{proof}
\textbf{(of Theorem \ref{flag duality}):} If $f\in \mathcal{M}%
_{flag}^{M+\delta }$ and $g\in CMO_{flag}^{p}$, let $\ell _{g}=\left\langle
f,g\right\rangle $. Then then the discrete Calder\'{o}n reproducing formula,
Theorem \ref{P-PCar} and Theorem \ref{sequenceduality} imply 
\begin{eqnarray*}
\left\vert \ell _{g}\right\vert &=&\left\vert \left\langle f,g\right\rangle
\right\vert =\left\vert \sum\limits_{R=I\times J}|I||J|\psi _{R}\ast
f(x_{I},y_{J}){\widetilde{\psi }}_{R}(g)(x_{I},y_{J})\right\vert \\
&\leq &C\Vert f\Vert _{H_{flag}^{p}}\Vert g\Vert _{CMO_{F}^{p}}.
\end{eqnarray*}

Because $\mathcal{M}_{flag}^{M+\delta }$ is dense in $H_{flag}^{p},$ this
shows that the map $\ell _{g}=\left\langle f,g\right\rangle ,$ defined
initially for $f\in \mathcal{M}_{flag}^{M+\delta }$ can be extended to a
continuous linear functional on $H_{flag}^{p}$ with $\Vert \ell _{g}\Vert
\leq C\Vert g\Vert _{CMO_{flag}^{p}}$.

Conversely, let $\ell \in (H_{flag}^{p})^{\ast }$ and set $\ell _{1}=\ell
\circ T,$ where $T$ is defined as in Theorem \ref{sequenceduality}. Then, by
theorem \ref{sequenceduality}, $\ell _{1}\in (s^{p})^{\ast }$, so by Theorem %
\ref{P-PCar}, there exists $t=\{t_{I\times J}\}$ such that $\ell
_{1}(s)=\sum\limits_{I\times J}s_{I\times J}{\overline{t}}_{I\times J}$ for
all $s=\{s_{I\times J}\}$, and where 
\begin{equation*}
\Vert t\Vert _{c^{p}}\approx \Vert \ell _{1}\Vert \leq C\Vert \ell \Vert ,
\end{equation*}%
because $T$ is bounded. Again, by Theorem \ref{sequenceduality}, $\ell =\ell
\circ T\circ S=\ell _{1}\circ S$. Hence, with 
\begin{equation*}
f\in \mathcal{M}_{flag}^{M+\delta }\text{ and\ }g=\sum\limits_{I\times
J}t_{I\times J}\psi _{R}(\left( z,u\right) \circ \left( x_{I},y_{J}\right)
^{-1}),
\end{equation*}
and where, without loss the generality we may assume that $\psi $ is a
radial function, we have 
\begin{equation*}
\ell (f)=\ell _{1}(S(f))=\left\langle S(f),t\right\rangle =\left\langle
f,g\right\rangle .
\end{equation*}%
This proves $\ell =\ell _{g}$, and by Theorem \ref{sequenceduality} we have%
\begin{equation*}
\Vert g\Vert _{CMO_{flag}^{p}}\leq C\Vert t\Vert _{c^{p}}\leq C\Vert \ell
_{g}\Vert \mathbf{.}
\end{equation*}
\end{proof}

\begin{proof}
\textbf{(of Theorem \ref{BMO bound}):} By Corollary \ref{4}, $H_{flag}^{1}$
is a subspace of $L^{1}$. By the duality of $H_{flag}^{1}$ and $BMO_{flag}$,
we now conclude that $L^{\infty }$ is a subspace of $BMO_{flag}$, and from
the boundedness of flag singular integrals on $H_{flag}^{1}$, we obtain that
flag singular integrals are bounded on $BMO_{flag}$ and hence also from $%
L^{\infty }$ to $BMO_{flag}$. This completes the proof of Theorem \ref{BMO
bound}.
\end{proof}

\section{Calder\'{o}n-Zygmund decomposition and interpolation\label{s CZ
decomposition}}

In this section we derive a Calder\'{o}n-Zygmund decomposition using
functions in flag Hardy spaces. As an application, we prove an interpolation
theorem for the spaces $H_{flag}^{p}\left( \mathbb{H}^{n}\right) $.

We first recall that A. Chang and R. Fefferman established the following
Calder\'{o}n-Zygmund decomposition on the pure product domain $\mathbb{R}%
_{+}^{2}\times \mathbb{R}_{+}^{2}$ (\cite{CF2}).

\begin{lemma}[\textbf{Calder\'{o}n-Zygmund Lemma}]
Let $\alpha >0$ be given and $f\in L^{p}(\mathbb{R}^{2})$, $1<p<2$. Then we
may write $f=g+b$ where $g\in L^{2}(\mathbb{R}^{2})$ and $b\in H^{1}(\mathbb{%
R}_{+}^{2}\times \mathbb{R}_{+}^{2})$ with $||g||_{2}^{2}\leq \alpha
^{2-p}||f||_{p}^{p}$ and $||b||_{H^{1}(\mathbb{R}_{+}^{2}\times \mathbb{R}%
_{+}^{2})}\leq C\alpha ^{1-p}||f||_{p}^{p}$, where $c$ is an absolute
constant.
\end{lemma}

We now prove the Calder\'{o}n-Zygmund decomposition in the setting of flag
Hardy spaces on the Heisenberg group.

\begin{proof}
\textbf{(of Theorem \ref{CZ decomposition}):} We first assume $f\in L^{2}(%
\mathbb{H}^{n})\cap H_{flag}^{p}\left( \mathbb{H}^{n}\right) .$ Let $\alpha
>0$ and 
\begin{equation*}
\Omega _{\ell }=\{(z,u)\in \mathbb{H}^{n}:S(f)(z,u)>\alpha 2^{\ell }\},
\end{equation*}%
where, as in Corollary \ref{L2 dense}, 
\begin{equation*}
S(f)(z,u)=\left\{ \sum\limits_{j,k}\sum\limits_{I,J}|\phi _{jk}\ast \left(
T_{N}^{-1}(f)\right) (x_{I},y_{J})|^{2}\chi _{I}(z)\chi _{J}(u)\right\} ^{%
\frac{1}{2}}.
\end{equation*}%
It was shown in Corollary \ref{L2 dense} that for $f\in L^{2}(\mathbb{H}%
^{n})\cap H_{flag}^{p}\left( \mathbb{H}^{n}\right) $, we have $%
||f||_{H_{flag}^{p}}\approx ||S(f)||_{p}$.

In the following we denote dyadic rectangles in $\mathbb{H}^{n}$ by $%
R=I\times J$ with $\ell \left( I\right) =2^{-j-N}$ and $\ell \left( J\right)
=2^{-j-N}+2^{-k-N}$, where $j,k$ are integers and $N$ is sufficiently large.
Let 
\begin{equation*}
\mathcal{R}_{0}=\left\{ R=I\times J,\,\,\text{such that}\,\,|R\cap \Omega
_{0}|<\frac{1}{2}|R|\right\}
\end{equation*}%
and for $\ell \geq 1$ 
\begin{equation*}
\mathcal{R}_{\ell }=\left\{ R=I\times J,\,\,\text{such that}\,\,|R\cap
\Omega _{\ell -1}|\geq \frac{1}{2}|R|\,\,\text{but}\,\,|R\cap \Omega _{\ell
}|<\frac{1}{2}|R|\right\} .
\end{equation*}%
By the discrete Calder\'{o}n reproducing formula in Theorem \ref{19}, 
\begin{eqnarray*}
f(z,u) &=&\sum\limits_{j,k}\sum\limits_{I,J}|I||J|{\widetilde{\phi }}%
_{jk}(\left( z,u\right) \circ \left( x_{I},y_{J}\right) ^{-1})\phi _{jk}\ast
\left( T_{N}^{-1}(f)\right) (x_{I},y_{J}) \\
&=&\sum\limits_{\ell \geq 1}\sum\limits_{I\times J\in \mathcal{R}_{\ell
}}|I||J|{\widetilde{\phi }}_{jk}(\left( z,u\right) \circ \left(
x_{I},y_{J}\right) ^{-1})\phi _{jk}\ast \left( T_{N}^{-1}(f)\right)
(x_{I},y_{J}) \\
&&+\sum\limits_{I\times J\in \mathcal{R}_{0}}|I||J|{\widetilde{\phi }}%
_{jk}(\left( z,u\right) \circ \left( x_{I},y_{J}\right) ^{-1})\phi _{jk}\ast
\left( T_{N}^{-1}(f)\right) (x_{I},y_{J}) \\
&=&b(z,u)+g(z,u)
\end{eqnarray*}%
When $p_{1}>1$, using a duality argument, it is easy to show 
\begin{equation*}
||g||_{p_{1}}\leq C\left\Vert \left\{ \sum\limits_{R=I\times J\in \mathcal{R}%
_{0}}|\phi _{jk}\ast \left( T_{N}^{-1}(f)\right) (x_{I},y_{J})|^{2}\chi
_{I}\chi _{J}\right\} ^{\frac{1}{2}}\right\Vert _{p_{1}}.
\end{equation*}%
Next, we estimate $||g||_{H_{flag}^{p_{1}}}$ when $0<p_{1}\leq 1$. Clearly,
the duality argument will not work here. Nevertheless, we can estimate the $%
H_{flag}^{p_{1}}$ norm directly by using the discrete Calder\'{o}n
reproducing formula in Theorem \ref{19}. To this end, we note that 
\begin{equation*}
||g||_{H_{flag}^{p_{1}}}\leq \left\Vert \left\{ \sum\limits_{j^{\prime
},k^{\prime }}\sum\limits_{I^{\prime },J^{\prime }}|\left( \psi _{j^{\prime
}k^{\prime }}\ast g\right) (x_{I^{\prime }},y_{J^{\prime }})|^{2}\chi
_{I^{\prime }}(z)\chi _{J^{\prime }}(u)\right\} ^{\frac{1}{2}}\right\Vert
_{L^{p_{1}}}.
\end{equation*}%
Since 
\begin{eqnarray*}
&&\left( \psi _{j^{\prime },k^{\prime }}\ast g\right) (x_{I^{\prime
}},y_{J^{\prime }}) \\
&=&\sum\limits_{I\times J\in \mathcal{R}_{0}}|I||J|\left( \psi _{j^{\prime
}k^{\prime }}\ast {\widetilde{\phi }}_{jk}\right) (\left( x_{I^{\prime
}},y_{J^{\prime }}\right) \circ \left( x_{I},y_{J}\right) ^{-1})\phi
_{jk}\ast \left( T_{N}^{-1}(f)\right) (x_{I},y_{J}),
\end{eqnarray*}%
we can repeat the argument in the proof of Theorem \ref{L<H} to obtain 
\begin{eqnarray*}
&&|\left\Vert \left\{ \sum\limits_{j^{\prime },k^{\prime
}}\sum\limits_{I^{\prime },J^{\prime }}|\left( \psi _{j^{\prime }k^{\prime
}}\ast g\right) (x_{I^{\prime }},y_{J^{\prime }})|^{2}\chi _{I^{\prime
}}(z)\chi _{J^{\prime }}(u)\right\} ^{\frac{1}{2}}\right\Vert |_{L^{p_{1}}}
\\
&\leq &C\left\Vert \left\{ \sum\limits_{R=I\times J\in \mathcal{R}_{0}}|\phi
_{jk}\ast \left( T_{N}^{-1}(f)\right) (x_{I},y_{J})|^{2}\chi _{I}\chi
_{J}\right\} ^{\frac{1}{2}}\right\Vert _{p_{1}}.
\end{eqnarray*}%
This shows that for all $0<p_{1}<\infty $ 
\begin{equation*}
||g||_{H_{flag}^{p_{1}}}\leq C\left\Vert \left\{ \sum\limits_{R=I\times J\in 
\mathcal{R}_{0}}|\phi _{jk}\ast \left( T_{N}^{-1}(f)\right)
(x_{I},y_{J})|^{2}\chi _{I}\chi _{J}\right\} ^{\frac{1}{2}}\right\Vert
_{p_{1}}.
\end{equation*}

\medskip

\textbf{Claim 1:} We have 
\begin{equation*}
\int_{S(f)(z,u)\leq \alpha }S^{p_{1}}(f)(z,u)dzdu\geq C\left\Vert \left\{
\sum\limits_{R=I\times J\in \mathcal{R}_{0}}|\phi _{jk}\ast \left(
T_{N}^{-1}(f)\right) (x_{I},y_{J})|^{2}\chi _{I}\chi _{J}\right\} ^{\frac{1}{%
2}}\right\Vert _{p_{1}}.
\end{equation*}%
This claim implies 
\begin{eqnarray*}
||g||_{p_{1}} &\leq &C\int_{S(f)(z,u)\leq \alpha }S^{p_{1}}(f)(z,u)dzdu \\
&\leq &C\alpha ^{p_{1}-p}\int_{S(f)(z,u)\leq \alpha }S^{p}(f)(z,u)dzdu \\
&\leq &C\alpha ^{p_{1}-p}||f||_{H_{flag}^{p}\left( \mathbb{H}^{n}\right)
}^{p}.
\end{eqnarray*}%
To prove Claim 1, we let $R=I\times J\in \mathcal{R}_{0}$. Choose $0<q<p_{1}$
and note that 
\begin{eqnarray*}
&&\int_{S(f)(z,u)\leq \alpha }S^{p_{1}}(f)(z,u)dzdu \\
&=&\int_{S(f)(z,u)\leq \alpha }\left\{
\sum\limits_{j,k}\sum\limits_{I,J}|\phi _{jk}\ast \left(
T_{N}^{-1}(f)\right) (x_{I},y_{J})|^{2}\chi _{I}(z)\chi _{J}(u)\right\} ^{%
\frac{p_{1}}{2}}dzdu \\
&\geq &C\int_{\Omega _{0}^{c}}\left\{ \sum\limits_{R\in \mathcal{R}%
_{0}}|\phi _{jk}\ast \left( T_{N}^{-1}(f)\right) (x_{I},y_{J})|^{2}\chi
_{I}\chi _{J}\right\} ^{\frac{p_{1}}{2}}dzdu \\
&=&C\int_{\mathbb{H}^{n}}\left\{ \sum\limits_{R\in \mathcal{R}_{0}}|\phi
_{jk}\ast \left( T_{N}^{-1}(f)\right) (x_{I},y_{J})|^{2}\chi _{R\cap \Omega
_{0}^{c}}(z,u)\right\} ^{\frac{p_{1}}{2}}dzdu \\
&\geq &C\int_{\mathbb{H}^{n}}\left\{ \left\{ \sum\limits_{R\in \mathcal{R}%
_{0}}\left( M_{S}\left( |\phi _{jk}\ast \left( T_{N}^{-1}(f)\right)
(x_{I},y_{J})|^{q}\chi _{R\cap \Omega _{0}^{c}}\right) (z,u)\right) ^{\frac{2%
}{q}}\right\} ^{\frac{q}{2}}\right\} ^{\frac{p_{1}}{q}}dzdu \\
&\geq &C\int_{\mathbb{H}^{n}}\left\{ \sum\limits_{R\in \mathcal{R}_{0}}|\phi
_{jk}\ast \left( T_{N}^{-1}(f)\right) (x_{I},y_{J})|^{2}\chi
_{R}(z,u)\right\} ^{\frac{p_{1}}{2}}dzdu
\end{eqnarray*}%
In the last inequality above we have used the fact that $|\Omega
_{0}^{c}\cap (I\times J)|\geq \frac{1}{2}|I\times J|$ for $I\times J\in 
\mathcal{R}_{0}$, and thus 
\begin{equation*}
\chi _{R}(z,u)\leq 2^{\frac{1}{q}}M_{S}(\chi _{R\cap \Omega _{0}^{c}})^{%
\frac{1}{q}}(z,u).
\end{equation*}%
In the second to last inequality above we have used the vector-valued
Fefferman-Stein inequality for the strong maximal function 
\begin{equation*}
\left\Vert \left( \sum\limits_{k=1}^{\infty }(M_{S}f_{k})^{r}\right) ^{\frac{%
1}{r}}\right\Vert _{p}\leq C\left\Vert \left( \sum\limits_{k=1}^{\infty
}|f_{k}|^{r}\right) ^{\frac{1}{r}}\right\Vert _{p},
\end{equation*}%
with the exponents $r=2/q>1$ and $p=p_{1}/q>1$. Thus Claim 1 follows.

We now recall $\widetilde{\Omega _{\ell }}=\{(z,u)\in \mathbb{H}%
^{n}:M_{S}(\chi _{\Omega _{\ell }})>\frac{1}{2}\}$.

\medskip

\textbf{Claim 2:} For $p_{2}\leq 1$, 
\begin{equation*}
\left\Vert \sum\limits_{I\times J\in \mathcal{R}_{\ell }}|I||J|{\widetilde{%
\phi }}_{jk}(\left( z,u\right) \circ \left( x_{I},y_{J}\right) ^{-1})\phi
_{jk}\ast \left( T_{N}^{-1}f\right) (x_{I},y_{J})\right\Vert
_{H_{flag}^{p_{2}}}^{p_{2}}\leq C(2^{\ell }\alpha )^{p_{2}}|\widetilde{%
\Omega _{\ell -1}}|.
\end{equation*}%
Claim 2 implies 
\begin{eqnarray*}
||b||_{H_{flag}^{p_{2}}}^{p_{2}} &\leq &\sum\limits_{\ell \geq 1}(2^{\ell
}\alpha )^{p_{2}}|\widetilde{\Omega _{\ell -1}}|\leq C\sum\limits_{\ell \geq
1}(2^{\ell }\alpha )^{p_{2}}|\Omega _{\ell -1}| \\
&\leq &C\int_{S(f)(z,u)>\alpha }S^{p_{2}}f(z,u)dzdu \\
&\leq &C\alpha ^{p_{2}-p}\int_{S(f)(z,u)>\alpha }S^{p}f(z,u)dzdu\leq C\alpha
^{p_{2}-p}||f||_{H_{flag}^{p}}^{p}.
\end{eqnarray*}%
To prove Claim 2, we again have 
\begin{eqnarray*}
&&\left\Vert \sum\limits_{I\times J\in \mathcal{R}_{\ell }}|I||J|{\widetilde{%
\phi }}_{jk}(\left( z,u\right) \circ \left( x_{I},y_{J}\right) ^{-1})\phi
_{jk}\ast \left( T_{N}^{-1}f\right) (x_{I},y_{J})\right\Vert
_{H_{flag}^{p_{2}}}^{p_{2}} \\
&\leq &C|\left\Vert \left\{ \sum\limits_{j^{\prime }k^{\prime
}}\sum\limits_{I^{\prime },J^{\prime }}\left\vert \sum\limits_{I\times J\in 
\mathcal{R}_{\ell }}|I||J|\left( \psi _{j^{\prime }k^{\prime }}\ast {%
\widetilde{\phi }}_{jk}\right) (\left( x_{I^{\prime }},y_{J^{\prime
}}\right) \circ \left( x_{I},y_{J}\right) ^{-1})\phi _{jk}\ast \left(
T_{N}^{-1}f\right) (x_{I},y_{J})\right\vert ^{2}\right\} ^{\frac{1}{2}%
}\right\Vert _{L^{p_{2}}} \\
&\leq &C\left\Vert \left\{ \sum\limits_{R=I\times J\in \mathcal{R}_{\ell
}}|\phi _{jk}\ast \left( T_{N}^{-1}f\right) (x_{I},y_{J})|^{2}\chi _{I}\chi
_{J}\right\} ^{\frac{1}{2}}\right\Vert _{p_{2}}
\end{eqnarray*}%
where we can use an argument similar to that in the proof of Theorem \ref%
{P-P} to prove the last inequality.

However, 
\begin{eqnarray*}
&&\sum\limits_{\ell =1}^{\infty }(2^{\ell }\alpha )^{p_{2}}|\widetilde{%
\Omega }_{\ell -1}| \\
&\geq &\int_{\widetilde{\Omega }_{\ell -1}\backslash \Omega _{\ell
}}S(f)^{p_{2}}(z,u)dzdu \\
&=&\int_{\widetilde{\Omega }_{\ell -1}\backslash \Omega _{\ell }}\left\{
\sum\limits_{j,k}\sum\limits_{I,J}|\phi _{jk}\ast \left(
T_{N}^{-1}(f)\right) (x_{I},y_{J})|^{2}\chi _{I}(z)\chi _{J}(u)\right\} ^{%
\frac{p_{2}}{2}}dzdu \\
&=&\int_{\mathbb{H}^{n}}\left\{ \sum\limits_{j,k}\sum\limits_{I,J}|\phi
_{jk}\ast \left( T_{N}^{-1}(f)\right) (x_{I},y_{J})|^{2}\chi _{(I\times
J)\cap \widetilde{\Omega }_{\ell -1}\backslash \Omega _{\ell
})}(z,u)\right\} ^{\frac{p_{2}}{2}}dzdu \\
&\geq &\int_{\mathbb{H}^{n}}\left\{ \sum\limits_{I\times J\in \mathcal{R}%
_{\ell }}|\phi _{jk}\ast \left( T_{N}^{-1}(f)\right) (x_{I},y_{J})|^{2}\chi
_{(I\times J)\cap \widetilde{\Omega }_{\ell -1}\backslash \Omega _{\ell
})}(z,u)\right\} ^{\frac{p_{2}}{2}}dzdu \\
&\geq &\int_{\mathbb{H}^{n}}\left\{ \sum\limits_{I\times J\in \mathcal{R}%
_{\ell }}|\phi _{jk}\ast \left( T_{N}^{-1}(f)\right) (x_{I},y_{J})|^{2}\chi
_{I}(z)\chi _{J}(u)\right\} ^{\frac{p_{2}}{2}}dzdu
\end{eqnarray*}

In the above string of inequalities, we have used the fact that for $R\in 
\mathcal{R}_{\ell }$ we have 
\begin{equation*}
|R\cap \Omega _{\ell -1}|>\frac{1}{2}|R|\,\,\,\text{and}\,\,|R\cap \Omega
_{\ell }|\leq \frac{1}{2}|R|
\end{equation*}%
and consequently $R\subset \widetilde{\Omega }_{\ell -1}$. Therefore $|R\cap
(\widetilde{\Omega }_{\ell -1}\backslash \Omega _{\ell })|>\frac{1}{2}|R|$.
Thus the same argument applies here to conclude the last inequality above.
Finally, since $L^{2}(\mathbb{H}^{n})$ is dense in $H_{flag}^{p}\left( 
\mathbb{H}^{n}\right) $, Theorem \ref{CZ decomposition} is proved.
\end{proof}

We are now ready to prove the interpolation theorem on Hardy spaces $%
H_{flag}^{p}$ for all $0<p<\infty $.

\begin{proof}
\textbf{(of Theorem \ref{interpolation}):} Suppose that $T$ is bounded from $%
H_{flag}^{p_{2}}$ to $L^{p_{2}}$ and from $H_{flag}^{p_{1}}$ to $L^{p_{1}}$.
For any given $\lambda >0$ and $f\in H_{flag}^{p}$, by the Calder\'{o}%
n-Zygmund decomposition, 
\begin{equation*}
f(z,u)=g(z,u)+b(z,u)
\end{equation*}%
with 
\begin{equation*}
||g||_{H_{flag}^{p_{1}}}^{p_{1}}\leq C\lambda
^{p_{1}-p}||f||_{H_{flag}^{p}}^{p}\,\,\,\text{and}\,%
\,||b||_{H_{flag}^{p_{2}}}^{p_{2}}\leq C\lambda
^{p_{2}-p}||f||_{H_{flag}^{p}}^{p}.
\end{equation*}

Moreover, we have proved the estimates 
\begin{equation*}
||g||_{H_{flag}^{p_{1}}}^{p_{1}}\leq C\int_{S(f)(z,u)\leq \alpha
}S(f)^{p_{1}}(z,u)dzdu
\end{equation*}%
and 
\begin{equation*}
||b||_{H_{flag}^{p_{2}}}^{p_{2}}\leq C\int_{S(f)(z,u)>\alpha
}S(f)^{p_{2}}(z,u)dzdu
\end{equation*}%
which implies that 
\begin{eqnarray*}
||Tf||_{p}^{p} &=&p\int_{0}^{\infty }\alpha ^{p-1}|\left\{
(z,u):|Tf(z,u)|>\lambda \right\} |d\alpha \\
&\leq &p\int_{0}^{\infty }\alpha ^{p-1}|\left\{ (z,u):|Tg(z,u)|>\frac{%
\lambda }{2}\right\} |d\alpha +p\int_{0}^{\infty }\alpha ^{p-1}|\left\{
(z,u):|Tb(z,u)|>\frac{\lambda }{2}\right\} |d\alpha \\
&\leq &p\int_{0}^{\infty }\alpha ^{p-1}\int_{S(f)(z,u)\leq \alpha
}S(f)^{p_{1}}(z,u)dzdud\alpha +p\int_{0}^{\infty }\alpha
^{p-1}\int_{S(f)(z,u)>\alpha }S(f)^{p_{2}}(z,u)dzdud\alpha \\
&\leq &C||f||_{H_{flag}^{p}}^{p}.
\end{eqnarray*}%
Thus, 
\begin{equation*}
||Tf||_{p}\leq C||f||_{H_{flag}^{p}}
\end{equation*}%
for any $p_{2}<p<p_{1}$. Hence, $T$ is bounded from $H_{flag}^{p}$ to $L^{p}$%
.

Now we prove the second assertion that $T$ is bounded on $H_{flag}^{p}$ for $%
p_{2}<p<p_{1}$. For any given $\lambda >0$ and $f\in H_{flag}^{p}$, we have
by the Calder\'{o}n-Zygmund decomposition again, 
\begin{eqnarray*}
&&|\left\{ (z,u):|g(Tf)(z,u)|>\alpha \right\} | \\
&\leq &\left\vert \left\{ (z,u):|g(Tg)(z,u)|>\frac{\alpha }{2}\right\}
\right\vert +\left\vert \left\{ (z,u):|g(Tb)(z,u)|>\frac{\alpha }{2}\right\}
\right\vert \\
&\leq &C\alpha ^{-p_{1}}||Tg||_{H_{flag}^{p_{1}}}^{p_{1}}+C\alpha
^{-p_{2}}||Tb||_{H_{flag}^{p_{2}}}^{p_{2}} \\
&\leq &C\alpha ^{-p_{1}}||g||_{H_{flag}^{p_{1}}}^{p_{1}}+C\alpha
^{-p_{2}}||b||_{H_{flag}^{p_{2}}}^{p_{2}} \\
&\leq &C\alpha ^{-p_{1}}\int_{S(f)(z,u)\leq \alpha
}(Sf)^{p_{1}}(z,u)dzdu+C\alpha ^{-p_{2}}\int_{S(f)(z,u)>\alpha
}(Sf)^{p_{2}}(z,u)dzdu,
\end{eqnarray*}%
which, as above, shows that $||Tf||_{H_{flag}^{p}}\leq C||g(TF)||_{p}\leq
C||f||_{H_{flag}^{p}}$ for any $p_{2}<p<p_{1}$.
\end{proof}

\section{Embeddings and quotients of flag and moment molecular spaces\label%
{s embeddings}}

Our purpose in this final section of Part 2 is to give the proof of Lemma %
\ref{containments}, and then prove the inclusion $Q_{flag}^{p}\left( \mathbb{%
H}^{n}\right) \hookrightarrow Q^{p}\left( \mathbb{H}^{n}\right) $ of the
quotient spaces 
\begin{eqnarray*}
Q_{flag}^{p}\left( \mathbb{H}^{n}\right) &\equiv &H_{flag}^{p}\left( \mathbb{%
H}^{n}\right) /\mathsf{M}_{F}^{M^{\prime }+\delta ,M_{1}^{\prime
},M_{2}^{\prime }}\left( \mathbb{H}^{n}\right) ^{\bot }, \\
Q^{p}\left( \mathbb{H}^{n}\right) &\equiv &H^{p}\left( \mathbb{H}^{n}\right)
/\mathsf{M}_{F}^{M^{\prime }+\delta ,M_{1}^{\prime },M_{2}^{\prime }}\left( 
\mathbb{H}^{n}\right) ^{\bot }.
\end{eqnarray*}

\subsection{Proof of Lemma \protect\ref{containments}}

We begin with the second containment of Lemma \ref{containments}. Suppose
that $f\in \mathcal{M}_{flag}^{M,M_{1},M_{2}}\left( \mathbb{H}^{n}\right) $
and that $F\in \mathcal{M}_{product}^{M,M_{1},M_{2}}\left( \mathbb{H}%
^{n}\times \mathbb{R}\right) $ satisfies%
\begin{equation}
f=\pi F\text{ and }\left\Vert F\right\Vert _{\mathcal{M}%
_{product}^{M,M_{1},M_{2}}\left( \mathbb{H}^{n}\times \mathbb{R}\right)
}\leq 2\left\Vert f\right\Vert _{\mathcal{M}_{flag}^{M,M_{1},M_{2}}\left( 
\mathbb{H}^{n}\right) }.  \label{near optimal}
\end{equation}%
We first verify the moment conditions required for membership in $\mathsf{M}%
_{F}^{M,M_{1},M_{2}}\left( \mathbb{H}^{n}\right) $\textbf{. }For $\left\vert
\alpha \right\vert \leq M_{1}$ and $\left\vert \alpha \right\vert +2\beta
\leq 2M_{1}+2$, we have that 
\begin{eqnarray*}
\int_{\mathbb{H}^{n}}z^{\alpha }u^{\beta }f\left( z,u\right) dzdu &=&\int_{%
\mathbb{H}^{n}}z^{\alpha }u^{\beta }\int_{\mathbb{R}}F\left( \left(
z,u-v\right) ,v\right) dvdzdu \\
&=&\int_{\mathbb{H}^{n}}z^{\alpha }\left( u+v\right) ^{\beta }\int_{\mathbb{R%
}}F\left( \left( z,u\right) ,v\right) dvdzdu \\
&=&\sum_{\beta =\gamma +\delta }c_{\gamma ,\delta }\int_{\mathbb{H}%
^{n}}z^{\alpha }u^{\gamma }v^{\delta }\int_{\mathbb{R}}F\left( \left(
z,u\right) ,v\right) dvdzdu \\
&=&\left\{ \sum_{\beta =\gamma +\delta :\left\vert \alpha \right\vert
+2\gamma \leq M}+\sum_{\beta =\gamma +\delta :2\delta \leq M}\right\}
c_{\gamma ,\delta }\int_{\mathbb{H}^{n}}z^{\alpha }u^{\gamma }\int_{\mathbb{R%
}}v^{\delta }F\left( \left( z,u\right) ,v\right) dvdzdu \\
&=&\sum_{\beta =\gamma +\delta :\left\vert \alpha \right\vert +2\gamma \leq
M}c_{\gamma ,\delta }\int_{\mathbb{R}}v^{\delta }\left\{ \int_{\mathbb{H}%
^{n}}z^{\alpha }u^{\gamma }F\left( \left( z,u\right) ,v\right) dzdu\right\}
dv \\
&&+\sum_{\beta =\gamma +\delta :2\delta \leq M}c_{\gamma ,\delta }\int_{%
\mathbb{H}^{n}}z^{\alpha }u^{\gamma }\left\{ \int_{\mathbb{R}}v^{\delta
}F\left( \left( z,u\right) ,v\right) dv\right\} dzdu
\end{eqnarray*}%
vanishes, since%
\begin{equation*}
\int_{\mathbb{H}^{n}}z^{\alpha }u^{\gamma }F\left( \left( z,u\right)
,v\right) dzdu=0\text{ if }\left\vert \alpha \right\vert +2\gamma \leq M_{1},
\end{equation*}%
and%
\begin{equation*}
\int_{\mathbb{R}}v^{\delta }F\left( \left( z,u\right) ,v\right) dv=0\text{
if }2\delta \leq M_{1}.
\end{equation*}%
Note that since $\left\vert \alpha \right\vert \leq M_{1}$ and $\left\vert
\alpha \right\vert +2\beta \leq 2M_{1}+2$, then%
\begin{equation*}
\left\vert \alpha \right\vert +2\gamma +2\delta =\left\vert \alpha
\right\vert +2\beta \leq 2M_{1}+2
\end{equation*}%
implies that at least one of the inequalities $\left\vert \alpha \right\vert
+2\gamma \leq M_{1}$ and $2\delta \leq M_{1}$ must hold (if they both fail,
then $\left\vert \alpha \right\vert +2\gamma \geq M_{1}+1$ and $2\delta \geq
M_{1}+2$).

Next, for $2\gamma \leq M_{1}$ we have that 
\begin{eqnarray*}
\int_{\mathbb{R}}u^{\gamma }f\left( z,u\right) du &=&\int_{\mathbb{R}%
}u^{\gamma }\int_{\mathbb{R}}F\left( \left( z,u-v\right) ,v\right) dvdu \\
&=&\int_{\mathbb{R}}\left( u+v\right) ^{\gamma }\int_{\mathbb{R}}F\left(
\left( z,u\right) ,v\right) dvdu \\
&=&\sum_{\gamma =\delta +\eta }c_{\delta ,\eta }\int_{\mathbb{R}}u^{\delta
}\left\{ \int_{\mathbb{R}}v^{\eta }F\left( \left( z,u\right) ,v\right)
dv\right\} du
\end{eqnarray*}%
vanishes since%
\begin{equation*}
\int_{\mathbb{R}}v^{\eta }F\left( \left( z,u\right) ,v\right) dv=0\text{ if }%
2\eta \leq 2\gamma \leq M_{1}.
\end{equation*}

We now turn to proving the norm inequality%
\begin{equation*}
\left\Vert f\right\Vert _{\mathsf{M}_{F}^{M,M_{1},M_{2}}\left( \mathbb{H}%
^{n}\right) }\lesssim \left\Vert F\right\Vert _{\mathcal{M}%
_{product}^{M,M_{1},M_{2}}\left( \mathbb{H}^{n}\times \mathbb{R}\right) }.
\end{equation*}%
For $\left\vert \alpha \right\vert +2\beta \leq M_{2}$ we have 
\begin{eqnarray*}
\left\vert \partial _{z}^{\alpha }\partial _{u}^{\beta }f\left( z,u\right)
\right\vert &=&\left\vert \partial _{z}^{\alpha }\partial _{u}^{\beta }\pi
F\left( z,u\right) \right\vert =\left\vert \int_{\mathbb{R}}\partial
_{z}^{\alpha }\partial _{u}^{\beta }F\left( \left( z,u-v\right) ,v\right)
dv\right\vert \\
&\lesssim &\int_{\mathbb{R}}\frac{1}{\left( 1+\left\vert z\right\vert
^{2}+\left\vert u-v\right\vert \right) ^{\frac{Q+M+\left\vert \alpha
\right\vert +2\beta }{2}}}\frac{1}{\left( 1+\left\vert v\right\vert \right)
^{1+M}}dv \\
&\lesssim &\frac{1}{\left( 1+\left\vert z\right\vert ^{2}+\left\vert
u\right\vert \right) ^{\frac{Q+M+\left\vert \alpha \right\vert +2\beta }{2}}}%
.
\end{eqnarray*}%
The first difference inequality is proved in the same way. Together with (%
\ref{near optimal}), these inequalities complete the proof that $\mathcal{M}%
_{flag}^{M,M_{1},M_{2}}\left( \mathbb{H}^{n}\right) $ is continuously
embedded in $\mathsf{M}_{F}^{M,M_{1},M_{2}}\left( \mathbb{H}^{n}\right) $.

\begin{remark}
We also have the following differential inequalities for $f\in \mathcal{M}%
_{flag}^{M,M_{1},M_{2}}\left( \mathbb{H}^{n}\right) $ with \emph{more}
derivatives but \emph{less} decay:%
\begin{eqnarray*}
&&\left\vert \partial _{z}^{\alpha }\partial _{u}^{\beta }f\left( z,u\right)
\right\vert \leq A\frac{1}{\left( 1+\left\vert z\right\vert ^{2}+\left\vert
u\right\vert \right) ^{\frac{Q+\widetilde{M}+\left\vert \alpha \right\vert
+2\beta }{2}}}, \\
&&\ \ \ \ \ \text{for all }\left\vert \alpha \right\vert \leq
M_{2},\left\vert \alpha \right\vert +2\beta \leq 2M_{2}\text{,} \\
&&\ \ \ \ \ \text{and where }\widetilde{M}=M-M_{2}.
\end{eqnarray*}%
Indeed, it is enough to prove the case $M_{2}\leq \left\vert \alpha
\right\vert +2\beta \leq 2M_{2}$, and so we write $\beta =\gamma +\delta $
where $\left\vert \alpha \right\vert +2\gamma =M_{2}$ and $2\delta \leq
M_{2} $. Now for any suitable function $G\left( z,u,v\right) $ we have%
\begin{eqnarray*}
\frac{\partial }{\partial v}G\left( z,u-v,v\right) &=&-G_{2}\left(
z,u-v,v\right) +G_{3}\left( z,u-v,v\right) , \\
\int_{\mathbb{R}}G_{2}\left( z,u-v,v\right) dv &=&\int_{\mathbb{R}%
}G_{3}\left( z,u-v,v\right) dv,
\end{eqnarray*}%
and iterating with $G=F\left( \left( z,u\right) ,v\right) $ we obtain%
\begin{eqnarray*}
\partial _{z}^{\alpha }\partial _{u}^{\beta }f\left( z,u\right) &=&\int_{%
\mathbb{R}}\partial _{z}^{\alpha }\partial _{2}^{\beta }F\left( \left(
z,u-v\right) ,v\right) dv \\
&=&\left( -1\right) ^{\delta }\int_{\mathbb{R}}\partial _{z}^{\alpha
}\partial _{2}^{\gamma }\partial _{3}^{\delta }F\left( \left( z,u-v\right)
,v\right) dv.
\end{eqnarray*}%
Thus we have%
\begin{eqnarray*}
\left\vert \partial _{z}^{\alpha }\partial _{u}^{\beta }f\left( z,u\right)
\right\vert &=&\left\vert \int_{\mathbb{R}}\partial _{z}^{\alpha }\partial
_{2}^{\gamma }\partial _{3}^{\delta }F\left( \left( z,u-v\right) ,v\right)
dv\right\vert \\
&\lesssim &\int_{\mathbb{R}}\frac{1}{\left( 1+\left\vert z\right\vert
^{2}+\left\vert u-v\right\vert \right) ^{\frac{Q+M+\left\vert \alpha
\right\vert +2\gamma }{2}}}\frac{1}{\left( 1+\left\vert v\right\vert \right)
^{1+M+\delta }}dv \\
&\lesssim &\frac{1}{\left( 1+\left\vert z\right\vert ^{2}+\left\vert
u-v\right\vert \right) ^{\frac{Q+M+M_{2}}{2}}} \\
&=&\frac{1}{\left( 1+\left\vert z\right\vert ^{2}+\left\vert u-v\right\vert
\right) ^{\frac{Q+\widetilde{M}+2M_{2}}{2}}} \\
&\leq &\frac{1}{\left( 1+\left\vert z\right\vert ^{2}+\left\vert
u-v\right\vert \right) ^{\frac{Q+\widetilde{M}+\left\vert \alpha \right\vert
+2\beta }{2}}}.
\end{eqnarray*}%
Note that the extra decay in the factor $\frac{1}{\left( 1+\left\vert
v\right\vert \right) ^{1+M+\delta }}$\ is `lost' here when we project.
\end{remark}

Now we turn to proving the first, and more difficult, containment in Lemma %
\ref{containments}. So suppose that $f\in \mathsf{M}%
_{F}^{3M+M_{2},M_{1},2M_{2}+4}\left( \mathbb{H}^{n}\right) $. We first
decompose $f$ according to annuli in $\mathbb{H}^{n}$. Let $1=\phi
_{0}\left( t\right) +\sum_{m=0}^{\infty }\phi _{m}\left( t\right) $ for $%
t\in \left[ 0,\infty \right) $ where $\phi _{0}$ is supported in $\left[ 0,1%
\right] $, $\phi _{m}$ is supported in $\left[ 2^{m-1},2^{m+1}\right] $ for $%
m\geq 1$, and 
\begin{equation*}
\left\vert \left( \frac{d}{dt}\right) ^{j}\phi _{m}\right\vert \leq
C_{j}2^{-jm},\ \ \ \ \ \text{for all }j,m\geq 0.
\end{equation*}%
Now set%
\begin{eqnarray*}
\varphi _{m}\left( z,u\right) &=&\phi _{m}\left( \sqrt{1+\left\vert
z\right\vert ^{4}+u^{2}}\right) , \\
f_{m}\left( z,u\right) &=&\varphi _{m}\left( z,u\right) f\left( z,u\right)
=\varphi _{m}f\left( z,u\right) ,
\end{eqnarray*}%
so that%
\begin{equation*}
f\left( z,u\right) =\sum_{m=0}^{\infty }\varphi _{m}f\left( z,u\right)
=\sum_{m=0}^{\infty }f_{m}\left( z,u\right) .
\end{equation*}

We now decompose each $f_{m}\left( z,u\right) $ in the $u$ variable using
the Calder\'{o}n reproducing formula in $\mathbb{R}$. Let $\chi $ and $\psi $
be smooth functions on the real line $\mathbb{R}$ supported in $\left[ -1,1%
\right] $, satisfying $\int \chi \left( x\right) dx=1$ and $\int x^{\beta
}\psi \left( x\right) dx=0$ for $0\leq \beta \leq M_{1}$, and for each $m$
the reproducing formula%
\begin{equation*}
\delta _{0}=\chi _{-m}\ast \chi _{-m}+\sum_{k=-m+1}^{\infty }\psi _{k}\ast
\psi _{k},
\end{equation*}%
where $\psi _{k}\left( v\right) =2^{k}\psi \left( 2^{k}v\right) $ and $\chi
_{-m}\left( v\right) =2^{-m}\chi \left( 2^{-m}v\right) $. Let $f_{m,z}\left(
u\right) =f_{m}\left( z,u\right) $ and define 
\begin{eqnarray*}
F_{m}\left( \left( z,u\right) ,v\right) &=&\left( f_{m,z}\ast \chi \right)
\left( u\right) \chi \left( v\right) +\sum_{k=-m+1}^{\infty }\left(
f_{m,z}\ast \psi _{k}\right) \left( u\right) \psi _{k}\left( v\right) \\
&=&\left( f_{m}\ast _{2}\chi \right) \left( z,u\right) \chi \left( v\right)
+\sum_{k=-m+1}^{\infty }\left( f_{m}\ast _{2}\psi _{k}\right) \left(
z,u\right) \psi _{k}\left( v\right) ,
\end{eqnarray*}%
for $\left( \left( z,u\right) ,v\right) \in \mathbb{H}^{n}\times \mathbb{R}$%
. Note that 
\begin{eqnarray*}
\pi F_{m}\left( z,u\right) &=&\int_{\mathbb{R}}F_{m}\left( \left(
z,u-v\right) ,v\right) dv \\
&=&f_{m,z}\ast \chi \ast \chi \left( u\right) +f_{m,z}\ast \left(
\sum_{k=-m+1}^{\infty }\psi _{k}\ast \psi _{k}\right) \left( u\right) \\
&=&f_{m,z}\ast \delta _{0}\left( u\right) =f_{m}\left( z,u\right) .
\end{eqnarray*}

Now $\psi _{k}$ has vanishing moments but $\chi _{-m}$ does not. We remedy
this lack of vanishing moments for $\chi _{-m}$ by noting that for $0\leq
2\ell \leq M_{1}+2$ we have by integration by parts in $v$, 
\begin{eqnarray*}
\pi F_{m}\left( z,u\right) &=&\int_{\mathbb{R}}I^{\ell }\left( f_{m}\ast
_{2}\chi _{-m}\right) \left( z,u-v\right) D^{\ell }\chi _{-m}\left( v\right)
dv \\
&&+\sum_{k=-m+1}^{\infty }\int_{\mathbb{R}}\left( f_{m}\ast _{2}\psi
_{k}\right) \left( z,u-v\right) \psi _{k}\left( v\right) dv,
\end{eqnarray*}%
where $D^{\ell }\chi _{-m}\left( v\right) =\left( \frac{d}{dv}\right) ^{\ell
}\chi _{-m}\left( v\right) $ and $I^{\ell }g\left( z,t\right) =\ell
\int_{-\infty }^{t}g\left( z,s\right) \left( t-s\right) ^{\ell -1}ds$.
Observe that $\lim_{t\rightarrow \infty }I^{\ell }g\left( z,t\right) =0$
provided $g\left( z,s\right) $ satisfies $\int_{\mathbb{R}}g\left(
z,s\right) s^{j}ds=0$ for $0\leq j\leq \ell -1$. We now fix%
\begin{equation*}
\ell =\frac{M_{1}}{2}+1,
\end{equation*}%
but will continue to write $\ell $ for clarity. Then with%
\begin{equation}
F_{m}^{\ell }\left( \left( z,u\right) ,v\right) \equiv I^{\ell }\left(
f_{m}\ast _{2}\chi _{-m}\right) \left( z,u\right) D^{\ell }\chi _{-m}\left(
v\right) +\sum_{k=-m+1}^{\infty }\left( f_{m}\ast _{2}\psi _{k}\right)
\left( z,u\right) \psi _{k}\left( v\right) ,  \label{Flr}
\end{equation}%
we have%
\begin{equation*}
\pi F_{m}^{\ell }\left( z,u\right) =\pi F_{m}\left( z,u\right) .
\end{equation*}

Now we check that $F_{m}^{\ell }$ satisfies the moment conditions required
for membership in $\mathcal{M}_{product}^{M,M_{1},M_{2}}\left( \mathbb{H}%
^{n}\times \mathbb{R}\right) $. We write $F_{m}^{\ell }=G_{m}^{\ell }+H_{m}$
where%
\begin{equation*}
G_{m}^{\ell }=I^{\ell }\left( f_{m}\ast _{2}\chi _{-m}\right) \left(
z,u\right) D^{\ell }\chi _{-m}\left( v\right) \text{ and }%
H_{m}=\sum_{k=-m+1}^{\infty }\left( f_{m}\ast _{2}\psi _{k}\right) \left(
z,u\right) \psi _{k}\left( v\right) .
\end{equation*}%
It is clear that $H_{m}$ satisfies (more than) the required number of
vanishing moments since the functions $\psi _{k}$ each have $M_{1}$
vanishing moments, so we turn our attention to $G_{m}^{\ell }$. For $%
\left\vert \alpha \right\vert +2\beta \leq M_{1}$ we have that%
\begin{eqnarray*}
\int_{\mathbb{H}^{n}}z^{\alpha }u^{\beta }G_{m}^{\ell }\left( \left(
z,u\right) ,v\right) dzdu &=&\int_{\mathbb{H}^{n}}z^{\alpha }u^{\beta
}I^{\ell }\left( f_{m}\ast _{2}\chi _{-m}\right) \left( z,u\right) D^{\ell
}\chi _{-m}\left( v\right) dzdu \\
&=&D^{\ell }\chi _{-m}\left( v\right) \int_{\mathbb{H}^{n}}z^{\alpha
}\left\{ I^{\ell }u^{\beta }\right\} \left( f_{m}\ast _{2}\chi _{-m}\right)
\left( z,u\right) dzdu \\
&=&c_{\ell ,\beta }D^{\ell }\chi _{-m}\left( v\right) \int_{\mathbb{H}%
^{n}}z^{\alpha }u^{\beta +\ell }\left\{ \int_{\mathbb{R}}f_{m}\left(
z,u-s\right) \chi _{-m}\left( s\right) ds\right\} dzdu \\
&=&c_{\ell ,\beta }D^{\ell }\chi _{-m}\left( v\right) \int_{\mathbb{R}%
}\left\{ \int_{\mathbb{H}^{n}}z^{\alpha }u^{\beta +\ell }f\left(
z,u-s\right) dzdu\right\} \chi _{-m}\left( s\right) ds
\end{eqnarray*}%
vanishes since%
\begin{equation*}
\int_{\mathbb{H}^{n}}z^{\alpha }u^{\beta +\ell }f\left( z,u-s\right)
dzdu=\int_{\mathbb{H}^{n}}z^{\alpha }u^{\beta +\ell }f\left( z,u\right)
dzdu=0
\end{equation*}%
when%
\begin{equation*}
\left\vert \alpha \right\vert +2\left( \beta +\ell \right) =\left\vert
\alpha \right\vert +2\beta +M_{1}+2\leq 2M_{1}+2.
\end{equation*}%
Also, for $2\gamma \leq M_{1}$ we have

\begin{eqnarray*}
\int_{\mathbb{R}}v^{\gamma }G_{m}^{\ell }\left( \left( z,u\right) ,v\right)
dv &=&\int_{\mathbb{R}}v^{\gamma }I^{\ell }f_{m}\left( z,u\right) D^{\ell
}\chi _{-m}\left( v\right) dv \\
&=&I^{\ell }f_{m}\left( z,u\right) \left( -1\right) ^{\ell }\int_{\mathbb{R}%
}\left( D^{\ell }v^{\gamma }\right) \chi _{-m}\left( v\right) dv
\end{eqnarray*}%
vanishes since $\gamma \leq \frac{M_{1}}{2}<\frac{M_{1}}{2}+1=\ell $.

\bigskip

It remains to verify the norm inequality%
\begin{equation}
\left\Vert F^{\ell }\right\Vert _{\mathcal{M}_{product}^{M,M_{1},M_{2}}%
\left( \mathbb{H}^{n}\times \mathbb{R}\right) }\lesssim \left\Vert
f\right\Vert _{\mathsf{M}_{F}^{3M+M_{2},M_{1},2M_{2}+4}\left( \mathbb{H}%
^{n}\right) }.  \label{flag norm estimate}
\end{equation}%
Our first task is to prove the differential inequalities%
\begin{eqnarray*}
&&\left\vert \partial _{z}^{\alpha }\partial _{u}^{\beta }\partial
_{v}^{\gamma }F^{\ell }\left( \left( z,u\right) ,v\right) \right\vert \leq A%
\frac{1}{\left( 1+\left\vert z\right\vert ^{2}+\left\vert u\right\vert
\right) ^{\frac{Q+M+\left\vert \alpha \right\vert +2\beta }{2}}}\frac{1}{%
\left( 1+\left\vert v\right\vert \right) ^{1+M+\gamma }}, \\
&&\ \ \ \ \ \ \ \ \ \ \text{for all }\left\vert \alpha \right\vert +2\beta
\leq M_{2}\text{ and }2\gamma \leq M_{2}.
\end{eqnarray*}%
We have%
\begin{eqnarray*}
\partial _{z}^{\alpha }\partial _{u}^{\beta }I^{\ell }\left( f_{m}\ast
_{2}\chi _{-m}\right) \left( z,u\right) &=&\partial _{z}^{\alpha }I^{\ell
-\beta }\left( f_{m}\ast _{2}\chi _{-m}\right) \left( z,u\right) \\
&=&\left[ \left( \partial _{z}^{\alpha }\partial _{u}^{\beta }I^{\ell
}f_{m}\right) \ast _{2}\chi _{-m}\right] \left( z,u\right) ,
\end{eqnarray*}%
and%
\begin{equation*}
\left\vert \partial _{z}^{\alpha }\partial _{u}^{\beta }I^{\ell }f_{m}\left(
z,u\right) \right\vert =\left\vert \partial _{z}^{\alpha }\partial
_{u}^{\beta -\ell }f_{m}\left( z,u\right) \right\vert \lesssim \frac{1}{%
\left( 1+\left\vert z\right\vert ^{2}+\left\vert u\right\vert \right) ^{%
\frac{Q+3M+M_{2}+\left\vert \alpha \right\vert +2\beta -2\ell }{2}}},
\end{equation*}%
for $\left\vert \alpha \right\vert +2\left( \beta -\ell \right) \leq
M_{2}\leq 2M_{2}+4$ since $f\in \mathsf{M}_{F}^{3M+M_{2},M_{1},2M_{2}+4}%
\left( \mathbb{H}^{n}\right) $. This inequality follows immediately from the
definition if $\beta \geq \ell $, and the case $\beta <\ell $ is an easy
exercise using that $2\left( \ell -1\right) =M_{1}$.

This leads to%
\begin{eqnarray*}
\left\vert \partial _{z}^{\alpha }\partial _{u}^{\beta }\partial
_{v}^{\gamma }G_{m}^{\ell }\left( \left( z,u\right) ,v\right) \right\vert
&=&\left\vert \partial _{z}^{\alpha }\partial _{u}^{\beta }I^{\ell }\left(
f_{m}\ast _{2}\chi _{-m}\right) \left( z,u\right) \right\vert \left\vert
\partial _{\nu }^{\gamma }D^{\ell }\chi _{-m}\left( v\right) \right\vert \\
&\lesssim &A\chi _{\left\{ 1+\left\vert z\right\vert ^{2}+\left\vert
u\right\vert \approx 2^{m}\right\} }\frac{1}{\left( 1+\left\vert
z\right\vert ^{2}+\left\vert u\right\vert \right) ^{\frac{%
Q+3M+M_{2}+\left\vert \alpha \right\vert +2\beta }{2}-\ell }} \\
&&\times \frac{\left( 2^{-m}\right) ^{\ell +\gamma +1}}{\left(
1+2^{-m}\left\vert v\right\vert \right) ^{1+M+\gamma }},
\end{eqnarray*}%
since $\left\vert \partial _{v}^{\gamma }D^{\ell }\chi _{-m}\left( v\right)
\right\vert \lesssim \chi _{\left\{ \left\vert v\right\vert \leq
2^{m}\right\} }\left( 2^{-m}\right) ^{\ell +\gamma +1}$ follows from the
fact that $2^{-m}\left\vert v\right\vert \leq 1$ if $\chi _{-m}\left(
v\right) \neq 0$. Using that $1+\left\vert z\right\vert ^{2}+\left\vert
u\right\vert \approx 2^{m}$ we thus obtain the estimate%
\begin{eqnarray*}
&&\left\vert \partial _{z}^{\alpha }\partial _{u}^{\beta }\partial
_{v}^{\gamma }G_{m}^{\ell }\left( \left( z,u\right) ,v\right) \right\vert \\
&\lesssim &A\chi _{\left\{ 1+\left\vert z\right\vert ^{2}+\left\vert
u\right\vert \approx 2^{m}\right\} }\frac{1}{\left( 1+\left\vert
z\right\vert ^{2}+\left\vert u\right\vert \right) ^{\frac{Q+M+\left\vert
\alpha \right\vert +2\beta }{2}}\left( 2^{m}\right) ^{M+\frac{M_{2}}{2}-\ell
}}\chi _{\left\{ \left\vert v\right\vert \leq 2^{m}\right\} }\left(
2^{-m}\right) ^{\ell +\gamma +1} \\
&=&A\chi _{\left\{ 1+\left\vert z\right\vert ^{2}+\left\vert u\right\vert
\approx 2^{m}\right\} }\frac{1}{\left( 1+\left\vert z\right\vert
^{2}+\left\vert u\right\vert \right) ^{\frac{Q+M+\left\vert \alpha
\right\vert +2\beta }{2}}}\chi _{\left\{ \left\vert v\right\vert \leq
2^{m}\right\} }\left( 2^{-m}\right) ^{1+M+\gamma +\frac{M_{2}}{2}} \\
&\lesssim &A\chi _{\left\{ 1+\left\vert z\right\vert ^{2}+\left\vert
u\right\vert \approx 2^{m}\right\} }\frac{1}{\left( 1+\left\vert
z\right\vert ^{2}+\left\vert u\right\vert \right) ^{\frac{Q+M+\left\vert
\alpha \right\vert +2\beta }{2}}}\frac{1}{\left( 1+\left\vert v\right\vert
\right) ^{1+M+\gamma +\frac{M_{2}}{2}}},
\end{eqnarray*}%
in which the power of $1+\left\vert v\right\vert $ is \emph{better} by $%
\frac{M_{2}}{2}$.

\begin{remark}
The above inequality uses $3M$\ decay for $f$ in the space $%
M_{F}^{3M+M_{2},M_{1},2M_{2}+4}\left( \mathbb{H}^{n}\right) $. The extra
decay of $M_{2}$ is not needed here.
\end{remark}

To estimate $\left\vert \partial _{z}^{\alpha }\partial _{u}^{\beta
}\partial _{v}^{\gamma }H_{m}\left( \left( z,u\right) ,v\right) \right\vert $
we first note that by Taylor's formula there is $0<\theta <1$ such that%
\begin{eqnarray*}
&&\left\vert \partial _{z}^{\alpha }\partial _{u}^{\beta }\left( f_{m}\ast
_{2}\psi _{k}\right) \left( z,u\right) \right\vert =\left\vert \int_{\mathbb{%
R}}\partial _{z}^{\alpha }\partial _{u}^{\beta }f_{m,z}\left( u-w\right)
\psi _{k}\left( w\right) dw\right\vert \\
&&\ \ \ \ \ =\left\vert \int_{\mathbb{R}}\left\{ \sum_{j=0}^{\gamma
+1}\left( \partial _{z}^{\alpha }\partial _{u}^{\beta }f_{m,z}\right)
^{\left( j\right) }\left( u\right) \frac{\left( -w\right) ^{j}}{j!}+\left(
\partial _{z}^{\alpha }\partial _{u}^{\beta }f_{m,z}\right) ^{\left( \gamma
+2\right) }\left( u-\theta w\right) \frac{\left( -w\right) ^{\gamma +2}}{%
\left( \gamma +2\right) !}\right\} \psi _{k}\left( w\right) dw\right\vert \\
&&\ \ \ \ \ =\left\vert \int_{\mathbb{R}}\left( \partial _{z}^{\alpha
}\partial _{u}^{\beta }f_{m,z}\right) ^{\left( \gamma +2\right) }\left(
u-\theta w\right) \frac{\left( -w\right) ^{\gamma +2}}{\left( \gamma
+2\right) !}\psi _{k}\left( w\right) dw\right\vert \\
&&\ \ \ \ \ \lesssim \frac{2^{-k\left( \gamma +2\right) }}{\left(
1+\left\vert z\right\vert ^{2}+\left\vert u\right\vert \right) ^{\frac{%
Q+3M+\left\vert \alpha \right\vert +2\beta +2\gamma }{2}+2}},
\end{eqnarray*}%
since $\int w^{m}\psi _{k}\left( w\right) dw=0$ for $m\leq \gamma +1$, $%
\left\vert w\right\vert ^{\gamma +2}\leq 2^{-k\left( \gamma +2\right) }$ if $%
\psi _{k}\left( w\right) \neq 0$, and $\int_{\mathbb{R}}\left\vert \psi
_{k}\left( w\right) \right\vert dw\leq C$. Here we have used the
differential inequality for the function 
\begin{equation*}
\left( \partial _{z}^{\alpha }\partial _{u}^{\beta }f_{z}\right) ^{\left(
\gamma +2\right) }\left( v\right) =\partial _{z}^{\alpha }\partial
_{u}^{\beta +\gamma +2}f\left( z,u\right) ,
\end{equation*}%
which holds since $f\in \mathsf{M}_{F}^{3M+M_{2},M_{1},2M_{2}+4}\left( 
\mathbb{H}^{n}\right) $ and%
\begin{equation*}
\left\vert \alpha \right\vert +2\left( \beta +\gamma +2\right) =\left(
\left\vert \alpha \right\vert +2\beta \right) +2\gamma +4\leq
M_{2}+M_{2}+4=2M_{2}+4.
\end{equation*}

\begin{remark}
The above inequality uses the full $2M_{2}+4$\ derivatives for $f$ in the
space $M_{F}^{3M+M_{2},M_{1},2M_{2}+4}\left( \mathbb{H}^{n}\right) $.
\end{remark}

Using this and $\left\vert \partial _{v}^{\gamma }\psi _{k}\left( v\right)
\right\vert \leq \chi _{\left\{ \left\vert v\right\vert \leq 2^{-k}\right\}
}2^{k\left( \gamma +1\right) }$ we thus have%
\begin{eqnarray*}
&&\left\vert \partial _{z}^{\alpha }\partial _{u}^{\beta }\partial
_{v}^{\gamma }H_{m}\left( \left( z,u\right) ,v\right) \right\vert
=\left\vert \sum_{k=-m+1}^{\infty }\left( \partial _{z}^{\alpha }\partial
_{u}^{\beta }f_{m}\ast _{2}\psi _{k}\right) \left( z,u\right) \partial
_{v}^{\gamma }\psi _{k}\left( v\right) \right\vert \\
&\leq &\sum_{k=-m+1}^{\infty }\left\vert \left( \partial _{z}^{\alpha
}\partial _{u}^{\beta }f_{m}\ast _{2}\psi _{k}\right) \left( z,u\right)
\right\vert \left\vert \partial _{v}^{\gamma }\psi _{k}\left( v\right)
\right\vert \\
&\lesssim &\sum_{k=-m+1}^{\infty }\chi _{\left\{ 1+\left\vert z\right\vert
^{2}+\left\vert u\right\vert \approx 2^{m}\right\} }\frac{2^{-k\left( \gamma
+2\right) }}{\left( 1+\left\vert z\right\vert ^{2}+\left\vert u\right\vert
\right) ^{\frac{Q+3M+M_{2}+\left\vert \alpha \right\vert +2\beta }{2}+2}}%
\chi _{\left\{ \left\vert v\right\vert \leq 2^{-k}\right\} }2^{k\left(
\gamma +1\right) } \\
&\lesssim &\sum_{k=-m+1}^{\infty }2^{-k}\chi _{\left\{ 1+\left\vert
z\right\vert ^{2}+\left\vert u\right\vert \approx 2^{m}\right\} }\frac{1}{%
\left( 1+\left\vert z\right\vert ^{2}+\left\vert u\right\vert \right) ^{%
\frac{Q+M+\left\vert \alpha \right\vert +2\beta }{2}}}\frac{1}{\left(
2^{m}\right) ^{M+\frac{M_{2}}{2}+2}}\chi _{\left\{ \left\vert v\right\vert
\leq 2^{-k}\right\} } \\
&\lesssim &\chi _{\left\{ 1+\left\vert z\right\vert ^{2}+\left\vert
u\right\vert \approx 2^{m}\right\} }\frac{1}{\left( 1+\left\vert
z\right\vert ^{2}+\left\vert u\right\vert \right) ^{\frac{Q+M+\left\vert
\alpha \right\vert +2\beta }{2}}}\frac{1}{\left( 1+\left\vert v\right\vert
\right) ^{M+1+\gamma }},
\end{eqnarray*}%
since $2\gamma \leq M_{2}$ and $\sum_{k=-m+1}^{\infty }2^{-k}\approx 2^{m}$.

\begin{remark}
The above display uses\ the full $3M+M_{2}$ decay for $f$ in the space $%
M_{F}^{3M+M_{2},M_{1},2M_{2}+4}\left( \mathbb{H}^{n}\right) $.
\end{remark}

Now we use the fact that the sets $\left\{ 1+\left\vert z\right\vert
^{2}+\left\vert u\right\vert \approx 2^{m}\right\} $ have bounded overlap to
obtain the desired estimate:%
\begin{eqnarray*}
\left\vert \partial _{z}^{\alpha }\partial _{u}^{\beta }\partial
_{v}^{\gamma }F^{\ell }\left( \left( z,u\right) ,v\right) \right\vert
&=&\left\vert \sum_{m=0}^{\infty }\partial _{z}^{\alpha }\partial
_{u}^{\beta }\partial _{v}^{\gamma }F_{m}^{\ell }\left( \left( z,u\right)
,v\right) \right\vert \lesssim \sum_{m=0}^{\infty }\left\vert \partial
_{z}^{\alpha }\partial _{u}^{\beta }\partial _{v}^{\gamma }F_{m}^{\ell
}\left( \left( z,u\right) ,v\right) \right\vert \\
&\leq &\sum_{m=0}^{\infty }\left\{ \left\vert \partial _{z}^{\alpha
}\partial _{u}^{\beta }\partial _{v}^{\gamma }G_{m}^{\ell }\left( \left(
z,u\right) ,v\right) \right\vert +\left\vert \partial _{z}^{\alpha }\partial
_{u}^{\beta }\partial _{v}^{\gamma }H_{m}\left( \left( z,u\right) ,v\right)
\right\vert \right\} \\
&\lesssim &\sum_{m=0}^{\infty }\chi _{\left\{ 1+\left\vert z\right\vert
^{2}+\left\vert u\right\vert \approx 2^{m}\right\} }\frac{1}{\left(
1+\left\vert z\right\vert ^{2}+\left\vert u\right\vert \right) ^{\frac{%
Q+M+\left\vert \alpha \right\vert +2\beta }{2}}}\frac{1}{\left(
2^{m}+\left\vert v\right\vert \right) ^{1+M+\gamma }} \\
&\lesssim &\frac{1}{\left( 1+\left\vert z\right\vert ^{2}+\left\vert
u\right\vert \right) ^{\frac{Q+M+\left\vert \alpha \right\vert +2\beta }{2}}}%
\frac{1}{\left( 1+\left\vert v\right\vert \right) ^{1+M+\gamma }}.
\end{eqnarray*}

\bigskip

It remains to estimate the first order differences and second order
differences in Definition \ref{product molecular}. The first order
difference estimates%
\begin{eqnarray*}
&&\left\vert \partial _{z}^{\alpha }\partial _{u}^{\beta }F\left( \left(
z,u\right) ,v\right) -\partial _{z}^{\alpha }\partial _{u}^{\beta }F\left(
\left( z^{\prime },u^{\prime }\right) ,v\right) \right\vert \\
&&\ \ \ \ \ \ \ \ \ \ \ \ \ \ \ \leq A\frac{\left\vert \left( z,u\right)
\circ \left( z^{\prime },u^{\prime }\right) ^{-1}\right\vert }{\left(
1+\left\vert z\right\vert ^{2}+\left\vert u\right\vert \right) ^{\frac{%
Q+M+M_{2}+1}{2}}}\frac{1}{\left( 1+\left\vert v\right\vert \right) ^{1+M}} \\
&&\text{for all }\left\vert \alpha \right\vert +2\beta =M_{2}\text{ and }%
\left\vert \left( z,u\right) \circ \left( z^{\prime },u^{\prime }\right)
^{-1}\right\vert \leq \frac{1}{2}\left( 1+\left\vert z\right\vert
^{2}+\left\vert u\right\vert \right) ^{\frac{1}{2}}.
\end{eqnarray*}%
\begin{eqnarray*}
&&\left\vert \partial _{z}^{\alpha }\partial _{u}^{\beta }\partial
_{v}^{\gamma }F\left( \left( z,u\right) ,v\right) -\partial _{z}^{\alpha
}\partial _{u}^{\beta }\partial _{v}^{\gamma }F\left( \left( z,u\right)
,v^{\prime }\right) \right\vert \\
&&\ \ \ \ \ \ \ \ \ \ \leq A\frac{1}{\left( 1+\left\vert z\right\vert
^{2}+\left\vert u\right\vert \right) ^{\frac{Q+M+\left\vert \alpha
\right\vert +2\beta }{2}}}\frac{\left\vert v-v^{\prime }\right\vert }{\left(
1+\left\vert v\right\vert \right) ^{2+M}}, \\
&&\ \ \ \ \ \ \ \ \ \ \text{for all }\left\vert \alpha \right\vert +2\beta
\leq M_{2}\text{ and }2\gamma =M_{2} \\
&&\ \ \ \ \ \ \ \ \ \ \ \ \ \ \ \text{and }\left\vert v-v^{\prime
}\right\vert \leq \frac{1}{2}\left( 1+\left\vert v\right\vert \right) ,
\end{eqnarray*}%
with $F$ replaced by first $G_{m}^{\ell }$ and then $H_{m}$, are easily
proved using the above methods and the first order estimates for $f_{m}$.
Then we can sum in $m$ using the bounded overlap of the supports of $%
G_{m}^{\ell }$ and $H_{m}$ to obtain the required inequality.

The second order difference estimate%
\begin{eqnarray*}
&&\left\vert \left[ \partial _{z}^{\alpha }\partial _{u}^{\beta }\partial
_{v}^{\gamma }F\left( \left( z,u\right) ,v\right) -\partial _{z}^{\alpha
}\partial _{u}^{\beta }\partial _{v}^{\gamma }F\left( \left( z^{\prime
},u^{\prime }\right) ,v\right) \right] \right. \\
&&\ \ \ \ \ \left. -\left[ \partial _{z}^{\alpha }\partial _{u}^{\beta
}\partial _{v}^{\gamma }F\left( \left( z,u\right) ,v^{\prime }\right)
-\partial _{z}^{\alpha }\partial _{u}^{\beta }\partial _{v}^{\gamma }F\left(
\left( z^{\prime },u^{\prime }\right) ,v^{\prime }\right) \right] \right\vert
\\
&&\ \ \ \ \ \ \ \ \ \ \leq A\frac{\left\vert \left( z,u\right) \circ \left(
z^{\prime },u^{\prime }\right) ^{-1}\right\vert }{\left( 1+\left\vert
z\right\vert ^{2}+\left\vert u\right\vert \right) ^{\frac{Q+2M}{2}}}\frac{%
\left\vert v-v^{\prime }\right\vert }{\left( 1+\left\vert v\right\vert
\right) ^{1+M}} \\
&&\ \ \ \ \ \ \ \ \ \ \text{for all }\left\vert \alpha \right\vert +2\beta
=M_{2}\text{ and }2\gamma =M_{2} \\
&&\ \ \ \ \ \ \ \ \ \ \ \ \ \ \ \text{and }\left\vert \left( z,u\right)
\circ \left( z^{\prime },u^{\prime }\right) ^{-1}\right\vert \leq \frac{1}{2}%
\left( 1+\left\vert z\right\vert ^{2}+\left\vert u\right\vert \right) ^{%
\frac{1}{2}} \\
&&\ \ \ \ \ \ \ \ \ \ \ \ \ \ \ \text{and }\left\vert v-v^{\prime
}\right\vert \leq \frac{1}{2}\left( 1+\left\vert v\right\vert \right) ,
\end{eqnarray*}%
follows from the first order difference estimates using the product form (%
\ref{Flr}) of the functions $G_{m}^{\ell }$ and $H_{m}$. For example,%
\begin{eqnarray*}
&&\left\vert \left[ \partial _{z}^{\alpha }\partial _{u}^{\beta }\partial
_{v}^{\gamma }G_{m}^{\ell }\left( \left( z,u\right) ,v\right) -\partial
_{z}^{\alpha }\partial _{u}^{\beta }\partial _{v}^{\gamma }G_{m}^{\ell
}\left( \left( z^{\prime },u^{\prime }\right) ,v\right) \right] \right. \\
&&\ \ \ \ \ \left. -\left[ \partial _{z}^{\alpha }\partial _{u}^{\beta
}\partial _{v}^{\gamma }G_{m}^{\ell }\left( \left( z,u\right) ,v^{\prime
}\right) -\partial _{z}^{\alpha }\partial _{u}^{\beta }\partial _{v}^{\gamma
}G_{m}^{\ell }\left( \left( z^{\prime },u^{\prime }\right) ,v^{\prime
}\right) \right] \right\vert \\
&=&\left\vert \left[ \partial _{z}^{\alpha }\partial _{u}^{\beta }I^{\ell
}\left( f_{m}\ast _{2}\chi _{-m}\right) \left( z,u\right) \partial
_{v}^{\gamma }D^{\ell }\chi _{-m}\left( v\right) -\partial _{z}^{\alpha
}\partial _{u}^{\beta }I^{\ell }\left( f_{m}\ast _{2}\chi _{-m}\right)
\left( z^{\prime },u^{\prime }\right) \partial _{v}^{\gamma }D^{\ell }\chi
_{-m}\left( v\right) \right] \right. \\
&&\left. -\left[ \partial _{z}^{\alpha }\partial _{u}^{\beta }I^{\ell
}\left( f_{m}\ast _{2}\chi _{-m}\right) \left( z,u\right) \partial
_{v}^{\gamma }D^{\ell }\chi _{-m}\left( v^{\prime }\right) -\partial
_{z}^{\alpha }\partial _{u}^{\beta }I^{\ell }\left( f_{m}\ast _{2}\chi
_{-m}\right) \left( z^{\prime },u^{\prime }\right) \partial _{v}^{\gamma
}D^{\ell }\chi _{-m}\left( v^{\prime }\right) \right] \right\vert \\
&=&\left\vert \partial _{z}^{\alpha }\partial _{u}^{\beta }I^{\ell }\left(
f_{m}\ast _{2}\chi _{-m}\right) \left( z,u\right) -\partial _{z}^{\alpha
}\partial _{u}^{\beta }I^{\ell }\left( f_{m}\ast _{2}\chi _{-m}\right)
\left( z^{\prime },u^{\prime }\right) \right\vert \\
&&\ \ \ \ \ \ \ \ \ \ \ \ \ \ \ \ \ \ \ \ \times \left\vert \partial
_{v}^{\gamma }D^{\ell }\chi _{-m}\left( v\right) -\partial _{v}^{\gamma
}D^{\ell }\chi _{-m}\left( v^{\prime }\right) \right\vert ,
\end{eqnarray*}%
and we can now apply first order difference estimates for the functions $%
\partial _{z}^{\alpha }\partial _{u}^{\beta }I^{\ell }\left( f_{m}\ast
_{2}\chi _{-m}\right) $ and $\partial _{v}^{\gamma }D^{\ell }\chi _{-m}$.
The same argument works for $H_{m}$ in place of $G_{m}^{\ell }$. Finally, we
sum the resulting estimates in $m$ using the finite overlap of the supports
of $G_{m}^{\ell }$ and $H_{m}$ to obtain the required inequality.

\subsection{A comparison of quotient spaces}

Here we put the space of `projected' flag test functions $\mathcal{M}%
_{flag}^{M+\delta ,M_{1},M_{2}}\left( \mathbb{H}^{n}\right) $ into
perspective. For the purpose of comparison, we first recall the usual
molecular space $\mathcal{M}^{M+\delta ,M_{1},M_{2}}\left( \mathbb{H}%
^{n}\right) $ associated with the one-parameter autormorphic group of
dilations on $\mathbb{H}^{n}$.

\begin{definition}
\label{one para mol}Let $M,M_{1},M_{2}\in \mathbb{N}$ be positive integers, $%
0<\delta \leq 1$, and let $Q=2n+2$ denote the homogeneous dimension of $%
\mathbb{H}^{n}$. The \emph{one-parameter} molecular space $\mathcal{M}%
^{M+\delta ,M_{1},M_{2}}\left( \mathbb{H}^{n}\right) $ consists of all
functions $f\left( z,u\right) $ on $\mathbb{H}^{n}$ satisfying the moment
conditions%
\begin{equation*}
\int_{\mathbb{H}^{n}}z^{\alpha }u^{\beta }f\left( z,u\right) dzdu=0\text{
for all }\left\vert \alpha \right\vert +2\left\vert \beta \right\vert \leq
M_{1},
\end{equation*}%
and such that there is a nonnegative constant $A$ satisfying the following
two differential inequalities:%
\begin{equation*}
\left\vert \partial _{z}^{\alpha }\partial _{u}^{\beta }f\left( z,u\right)
\right\vert \leq A\frac{1}{\left( 1+\left\vert z\right\vert ^{2}+\left\vert
u\right\vert \right) ^{\frac{Q+M+\left\vert \alpha \right\vert +2\left\vert
\beta \right\vert +\delta }{2}}}\text{ for all }\left\vert \alpha
\right\vert +2\left\vert \beta \right\vert \leq M_{2},
\end{equation*}%
\begin{eqnarray*}
&&\left\vert \partial _{z}^{\alpha }\partial _{u}^{\beta }f\left( z,u\right)
-\partial _{z}^{\alpha }\partial _{u}^{\beta }f\left( z^{\prime },u^{\prime
}\right) \right\vert \leq A\frac{\left\vert \left( z,u\right) \circ \left(
z^{\prime },u^{\prime }\right) ^{-1}\right\vert ^{\delta }}{\left(
1+\left\vert z\right\vert ^{2}+\left\vert u\right\vert \right) ^{\frac{%
Q+M+\delta +M_{2}+2\delta }{2}}} \\
&&\ \ \ \ \ \text{for all }\left\vert \alpha \right\vert +2\left\vert \beta
\right\vert =M_{2}\text{ and }\left\vert \left( z,u\right) \circ \left(
z^{\prime },u^{\prime }\right) ^{-1}\right\vert \leq \frac{1}{2}\left(
1+\left\vert z\right\vert ^{2}+\left\vert u\right\vert \right) ^{\frac{1}{2}%
}.
\end{eqnarray*}
\end{definition}

\begin{remark}
\label{standard on flag}The intertwining formula in Lemma \ref{intertwine}
shows that if $T$ is a convolution operator on $\mathbb{H}^{n}$ that is
bounded on the one-parameter molecular space $\mathcal{M}^{M+\delta
,M_{1},M_{2}}\left( \mathbb{H}^{n}\right) $, then it is also bounded on the
flag molecular space $\mathcal{M}_{flag}^{M+\delta ,M_{1},M_{2}}\left( 
\mathbb{H}^{n}\right) $. Indeed, one simply notes that the extension $%
\widetilde{T}$ of $T$ to the group $\mathbb{H}^{n}\times \mathbb{R}$ is
bounded on the product molecular space $\mathcal{M}_{product}^{M+\delta
,M_{1},M_{2}}\left( \mathbb{H}^{n}\times \mathbb{R}\right) $ by fixing $v$
and $v^{\prime }$ in conditions (\ref{Mo1}) - (\ref{DI4}) in Definition \ref%
{product molecular} and using the boundedness of $T$ in the remaining
variables in the one-parameter space $\mathcal{M}^{M+\delta
,M_{1},M_{2}}\left( \mathbb{H}^{n}\right) $.
\end{remark}

The following lemma provides the fundamental inclusions between these spaces
of test functions.

\begin{notation}
Given positive integers $M,M_{1},M_{2}\in \mathbb{N}$, set%
\begin{eqnarray*}
M^{\prime } &=&3M+M_{2}, \\
M_{1}^{\prime } &=&M_{1}, \\
M_{2}^{\prime } &=&2M_{2}+4.
\end{eqnarray*}
\end{notation}

\begin{lemma}
\label{inclusion}Let $M,M_{1},M_{2}\in \mathbb{N}$ be positive integers,and $%
0<\delta \leq 1$. Then the inclusions%
\begin{equation*}
\mathsf{M}_{F}^{M^{\prime }+\delta ,M_{1}^{\prime },M_{2}^{\prime }}\left( 
\mathbb{H}^{n}\right) \subset \mathcal{M}_{flag}^{M+\delta
,M_{1},M_{2}}\left( \mathbb{H}^{n}\right) \subset \mathsf{M}_{F}^{M+\delta
,M_{1},M_{2}}\left( \mathbb{H}^{n}\right) \subset \mathcal{M}^{M+\delta
,M_{1},M_{2}}\left( \mathbb{H}^{n}\right)
\end{equation*}%
are continuous with closed range.
\end{lemma}

The proof of Lemma \ref{inclusion} follows immediately from the continuous
containments in Lemma \ref{containments}. As a consequence of Lemma \ref%
{inclusion} we have the following relationship between the corresponding
dual spaces.

\begin{lemma}
\label{dual inclusion}For every $\Lambda \in \mathcal{M}_{flag}^{M+\delta
,M_{1},M_{2}}\left( \mathbb{H}^{n}\right) ^{\prime }$ there is $\Phi \in 
\mathcal{M}^{M+\delta ,M_{1},M_{2}}\left( \mathbb{H}^{n}\right) ^{\prime }$
such that 
\begin{equation*}
\Phi \mid _{\mathsf{M}_{F}^{M^{\prime }+\delta ,M_{1}^{\prime
},M_{2}^{\prime }}\left( \mathbb{H}^{n}\right) }=\Lambda _{\mathsf{M}%
_{F}^{M^{\prime }+\delta ,M_{1}^{\prime },M_{2}^{\prime }}\left( \mathbb{H}%
^{n}\right) }.
\end{equation*}
\end{lemma}

\begin{proof}
Suppose that $\Lambda \in \mathcal{M}_{flag}^{M+\delta ,M_{1},M_{2}}\left( 
\mathbb{H}^{n}\right) ^{\prime }$. Since the inclusion%
\begin{equation*}
\mathsf{M}_{F}^{M^{\prime }+\delta ,M_{1}^{\prime },M_{2}^{\prime }}\left( 
\mathbb{H}^{n}\right) \subset \mathcal{M}_{flag}^{M+\delta
,M_{1},M_{2}}\left( \mathbb{H}^{n}\right)
\end{equation*}%
is continuous, we see that 
\begin{equation*}
\Lambda _{1}\equiv \Lambda \mid _{\mathsf{M}_{F}^{M^{\prime }+\delta
,M_{1}^{\prime },M_{2}^{\prime }}\left( \mathbb{H}^{n}\right) }
\end{equation*}%
is a continuous linear functional on $\mathsf{M}_{F}^{M^{\prime }+\delta
,M_{1}^{\prime },M_{2}^{\prime }}\left( \mathbb{H}^{n}\right) $. The
Hahn-Banach theorem yields a continuous extension $\Phi $ of $\Lambda _{1}$
first to the space $\mathcal{M}^{M^{\prime }+\delta ,M_{1}^{\prime
},M_{2}^{\prime }}\left( \mathbb{H}^{n}\right) $ and then to the larger
space $\mathcal{M}^{M+\delta ,M_{1},M_{2}}\left( \mathbb{H}^{n}\right) $.
Thus $\Phi \in \mathcal{M}^{M+\delta ,M_{1},M_{2}}\left( \mathbb{H}%
^{n}\right) ^{\prime }$ and of course%
\begin{equation*}
\Phi \mid _{\mathsf{M}_{F}^{M^{\prime }+\delta ,M_{1}^{\prime
},M_{2}^{\prime }}\left( \mathbb{H}^{n}\right) }=\Lambda _{1}=\Lambda \mid _{%
\mathsf{M}_{F}^{M^{\prime }+\delta ,M_{1}^{\prime },M_{2}^{\prime }}\left( 
\mathbb{H}^{n}\right) }.
\end{equation*}
\end{proof}

While we cannot say that $H_{flag}^{p}\left( \mathbb{H}^{n}\right) $ is a 
\emph{subspace} of the one-parameter Hardy space $H^{p}\left( \mathbb{H}%
^{n}\right) $, Lemma \ref{dual inclusion} shows that a certain quotient
space $Q_{flag}^{p}\left( \mathbb{H}^{n}\right) $ of $H_{flag}^{p}\left( 
\mathbb{H}^{n}\right) $ can be identified with a closed subspace of the
corresponding quotient space $Q^{p}\left( \mathbb{H}^{n}\right) $ of $%
H^{p}\left( \mathbb{H}^{n}\right) $; the quotient spaces in question are%
\begin{eqnarray*}
Q_{flag}^{p}\left( \mathbb{H}^{n}\right) &\equiv &H_{flag}^{p}\left( \mathbb{%
H}^{n}\right) /\mathsf{M}_{F}^{M^{\prime }+\delta ,M_{1}^{\prime
},M_{2}^{\prime }}\left( \mathbb{H}^{n}\right) ^{\bot }, \\
Q^{p}\left( \mathbb{H}^{n}\right) &\equiv &H^{p}\left( \mathbb{H}^{n}\right)
/\mathsf{M}_{F}^{M^{\prime }+\delta ,M_{1}^{\prime },M_{2}^{\prime }}\left( 
\mathbb{H}^{n}\right) ^{\bot }.
\end{eqnarray*}%
Indeed, if $f\in H_{flag}^{p}\left( \mathbb{H}^{n}\right) $, then by Lemma %
\ref{dual inclusion} there is 
\begin{equation*}
F\in \mathcal{M}^{M+\delta ,M_{1},M_{2}}\left( \mathbb{H}^{n}\right)
^{\prime }
\end{equation*}%
such that%
\begin{equation}
F\mid _{\mathsf{M}_{F}^{M^{\prime }+\delta ,M_{1}^{\prime },M_{2}^{\prime
}}\left( \mathbb{H}^{n}\right) }=f_{\mathsf{M}_{F}^{M^{\prime }+\delta
,M_{1}^{\prime },M_{2}^{\prime }}\left( \mathbb{H}^{n}\right) }.
\label{coincide}
\end{equation}%
But then since the component functions in $E_{k}D_{j}$ can be chosen to
belong to $\mathsf{M}_{F}^{M^{\prime }+\delta ,M_{1}^{\prime },M_{2}^{\prime
}}\left( \mathbb{H}^{n}\right) $, we have $E_{k}D_{j}F=E_{k}D_{j}f$ and%
\begin{eqnarray}
g\left( F\right) &=&\left\{ \sum\limits_{k=-\infty }^{\infty }\left\vert
E_{k}F\right\vert ^{2}\right\} ^{{\frac{{1}}{{2}}}}\lesssim \left\{
\sum\limits_{j,k=-\infty }^{\infty }\left\vert E_{k}D_{j}F\right\vert
^{2}\right\} ^{{\frac{{1}}{{2}}}}  \label{comp g} \\
&=&\left\{ \sum\limits_{j,k=-\infty }^{\infty }\left\vert
E_{k}D_{j}f\right\vert ^{2}\right\} ^{{\frac{{1}}{{2}}}}=g_{flag}\left(
f\right) .  \notag
\end{eqnarray}%
If $\left[ F\right] $ (respectively $\left[ f\right] $) denotes the
equivalence class in $Q^{p}\left( \mathbb{H}^{n}\right) $ (respectively $%
Q_{flag}^{p}\left( \mathbb{H}^{n}\right) $) that contains the distribution $%
F $ (respectively $f$), then (\ref{coincide}) shows that there is a
well-defined one-to-one linear map%
\begin{equation*}
T:Q_{flag}^{p}\left( \mathbb{H}^{n}\right) \rightarrow Q^{p}\left( \mathbb{H}%
^{n}\right) ,
\end{equation*}%
given by $Tq=\left[ F\right] $ if $q=\left[ f\right] $ and (\ref{coincide})
holds. Then (\ref{comp g}) implies that 
\begin{equation*}
\left\Vert Tq\right\Vert _{Q^{p}\left( \mathbb{H}^{n}\right) }\lesssim
\left\Vert q\right\Vert _{Q_{flag}^{p}\left( \mathbb{H}^{n}\right) },\ \ \ \
\ q\in Q_{flag}^{p}\left( \mathbb{H}^{n}\right) .
\end{equation*}%
The map $T$ thus identifies $Q_{flag}^{p}\left( \mathbb{H}^{n}\right) $ as a
closed subspace of $Q^{p}\left( \mathbb{H}^{n}\right) $.

Note that if we could identify $\mathcal{M}^{M+\delta ,M_{1},M_{2}}\left( 
\mathbb{H}^{n}\right) $ with $\mathsf{M}_{F}^{M^{\prime }+\delta
,M_{1}^{\prime },M_{2}^{\prime }}\left( \mathbb{H}^{n}\right) $ for some
choice of parameters $M^{\prime }+\delta ,M_{1}^{\prime },M_{2}^{\prime }$,
then we could conclude that $H_{flag}^{p}\left( \mathbb{H}^{n}\right)
=Q_{flag}^{p}\left( \mathbb{H}^{n}\right) $ is itself a closed subspace of
the quotient space $Q^{p}\left( \mathbb{H}^{n}\right) $.

\section{A counterexample for the one-parameter Hardy space\label{s
counterexample}}

Recall that $\mathbb{H}^{n}=\mathbb{C}^{n}\times \mathbb{R}$ is the
Heisenberg group with group multiplication%
\begin{equation*}
\left( \zeta ,t\right) \cdot \left( \eta ,s\right) =\left( \zeta +\eta ,t+s+2%
\func{Im}\left( \zeta \cdot \overline{\eta }\right) \right) ,\ \ \ \ \
\left( \zeta ,t\right) ,\left( \eta ,s\right) \in \mathbb{C}^{n}\times 
\mathbb{R},
\end{equation*}%
and that $\left( \eta ,s\right) ^{-1}=\left( -\eta ,-s\right) $. Consider
the mixed kernel $K\left( z,t\right) =K_{1}\left( z\right) K_{2}\left(
z,t\right) $ for $\left( z,t\right) \in \mathbb{H}^{n}=\mathbb{C}^{n}\times 
\mathbb{R}$ given by%
\begin{equation*}
K_{1}\left( z\right) =\frac{\Omega \left( z\right) }{\left\vert z\right\vert
^{2n}}\text{ and }K_{2}\left( z,t\right) =\frac{1}{\left\vert z\right\vert
^{2}+it},
\end{equation*}%
where $\Omega $ is smooth with mean zero on the unit sphere in $\mathbb{C}%
^{n}$. We show in the subsection below that $K$ satisfies the smoothness and
cancellation conditions required of a flag kernel. It then follows from
Muller-Ricci-Stein \cite{MRS} that there is an operator $T$ having kernel $K$
such that for each $1<p<\infty $,%
\begin{equation*}
\left\Vert Tf\right\Vert _{L^{p}\left( \mathbb{H}^{n}\right) }\leq
C_{p,n}\left\Vert f\right\Vert _{L^{p}\left( \mathbb{H}^{n}\right) },\ \ \ \
\ f\in L^{p}\left( \mathbb{H}^{n}\right) .
\end{equation*}%
The action of the corresponding singular integral operator $Tf=K\ast f$ is
given by 
\begin{eqnarray*}
Tf\left( \zeta ,t\right) &=&K\ast _{\mathbb{H}^{n}}f\left( \zeta ,t\right)
=\int_{\mathbb{H}^{n}}K\left( \left( \zeta ,t\right) \circ \left( \eta
,s\right) ^{-1}\right) \ f\left( \eta ,s\right) d\eta ds \\
&=&\int_{\mathbb{H}^{n}}f\left( \eta ,s\right) \ K\left( \zeta -\eta ,t-s-2%
\func{Im}\left( \zeta \cdot \overline{\eta }\right) \right) \ d\eta ds \\
&=&\int_{\mathbb{H}^{n}}f\left( \eta ,s\right) \ \frac{\Omega \left( \zeta
-\eta \right) }{\left\vert \zeta -\eta \right\vert ^{2n}}\frac{1}{\left\vert
\zeta -\eta \right\vert ^{2}+i\left( t-s-2\func{Im}\left( \zeta \cdot 
\overline{\eta }\right) \right) }\ d\eta ds.
\end{eqnarray*}

\begin{theorem}
\label{endpoint fails}There is a smooth function $\Omega $ with mean zero on
the unit sphere in $\mathbb{C}^{n}$ such that there is \textbf{no} operator $%
T$ having kernel $K$ that is bounded from $H^{1}\left( \mathbb{H}^{n}\right) 
$ to $L^{1}\left( \mathbb{H}^{n}\right) $.
\end{theorem}

To prove the theorem, we fix $f\left( z,u\right) =\psi \left( z\right)
\varphi \left( u\right) $ where

\begin{enumerate}
\item $\psi $ is smooth with support in the unit ball of $\mathbb{C}^{n}$,

\item $\varphi $ is smooth with support in $\left( -1,1\right) $,

\item $\int_{\mathbb{C}^{n}}\psi \left( z\right) dz=0$ and $\int_{\mathbb{R}%
}\varphi \left( u\right) du=1$.
\end{enumerate}

Such a function $f$ is clearly in $H^{1}\left( \mathbb{H}^{n}\right) $ since 
$f$ is smooth, compactly supported and has mean zero: 
\begin{equation*}
\int_{\mathbb{H}^{n}}f\left( z,u\right) dzdu=\int_{\mathbb{R}}\left\{ \int_{%
\mathbb{C}^{n}}\psi \left( z\right) dz\right\} \varphi \left( u\right)
du=\int_{\mathbb{R}}\left\{ 0\right\} \varphi \left( u\right) du=0.
\end{equation*}%
We next show that $T$ fails to be bounded from $H^{1}\left( \mathbb{H}%
^{n}\right) $ to $L^{1}\left( \mathbb{H}^{n}\right) $, and then that $T$ is
a flag singular integral.

\subsection{Failure of boundedness of $T$}

For 
\begin{eqnarray*}
\zeta &\in &B\left( \left( 100,\mathbf{0}\right) ,0\right) =\left\{ \left(
\zeta _{1},\mathbf{\zeta }^{\prime }\right) \in \mathbb{R}\times \mathbb{C}%
^{n-1}:\left( \zeta _{1}-100\right) ^{2}+\left\vert \mathbf{\zeta }^{\prime
}\right\vert ^{2}<1\right\} , \\
\left\vert t\right\vert &>&10^{6},
\end{eqnarray*}%
we have%
\begin{equation*}
\left\vert Tf\left( \zeta ,t\right) \right\vert \approx \int \psi \left(
\eta \right) \varphi \left( s\right) \frac{\Omega \left( \zeta -\eta \right) 
}{\left\vert \zeta \right\vert ^{2n}}\frac{1}{\left\vert \zeta \right\vert
^{2}+i\left( t-2\left\vert \zeta \right\vert ^{2}\right) }d\eta ds\approx 
\frac{1}{\left\vert \zeta \right\vert ^{2n}\left\vert t\right\vert },
\end{equation*}%
since for $\zeta \in B\left( \left( 100,\mathbf{0}\right) ,0\right) $ we have%
\begin{equation*}
\left\vert \int \psi \left( \eta \right) \Omega \left( \zeta -\eta \right)
d\eta \right\vert \geq c>0,
\end{equation*}%
for an appropriately chosen $\Omega $ with mean zero on the sphere. The
point is that both functions $\psi $ and $\Omega $ have mean zero on their
respective domains, but the product can destroy enough of the cancellation.
For example, when $n=1$ we can take%
\begin{eqnarray*}
\Omega \left( x,y\right) &=&\frac{y}{\sqrt{x^{2}+y^{2}}}, \\
\psi \left( x,y\right) &=&y\psi _{1}\left( x\right) \psi _{2}\left( y\right)
,
\end{eqnarray*}%
where $\psi _{i}$ is an even function identically one on $\left( -\frac{1}{2}%
,\frac{1}{2}\right) $ and supported in $\left( -\frac{1}{\sqrt{2}},\frac{1}{%
\sqrt{2}}\right) $. Then for 
\begin{equation*}
\zeta =\left( 100+\nu ,\omega \right) ,\ \ \ \ \ \left\vert \nu \right\vert
^{2}+\left\vert \omega \right\vert ^{2}\leq 1,
\end{equation*}%
we have%
\begin{eqnarray*}
\int \psi \left( \eta \right) \Omega \left( \zeta -\eta \right) d\eta
&=&\int y\psi _{1}\left( x\right) \psi _{2}\left( y\right) \Omega \left(
100+\nu -x,\omega -y\right) \\
&=&\int y\psi _{1}\left( x\right) \psi _{2}\left( y\right) \frac{\omega -y}{%
\sqrt{\left( 100+\nu -x\right) ^{2}+\left( \omega -y\right) ^{2}}} \\
&=&\omega \int \frac{y\psi _{1}\left( x\right) \psi _{2}\left( y\right) }{%
\sqrt{\left( 100+\nu -x\right) ^{2}+\left( \omega -y\right) ^{2}}} \\
&&-\int \frac{y^{2}\psi _{1}\left( x\right) \psi _{2}\left( y\right) }{\sqrt{%
\left( 100+\nu -x\right) ^{2}+\left( \omega -y\right) ^{2}}} \\
&\approx &-\frac{1}{100}.
\end{eqnarray*}

We conclude from the above that 
\begin{equation*}
\int_{\mathbb{H}^{n}}\left\vert Tf\left( \zeta ,t\right) \right\vert d\zeta
dt\gtrsim \int_{\left\{ \zeta \in B\left( \left( 100,\mathbf{0}\right)
,0\right) \text{ and }\left\vert t\right\vert >10^{6}\right\} }\frac{1}{%
\left\vert \zeta \right\vert ^{2n}\left\vert t\right\vert }d\zeta dt=\infty .
\end{equation*}

\subsection{$T$ is a flag singular integral}

Let $K$ be the kernel%
\begin{equation*}
K\left( z,t\right) =\frac{\Omega \left( z\right) }{\left\vert z\right\vert
^{2n}}\frac{1}{\left\vert z\right\vert ^{2}+it},\ \ \ \ \ \left( z,t\right)
\in \mathbb{H}^{n}.
\end{equation*}%
In order to show that $K$ is a flag kernel we must establish the following
smoothness and cancellation conditions.

\begin{enumerate}
\item (\emph{Differential Inequalities}) For any multi-indices $\alpha
=(\alpha _{1},\cdots ,\alpha _{n})$, $\beta =(\beta _{1},\cdots ,\beta _{m})$
\begin{equation*}
\left\vert \partial _{z}^{\alpha }\partial _{u}^{\beta }K(z,u)\right\vert
\leq C_{\alpha ,\beta }\left\vert z\right\vert ^{-2n-|\alpha |}\cdot \left(
\left\vert z\right\vert ^{2}+\left\vert u\right\vert \right) ^{-1-|\beta |}
\end{equation*}%
for all $(z,u)\in \mathbb{H}^{n}$ with $z\neq 0$.

\item (\emph{Cancellation Condition}) 
\begin{equation*}
\left\vert \int_{\mathbb{R}}\partial _{z}^{\alpha }K(z,u)\phi _{1}(\delta
u)du\right\vert \leq C_{\alpha }|z|^{-2n-|\alpha |}
\end{equation*}%
for every multi-index $\alpha $ and every normalized bump function $\phi
_{1} $ on $\mathbb{R}$ and every $\delta >0$; 
\begin{equation*}
\left\vert \int_{\mathbb{C}^{n}}\partial _{u}^{\beta }K(z,u)\phi _{2}(\delta
z)dz\right\vert \leq C_{\gamma }|u|^{-1-|\beta |}
\end{equation*}%
for every multi-index $\beta $ and every normalized bump function $\phi _{2}$
on $\mathbb{C}^{n}$ and every $\delta >0$; and 
\begin{equation*}
\left\vert \int_{\mathbb{H}^{n}}K(z,u)\phi _{3}(\delta _{1}z,\delta
_{2}u)dzdu\right\vert \leq C
\end{equation*}%
for every normalized bump function $\phi _{3}$ on $\mathbb{H}^{n}$ and every 
$\delta _{1}>0$ and $\delta _{2}>0$.
\end{enumerate}

\bigskip

The differential inequalities in (1) follow immediately from the definition
of $K$.

The first cancellation condition in (2) exploits the fact that $t$ is an odd
function. For convenience we assume $\alpha =0$. We then have%
\begin{eqnarray*}
\left\vert \int_{\mathbb{R}}K(z,t)\phi _{1}(\delta t)dt\right\vert
&=&\left\vert \int_{\mathbb{R}}\frac{\Omega \left( z\right) }{\left\vert
z\right\vert ^{2n}}\left\{ \frac{\left\vert z\right\vert ^{2}}{\left\vert
z\right\vert ^{4}+t^{2}}-\frac{it}{\left\vert z\right\vert ^{4}+t^{2}}%
\right\} \phi _{1}(\delta t)dt\right\vert \\
&\leq &\int_{\mathbb{R}}\frac{1}{\left\vert z\right\vert ^{2n}}\frac{%
\left\vert z\right\vert ^{2}}{\left\vert z\right\vert ^{4}+t^{2}}\left\vert
\phi _{1}\left( \delta t\right) \right\vert dt \\
&&+\left\vert \int_{\mathbb{R}}\frac{\Omega \left( z\right) }{\left\vert
z\right\vert ^{2n}}\frac{it}{\left\vert z\right\vert ^{4}+t^{2}}\left\{ \phi
_{1}\left( \delta t\right) -\phi _{1}\left( 0\right) \right\} dt\right\vert
\\
&\lesssim &\frac{1}{\left\vert z\right\vert ^{2n-2}}\int_{0}^{\infty }\frac{1%
}{\left\vert z\right\vert ^{4}+t^{2}}dt+\frac{1}{\left\vert z\right\vert
^{2n}}\int_{0}^{\frac{1}{\delta }}\frac{\delta t^{2}}{\left\vert
z\right\vert ^{4}+t^{2}}dt.
\end{eqnarray*}%
Now%
\begin{equation*}
\frac{1}{\left\vert z\right\vert ^{2n-2}}\int_{0}^{\infty }\frac{1}{%
\left\vert z\right\vert ^{4}+t^{2}}dt\lesssim \frac{1}{\left\vert
z\right\vert ^{2n-2}}\left( \int_{0}^{\left\vert z\right\vert ^{2}}\frac{1}{%
\left\vert z\right\vert ^{4}}dt+\int_{\left\vert z\right\vert ^{2}}^{\infty }%
\frac{1}{t^{2}}dt\right) \lesssim \frac{1}{\left\vert z\right\vert ^{2n}},
\end{equation*}%
and for $\left\vert z\right\vert ^{2}\leq \frac{1}{\delta }$, we have 
\begin{equation*}
\int_{0}^{\frac{1}{\delta }}\frac{\delta t^{2}}{\left\vert z\right\vert
^{4}+t^{2}}dt\lesssim \int_{0}^{\left\vert z\right\vert ^{2}}\frac{\delta
t^{2}}{\left\vert z\right\vert ^{4}}dt+\int_{\left\vert z\right\vert ^{2}}^{%
\frac{1}{\delta }}\frac{\delta t^{2}}{t^{2}}dt\lesssim \delta \frac{%
\left\vert z\right\vert ^{6}}{\left\vert z\right\vert ^{4}}+1\lesssim 1,
\end{equation*}%
while for $\left\vert z\right\vert ^{2}>\frac{1}{\delta }$, we have%
\begin{equation*}
\int_{0}^{\frac{1}{\delta }}\frac{\delta t^{2}}{\left\vert z\right\vert
^{4}+t^{2}}dt\lesssim \int_{0}^{\frac{1}{\delta }}\frac{\delta t^{2}}{%
\left\vert z\right\vert ^{4}}dt\lesssim \delta \frac{\left( \frac{1}{\delta }%
\right) ^{3}}{\left\vert z\right\vert ^{4}}\lesssim 1.
\end{equation*}%
Altogether we have $\left\vert \int_{\mathbb{R}}K(z,t)\phi _{1}(\delta
t)dt\right\vert \lesssim \left\vert z\right\vert ^{-2n}$ as required.

The second cancellation condition in (2) uses the assumption that $\Omega $
has mean zero on the sphere. For convenience we take $\beta =0$. Then we
have 
\begin{eqnarray*}
\left\vert \int_{\mathbb{C}^{n}}K(z,t)\phi _{2}(\delta z)dz\right\vert
&=&\left\vert \int_{\mathbb{C}^{n}}\frac{\Omega \left( z\right) }{\left\vert
z\right\vert ^{2n}}\frac{1}{\left\vert z\right\vert ^{2}+it}\left\{ \phi
_{2}\left( \delta z\right) -\phi _{2}\left( 0\right) \right\} dz\right\vert
\\
&\lesssim &\delta \int_{\left\{ \left\vert z\right\vert \leq \frac{1}{\delta 
}\right\} }\frac{1}{\left\vert z\right\vert ^{2n}}\frac{1}{\left\vert
z\right\vert ^{2}+\left\vert t\right\vert }\left\vert z\right\vert dz \\
&\lesssim &\frac{\delta }{\left\vert t\right\vert }\int_{0}^{\frac{1}{\delta 
}}\frac{1}{r^{2n}}r\left( r^{2n-1}dr\right) \approx \left\vert t\right\vert
^{-1},
\end{eqnarray*}%
as required.

The third cancellation condition in (2) is handled similarly. We have%
\begin{eqnarray*}
&&\int_{\mathbb{H}^{n}}K(z,t)\phi _{3}(\delta _{1}z,\delta _{2}t)dzdt \\
&=&\int_{\mathbb{H}^{n}}\frac{\Omega \left( z\right) }{\left\vert
z\right\vert ^{2n}}\left\{ \frac{\left\vert z\right\vert ^{2}}{\left\vert
z\right\vert ^{4}+t^{2}}-\frac{it}{\left\vert z\right\vert ^{4}+t^{2}}%
\right\} \left\{ \phi _{3}\left( \delta _{1}z,\delta _{2}t\right) -\phi
_{3}\left( 0,\delta _{2}t\right) \right\} dzdt \\
&=&\int_{\mathbb{H}^{n}}\frac{\Omega \left( z\right) }{\left\vert
z\right\vert ^{2n}}\frac{\left\vert z\right\vert ^{2}}{\left\vert
z\right\vert ^{4}+t^{2}}\left\{ \phi _{3}\left( \delta _{1}z,\delta
_{2}t\right) -\phi _{3}\left( 0,\delta _{2}t\right) \right\} dzdt \\
&&-\int_{\mathbb{H}^{n}}\frac{\Omega \left( z\right) }{\left\vert
z\right\vert ^{2n}}\frac{it}{\left\vert z\right\vert ^{4}+t^{2}}\left\{ \phi
_{3}\left( \delta _{1}z,\delta _{2}t\right) -\phi _{3}\left( 0,\delta
_{2}t\right) -\phi _{3}\left( \delta _{1}z,0\right) +\phi _{3}\left(
0,0\right) \right\} dzdt,
\end{eqnarray*}%
and so%
\begin{eqnarray*}
&&\left\vert \int_{\mathbb{H}^{n}}K(z,t)\phi _{3}(\delta _{1}z,\delta
_{2}t)dzdt\right\vert \\
&\lesssim &\int_{\left\vert t\right\vert \leq \frac{1}{\delta _{2}}%
}\int_{\left\vert z\right\vert \leq \frac{1}{\delta _{1}}}\frac{1}{%
\left\vert z\right\vert ^{2n}}\frac{\left\vert z\right\vert ^{2}}{\left\vert
z\right\vert ^{4}+t^{2}}\delta _{1}\left\vert z\right\vert
dzdt+\int_{\left\vert t\right\vert \leq \frac{1}{\delta _{2}}%
}\int_{\left\vert z\right\vert \leq \frac{1}{\delta _{1}}}\frac{1}{%
\left\vert z\right\vert ^{2n}}\frac{\left\vert t\right\vert }{\left\vert
z\right\vert ^{4}+t^{2}}\delta _{1}\left\vert z\right\vert \delta
_{2}\left\vert t\right\vert dzdt \\
&=&I+II.
\end{eqnarray*}%
Now if $\frac{1}{\delta _{2}}\leq \left\vert z\right\vert ^{2}$, then 
\begin{equation*}
I\lesssim \delta _{1}\int_{\left\vert z\right\vert \leq \frac{1}{\delta _{1}}%
}\frac{1}{\left\vert z\right\vert ^{2n-3}}\left\{ \int_{0}^{\frac{1}{\delta
_{2}}}\frac{1}{\left\vert z\right\vert ^{4}}dt\right\} dz\lesssim \delta
_{1}\int_{\left\vert z\right\vert \leq \frac{1}{\delta _{1}}}\frac{1}{%
\left\vert z\right\vert ^{2n-1}}dz\approx \delta _{1}\int_{0}^{\frac{1}{%
\delta _{1}}}dr=1,
\end{equation*}%
while if $\frac{1}{\delta _{2}}>\left\vert z\right\vert ^{2}$, then%
\begin{equation*}
I\lesssim \delta _{1}\int_{\left\vert z\right\vert \leq \frac{1}{\delta _{1}}%
}\frac{1}{\left\vert z\right\vert ^{2n-3}}\left\{ \int_{0}^{\left\vert
z\right\vert ^{2}}\frac{1}{\left\vert z\right\vert ^{4}}dt+\int_{\left\vert
z\right\vert ^{2}}^{\frac{1}{\delta _{2}}}\frac{1}{t^{2}}dt\right\}
dz\lesssim \delta _{1}\int_{\left\vert z\right\vert \leq \frac{1}{\delta _{1}%
}}\frac{1}{\left\vert z\right\vert ^{2n-1}}dz\approx 1.
\end{equation*}%
Finally, we have%
\begin{equation*}
II\lesssim \delta _{1}\int_{\left\vert z\right\vert \leq \frac{1}{\delta _{1}%
}}\frac{1}{\left\vert z\right\vert ^{2n-1}}\left\{ \delta
_{2}\int_{\left\vert t\right\vert \leq \frac{1}{\delta _{2}}}\frac{t^{2}}{%
\left\vert z\right\vert ^{4}+t^{2}}dt\right\} dz\lesssim \delta
_{1}\int_{\left\vert z\right\vert \leq \frac{1}{\delta _{1}}}\frac{1}{%
\left\vert z\right\vert ^{2n-1}}dz\approx 1.
\end{equation*}

\part{Spaces of homogenous type}

Here in Part 3, we turn to a modest beginnning of an extension of the
implicit two-parameter theory to the general context of spaces of
homogeneous type. Recall that the general theory of Hardy spaces in spaces
of homogeneous type is limited to a single vanishing moment condition, and
hence to $p_{0}<p\leq 1$ where $p_{0}$ is the exponent determined by having
just one vanishing moment condition. The tools required for an extension of
the theory to two implict parameters will include an appropriate \emph{flag}
dyadic decomposition of the space, and an analogue of the covering lemma of
Journ\'{e} and Pipher.

We begin by constructing a \emph{flag} dyadic decompositon of the Heisenberg
group using two different proofs. The first uses the tiling theorem of
Strichartz, and the second uses a `hands-on' construction, which has the
advantage that it generalizes to certain products of spaces of homogeneous
type. We end by indicating how to extend Journ\'{e}'s covering lemma to the
Heisenberg group, and more generally to the aforementioned product spaces.

\section{The Heisenberg grid\label{s Heisenberg grid}}

Let $\mathbb{H}^{n}=\mathbb{C}^{n}\times \mathbb{R}$ be the Heisenberg group
with group multiplication%
\begin{equation*}
\left( \zeta ,t\right) \cdot \left( \eta ,s\right) =\left( \zeta +\eta ,t+s+2%
\func{Im}\left( \zeta \cdot \overline{\eta }\right) \right) ,\ \ \ \ \
\left( \zeta ,t\right) ,\left( \eta ,s\right) \in \mathbb{C}^{n}\times 
\mathbb{R}.
\end{equation*}%
Note that $\left( \eta ,s\right) ^{-1}=\left( -\eta ,-s\right) $. Relative
to this multiplication we define the dilation%
\begin{equation*}
\delta _{\lambda }\left( \zeta ,t\right) =\left( \lambda \zeta ,\lambda
^{2}t\right) ,
\end{equation*}%
and its corresponding "norm" on $\mathbb{H}^{n}$ by%
\begin{equation*}
\rho \left( \zeta ,t\right) =\sqrt[4]{\left\vert \zeta \right\vert ^{4}+t^{2}%
}.
\end{equation*}%
Then we define a symmetric quasimetric $d$ on $\mathbb{H}^{n}$ by%
\begin{equation*}
d\left( \left( \zeta ,t\right) ,\left( \eta ,s\right) \right) =\rho \left(
\left( \zeta ,t\right) \cdot \left( \eta ,s\right) ^{-1}\right) ,
\end{equation*}%
and note that%
\begin{equation*}
d\left( \delta _{\lambda }\left( \zeta ,t\right) ,\delta _{\lambda }\left(
\eta ,s\right) \right) =\lambda d\left( \left( \zeta ,t\right) ,\left( \eta
,s\right) \right) .
\end{equation*}

The center of the group $\mathbb{H}^{n}$ is 
\begin{equation*}
\mathcal{Z}^{n}=\left\{ \left( \zeta ,t\right) \in \mathbb{H}^{n}:\zeta
=0\right\} ,
\end{equation*}%
which is isomorphic to the abelian group $\mathbb{R}$. The quotient group $%
\mathbb{Q}^{n}=\mathbb{H}^{n}/\mathcal{Z}^{n}$ consists of equivalence
classes $\left[ \left( \zeta ,t\right) \right] $ such that $\left[ \left(
\zeta ,t\right) \right] =\left[ \left( \eta ,s\right) \right] $ if and only
if%
\begin{equation*}
\left( \zeta ,t\right) \cdot \left( \eta ,s\right) ^{-1}\in \mathcal{Z}%
^{n},\ \ \ \ \ i.e.\ \zeta =\eta .
\end{equation*}%
Thus we may identify $\mathbb{Q}^{n}$ with $\mathbb{C}^{n}$ as abelian
groups. Thus we see that $\mathbb{H}^{n}=\mathbb{C}^{n}\otimes _{twist}%
\mathbb{R}$ is a twisted group product of the abelian groups $\mathbb{C}^{n}$
and $\mathbb{R}$.

Now we apply the usual dyadic decomposition to the quotient metric space $%
\mathbb{Q}^{n}=\mathbb{C}^{n}$ to obtain a grid of "almost balls" (which are
actually cubes here)%
\begin{equation*}
\left\{ I\right\} _{I\text{ dyadic}}=\left\{ I_{\alpha }^{j}\right\} _{j\in 
\mathbb{Z}\text{ and }\alpha \in 2^{j}\mathbb{Z}^{2n}}
\end{equation*}%
where $I_{0}^{j}=\left[ 0,2^{j}\right) ^{2n}$ and $I_{\alpha
}^{j}=I_{0}^{j}+\alpha $ for $j\in \mathbb{Z}$ and $\alpha \in 2^{j}\mathbb{Z%
}^{2n}$, so that $\ell \left( I_{\alpha }^{j}\right) =2^{j}$. By a grid of
almost balls we mean that the sets $I_{\alpha }^{j}$ decompose $\mathbb{C}%
^{n}$ at each scale $2^{j}$, are almost balls, and are nested at differing
scales, i.e. there are positive constants $C_{1},C_{2}$ and points $%
c_{I_{\alpha }^{j}}\in I_{\alpha }^{j}$ such that%
\begin{eqnarray}
\mathbb{C}^{n} &=&\overset{\cdot }{\cup }_{\alpha }I_{\alpha }^{j},\ \ \ \ \
j\in \mathbb{Z},  \label{grid1} \\
B\left( c_{I_{\alpha }^{j}},C_{1}2^{j}\right) &\subset &I_{\alpha
}^{j}\subset B\left( c_{I_{\alpha }^{j}},C_{2}2^{j}\right) \ \ \ \ \ j\in 
\mathbb{Z},\alpha \in 2^{j}\mathbb{Z}^{2n},  \notag \\
\text{either }I_{\alpha ^{\prime }}^{j^{\prime }} &\subset &I_{\alpha }^{j}%
\text{, }I_{\alpha }^{j}\subset I_{\alpha ^{\prime }}^{j^{\prime }}\text{ or 
}I_{\alpha ^{\prime }}^{j^{\prime }}=I_{\alpha }^{j}.  \notag
\end{eqnarray}%
Here we can take $c_{I}$ to be the center of the cube $I$, and $C_{1}=\frac{1%
}{2}$, $C_{2}=\frac{\sqrt{2n}}{2}=\sqrt{\frac{n}{2}}$. We also have the
usual dyadic grid $\left\{ J_{\tau }^{k}\right\} _{k\in \mathbb{Z}\text{ and 
}\tau \in 2^{k}\mathbb{Z}}$ for $\mathbb{R}$ where $J_{0}^{k}=\left[
0,2^{k}\right) $ and $I_{\tau }^{k}=I_{0}^{k}+\tau $ for $k\in \mathbb{Z}$
and $\tau \in 2^{k}\mathbb{Z}$.

In order to use these grids to construct a "product-like" grid for $\mathbb{H%
}^{n}$ we must take into account the twisted structure of the product $%
\mathbb{H}^{n}=\mathbb{C}^{n}\otimes _{twist}\mathbb{R}$. Here is our
theorem on the existence of a twisted grid for $\mathbb{H}^{n}$.

\begin{theorem}
\label{Heisenberg grid}There is a positive integer $m$ and positive
constants $C_{1},C_{2}$, such that for each $j\in m\mathbb{Z}$ and%
\begin{equation*}
\left( \alpha ,\tau \right) \in K_{j}\equiv 2^{j}\mathbb{Z}^{2n}\times 2^{2j}%
\mathbb{Z},
\end{equation*}%
there are subsets $\mathcal{S}_{j,\alpha ,\tau }$ of $\mathbb{H}^{n}$\
satisfying%
\begin{eqnarray}
\mathbb{H}^{n} &=&\overset{\cdot }{\cup }_{\left( \alpha ,\tau \right) \in
K_{j}}\mathcal{S}_{j,\alpha ,\tau },\ \ \ \ \ \text{for each }j\in m\mathbb{Z%
},  \label{productgrid} \\
P_{\mathbb{C}^{n}}\mathcal{S}_{j,\alpha ,\tau } &=&I_{\alpha }^{j},\ \ \ \ \
j\in m\mathbb{Z},\left( \alpha ,\tau \right) \in K_{j},  \notag \\
B_{d}\left( c_{j,\alpha ,\tau },C_{1}2^{j}\right) &\subset &\mathcal{S}%
_{j,\alpha ,\tau }\subset B_{d}\left( c_{j,\alpha ,\tau },C_{2}2^{j}\right)
\ \ \ \ \ j\in m\mathbb{Z},\left( \alpha ,\tau \right) \in K_{j},  \notag \\
\text{either }\mathcal{S}_{j,\alpha ,\tau } &\subset &\mathcal{S}_{j^{\prime
},\alpha ^{\prime },\tau ^{\prime }}\text{, }\mathcal{S}_{j^{\prime },\alpha
^{\prime },\tau ^{\prime }}\subset \mathcal{S}_{j,\alpha ,\tau }\text{ or }%
\mathcal{S}_{j,\alpha ,\tau }\cap \mathcal{S}_{j^{\prime },\alpha ^{\prime
},\tau ^{\prime }}=\phi ,  \notag \\
c_{j,\alpha ,\tau } &=&\left( P_{j,\alpha },\tau +\frac{1}{2}2^{2j}\right) ,
\notag
\end{eqnarray}%
where $P_{j,\alpha }=c_{I_{\alpha }^{j}}$ and $P_{\mathbb{C}^{n}}$ denotes
orthogonal projection of $\mathbb{H}^{n}$ onto $\mathbb{C}^{n}$.
\end{theorem}

Thus at each dyadic scale $2^{j}$ with $j\in m\mathbb{Z}$ we have a pairwise
disjoint decomposition of $\mathbb{H}^{n}$ into sets $\mathcal{S}_{j,\alpha
,\tau }$ that are almost Heisenberg balls of radius $2^{j}$. These
decompositions are nested, and moreover are \emph{product-like} in the sense
that the sets $\mathcal{S}_{j,\alpha ,\tau }$ project onto the usual dyadic
grid in the factor $\mathbb{C}^{n}$, and have centers $c_{j,\alpha ,\tau
}=\left( P_{j,\alpha },\tau +\frac{1}{2}2^{2j}\right) $ that for each $j$
form a product set indexed by $K_{j}\equiv 2^{j}\mathbb{Z}^{2n}\times 2^{2j}%
\mathbb{Z}$ and satisfy%
\begin{equation*}
\left\vert c_{j,\alpha ,\tau }-c_{j,\alpha ^{\prime },\tau }\right\vert
=2^{j}\text{ and }\left\vert c_{j,\alpha ,\tau }-c_{j,\alpha ,\tau ^{\prime
}}\right\vert =2^{2j},
\end{equation*}%
if $\alpha $ and $\alpha ^{\prime }$ are neighbours in $2^{j}\mathbb{Z}^{2n}$%
, and if $\tau $ and $\tau ^{\prime }$ are neighbours in $2^{2j}\mathbb{Z}$.

\subsection{Self-similar tilings of the Heisenberg group}

Theorem \ref{Heisenberg grid} follows easily from the theory of self-similar
tilings (neatly stacked over dyadic cubes) in Strichartz \cite{Str}. An
excellent source for this material is pages 39 to 42 of Tyson \cite{Tys},
which we now briefly recall.

Let $b=2n+1$ be a dyadic division factor, and let $Q=2n+2$ be the
homogeneous dimension of the Heisenberg group $\mathbb{H}^{n}$. In order to
fix the geometry, we suppose that $n=1$, $b=3$ and $Q=4$. Let $\mathbf{k}%
=\left( k_{1},k_{2}\right) \in \left\{ 1,2,3\right\} ^{2}$ and $\ell \in
\left\{ 1,2,...,9\right\} $. As on page 39 of \cite{Tys} consider the
following collection of $81$ contractive similarities from $\mathbb{H}^{n}$
to itself:%
\begin{equation*}
F_{\mathbf{k},\ell }\left( z,t\right) =\left( z_{\mathbf{k}},t_{\ell
}\right) \delta _{\frac{1}{3}}\left( z,t\right) ,\ \ \ \ \ \left( z,t\right)
\in \mathbb{H}^{n},
\end{equation*}%
where 
\begin{equation*}
z_{\mathbf{k}}=\frac{k_{1}-2}{3}+i\frac{k_{2}-2}{3},\ \ \ \ \ 1\leq
k_{1},k_{2}\leq 3,
\end{equation*}%
and%
\begin{equation*}
t_{\ell }=\frac{\ell -5}{9},\ \ \ \ \ 1\leq \ell \leq 9.
\end{equation*}%
Each similarity has contraction ratio $\frac{1}{3}$ and they differ only in
the $81$ Heisenberg group translations. The corresponding iterated function
system (IFS) has a unique nonempty compact invariant set $T_{o}\subset 
\mathbb{H}^{n}$ characterized by the identity%
\begin{equation*}
T_{o}=\bigcup_{\mathbf{k},\ell }F_{\mathbf{k},\ell }\left( T_{o}\right) .
\end{equation*}

Now following page 41 of \cite{Tys} let $\mathbb{H}_{\mathbb{Z}}^{n}=\left( 
\mathbb{Z}+i\mathbb{Z}\right) \times \mathbb{Z}$ denote the integral
Heisenberg group and note that%
\begin{equation*}
\delta _{\frac{1}{3}}T_{o}=\bigcup_{p\in D}pT_{o},
\end{equation*}%
where $D$ is the set of points $p=\left( z,t\right) \in \mathbb{H}_{\mathbb{Z%
}}^{n}$ such that $\left\vert z\right\vert \leq 1$ and $\left\vert
t\right\vert \leq 4$ (for the relevance of these constants see Lemma 3.3 on
page 40 of \cite{Tys}). Iterating and passing to the limit we obtain the
following decompostion at scale one of $\mathbb{H}^{n}$:%
\begin{equation*}
\mathbb{H}^{n}=\bigcup_{p\in \mathbb{H}_{\mathbb{Z}}^{n}}pT_{o}.
\end{equation*}%
For $m\in \mathbb{Z}$ and $p\in \delta _{\frac{1}{3^{m}}}\mathbb{H}_{\mathbb{%
Z}}^{n}$ we obtain a decomposition at scale $3^{m}$:%
\begin{equation*}
\mathbb{H}^{n}=\bigcup_{p\in \delta _{\frac{1}{3^{m}}}\mathbb{H}_{\mathbb{Z}%
}^{n}}p\delta _{\frac{1}{3^{m}}}\left( T_{o}\right) .
\end{equation*}%
These decompositions are nested and Lemma 3.3 on page 40 of \cite{Tys} shows
that the sets $p\delta _{\frac{1}{3^{m}}}\left( T_{o}\right) $ are "almost
Heisenberg balls". Together with Lemma 3.4 on page 42 of \cite{Tys}, this
can be used to prove Theorem \ref{Heisenberg grid} with the sets $p\delta _{%
\frac{1}{3^{m}}}\left( T_{o}\right) $ playing the role of the sets $\mathcal{%
S}_{j,\alpha ,\tau }$ (with appropriate translation of notation).

\begin{remark}
The self-similarity approach also works more generally for nilpotent Lie
groups.
\end{remark}

\section{A grid in semiproducts of quasimetric spaces}

Theorem \ref{Heisenberg grid} can be generalized to the following setting of 
\emph{semiproducts} of quasimetric spaces where there is no group structure,
hence no self-similarity. Suppose that $\left( X,d_{X}\right) $ and $\left(
Y,d_{Y}\right) $ are quasimetric spaces. Suppose moreover that $d_{Z}$ is a
quasimetric on the product set $Z=X\times Y$ that satisfies the following
"semiproduct axiom":

\begin{axiom}
\label{semiquasi}There are positive constants $c$ and $C$ for which the
following holds: for each ball $B_{d_{X}}\left( x,r\right) $ in $\left(
X,d_{X}\right) $ there is $R>0$ and a collection of points $\left\{
y_{j}\right\} \subset Y$ satisfying both%
\begin{eqnarray*}
B_{d_{Y}}\left( y_{i},cR\right) \cap B_{d_{Y}}\left( y_{j},cR\right)
&=&\emptyset ,\ \ \ \ \ i\neq j, \\
\bigcup_{j}B_{d_{Y}}\left( y_{j},CR\right) &\supset &Y,
\end{eqnarray*}%
and%
\begin{equation*}
\overset{\cdot }{\bigcup_{j}}B_{d_{Z}}\left( \left( x,y_{j}\right)
,cr\right) \subset B_{X}\left( x,r\right) \times Y\subset
\bigcup_{j}B_{d_{Z}}\left( \left( x,y_{j}\right) ,Cr\right) .
\end{equation*}
\end{axiom}

The notation $\overset{\cdot }{\bigcup_{j}}$ means that the union over $j$
is pairwise disjoint. Thus the above axiom postulates that every "vertical
tube" $B_{X}\left( x,r\right) \times Y$ in the product space $X\times Y$ can
be covered by balls $B_{d_{Z}}\left( \left( x,y_{j}\right) ,Cr\right) $ with
the property that the smaller balls $B_{d_{Z}}\left( \left( x,y_{j}\right)
,cr\right) $ are pairwise disjoint, and moreover that the corresponding
balls $B_{d_{Y}}\left( y_{j},R\right) $ in $Y$ form an "almost
decomposition" of $Y$.

\begin{theorem}
\label{quasi grid}Suppose that $\left( X,d_{X}\right) $, $\left(
Y,d_{Y}\right) $ and $\left( Z,d_{Z}\right) $ are quasimetric spaces
satisfying Axiom \ref{semiquasi} with $Z=X\times Y$. Then there is a grid
for the quasimetric space $\left( Z,d_{Z}\right) $ that satisfies the
analogue of (\ref{semigrid}).
\end{theorem}

This theorem generalizes the construction of Christ in \cite{Chr}, and the
variant of Sawyer and Wheeden in \cite{SaWh}. Note also that the Heisenberg
group $\mathbb{H}^{n}$ is an example of a semiproduct of $\mathbb{C}^{n}$
and $\mathbb{R}$ in that Axiom \ref{semiquasi} holds with $X=\mathbb{C}^{n}$%
, $Y=\mathbb{R}$ and $Z=\mathbb{H}^{n}=\mathbb{C}^{n}\times \mathbb{R}$
together with the corresponding metrics.

\begin{remark}
There is a generalization of the M\"{u}ller-Ricci-Stein theory to this more
general setting.
\end{remark}

We will now give a proof of Theorem \ref{Heisenberg grid} that can be
generalized to prove Theorem \ref{quasi grid}.

\subsection{A twisted semigrid}

In order to construct the twisted product grid in (\ref{productgrid}), we
must first compute the shape of the Heisenberg balls. For small $r>0$ the
Heisenberg ball $B_{\rho }\left( \left( \zeta ,t\right) ,r\right) $ centered
at $\left( \zeta ,t\right) $ with radius $r$ is given by%
\begin{eqnarray*}
B_{\rho }\left( \left( \zeta ,t\right) ,r\right) &=&\left\{ \left( \eta
,s\right) \in \mathbb{H}^{n}:\rho \left[ \left( \eta ,s\right) \cdot \left(
\zeta ,t\right) ^{-1}\right] <r\right\} \\
&=&\left\{ \left( \eta ,s\right) \in \mathbb{H}^{n}:\rho \left[ \left( \eta
-\zeta ,s-t-2\func{Im}\left( \eta \cdot \overline{\zeta }\right) \right) %
\right] <r\right\} \\
&=&\left\{ \left( \eta ,s\right) \in \mathbb{H}^{n}:\left\vert \eta -\zeta
\right\vert ^{4}+\left( s-t-2\func{Im}\left( \eta \cdot \overline{\zeta }%
\right) \right) ^{2}<r^{4}\right\} .
\end{eqnarray*}%
Now take $n=1$ for the moment, let $\left( \zeta ,t\right) =P_{R}=\left(
R,0\right) \in \mathbb{C}\times \mathbb{R}$ and $\eta =x+iy$ and consider
the equation for the surface $\mathbf{\sigma }=\partial B_{\rho }\left(
P_{R},r\right) $:%
\begin{equation}
\left[ \left( x-R\right) ^{2}+y^{2}\right] ^{2}+\left( s-2Ry\right)
^{2}=r^{4}.  \label{equation}
\end{equation}%
For $r$ small (large), $\mathbf{\sigma }$ is a pancake (cigar) shaped
surface straddling the elliptical disk $D\left( P_{R},r\right) $ having
boundary given by the ellipse 
\begin{equation*}
\left\vert \left( x-R,y\right) \right\vert =r,s=2Ry,
\end{equation*}%
in three dimensions. The unit normal vector to the elliptical disk $D\left(
P_{R},r\right) $ is 
\begin{equation*}
\mathbf{n}_{\left( R,0\right) }=\left( 0,\frac{-2R}{\sqrt{1+4R^{2}}},\frac{1%
}{\sqrt{1+4R^{2}}}\right) .
\end{equation*}%
The Heisenberg balls $B_{\rho }\left( \left( \zeta ,t\right) ,r\right) $ are
rotation invariant in $\zeta $ and translation invariant in $t$, and so we
obtain that if $\left( \zeta ,t\right) =\left( R\cos \theta +iR\sin \theta
,t\right) =Rot_{\theta }\left( R,t\right) $, then $D\left( \left( \zeta
,t\right) ,r\right) =Rot_{\theta }D\left( P_{R},r\right) $ and the normal to 
$D\left( \left( \zeta ,t\right) ,r\right) $ is 
\begin{equation}
\mathbf{n}_{\left( \zeta ,t\right) }=Rot_{\theta }\mathbf{n}_{\left(
R,0\right) }=\left( \frac{2R}{\sqrt{1+4R^{2}}}\sin \theta ,\frac{-2R}{\sqrt{%
1+4R^{2}}}\cos \theta ,\frac{1}{\sqrt{1+4R^{2}}}\right) .  \label{normal}
\end{equation}%
We will refer to $D\left( \left( \zeta ,t\right) ,r\right) $ as the \emph{%
straddling disk} for the ball $B_{d}\left( \left( \zeta ,t\right) ,r\right) $%
. The situation is similar in $\mathbb{C}^{n}$ for $n>1$.

Now fix a dyadic cube $I_{\alpha }^{j}$ of side length $2^{j}$ with centre $%
P_{j,\alpha }$ in $\mathbb{C}^{n}$ and consider the infinite rectangular box$%
\ \mathcal{T}_{j,\alpha }\equiv I_{\alpha }^{j}\times \mathbb{R}$. For each
point $\tau \in 2^{2j}\mathbb{Z}$, let $H_{j,\alpha ,\tau }$ be the
hyperplane through the point $P_{j,\alpha }+\left( 0,\tau \right) \in 
\mathbb{H}^{n}$ with normal vector $\mathbf{n}_{P_{j,\alpha }}$. Let $%
\mathcal{H}_{j,\alpha ,\tau }$ be the region between the hyperplanes $%
H_{j,\alpha ,\tau }$ and $H_{j,\alpha ,\tau +2^{2j}}$ including $H_{j,\alpha
,\tau }$ but not $H_{j,\alpha ,\tau +2^{2j}}$. Then for each $j,\alpha $ we
have $\mathbb{H}^{n}=\overset{\cdot }{\cup }_{\tau \in 2^{2j}\mathbb{Z}}%
\mathcal{H}_{j,\alpha ,\tau }$. We decompose the rectangular box $\mathcal{T}%
_{j,\alpha }$ into preliminary pairwise disjoint slabs%
\begin{equation*}
\mathcal{S}_{j,\alpha ,\tau }\equiv \mathcal{T}_{j,\alpha }\cap \mathcal{H}%
_{j,\alpha ,\tau }
\end{equation*}%
so that%
\begin{equation}
\mathcal{T}_{j,\alpha }=\overset{\cdot }{\cup }_{\tau \in 2^{2j}\mathbb{Z}%
}\left\{ \mathcal{T}_{j,\alpha }\cap \mathcal{H}_{j,\alpha ,\tau }\right\} =%
\overset{\cdot }{\cup }_{\tau \in 2^{2j}\mathbb{Z}}\mathcal{S}_{j,\alpha
,\tau }.  \label{tube}
\end{equation}

At this point we note that the collection of slabs $\left\{ \mathcal{S}%
_{j,\alpha ,\tau }\right\} _{j\in \mathbb{Z},\alpha \in 2^{j}\mathbb{Z}%
^{2n},\tau \in 2^{2j}\mathbb{Z}}$ is a \emph{semigrid} of almost balls in $%
\mathbb{H}^{n}$ in the following sense.

\begin{lemma}
\label{semigrid}The collection $\left\{ \mathcal{S}_{j,\alpha ,\tau
}\right\} _{j\in \mathbb{Z},\alpha \in 2^{j}\mathbb{Z}^{2n},\tau \in 2^{2j}%
\mathbb{Z}}$ satisfies 
\begin{eqnarray}
\mathbb{H}^{n} &=&\overset{\cdot }{\cup }_{\alpha \in 2^{j}\mathbb{Z}%
^{2n},\tau \in 2^{2j}\mathbb{Z}}\mathcal{S}_{j,\alpha ,\tau },\ \ \ \ \ 
\text{for each }j\in \mathbb{Z},  \label{grid2} \\
B_{\rho }\left( c_{j,\alpha ,\tau },C_{1}2^{j}\right) &\subset &\mathcal{S}%
_{j,\alpha ,\tau }\subset B_{\rho }\left( c_{j,\alpha ,\tau
},C_{2}2^{j}\right) \ \ \ \ \ \text{for each }j\in \mathbb{Z},\alpha \in
2^{j}\mathbb{Z}^{2n},\tau \in 2^{2j}\mathbb{Z},  \notag
\end{eqnarray}%
for some positive constants $C_{1},C_{2}$ and where $c_{j,\alpha ,\tau }$ is
the center of the slab $\mathcal{S}_{j,\alpha ,\tau }$. We may take $C_{1}=%
\frac{1}{\sqrt{2}}$ and $C_{2}=\frac{\sqrt[4]{n^{2}+2}}{\sqrt{2}}$.
\end{lemma}

\textbf{Proof}: We have already observed just after (\ref{grid1}) that\ in $%
\mathbb{C}^{n}$ we have%
\begin{equation*}
B\left( c_{I_{\alpha }^{j}},\frac{1}{2}2^{j}\right) \subset I_{\alpha
}^{j}\subset B\left( c_{I_{\alpha }^{j}},\sqrt{\frac{n}{2}}2^{j}\right) \ \
\ \ \ j\in \mathbb{Z},\alpha \in 2^{j}\mathbb{Z}^{2n}.
\end{equation*}%
Thus in order to prove the second containment in the second line of (\ref%
{grid2}) it suffices to show that for each point $\left( x,y\right) \in
B\left( c_{I_{\alpha }^{j}},\sqrt{\frac{n}{2}}2^{j}\right) $, the
intersection of the vertical line $\mathsf{L}_{\left( x,y\right) }$ through $%
\left( x,y\right) $ with the slab $\mathcal{S}_{j,\alpha ,\tau }$ is
contained in $\mathsf{L}_{\left( x,y\right) }\cap B_{\rho }\left(
c_{j,\alpha ,\tau },C_{2}2^{j}\right) $ provided $C_{2}$ is chosen large
enough. For convenience we suppose that $c_{I_{\alpha }^{j}}=R$ and $%
c_{j,\alpha ,\tau }=\left( R,0\right) $. But then from the definition of $%
\mathcal{S}_{j,\alpha ,\tau }$ we have 
\begin{equation*}
\mathsf{L}_{\left( x,y\right) }\cap \mathcal{S}_{j,\alpha ,\tau }=\left\{
\left( x,y,s\right) :2Ry-\frac{1}{2}2^{2j}\leq s<2Ry+\frac{1}{2}%
2^{2j}\right\} ,
\end{equation*}%
and from the equation (\ref{equation}) we have%
\begin{eqnarray*}
\left( s-2Ry\right) ^{2} &=&\left( C_{2}2^{j}\right) ^{4}-\left[ \left(
x-R\right) ^{2}+y^{2}\right] ^{2} \\
&\geq &\left( C_{2}2^{j}\right) ^{4}-\left( \sqrt{\frac{n}{2}}2^{j}\right)
^{4}=\left( C_{2}^{4}-\frac{n^{2}}{4}\right) 2^{4j},
\end{eqnarray*}%
and so%
\begin{eqnarray*}
&&\mathsf{L}_{\left( x,y\right) }\cap B_{\rho }\left( \left( R,0\right)
,C_{2}2^{j}\right) \\
&\supset &\left\{ \left( x,y,s\right) :2Ry-\sqrt{C_{2}^{4}-\frac{n^{2}}{4}}%
2^{2j}<s<2Ry+\sqrt{C_{2}^{4}-\frac{n^{2}}{4}}2^{2j}\right\} .
\end{eqnarray*}%
Altogether then we obtain%
\begin{equation*}
\mathsf{L}_{\left( x,y\right) }\cap \mathcal{S}_{j,\alpha ,\tau }\subset 
\mathsf{L}_{\left( x,y\right) }\cap B_{\rho }\left( \left( R,0\right)
,C_{2}2^{j}\right)
\end{equation*}%
provided $C_{2}^{4}-\frac{n^{2}}{4}\geq \frac{1}{2}$ or $C_{2}\geq \frac{%
\sqrt[4]{n^{2}+2}}{\sqrt{2}}$.

Turning to the first containment in (\ref{grid2}), we note that for each
point $\left( x,y\right) \in B\left( R,C_{1}2^{j}\right) $, (\ref{equation})
yields%
\begin{equation*}
\left( s-2Ry\right) ^{2}=\left( C_{1}2^{j}\right) ^{4}-\left[ \left(
x-R\right) ^{2}+y^{2}\right] ^{2}\leq C_{1}^{4}2^{4j}
\end{equation*}%
and so%
\begin{eqnarray*}
&&\mathsf{L}_{\left( x,y\right) }\cap B_{\rho }\left( \left( R,0\right)
,C_{1}2^{j}\right) \\
&\subset &\left\{ \left( x,y,s\right)
:2Ry-C_{1}^{2}2^{2j}<s<2Ry+C_{1}^{2}2^{2j}\right\} .
\end{eqnarray*}%
Thus we have%
\begin{equation*}
\mathsf{L}_{\left( x,y\right) }\cap B_{\rho }\left( \left( R,0\right)
,C_{1}2^{j}\right) \subset \mathsf{L}_{\left( x,y\right) }\cap \mathcal{S}%
_{j,\alpha ,\tau }
\end{equation*}%
provided $C_{1}^{2}\leq \frac{1}{2}$ or $C_{1}\leq \frac{1}{\sqrt{2}}$.

Thus we may take $C_{1}=\frac{1}{\sqrt{2}}$ and $C_{2}=\frac{\sqrt[4]{n^{2}+2%
}}{\sqrt{2}}$ in (\ref{grid2}), and this completes the proof of Lemma \ref%
{semigrid}.

\bigskip

However, the collection of slabs in Lemma \ref{semigrid} fails to satisfy
the corresponding nesting property since the slabs $\mathcal{S}_{j-1,\alpha
^{\prime },\tau ^{\prime }}$ corresponding to dyadic subcubes $I_{\alpha
^{\prime }}^{j-1}$ of a given dyadic cube $I_{\alpha }^{j}$ have different
normal vectors. Nevertheless, (\ref{normal}) shows that these normal vectors
are very close and indeed, the fact that $d$ is a quasimetric can be used to
modify the slabs by adding and subtracting portions near the boundary in
such a way as to preserve the semigrid properties while achieving the
nesting property. We now turn to the details.

\subsection{A truncated twisted grid}

Recall\ the index set $K_{j}=2^{j}\mathbb{Z}^{2n}\times 2^{2j}\mathbb{Z}$.
Fix a large integer $M$ (the integer of truncation) and consider the
decomposition of $\mathbb{H}^{n}$ given by the first line in (\ref{grid2})
with $j=-M$, i.e. $\mathbb{H}^{n}=\overset{\cdot }{\cup }\left\{ \mathcal{S}%
_{-M,\alpha ,\tau }:\left( \alpha ,\tau \right) \in K_{-M}\right\} $. Let $m$
be a positive integer that will be chosen sufficiently large below. We
construct new slabs $\widetilde{\mathcal{S}}_{m-M,\alpha ,\tau }$ for $%
\left( \alpha ,\tau \right) \in K_{m-M}$ so that every slab $\mathcal{S}%
_{-M,\alpha ^{\prime },\tau ^{\prime }}$ with $\left( \alpha ^{\prime },\tau
^{\prime }\right) \in K_{-M}$ is contained in a new slab $\widetilde{%
\mathcal{S}}_{m-M,\alpha ,\tau }$ for some $\left( \alpha ,\tau \right) \in
K_{m-M}$. We perform the construction of $\widetilde{\mathcal{S}}%
_{m-M,\alpha ,\tau }$ within the rectangular box $\mathcal{T}_{m-M,\alpha }$%
. So fix $m-M$ and $\alpha $.

First we note that no slab $\mathcal{S}_{-M,\alpha ,\tau }$ at level $-M$
can simultaneously intersect two \emph{different} balls $B_{\rho }\left(
c_{m-M,\alpha ^{\prime },\tau ^{\prime }},C_{1}2^{m-M}\right) $ and $B_{\rho
}\left( c_{m-M,\alpha ^{\prime },\tau ^{\prime \prime }},C_{1}2^{m-M}\right) 
$ at level $n-M$ if $C_{1}$ is small enough and $m$ is large enough. Indeed,
let $K$ be the quasimetric constant for $d$, and suppose that $\mathcal{S}%
_{-M,\alpha ,\tau }$ has nonempty intersection with the ball $B_{\rho
}\left( c_{m-M,\alpha ^{\prime },\tau ^{\prime }},C_{1}2^{m-M}\right) $. By
the second containment in (\ref{grid2}) we have%
\begin{equation*}
\mathcal{S}_{-M,\alpha ,\tau }\subset B_{\rho }\left( c_{-M,\alpha ,\tau
},C_{2}2^{-M}\right) ,\ \ \ \ \ C_{2}=\frac{\sqrt[4]{n^{2}+2}}{\sqrt{2}},
\end{equation*}%
and so the triangle inequality shows that every point $\left( \zeta
,t\right) \in \mathcal{S}_{-M,\alpha ,\tau }$ satisfies%
\begin{equation*}
d\left( \left( \zeta ,t\right) ,c_{m-M,\alpha ^{\prime },\tau ^{\prime
}}\right) \leq K\left[ C_{2}2^{-M}+C_{1}2^{m-M}\right] .
\end{equation*}%
Now assume that the point $\left( \zeta ,t\right) $ also lies in the other
ball $B_{\rho }\left( c_{m-M,\alpha ^{\prime },\tau ^{\prime \prime
}},C_{1}2^{m-M}\right) $. Then we would obtain by the triangle inequality
that%
\begin{eqnarray*}
d\left( c_{m-M,\alpha ^{\prime },\tau ^{\prime \prime }},c_{m-M,\alpha
^{\prime },\tau ^{\prime }}\right) &\leq &K\left[ d\left( c_{m-M,\alpha
^{\prime },\tau ^{\prime \prime }},\left( \zeta ,t\right) \right) +d\left(
\left( \zeta ,t\right) ,c_{m-M,\alpha ^{\prime },\tau ^{\prime }}\right) %
\right] \\
&\leq &K\left[ C_{1}2^{m-M}+K\left[ C_{2}2^{-M}+C_{1}2^{m-M}\right] \right] .
\end{eqnarray*}%
However, since the balls are \emph{different} there is a positive constant
depending only on $n$ such that%
\begin{equation*}
d\left( c_{m-M,\alpha ^{\prime },\tau ^{\prime \prime }},c_{m-M,\alpha
^{\prime },\tau ^{\prime }}\right) \geq c2^{m-M}.
\end{equation*}%
Combining the latter two inequalities we obtain%
\begin{equation}
c2^{m-M}\leq \left( K+K^{2}\right) C_{1}2^{m-M}+K^{2}C_{2}2^{-M},\ \ \ \ \
C_{2}=\frac{\sqrt[4]{n^{2}+2}}{\sqrt{2}}.  \label{contrainequ}
\end{equation}%
Clearly (\ref{contrainequ}) cannot hold if we take $C_{1}\leq \frac{1}{\sqrt{%
2}}$ small enough and $m$ large enough, e.g. if%
\begin{equation*}
\left( K+K^{2}\right) C_{1}\leq \frac{c}{3}\text{ and }K^{2}\frac{\sqrt[4]{%
n^{2}+2}}{\sqrt{2}}\leq \frac{c}{3}2^{m}.
\end{equation*}

Moreover, every slab $\mathcal{S}_{-M,\alpha ,\tau }$ is contained in some
ball $B_{\rho }\left( c_{m-M,\alpha ^{\prime },\tau ^{\prime
}},C_{2}2^{m-M}\right) $ if $C_{2}$ is large enough. Thus we can assign each
slab $\mathcal{S}_{-M,\alpha ,\tau }$ with $\left( \alpha ,\tau \right) \in
K_{-M}$ to one of the slabs $\mathcal{S}_{m-M,\alpha ^{\prime },\tau
^{\prime }}$ with $\left( \alpha ^{\prime },\tau ^{\prime }\right) \in
K_{m-M}$ in such a way that if $\widetilde{\mathcal{S}}_{m-M,\alpha ^{\prime
},\tau ^{\prime }}$ is the union of all the slabs $\mathcal{S}_{-M,\alpha
,\tau }$ that have been assigned to $\mathcal{S}_{m-M,\alpha ^{\prime },\tau
^{\prime }}$, then 
\begin{equation*}
B_{\rho }\left( c_{m-M,\alpha ,\tau },C_{1}2^{m-M}\right) \subset \widetilde{%
\mathcal{S}}_{m-M,\alpha ,\tau }\subset B_{\rho }\left( c_{m-M,\alpha ,\tau
},C_{2}2^{m-M}\right) ,\ \ \ \ \ \left( \alpha ,\tau \right) \in K_{m-M}.
\end{equation*}%
In fact we will use the following assignment scheme: if $\mathcal{S}%
_{-M,\alpha ,\tau }$ is contained in $\mathcal{S}_{m-M,\alpha ^{\prime
},\tau ^{\prime }}$ then we assign $\mathcal{S}_{-M,\alpha ,\tau }$ to $%
\mathcal{S}_{m-M,\alpha ^{\prime },\tau ^{\prime }}$. If $\mathcal{S}%
_{-M,\alpha ,\tau }$ intersects both $\mathcal{S}_{m-M,\alpha ^{\prime
},\tau ^{\prime }}$ and $\mathcal{S}_{m-M,\alpha ^{\prime },\tau ^{\prime
\prime }}$ where $\mathcal{S}_{m-M,\alpha ^{\prime },\tau ^{\prime }}$ lies
underneath $\mathcal{S}_{m-M,\alpha ^{\prime },\tau ^{\prime \prime }}$,
then we assign $\mathcal{S}_{-M,\alpha ,\tau }$ to $\mathcal{S}_{m-M,\alpha
^{\prime },\tau ^{\prime }}$.

By convention we set $\widetilde{\mathcal{S}}_{-M,\alpha ,\tau }=\mathcal{S}%
_{-M,\alpha ,\tau }$ for $\left( \alpha ,\tau \right) \in K_{-M}$. We now
inductively define in similar fashion new (rough) slabs $\widetilde{\mathcal{%
S}}_{j,\alpha ,\tau }$ for $\left( \alpha ,\tau \right) \in K_{j}$ and $%
j=2m-M,3m-M,...$ to be appropriate unions of the new slabs $\widetilde{%
\mathcal{S}}_{j-m,\alpha ^{\prime },\tau ^{\prime }}$ constructed in the
previous step. Provided $m$, $C_{1}$ and $C_{2}$ are chosen appropriately,
we obtain in this way a \emph{truncated} grid $\left\{ \widetilde{\mathcal{S}%
}_{j,\alpha ,\tau }\right\} _{j\in m\mathbb{Z}_{+}-M,\left( \alpha ,\tau
\right) \in K_{j}}$.

\begin{lemma}
\label{trungrid}The collection $\left\{ \widetilde{\mathcal{S}}_{j,\alpha
,\tau }\right\} _{j\in m\mathbb{Z}_{+}-M,\left( \alpha ,\tau \right) \in
K_{j}}$ satisfies 
\begin{eqnarray}
\mathbb{H}^{n} &=&\overset{\cdot }{\cup }_{\left( \alpha ,\tau \right) \in
K_{j}}\widetilde{\mathcal{S}}_{j,\alpha ,\tau },\ \ \ \ \ j\in m\mathbb{Z}%
_{+}-M,  \label{grid3} \\
B_{\rho }\left( c_{j,\alpha ,\tau },C_{1}2^{j}\right) &\subset &\widetilde{%
\mathcal{S}}_{j,\alpha ,\tau }\subset B_{\rho }\left( c_{j,\alpha ,\tau
},C_{2}2^{j}\right) \ \ \ \ \ j\in m\mathbb{Z}_{+}-M,\left( \alpha ,\tau
\right) \in K_{j},  \notag \\
\text{either }\widetilde{\mathcal{S}}_{j,\alpha ,\tau } &\subset &\widetilde{%
\mathcal{S}}_{j^{\prime },\alpha ^{\prime },\tau ^{\prime }}\text{, }%
\widetilde{\mathcal{S}}_{j^{\prime },\alpha ^{\prime },\tau ^{\prime
}}\subset \widetilde{\mathcal{S}}_{j,\alpha ,\tau }\text{ or }\widetilde{%
\mathcal{S}}_{j,\alpha ,\tau }\cap \widetilde{\mathcal{S}}_{j^{\prime
},\alpha ^{\prime },\tau ^{\prime }}=\phi .  \notag
\end{eqnarray}
\end{lemma}

The proof of Lemma \ref{trungrid} is a straightforward exercise.

\begin{remark}
\label{originalblock}We emphasize that each of the new slabs $\widetilde{%
\mathcal{S}}_{j,\alpha ,\tau }$ is a union of a subset of the original
building blocks $\mathcal{S}_{-M,\alpha ^{\prime },\tau ^{\prime }}$, $%
\left( \alpha ^{\prime },\tau ^{\prime }\right) \in K_{-M}$.
\end{remark}

\subsection{The full twisted grid}

It remains to extend this truncated grid to a full grid defined for all $%
j\in m\mathbb{Z}$ and $\left( \alpha ,\tau \right) \in K_{j}$. So pick $M=mk$
for a large positive integer $k$, and denote slabs $\widetilde{\mathcal{S}}%
_{j,\alpha ,\tau }$ in the truncated grid constructed in (\ref{grid3}) above
by $\mathcal{S}_{j,\alpha ,\tau }^{k}$. We now start the construction at the
next level down, namely $j=-m-M=-m\left( k+1\right) $, and obtain \emph{%
different} slabs $\mathcal{S}_{-mk,\alpha ,\tau }^{k+1}$ than the original
slabs $\mathcal{S}_{-mk,\alpha ,\tau }^{k}$ at the level $j=-M=-mk$.

Now comes the crucial point. We continue the inductive construction of the
new slabs $\mathcal{S}_{j,\alpha ,\tau }^{k+1}$ using \emph{exactly the same
assignments} as were used in the construction in (\ref{grid3}), but with the
original building blocks $\mathcal{S}_{-mk,\alpha ,\tau }^{k}=\widetilde{%
\mathcal{S}}_{-M,\alpha ,\tau }$ in Remark \ref{originalblock}\ replaced by
the new building blocks $\mathcal{S}_{-mk,\alpha ,\tau }^{k+1}$.

In this way we construct a \emph{new} truncated grid $\left\{ \mathcal{S}%
_{j,\alpha ,\tau }^{k+1}\right\} _{j\in m\mathbb{Z}_{+}-M,\left( \alpha
,\tau \right) \in K_{j}}$ that satisfies the properties in (\ref{grid3})
with the same constant $m$, and possibly new constants $C_{1}$ and $C_{2}$.
Moreover we have the crucial property that%
\begin{equation*}
dist_{\rho }\left( \mathcal{S}_{j,\alpha ,\tau }^{k+1},\mathcal{S}_{j,\alpha
,\tau }^{k}\right) \leq C2^{-M}=C2^{-mk},\ \ \ \ \ j\in m\mathbb{Z}_{+}-mk.
\end{equation*}%
Here the distance between sets $\mathcal{E}$ and $\mathcal{F}$ is defined as%
\begin{equation*}
dist_{\rho }\left( \mathcal{E},\mathcal{F}\right) =\sup_{x\in \mathcal{E}%
}\inf_{y\in \mathcal{F}}\rho \left( x,y\right) +\sup_{y\in \mathcal{F}%
}\inf_{x\in \mathcal{E}}\rho \left( x,y\right) .
\end{equation*}%
Continuing in this fashion we construct for each $\ell \in \mathbb{Z}_{+}$ a
truncated grid%
\begin{equation*}
\left\{ \mathcal{S}_{j,\alpha ,\tau }^{k+\ell }\right\} _{j\in m\mathbb{Z}%
_{+}-m\left( k+\ell \right) ,\left( \alpha ,\tau \right) \in K_{j}}
\end{equation*}%
satisfying the properties in (\ref{grid3}) uniformly in $\ell $, as well as
the inequality, 
\begin{equation}
dist_{\rho }\left( \mathcal{S}_{j,\alpha ,\tau }^{k+\ell +1},\mathcal{S}%
_{j,\alpha ,\tau }^{k+\ell }\right) \leq C2^{-m\left( k+\ell \right) },\ \ \
\ \ j\in m\mathbb{Z}_{+}-m\left( k+\ell \right) .  \label{dist}
\end{equation}

For convenience we take $k=0$ in (\ref{dist}). In the special case that $%
\rho $ is comparable to a metric, we conclude from (\ref{dist}) and the
triangle inequality that for $0\leq \ell \leq q$, we have 
\begin{equation}
dist_{\rho }\left( \mathcal{S}_{j,\alpha ,\tau }^{\ell },\mathcal{S}%
_{j,\alpha ,\tau }^{q}\right) \leq C2^{-m\ell },\ \ \ \ \ j\in m\mathbb{Z}%
_{+}-m\ell .  \label{stab}
\end{equation}%
This shows in particular that%
\begin{equation*}
B_{\rho }\left( c_{j,\alpha ,\tau },C_{1}2^{j}\right) \subset \mathcal{S}%
_{j,\alpha ,\tau }^{\ell }\subset B_{\rho }\left( c_{j,\alpha ,\tau
},C_{2}2^{j}\right) ,\ \ \ \ \ \ell \geq -j,
\end{equation*}%
perhaps with new constants $C_{1}$ and $C_{2}$. We can now let $\ell
\rightarrow \infty $ and define%
\begin{equation*}
\mathcal{S}_{j,\alpha ,\tau }^{\ast }=\lim \inf_{\ell \rightarrow \infty }%
\mathcal{S}_{j,\alpha ,\tau }^{\ell }\equiv \cup _{L=0}^{\infty }\cap _{\ell
=L}^{\infty }\mathcal{S}_{j,\alpha ,\tau }^{\ell }.
\end{equation*}

We have for each $j\in \mathbb{Z}$,%
\begin{eqnarray}
\mathcal{S}_{j,\alpha ,\tau }^{\ast }\cap \mathcal{S}_{j,\alpha ^{\prime
},\tau ^{\prime }}^{\ast } &=&\phi ,\ \ \ \ \ \left( \alpha ,\tau \right)
\neq \left( \alpha ^{\prime },\tau ^{\prime }\right) ,  \label{clearly} \\
\mathbb{H}^{n} &=&\cup _{\left( \alpha ,\tau \right) \in K_{j}}\overline{%
\mathcal{S}_{j,\alpha ,\tau }^{\ast }}.  \notag
\end{eqnarray}%
Indeed, the first line in (\ref{clearly}) is obvious. To see the second line
in (\ref{clearly}), suppose in order to derive a contradiction that there is 
$x\in \mathbb{H}^{n}\setminus \cup _{\left( \alpha ,\tau \right) \in K_{j}}%
\overline{\mathcal{S}_{j,\alpha ,\tau }^{\ast }}$. Since the sets $\overline{%
\mathcal{S}_{j,\alpha ,\tau }^{\ast }}$ have finite overlap, $\cup _{\left(
\alpha ,\tau \right) \in K_{j}}\overline{\mathcal{S}_{j,\alpha ,\tau }^{\ast
}}$ is closed and there is an open Euclidean ball $B\left( x,r\right) $
contained in $\mathbb{H}^{n}\setminus \cup _{\left( \alpha ,\tau \right) \in
K_{j}}\overline{\mathcal{S}_{j,\alpha ,\tau }^{\ast }}$. There are at most a
finite number of indices $\left( \alpha ,\tau \right) \in K_{j}$ such that
the slab $\mathcal{S}_{j,\alpha ,\tau }^{\ell }$ could have nonempty
intersection with $B\left( x,r\right) $. Since by (\ref{stab}) the slabs
stabilize as $\ell \rightarrow \infty $, there is an open subset $O$ of $%
B\left( x,r\right) $ and an index $\left( \alpha ,\tau \right) \in K_{j}$
such $O\subset \mathcal{S}_{j,\alpha ,\tau }^{\ell }$ for all sufficiently
large $\ell $. Thus $O\subset \mathcal{S}_{j,\alpha ,\tau }^{\ast }$ by
definition and this contradicts our assumption that $B\left( x,r\right) $
has empty intersection with all $\mathcal{S}_{j,\alpha ,\tau }^{\ast }$.

Moreover, we also have the nesting property for the $\mathcal{S}_{j,\alpha
,\tau }^{\ast }$ as well as the following monotonicity:%
\begin{equation*}
\cup _{\left( \alpha ,\tau \right) \in K_{j}}\mathcal{S}_{j,\alpha ,\tau
}^{\ast }\subset \cup _{\left( \alpha ,\tau \right) \in K_{j^{\prime }}}%
\mathcal{S}_{j,\alpha ,\tau }^{\ast }
\end{equation*}%
for $j<j^{\prime }$. Indeed, if $x\in \mathcal{S}_{j,\alpha ,\tau }^{\ast }$%
, then $x\in \mathcal{S}_{j,\alpha ,\tau }^{\ell }$ for all sufficiently
large $\ell $, and since $\left\{ \mathcal{S}_{j,\alpha ,\tau }^{k+\ell
}\right\} $ is a truncated grid, there is $\left( \alpha ^{\prime },\tau
^{\prime }\right) \in K_{j^{\prime }}$ such that $x\in \mathcal{S}%
_{j^{\prime },\alpha ^{\prime },\tau ^{\prime }}^{\ell }$ for all
sufficiently large $\ell $, i.e. $x\in \mathcal{S}_{j^{\prime },\alpha
^{\prime },\tau ^{\prime }}^{\ast }$.

Fix $j\in \mathbb{Z}$. Any point $x$ in%
\begin{equation*}
E_{j}\equiv \mathbb{H}^{n}\setminus \cup _{\left( \alpha ,\tau \right) \in
K_{j}}\mathcal{S}_{j,\alpha ,\tau }^{\ast }
\end{equation*}%
must be in the boundary of some $\mathcal{S}_{j,\alpha ,\tau }^{\ast }$ by (%
\ref{clearly}). It is now easy to inductively attach these points to slabs $%
\mathcal{S}_{j,\alpha ,\tau }^{\ast }$ for which they are already a limit
point, and in such a way that the new sets $\mathcal{S}_{j,\alpha ,\tau
}^{\maltese }$ with the attached points form a full grid $\left\{ \mathcal{S}%
_{j,\alpha ,\tau }^{\maltese }\right\} _{j\in m\mathbb{Z},\left( \alpha
,\tau \right) \in K_{j}}$ for $\mathbb{H}^{n}$:%
\begin{eqnarray}
\mathbb{H}^{n} &=&\overset{\cdot }{\cup }_{\left( \alpha ,\tau \right) \in
K_{j}}\mathcal{S}_{j,\alpha ,\tau }^{\maltese },\ \ \ \ \ j\in m\mathbb{Z},
\label{grid4} \\
B_{\rho }\left( c_{j,\alpha ,\tau },C_{1}2^{j}\right) &\subset &\mathcal{S}%
_{j,\alpha ,\tau }^{\maltese }\subset B_{\rho }\left( c_{j,\alpha ,\tau
},C_{2}2^{j}\right) \ \ \ \ \ j\in m\mathbb{Z},\left( \alpha ,\tau \right)
\in K_{j},  \notag \\
\text{either }\mathcal{S}_{j,\alpha ,\tau }^{\maltese } &\subset &\mathcal{S}%
_{j^{\prime },\alpha ^{\prime },\tau ^{\prime }}^{\maltese }\text{, }%
\mathcal{S}_{j^{\prime },\alpha ^{\prime },\tau ^{\prime }}^{\maltese
}\subset \mathcal{S}_{j,\alpha ,\tau }^{\maltese }\text{ or }\mathcal{S}%
_{j,\alpha ,\tau }^{\maltese }\cap \mathcal{S}_{j^{\prime },\alpha ^{\prime
},\tau ^{\prime }}^{\maltese }=\phi .  \notag
\end{eqnarray}

Indeed, the second line in (\ref{grid4}) will pose no problem for the points
in $E_{j}$. Let $j=0$ and write the indices in $K_{0}$ as a sequence $%
\left\{ \left( \alpha _{i},\tau _{i}\right) \right\} _{i=1}^{\infty }$. Now
attach all points in $E$ that are permissible boundary points of $\mathcal{S}%
_{0,\alpha _{1},\tau _{1}}^{\ast }$ to the slab $\mathcal{S}_{0,\alpha
_{1},\tau _{1}}^{\ast }$ to obtain $\mathcal{S}_{0,\alpha _{1},\tau
_{1}}^{\maltese }$. Here we say that a boundary point $x$ is \emph{%
permissible} for $\mathcal{S}_{0,\alpha _{1},\tau _{1}}^{\ast }$ if whenever 
$x$ lies in some larger slab $\mathcal{S}_{j^{\prime },\alpha ^{\prime
},\tau ^{\prime }}^{\ast }$, then $\mathcal{S}_{0,\alpha _{1},\tau
_{1}}^{\ast }$ also lies in the larger slab $\mathcal{S}_{j^{\prime },\alpha
^{\prime },\tau ^{\prime }}^{\ast }$. Next attach all points remaining in $E$
that are permissible boundary points of $\mathcal{S}_{0,\alpha _{2},\tau
_{2}}^{\ast }$ to the slab $\mathcal{S}_{0,\alpha _{2},\tau _{2}}^{\ast }$
to obtain $\mathcal{S}_{0,\alpha _{2},\tau _{2}}^{\maltese }$. Continue in
this way to define all the slabs $\mathcal{S}_{0,\alpha ,\tau }^{\maltese }$
with $\left( \alpha ,\tau \right) \in K_{0}$. This process exhausts the set $%
E_{0}$ of extra points at scale $j=0$. For $j\geq 1$, each slab $\mathcal{S}%
_{j,\alpha ,\tau }^{\ast }$ is a union of certain of the slabs $\mathcal{S}%
_{0,\alpha ^{\prime },\tau ^{\prime }}^{\ast }$, and we now define $\mathcal{%
S}_{j,\alpha ,\tau }^{\maltese }$ to be the union of the corresponding new
slabs $\mathcal{S}_{0,\alpha ^{\prime },\tau ^{\prime }}^{\maltese }$.

Fix $j=-1$. We must attach the points in $E_{-1}$ to the slabs $\mathcal{S}%
_{-1,\alpha ,\tau }^{\ast }$ in such a way that the nesting property holds.
Write the indices in $K_{-1}$ as a sequence $\left\{ \left( \alpha _{i},\tau
_{i}\right) \right\} _{i=1}^{\infty }$. The slab $\mathcal{S}_{-1,\alpha
_{1},\tau _{1}}^{\ast }$ is contained in some slab $\mathcal{S}_{0,\alpha
,\tau }^{\maltese }$. We attach all points in $E_{-1}$ that lie in $\mathcal{%
S}_{0,\alpha ,\tau }^{\maltese }$ and are boundary points of $\mathcal{S}%
_{-1,\alpha _{1},\tau _{1}}^{\ast }$ to the slab $\mathcal{S}_{-1,\alpha
_{1},\tau _{1}}^{\ast }$ to obtain $\mathcal{S}_{-1,\alpha _{1},\tau
_{1}}^{\maltese }$. The slab $\mathcal{S}_{-1,\alpha _{2},\tau _{2}}^{\ast }$
is contained in some slab $\mathcal{S}_{0,\alpha ^{\prime },\tau ^{\prime
}}^{\maltese }$. We attach all points remaining in $E_{-1}$ that lie in $%
\mathcal{S}_{0,\alpha ^{\prime },\tau ^{\prime }}^{\maltese }$\ and are
boundary points of $\mathcal{S}_{-1,\alpha _{2},\tau _{2}}^{\ast }$ to the
slab $\mathcal{S}_{-1,\alpha _{2},\tau _{2}}^{\ast }$ to obtain $\mathcal{S}%
_{-1,\alpha _{2},\tau _{2}}^{\maltese }$. We continue in this way to define
all the slabs $\mathcal{S}_{-1,\alpha ,\tau }^{\maltese }$ with $\left(
\alpha ,\tau \right) \in K_{-1}$. Now repeat the process with $j=-2,-3,...$
to complete the construction of a grid $\left\{ \mathcal{S}_{j,\alpha ,\tau
}^{\maltese }\right\} _{j\in m\mathbb{Z},\left( \alpha ,\tau \right) \in
K_{j}}$ for $\mathbb{H}^{n}$ satisfying (\ref{grid4}). We finally observe
that our construction has preserved the property that $P_{\mathbb{C}^{n}}%
\mathcal{S}_{j,\alpha ,\tau }^{\maltese }=I_{\alpha }^{j}$ for all $j,\alpha
,\tau $. We can thus use the sets $\left\{ \mathcal{S}_{j,\alpha ,\tau
}^{\maltese }\right\} _{j\in m\mathbb{Z},\left( \alpha ,\tau \right) \in
K_{j}}$ in the conclusion (\ref{productgrid}) of the theorem.

\section{Rectangles in the Heisenberg group}

Recall from Theorem \ref{Heisenberg grid} that at each dyadic scale $2^{j}$
with $j\in m\mathbb{Z}$ there is a pairwise disjoint decomposition of $%
\mathbb{H}^{n}$ into sets $\mathcal{S}_{j,\alpha ,\tau }$ that are "almost
Heisenberg ball" of radius $2^{j}$. We will refer to these sets as dyadic
cubes at scale $2^{j}$. These decompositions are nested, and moreover are 
\emph{product-like} in the sense that the cubes $\mathcal{S}_{j,\alpha ,\tau
}$ project onto $I_{\alpha }^{j}$ in the usual dyadic grid in the factor $%
\mathbb{C}^{n}$, and have centers $c_{j,\alpha ,\tau }=\left( P_{j,\alpha
},\tau +\frac{1}{2}2^{2j}\right) $ that for each $j$ form a product set
indexed by $K_{j}\equiv 2^{j}\mathbb{Z}^{2n}\times 2^{2j}\mathbb{Z}$ and
satisfy%
\begin{equation*}
\left\vert c_{j,\alpha ,\tau }-c_{j,\alpha ^{\prime },\tau }\right\vert
=2^{j}\text{ and }\left\vert c_{j,\alpha ,\tau }-c_{j,\alpha ,\tau ^{\prime
}}\right\vert =2^{2j},
\end{equation*}%
if $\alpha $ and $\alpha ^{\prime }$ are neighbours in $2^{j}\mathbb{Z}^{2n}$%
, and if $\tau $ and $\tau ^{\prime }$ are neighbours in $2^{2j}\mathbb{Z}$.

We now define vertical and horizontal dyadic rectangles relative to this
decomposition into dyadic cubes. The analogy with dyadic rectangles in the
plane $\mathbb{R}^{2}$ that we are pursuing here is that a dyadic rectangle $%
I=I_{1}\times I_{2}$ in the plane is vertical if $\left\vert
I_{2}\right\vert \geq \left\vert I_{1}\right\vert $, and is horizontal if $%
\left\vert I_{1}\right\vert \geq \left\vert I_{2}\right\vert $ (and both if
and only if $I$ is a dyadic square). If we consider the\ grid of dyadic
cubes $\left\{ \mathcal{S}_{j,\alpha ,\tau }\right\} $ in $\mathbb{H}^{n}$
in place of the grid of dyadic squares in $\mathbb{R}^{2}$, we are led to
the following definition.

\begin{definition}
\label{Hrect}Let $j,k\in m\mathbb{Z}$ with $j\leq k$ and let $\mathcal{S}%
_{j,\alpha ,\tau }$ and $\mathcal{S}_{k,\beta ,\upsilon }$ be dyadic cubes
in $\mathbb{H}^{n}$ with $\mathcal{S}_{j,\alpha ,\tau }\subset \mathcal{S}%
_{k,\beta ,\upsilon }$. The set%
\begin{equation*}
\mathcal{R}\left( ver\right) =\mathcal{R}_{\mathcal{S}_{j,\alpha ,\tau }}^{%
\mathcal{S}_{k,\beta ,\upsilon }}\left( ver\right) =\dbigcup \left\{ 
\mathcal{S}_{j,\alpha ,\tau ^{\prime }}:\mathcal{S}_{j,\alpha ,\tau ^{\prime
}}\subset \mathcal{S}_{k,\beta ,\upsilon }\right\}
\end{equation*}%
will be referred to as a \emph{vertical dyadic rectangle} or more precisely
the vertical dyadic rectangle in $\mathcal{S}_{k,\beta ,\upsilon }$\
containing $\mathcal{S}_{j,\alpha ,\tau }$. We define the \emph{base} of the
rectangle $\mathcal{R}\left( ver\right) $ to be the dyadic cube $I_{\alpha
}^{j}$ in $\mathbb{C}^{n}$ and we define the \emph{cobase} of the rectangle $%
\mathcal{R}\left( ver\right) $ to be the dyadic interval $J_{\upsilon }^{2k}$
in $\mathbb{R}$. We say the rectangle $\mathcal{R}\left( ver\right) $ has 
\emph{width} $2^{j}$ and \emph{height} $2^{2k}$. Similarly, the set%
\begin{equation*}
\mathcal{R}\left( hor\right) =\mathcal{R}_{\mathcal{S}_{j,\alpha ,\tau }}^{%
\mathcal{S}_{k,\beta ,\upsilon }}\left( hor\right) =\dbigcup \left\{ 
\mathcal{S}_{j,\alpha ^{\prime },\tau }:\mathcal{S}_{j,\alpha ^{\prime
},\tau }\subset \mathcal{S}_{k,\beta ,\upsilon }\right\}
\end{equation*}%
will be referred to as a \emph{horizontal dyadic rectangle} or more
precisely the horizontal dyadic rectangle in $\mathcal{S}_{k,\beta ,\upsilon
}$\ containing $\mathcal{S}_{j,\alpha ,\tau }$. We define the \emph{base} of
the rectangle $\mathcal{R}\left( hor\right) $ to be the dyadic cube $%
I_{\beta }^{k}$ in $\mathbb{C}^{n}$ and we define the \emph{cobase} of the
rectangle $\mathcal{R}\left( ver\right) $ to be the dyadic interval $J_{\tau
}^{2j}$ in $\mathbb{R}$. We say the rectangle $\mathcal{R}\left( hor\right) $
has \emph{width} $2^{k}$ and \emph{height} $2^{2j}$.
\end{definition}

We will usually write just $\mathcal{R}$ to denote a dyadic rectangle that
is either vertical or horizontal. Note that a dyadic rectangle $\mathcal{R}$
is both vertical and horizontal if and only if $\mathcal{R}$ is a dyadic
cube $\mathcal{S}_{j,\alpha ,\tau }$. Finally note that $\mathcal{R}_{%
\mathcal{S}_{j,\alpha ,\tau }}^{\mathcal{S}_{k,\beta ,\upsilon }}\left(
ver\right) $ can be thought of as a Heisenberg substitute for the Euclidean
rectangle $I_{\alpha }^{j}\times J_{\upsilon }^{2k}$ in $\mathbb{H}^{n}$
with width $2^{j}$ and height $2^{2k}$, and that $\mathcal{R}_{\mathcal{S}%
_{j,\alpha ,\tau }}^{\mathcal{S}_{k,\beta ,\upsilon }}\left( hor\right) $
can be thought of as a Heisenberg substitute for the Euclidean rectangle $%
I_{\beta }^{k}\times J_{\tau }^{2j}$ in $\mathbb{H}^{n}$ with width $2^{k}$
and height $2^{2j}$. The vertical Heisenberg rectangles are constructed by
stacking Heisenberg cubes neatly on top of each other, while the horizontal
Heisenberg rectangles are constructed by placing Heisenberg cubes next to
each other, although the placement is far from neat.

\begin{remark}
In applications to operators with flag kernels, or more generally a
semiproduct structure, it is appropriate to restrict attention to the set of 
\emph{vertical} dyadic rectangles.
\end{remark}

\subsection{The dyadic strong maximal function}

We define the strong \emph{dyadic} maximal function $M$ relative to these
dyadic rectangles in the usual way:%
\begin{equation*}
Mf\left( \zeta \right) =\sup_{\zeta \in \mathcal{R}}\frac{1}{\left\vert 
\mathcal{R}\right\vert }\int_{\mathcal{R}}\left\vert f\right\vert ,\ \ \ \ \
f\in L^{1}\left( \mathbb{H}^{n}\right) ,
\end{equation*}%
where the supremum is taken over all dyadic rectangles $\mathcal{R}$
containing $\zeta $. We then have the following strong maximal theorem.

\begin{theorem}
\label{maximal theorem}For $1<p<\infty $, we have%
\begin{equation}
\left\Vert Mf\right\Vert _{L^{p}\left( \mathbb{H}^{n}\right) }\leq
C_{n,p}\left\Vert f\right\Vert _{L^{p}\left( \mathbb{H}^{n}\right) },\ \ \ \
\ f\in L^{p}\left( \mathbb{H}^{n}\right) .  \label{max ineq}
\end{equation}
\end{theorem}

\begin{proof}
While $M$ in the form given here is not obviously a \emph{product} maximal
operator, it turns out that it can be dominated by an iteration of three
one-dimensional maximal operators in distinct variables (\cite{RiSo}), from
which (\ref{max ineq}) follows immediately. Alternatively, one can use the
more general strong maximal theorem in Christ (\cite{Ch}), whose proof is of
consequently more complicated. Here however, we can easily approximate an
iteration by simpler maximal operators as follows. For each $k\in m\mathbb{Z}
$ (we will eventually let $k\rightarrow -\infty $), we consider the \emph{%
truncated} maximal operator%
\begin{equation*}
M_{k}f\left( \zeta \right) =\sup_{\zeta \in \mathcal{R}:\ height\left( 
\mathcal{R}\right) \geq 2^{2k}}\frac{1}{\left\vert \mathcal{R}\right\vert }%
\int_{\mathcal{R}}\left\vert f\right\vert ,\ \ \ \ \ f\in L^{1}\left( 
\mathbb{H}^{n}\right) ,
\end{equation*}%
where the rectangles have height at least $2^{2k}$, as well as the \emph{%
level} maximal operator%
\begin{equation*}
\widetilde{M_{k}}f\left( \zeta \right) =\sup_{\zeta \in \mathcal{R}:\
height\left( \mathcal{R}\right) =2^{2k}}\frac{1}{\left\vert \mathcal{R}%
\right\vert }\int_{\mathcal{R}}\left\vert f\right\vert ,\ \ \ \ \ f\in
L^{1}\left( \mathbb{H}^{n}\right) ,
\end{equation*}%
where the rectangles have height exactly $2^{2k}$. We claim that if $%
\widehat{M}$ denotes the\ one-dimensional maximal operator in the variable $%
t $ where $\zeta =\left( z,t\right) \in \mathbb{C}^{n}\times \mathbb{R}$, we
have the pointwise estimate%
\begin{equation*}
M_{k}f\left( z,t\right) \leq C_{n}\widetilde{M_{k}}\left( \widehat{M}%
f\right) \left( z,t\right) .
\end{equation*}%
Here the level maximal operator $\widetilde{M_{k}}$ is playing the role of
an approximate maximal operator in the variable $z\in \mathbb{C}^{n}$.
Indeed, if $\mathcal{R}=\mathcal{R}_{\mathcal{S}_{j^{\prime },\alpha
^{\prime },\tau ^{\prime }}}^{\mathcal{S}_{j,\alpha ,\tau }}\left(
ver\right) $ contains $\left( z,t\right) $ and if $\widetilde{\mathcal{R}}$
is the unique dyadic rectangle with base $I_{\alpha ^{\prime }}^{j^{\prime
}} $ and height $2^{2k}$ that contains $\left( z,t\right) $, then 
\begin{eqnarray*}
\frac{1}{\left\vert \mathcal{R}\right\vert }\int_{\mathcal{R}}\left\vert
f\right\vert &=&\frac{1}{\left\vert I_{\alpha ^{\prime }}^{j^{\prime
}}\right\vert }\frac{1}{\left\vert J_{\tau }^{j}\right\vert }\int_{\mathcal{R%
}}\left\vert f\right\vert \leq C_{n}\frac{1}{\left\vert I_{\alpha ^{\prime
}}^{j^{\prime }}\right\vert }\frac{1}{\left\vert J_{\tau }^{j}\right\vert }%
\int_{I_{\alpha ^{\prime }}^{j^{\prime }}}\int_{\left\{ s:\left( w,s\right)
\in \mathcal{R}\right\} }\left\vert f\left( w,s\right) \right\vert dsdw \\
&\leq &C_{n}\frac{1}{\left\vert I_{\alpha ^{\prime }}^{j^{\prime
}}\right\vert }\int_{I_{\alpha ^{\prime }}^{j^{\prime }}}\left\{ \frac{1}{%
\left\vert \left\{ s:\left( w,s\right) \in \mathcal{R}\right\} \right\vert }%
\int_{\left\{ s:\left( w,s\right) \in \mathcal{R}\right\} }\left\vert
f\left( w,s\right) \right\vert ds\right\} dw \\
&\leq &C_{n}\frac{1}{\left\vert I_{\alpha ^{\prime }}^{j^{\prime
}}\right\vert }\int_{I_{\alpha ^{\prime }}^{j^{\prime }}}\left\{
\inf_{r:\left( w,r\right) \in \widetilde{\mathcal{R}}}\widehat{M}f\left(
w,r\right) \right\} dw\leq C_{n}\widetilde{M_{k}}\left( \widehat{M}f\right)
\left( z,t\right) .
\end{eqnarray*}

Now the level maximal operator $\widetilde{M_{k}}$ is trivially of weak type 
$\left( 1,1\right) $ since any collection of dyadic rectangles of \emph{fixed%
} height form a grid with the nesting property. By interpolation we obtain
that%
\begin{equation*}
\left\Vert \widetilde{M_{k}}g\right\Vert _{L^{p}\left( \mathbb{H}^{n}\right)
}\leq C_{n,p}\left\Vert g\right\Vert _{L^{p}\left( \mathbb{H}^{n}\right) },\
\ \ \ \ g\in L^{p}\left( \mathbb{H}^{n}\right) ,
\end{equation*}%
with a constant $C_{n,p}$ independent of $k$. Since the maximal operator $%
\widehat{M}$ is bounded on $L^{p}\left( \mathbb{R}\right) $, we conclude that%
\begin{eqnarray*}
\left\Vert M_{k}f\right\Vert _{L^{p}\left( \mathbb{H}^{n}\right) }^{p} &\leq
&\left\Vert C_{n}\widetilde{M_{k}}\left( \widehat{M}f\right) \right\Vert
_{L^{p}\left( \mathbb{H}^{n}\right) }^{p}\leq C_{n}^{p}C_{n,p}^{p}\left\Vert 
\widehat{M}f\right\Vert _{L^{p}\left( \mathbb{H}^{n}\right) }^{p} \\
&=&C_{n,p}\int_{\mathbb{C}^{n}}\left\{ \int_{\mathbb{R}}\left\vert \widehat{M%
}f\left( z,t\right) \right\vert ^{p}dt\right\} dz \\
&\leq &C_{n,p}\int_{\mathbb{C}^{n}}\int_{\mathbb{R}}\left\vert f\left(
z,t\right) \right\vert ^{p}dtdz=C_{n,p}\left\Vert f\right\Vert _{L^{p}\left( 
\mathbb{H}^{n}\right) }^{p},
\end{eqnarray*}%
with a constant $C_{n,p}$ independent of $k$. Now let $k\rightarrow -\infty $
and use the dominated convergence theorem to obtain (\ref{max ineq}).
\end{proof}

\subsection{Journ\'{e}'s covering lemma}

Using the Maximal Theorem \ref{maximal theorem}, we can obtain an analogue
of Journ\'{e}'s covering lemma for the Heisenberg group. Let $\Omega $ be an
open set in $\mathbb{H}^{n}$ and define%
\begin{equation*}
\Omega ^{\left( 1\right) }=\left\{ M\chi _{\Omega }>\frac{1}{2}\right\} 
\text{ and }\Omega ^{\left( 2\right) }=\left\{ M\chi _{\Omega ^{\left(
1\right) }}>\frac{1}{2}\right\} ,
\end{equation*}%
where $M$ is the strong dyadic maximal function on $\mathbb{H}^{n}$. Given a
rectangle $\mathcal{R}$ in $\Omega \,\ $we define the blowup rectangle $%
\widehat{\mathcal{R}}$ relative to $\Omega $ to be the following rectangle
contained in $\Omega ^{\left( 2\right) }$. Suppose that $\mathcal{R}$ has
cobase $J\subset \mathbb{R}$. First we define the rectangle $\widetilde{%
\mathcal{R}}$ to be the largest rectangle in $\Omega ^{\left( 1\right) }$
containing $\mathcal{R}$ and having cobase $J$. Let $\widetilde{\mathcal{R}}$
have base $I\subset \mathbb{C}^{n}$. Then we define the rectangle $\widehat{%
\mathcal{R}}$ to be the largest rectangle in $\Omega ^{\left( 2\right) }$
containing $\widetilde{\mathcal{R}}$ and having base $I$.

\begin{lemma}
Let $\Omega $ be an open set in $\mathbb{H}^{n}$, and for each rectangle $%
\mathcal{R}\subset \Omega $ let $\widehat{\mathcal{R}}\subset \Omega
^{\left( 2\right) }$ be defined as above. Then%
\begin{equation*}
\left\vert \dbigcup\limits_{\mathcal{R}}\widehat{\mathcal{R}}\right\vert
\leq C_{n}\left\vert \Omega \right\vert ,
\end{equation*}%
and for every $0<\varepsilon <1$, there is a positive constant $%
C_{n,\varepsilon }$ such that 
\begin{equation*}
\sum_{\mathcal{R}\subset \Omega }\left\vert \mathcal{R}\right\vert \left( 
\frac{\left\vert \mathcal{R}\right\vert }{\left\vert \widehat{\mathcal{R}}%
\right\vert }\right) ^{\varepsilon }\leq C_{n,\varepsilon }\left\vert \Omega
\right\vert .
\end{equation*}
\end{lemma}

\begin{proof}
Given that we have the Maximal Theorem \ref{maximal theorem} at our
disposal, the proof is an easy generalization of that in Journ\'{e} \cite{J2}%
.
\end{proof}

\end{document}